\documentclass[10pt,reqno]{amsart}
\usepackage[margin=1in]{geometry}

\usepackage{amsmath,amssymb,amsthm,bbm,bm,caption,comment,dsfont,enumerate,graphics,graphicx,mathrsfs,mathtools,multicol,pgfplots,subcaption,tikz,xcolor} 
\usepackage[nosort]{cite}
\usepackage[bookmarks=true,unicode,hidelinks]{hyperref} 

\pgfplotsset{compat=1.18}

\usetikzlibrary{3d,decorations,decorations.pathmorphing,decorations.pathreplacing,decorations.markings,patterns,shapes}

\numberwithin{equation}{section}

\newtheorem{thm}{Theorem}[section]
\newtheorem{lm}[thm]{Lemma}
\newtheorem{defn}[thm]{Definition}
\newtheorem{prop}[thm]{Proposition}

\newtheorem{cor}[thm]{Corollary}

\newtheorem{conj}[thm]{Conjecture}

\newcommand{\bbC}{\mathbb{C}}
\newcommand{\bbE}{\mathbb{E}}

\newcommand{\bbR}{\mathbb{R}}
\newcommand{\bbZ}{\mathbb{Z}}
\newcommand{\bfr}{\mathbf{r}}
\newcommand{\cE}{\mathcal{E}}
\newcommand{\cG}{\mathcal{G}}
\newcommand{\cH}{\mathcal{H}}
\newcommand{\cL}{\mathcal{L}}
\newcommand{\confH}{\mathbb{H}}
\newcommand{\cR}{\mathcal{R}}

\newcommand{\dom}{\Omega}
\newcommand{\HPT}{\mathsf{H}} 
\newcommand{\intZ}{\mathbb{Z}}
\newcommand{\ntau}{{\tau^*}}
\newcommand{\rch}{\mathrm{ch}}
\newcommand{\pch}{\mathrm{pch}}
\newcommand{\mdpch}{\mathrm{pch}^{\mathrm{mod}}}
\newcommand{\ra}{\mathrm{a}}

\newcommand{\rd}{\mathrm{d}}
\newcommand{\re}{e}
\newcommand{\ri}{\mathrm{i}}
\newcommand{\rI}{\mathrm{I}}
\newcommand{\rL}{\mathrm{L}}
\newcommand{\rp}{\mathrm{p}}

\newcommand{\rR}{\mathrm{R}}

\newcommand{\rz}{\mathrm{z}}
\newcommand{\rC}{\mathrm{C}}
\newcommand{\rD}{\mathrm{D}}
\newcommand{\rG}{\mathrm{G}}
\newcommand{\rK}{\mathrm{K}}

\newcommand{\bz}{\mathbf{z}}
\newcommand{\h}{\mathfrak{h}}
\newcommand{\B}{\mathbf{B}}

\newcommand{\bfone}{\mathbf{1}}
\newcommand{\prob}{\mathbb{P}}
\renewcommand{\Re}{\mathrm{Re}}

\newcommand{\pchlim}{\mathcal{X}} 
\newcommand{\pchlimline}{\mathcal{X}^0} 

\newcommand{\uc}{\mathrm{UC}}
\newcommand{\btau}{\mathcal{T}}
\newcommand{\bntau}{\mathcal{T}^*}
\newcommand{\TT}{\tau}
\newcommand{\XX}{\alpha}
\newcommand{\HH}{\beta}
\newcommand{\polylog}{\mathrm{Li}}
\newcommand{\I}{\mathrm{I}}
\newcommand{\rE}{\mathrm{E}}
\newcommand{\xib}{\bm{\xi}}
\newcommand{\etab}{\bm{\eta}}
\newcommand{\hftn}{\varphi} 
\newcommand{\sfx}{\mathsf{x}}
\newcommand{\sfy}{\mathsf{y}}
\newcommand{\sfz}{\mathsf{z}}
\newcommand{\sfu}{\mathsf{u}}
\newcommand{\sfS}{\mathsf{S}}
\newcommand{\sfT}{\mathsf{T}}

\newcommand{\Ke}{K^{\mathrm{en}}}
\newcommand{\limKe}{\mathrm{K}^{\mathrm{en}}}

\newcommand{\Cau}{\mathrm{ChDt}} 
\newcommand{\Sp}{\mathscr{S}}
\newcommand{\sx}{\mathsf{x}}
\newcommand{\sxx}{\mathsf{X}}

\newcommand{\bsigma}{\bm{\sigma}}
\newcommand{\sfG}{\mathsf{g}}

\DeclareMathOperator{\Pconf}{Pconf_{N,L}}

\newcommand{\mcG}{\mathfrak{G}}

\newcommand{\rhosr}{\rho_{L,0}}

\newcommand{\mpch}{\pch}

\DeclareMathOperator{\Fscaled}{F}
\DeclareMathOperator{\Pconfno}{Pconf}

\newcommand{\ncE}{\cE^{\mathrm{norm}}}
\newcommand{\npch}{\pch^{\mathrm{norm}}}

\newcommand{\bmz}{\bm{z}}
\newcommand{\Mx}{M}

\newcommand{\sfTp}{\sfT^+}

\title{Periodic KPZ fixed point with general initial conditions}

\author{Jinho Baik}
\address{Department of Mathematics, University of Michigan, Ann Arbor, MI 48109, USA}
\email{baik@umich.edu}

\author{Yuchen Liao}
\address{School of Mathematical Sciences, University of Science and Technology of China, Hefei, Anhui 230026, P.R. China}
\email{ycliao@ustc.edu.cn}

\author{Zhipeng Liu}
\address{Department of Mathematics, University of Kansas, Lawrence, KS 66045, USA}
\email{zhipengliumath@gmail.com}

\thanks{The work of J.B. was supported in part by NSF grant DMS-2246790. The work of Z.L. was supported in part by NSF grant DMS-2246683. The authors would like to thank Jhih-Huang Li, Konstantin Matetski, Leonid Petrov, Daniel Remenik, and Tejaswi Tripathi for helpful discussions and communications.}

\date{\today}

\begin{document}

\begin{abstract}
We consider the relaxation-time-scale limit of the periodic totally asymmetric simple exclusion process (PTASEP) with general initial conditions. For every sequence of initial conditions approximating a periodic upper semicontinuous function, we compute the limiting space-time multipoint distributions of the rescaled particle locations and height functions. The resulting finite-dimensional distributions are explicit and form a consistent family, thereby defining a spatially periodic space-time random field. We call this field the periodic KPZ fixed point with the corresponding initial condition. This extends earlier results for PTASEP with special initial conditions and defines the periodic analogue of the KPZ fixed point on the line.

The main technical novelty is a pair of new probabilistic representations for the energy function and the characteristic function, the two functions through which the initial condition enters the finite-time PTASEP multipoint distribution formula. Both representations are expressed in terms of a geometric random walk and two stopping times, namely the first hitting time of the initial profile and the first such hitting time at or after one full period, with the latter capturing the periodic geometry.
\end{abstract}

\maketitle

\tableofcontents

\section{Introduction and main results}

The Kardar--Parisi--Zhang (KPZ) universality class is a broad class of models in mathematical physics, including interacting particle systems, random interface growth, directed random polymers, and stochastic partial differential equations; see, e.g.,~\cite{Kardar-Parisi-Zhang86, Baik-Deift-Johansson99, Johansson00, Amir-Corwin-Quastel11, Seppalainen12, Corwin-OConnel-Seppalainen-Zygouras14} and the references therein. 
It is widely believed, and partly proved, that the height functions of these models share the characteristic $1:2:3$ scaling, namely, at large time $t$, height fluctuations are of order $t^{1/3}$ and spatial correlations are of order $t^{2/3}$. 
Under this scaling, the height functions are expected to converge to a universal space-time random field, the KPZ fixed point. 
The KPZ fixed point was constructed in~\cite{Matetski-Quastel-Remenik21}. 
See also~\cite{Matetski-Remenik23a, Matetski-Remenik23, BLSZ23} for modifications and generalizations. 
Convergence to the KPZ fixed point has been established for several models~\cite{Matetski-Quastel-Remenik21, Nica-Quastel-Remenik20, Virag20, Matetski-Quastel-Remenik25, Wu23, Aggarwal-Corwin-Hegde24, Dauvergne-Zhang24}.

This paper concerns KPZ models on a ring, or equivalently spatially periodic models.
Let $L$ be the period and $t$ the time, and consider the regime in which both $L$ and $t$ tend to infinity.
Since spatial correlations in KPZ models scale as $t^{2/3}$, the transition occurs when $t$ is of order $L^{3/2}$.
When $t\ll L^{3/2}$, the effect of periodicity is negligible, and the system behaves asymptotically as its infinite-line counterpart.
When $t\gg L^{3/2}$, spatial correlations become trivial, and the height function evolves asymptotically as Brownian motion.
The scale $t\sim L^{3/2}$ is called the \emph{relaxation time scale}.
At this scale, the height function is expected to converge to a nontrivial crossover field, which we call the \emph{periodic KPZ fixed point}.
This field is conjectured to interpolate, in suitable limiting regimes, between the KPZ fixed point and one-dimensional Brownian motion.

Periodic KPZ models have been studied in several physics papers~\cite{Gwa-Spohn92, Derrida-Lebowitz98, Priezzhev2003, Golinelli-Mallick04, Golinelli-Mallick05, Brankov-Papoyan-Poghosyan-Priezzhev06}, in particular by Prolhac~\cite{Prolhac16, Prolhac20} on the relaxation time scale. 
Rigorous results have been obtained in~\cite{Baik-Liu16, Baik-Liu18, Liu16, Baik-Liu19, Liu-Saenz-Wang20, Baik-Liu21, Liao22, Dunlap-Gu-Komorowski23, Li-Saenz25, Baik-Liu2024, Corwin-Gu-Sorensen26}. 
In particular, \cite{Baik-Liu19, Baik-Liu21} computed the multipoint limiting distributions of the periodic totally asymmetric simple exclusion process (PTASEP) on the relaxation time scale for several special initial conditions, thereby defining the periodic KPZ fixed point in those cases.\footnote{The multipoint limits for the periodic narrow-wedge initial condition were computed in~\cite{Baik-Liu19}, and their consistency was verified in~\cite[Appendix]{Baik-Liu2024}. The same argument applies to the other special initial conditions considered in~\cite{Baik-Liu21}.}

In this paper, we extend~\cite{Baik-Liu19, Baik-Liu21} and~\cite[Appendix]{Baik-Liu2024} to general initial conditions. 
We compute the limiting multipoint distributions of the PTASEP for every initial condition approximating a periodic upper semicontinuous function and prove that these finite-dimensional distributions are consistent. 
Together, these results define, for each periodic upper semicontinuous initial condition, a periodic KPZ fixed point.

\subsection{Relaxation-time-scale limit of PTASEP}\label{sec:result}

The \emph{periodic totally asymmetric simple exclusion process} (PTASEP) can be described either by particle positions or by its height-function representation. 
We state the result using both descriptions.

Let $L>N$ be positive integers, where $L$ is the period and $N$ is the number of particles per period. 
A PTASEP configuration is an element of
\begin{equation}\label{eq:def_conf}
	\Pconf := \{ Y=(y_i)_{i\in \bbZ} : y_i\in \bbZ,  y_{i+N}= y_i-L \text{ for all } i\in\bbZ,\ 
	y_N<\cdots<y_1 \}. 
\end{equation} 
The particle locations at time $t$ are denoted by $(\sx_i(t))_{i\in\bbZ}$, with labels chosen so that
\begin{equation*}
    \sx_{i+1}(t)<\sx_i(t),\qquad i\in\bbZ .
\end{equation*} 
The particles evolve according to the usual TASEP dynamics, subject to the periodicity constraint. The periodicity forces 
\begin{equation*}
    	\sx_{i+kN}(t) = \sx_i(t)-kL 
\end{equation*}
for all $i, k \in \bbZ$ and $t\ge 0$. 

\begin{defn}
Let $L>N$ be positive integers, and suppose $Y=(y_i)_{i\in\bbZ}\in\Pconf$. 
We denote by $\mathrm{PTASEP}_Y(L,N)$ the PTASEP of period $L$ with $N$ particles per period and initial condition $\sx_i(0)=y_i$ for all $i\in\bbZ$. 
\end{defn}

We assume that the initial condition approximates a periodic upper semicontinuous function. 

\begin{defn}
Define
\begin{equation*}
    \uc_1
    :=
    \left\{
    \h:\bbR\to\bbR\cup\{-\infty\} :
    \text{$\h$ is upper semicontinuous},\
    \h\not\equiv-\infty,\
    \h(\XX+1)=\h(\XX) \text{ for all }\XX\in\bbR
    \right\}.
\end{equation*}
We equip $\uc_1$ with the topology of local Hausdorff convergence of hypographs. 
The effective domain of $\h\in\uc_1$ is $\{\XX\in\bbR:\h(\XX)>-\infty\}$. 
\end{defn}

See~\cite[Section~3.1]{Matetski-Quastel-Remenik21} for the precise definition of this topology. 
In this topology, $\h_L\to\h$ in $\uc_1$ if and only if, for every $\XX\in\bbR$,
\begin{equation} \label{eq:UC_criterion}
\limsup_{L\to\infty}\h_L(\XX_L)\le\h(\XX)
\quad\text{for every }\XX_L\to\XX,
\qquad
\liminf_{L\to\infty}\h_L(\XX_L)\ge\h(\XX)
\quad\text{for some }\XX_L\to\XX .
\end{equation}

Our first main theorem establishes convergence of the multipoint distributions of the PTASEP for initial conditions approximating an arbitrary function in $\uc_1$. 
The period $L$, the number $N$ of particles per period, and the time $t$ all tend to infinity. The density $N/L$ remains bounded away from $0$ and $1$, and $t$ is of order $L^{3/2}$, the relaxation time scale.

\begin{thm}[Particle position version]
\label{thm:main}
Fix $0<\rho_-<\rho_+<1$. 
Let $N_L$ be a sequence of integers such that
\begin{equation*}
    \rho_L:=\frac{N_L}{L}\in[\rho_-,\rho_+]
\end{equation*} 
for all sufficiently large $L$. 
Let $\h\in\uc_1$. 
For each $L$, let $(\sx^{(L)}_k(t))_{k\in\bbZ}$ be distributed as $\mathrm{PTASEP}_{Y_L}(L,N_L)$, where the initial conditions 
$Y_L=(y_i^{(L)})_{i\in\bbZ}\in\Pconfno_{N_L,L}$ satisfy
\begin{equation}
\label{eq:initial_scaling}
    \h_L\to\h \quad \text{in } \uc_1,
    \qquad
    \h_L(\XX):=
    \frac{-y^{(L)}_{-\XX N_L+1}-1+\XX L}
    {\sqrt{\frac{1-\rho_L}{\rho_L}}\,L^{1/2}},
\end{equation}
as $L\to \infty$. Here $y_a^{(L)}$ for non-integer $a$ is defined by linear interpolation between 
$y_{\lfloor a\rfloor}^{(L)}$ and $y_{\lfloor a\rfloor+1}^{(L)}$. 
Fix $m\ge1$, points $(\XX_1,\TT_1),\ldots,(\XX_m,\TT_m)\in\bbR\times\bbR_{+}$\footnote{Here, $\bbR_+=(0,\infty)$.}, and 
$\bm{\HH}=(\HH_1,\ldots,\HH_m)\in\bbR^m$. 
Set
\begin{equation}
\label{eq:scaling_1}
    t_i:=\TT_i\frac{L^{3/2}}{\sqrt{\rho_L(1-\rho_L)}},
    \qquad
    k_i:=\bigl\lfloor \rho_L^2 t_i-\XX_i N_L \bigr\rfloor,
    \qquad 1\le i\le m.
\end{equation}
Then
\begin{equation}
\label{eq:lim_bbF_particle}
    \lim_{L\to\infty}
    \prob\left(
    \bigcap_{i=1}^m
    \left\{
    \frac{-\sx^{(L)}_{k_i}(t_i)+\XX_i L+(1-2\rho_L)t_i}
    {\sqrt{\frac{1-\rho_L}{\rho_L}}\,L^{1/2}}
    \le \HH_i
    \right\}
    \right)
    =
    \mathbb{F}_\h^{(m)}
    \bigl(\bm{\HH};(\XX_i,\TT_i)_{i=1}^m\bigr),
\end{equation}
where $\mathbb{F}_\h^{(m)}$ is defined in Definition~\ref{def:F}.
\end{thm}

\medskip

This limit theorem can also be stated in terms of the associated height function. 
Define for $L>N$ 
\begin{equation} \label{eq:confHdef}
    \confH(N,L):= \{H:\bbZ\to\bbZ : |H(x+1)-H(x)|=1
    \text{ and } H(x+L)-H(x)=L-2N \text{ for all }x\in\bbZ \}.
\end{equation}
The \emph{periodic corner growth model} of type $(N,L)$ is the Markov process $\HPT(\cdot,t)$ on $\confH(N,L)$ defined as follows. 
To each site $\ell\in\{1,\ldots,L\}$ we attach an independent rate-one Poisson clock. 
When the clock at $\ell$ rings, if $\ell$ is a local minimum of the current height function, then the heights at all sites in $\ell+L\bbZ$ are raised by $2$; otherwise, nothing happens. 
We extend $H\in\confH(N,L)$ to a function on $\bbR$ by linear interpolation.

For TASEP on the line, \cite{Matetski-Quastel-Remenik21} showed that convergence of the rescaled particle-location profile in the upper semicontinuous topology is equivalent to the convergence of the corresponding rescaled height function. 
Since this equivalence is local, the same argument applies in the periodic setting and yields the following consequence of Theorem~\ref{thm:main}. See Section~\ref{sec:particletoheight} for the proof.

\begin{thm}[Height-function version] 
\label{thm:mainheight}
Fix $0<\rho_-<\rho_+<1$. 
Let $N_L$ be a sequence of integers such that 
$\rho_L:=N_L/L \in[\rho_-,\rho_+]$ 
for all sufficiently large $L$.  Let $\h\in\uc_1$. 
For each $L$, let $\HPT_{L}$ be the periodic corner growth model of type $(N_L,L)$ with initial profile $\HPT_{L}(\cdot,0)\in \confH(N_L,L)$.  Assume that
\begin{equation}
\label{eq:initial_height_convergence}
    \frac{\HPT_{L}(\XX L,0)-(L-2N_L)\XX}
    {-2\sqrt{\rho_L(1-\rho_L)}\,L^{1/2}}
    \longrightarrow \h(\XX)
    \qquad \text{in $\uc_1$}
\end{equation}
as $L\to\infty$. Fix $m\ge1$, points $(\XX_1,\TT_1),\ldots,(\XX_m,\TT_m)\in\bbR\times\bbR_{+}$, and 
$\bm{\HH}=(\HH_1,\ldots,\HH_m)\in\bbR^m$. 
Set
\begin{equation*}
    t_i:=\frac{\TT_i L^{3/2}}{\sqrt{\rho_L(1-\rho_L)}},
    \qquad x_i:=\XX_i L+(1-2\rho_L)t_i,
    \qquad 1\le i\le m.
\end{equation*}
Then
\begin{equation}  \label{eq:lim_bbF_height}
    \lim_{L\to\infty} \prob\left( \bigcap_{i=1}^m \left\{
    \frac{ \HPT_{L}(x_i,t_i) -(1-2\rho_L)x_i -2\rho_L(1-\rho_L)t_i}{-2\sqrt{\rho_L(1-\rho_L)}\,L^{1/2}}
    \le \HH_i \right\} \right)
    = \mathbb{F}_\h^{(m)} \bigl(\bm{\HH};(\XX_i,\TT_i)_{i=1}^m\bigr).
\end{equation}
\end{thm}

The conclusions of Theorems~\ref{thm:main} and~\ref{thm:mainheight} were previously established for several special initial conditions. The periodic step initial condition $y_{i+kN_L}^{(L)}=-i-kL$ for $1\le i\le N_L$ and $k\in\bbZ$, which approximates the periodic narrow-wedge profile $\h_{\mathrm{pnw}}(\XX):=-\infty\mathbf{1}_{\XX\notin\bbZ}$, was addressed in~\cite{Baik-Liu19}. The flat initial condition $y_k^{(L)}=-\lfloor kL/N_L\rfloor$ for $k\in\bbZ$, which approximates $\h_{\mathrm{flat}}\equiv0$, was treated in~\cite{Baik-Liu21}; the periodic step-flat initial condition was considered there as well.

These earlier results, as well as the present work, are based on the finite-time PTASEP multipoint distribution formulas obtained in~\cite{Baik-Liu19,Baik-Liu21}. For the special initial conditions above, these formulas simplify sufficiently for direct asymptotic analysis. The paper~\cite{Baik-Liu21} also considered general initial conditions and proved convergence under certain technical assumptions, but these assumptions are restrictive and difficult to verify beyond the special cases.

In this paper, we prove Theorems~\ref{thm:main} and~\ref{thm:mainheight} for every $\h\in\uc_1$. The main new ingredient is a pair of probabilistic representations for the two initial-condition-dependent quantities in the finite-time formula. These representations are expressed in terms of geometric-random-walk hitting times and are suitable for asymptotic analysis. The limiting distribution functions $\mathbb{F}_\h^{(m)}$ are given explicitly by contour integrals involving Fredholm determinants and Brownian hitting-time terms. In Section~\ref{sec:examples}, we verify that our formulas recover the results of~\cite{Baik-Liu19,Baik-Liu21} for the periodic narrow-wedge and flat initial conditions.

\subsection{Periodic KPZ fixed point}\label{sec:PKPZ}

Theorems~\ref{thm:main} and~\ref{thm:mainheight} identify, for each $m\ge1$ and $\h\in\uc_1$, a function $\mathbb{F}_\h^{(m)}$ as the limiting multipoint distribution of the rescaled PTASEP height function. 
The next theorem records several properties of these functions.

\begin{thm} \label{thm:main2}
Let $\h\in\uc_1$, and for each $m\ge1$ let $\mathbb{F}_\h^{(m)}$ be the function from Definition~\ref{def:F}. 
Then the family $\{\mathbb{F}_\h^{(m)}\}_{m\ge1}$ satisfies the following properties.
\begin{enumerate}[(a)]
\item (Periodicity) 
$\mathbb{F}_\h^{(m)}\bigl(\bm{\HH};(\XX_i,\TT_i)_{i=1}^m\bigr)$ is unchanged if any one of the variables $\XX_k$ is replaced by $\XX_k+1$.

\item (Joint CDF) 
If $(\XX_1,\TT_1),\ldots,(\XX_m,\TT_m)$ are distinct, then 
$\mathbb{F}_\h^{(m)}\bigl(\bm{\HH};(\XX_i,\TT_i)_{i=1}^m\bigr)$ is a continuous joint distribution function in $\bm{\HH}\in\bbR^m$.

\item (Marginal compatibility) 
For $m\ge2$ and $1\le k\le m$,
\begin{equation*} 
	\lim_{\HH_k\to\infty}
	\mathbb{F}_\h^{(m)}\bigl(\bm{\HH};(\XX_i,\TT_i)_{i=1}^m\bigr)
	=
	\mathbb{F}_\h^{(m-1)}
	\bigl((\HH_i)_{i\ne k};\,(\XX_i,\TT_i)_{i\ne k}\bigr).
\end{equation*}   

\item (Symmetry) 
For every permutation $\sigma$ of $\{1,\ldots,m\}$,
\begin{equation*} 
	\mathbb{F}_\h^{(m)}
	\bigl((\HH_{\sigma(i)})_{i=1}^m;\,(\XX_{\sigma(i)},\TT_{\sigma(i)})_{i=1}^m\bigr)
	=
	\mathbb{F}_\h^{(m)}
	\bigl(\bm{\HH};(\XX_i,\TT_i)_{i=1}^m\bigr).
\end{equation*} 

\item (Joint continuity) 
$\mathbb{F}_\h^{(m)}$ is jointly continuous in 
$\bigl((\XX_i,\TT_i,\HH_i)\bigr)_{i=1}^m\in(\bbR\times\bbR_{+}\times\bbR)^m$.

\item (Initial condition) 
As $\TT_1,\ldots,\TT_m\to0$, 
\begin{equation}\label{eq:limit_tau0}
	\mathbb{F}_\h^{(m)}
	\bigl(\bm{\HH};(\XX_i,\TT_i)_{i=1}^m\bigr)
	\longrightarrow
    \begin{dcases}
        0, & \text{ if } \HH_i < \h(\XX_i) \text { for some }i, \\
        1, & \text{ if } \HH_i> \h(\XX_i) \text{ for all }i. 
    \end{dcases}
\end{equation}
\item (Shift invariance) 
If $\h'(\XX)=\h(\XX-a)+b$ for some $a,b\in\bbR$, then
\begin{equation}\label{eq:translation_invariance}
	\mathbb{F}_{\h'}^{(m)}
	\bigl(\bm{\HH}+b(1,\ldots,1);\,(\XX_i+a,\TT_i)_{i=1}^m\bigr)
	=
	\mathbb{F}_\h^{(m)}
	\bigl(\bm{\HH};(\XX_i,\TT_i)_{i=1}^m\bigr).
\end{equation}

\item (Monotonicity) 
If $\h,\h'\in\uc_1$ satisfy $\h(\XX)\le \h'(\XX)$ for all $\XX\in\bbR$, then
\begin{equation*}
	\mathbb{F}_{\h'}^{(m)}
	\bigl(\bm{\HH};(\XX_i,\TT_i)_{i=1}^m\bigr)
	\le
	\mathbb{F}_\h^{(m)}
	\bigl(\bm{\HH};(\XX_i,\TT_i)_{i=1}^m\bigr).
\end{equation*}
\end{enumerate}                
\end{thm} 

Properties (a)--(d) imply that $\{\mathbb{F}_\h^{(m)}\}_{m\ge1}$ is a consistent spatially periodic family of finite-dimensional distributions. 
By the Kolmogorov extension theorem, there exists a random field on $(\bbR/\bbZ)\times\bbR_{+}$ with these finite-dimensional distributions. 
We lift it periodically to $\bbR\times\bbR_{+}$ and call the resulting field the \emph{periodic KPZ fixed point with initial condition $\h$}. 
We denote it by
\begin{equation*}
	\cH^{\mathrm{PKPZ}}(\,\cdot\,,\,\cdot\,;\h)
	:=
	\bigl\{\cH^{\mathrm{PKPZ}}(\XX,\TT;\h):
	(\XX,\TT)\in\bbR\times\bbR_{+}\bigr\}.
\end{equation*}
The remaining properties (e)--(h) of Theorem~\ref{thm:main2} translate directly into the corresponding properties of $\cH^{\mathrm{PKPZ}}$. 

\medskip

The periodic KPZ fixed point is conjectured to interpolate between the KPZ fixed point on the line and Brownian motion. 
In our coordinates, the limits $\TT\to0$ and $\TT\to\infty$ correspond respectively to the sub-relaxation regime $t\ll L^{3/2}$ and the super-relaxation regime $t\gg L^{3/2}$. 
In these regimes, the periodic KPZ fixed point is expected to converge to the KPZ fixed point and to one-dimensional Brownian motion, respectively.
It is convenient to express these limits as $\rp\to\infty$ and $\rp\to0$ for a $\rp$-periodic version of the field. 
For $\rp>0$, define
\begin{equation*}
    \uc_\rp:=\{\h:\bbR\to\bbR\cup\{-\infty\} :
    \text{$\h$ is upper semicontinuous},\ 
    \h\not\equiv-\infty,\ 
    \h(\XX+\rp)=\h(\XX) \text{ for all }\XX\in\bbR \}.
\end{equation*}
For $\h\in\uc_\rp$, define the \emph{$\rp$-periodic KPZ fixed point} by 
\begin{equation} \label{eq:scaling_invariance}
     	\cH^{\mathrm{PKPZ}}_\rp(\XX,\TT;\h)
     	:=
     	\rp^{1/2}\,
     	\cH^{\mathrm{PKPZ}}
     	\bigl(\rp^{-1}\XX,\rp^{-3/2}\TT;\h^\star\bigr),
\end{equation}
where $\h^\star(\XX):=\rp^{-1/2}\h(\rp\XX)\in\uc_1$. 
Based on formal sub-relaxation and super-relaxation asymptotics, we make the following conjecture.

\begin{conj}\label{conj:largesmallPKPZ}
\begin{enumerate}[(a)]
\item (Large-period limit) 
Let $\h$ be an initial condition for the KPZ fixed point on the line, bounded above by a linear function. 
For $\rp>0$, define its $\rp$-periodic cutoff by
\begin{equation} \label{eq:def_hcutoff}
	\h^{\mathrm{cutoff}}_\rp(\XX)
	:=
	\h\bigl(\XX-\rp\lfloor \XX/\rp+1/2\rfloor\bigr),
\end{equation}
with boundary values understood in the upper-semicontinuous sense. 
Then, in the sense of finite-dimensional distributions,\footnote{The factor of $2$ in $2\alpha$ reflects our choice of spatial units, consistent with~\cite{Baik-Liu18,Baik-Liu19,Baik-Liu21}.}
\begin{equation}\label{eq:conjecture}
     \cH^{\mathrm{PKPZ}}_\rp(2\XX,\TT;\h^{\mathrm{cutoff}}_\rp)
     \xrightarrow[\rp\to\infty]{d}
     \cH^{\mathrm{KPZ}}(\XX,\TT;\h).
\end{equation}

\item (Small-period limit) 
For every family $\{\h_\rp\}_{\rp>0}$ with $\h_\rp\in\uc_\rp$ satisfying
\begin{equation*}
	\lim_{\rp\to0} \, 
	\rp^{1/4} \sup_{\XX:\,\h_\rp(\XX)\neq-\infty} \h_\rp(\XX)
	=0,
\end{equation*}
we have, in the sense of finite-dimensional distributions,
\begin{equation} \label{eq:conjecture_small_period} 
	\frac{\sqrt{2}\,\rp^{1/4}}{\pi^{1/4}}
	\left(
	    \cH^{\mathrm{PKPZ}}_\rp(\XX,\TT;\h_\rp)
	    +\rp^{-1}\TT
	\right)
	\xrightarrow[\rp\to0]{d}
	\B(\TT),
\end{equation}
where $\B$ is a standard Brownian motion.
\end{enumerate}
\end{conj}

For the periodic narrow wedge and flat initial conditions, the one-point large-period limit at $\XX=0$ and the one-point small-period limit were proved in~\cite{Baik-Liu-Silva22}.
The small-period result extends to general initial conditions $\h_\rp$.
The proof of the following proposition is given in Section~\ref{sec:proofofsmalltconj}.

\begin{prop}\label{prop:onepoint_b}
Conjecture~\ref{conj:largesmallPKPZ}\textnormal{(b)} holds for one-point distributions.
\end{prop}

The KPZ fixed point on the line admits a variational representation~\cite{Nica-Quastel-Remenik20}
in terms of the directed landscape~\cite{Dauvergne-Ortmann-Virag22}. It is therefore natural to
ask whether there exists a periodic directed landscape $\mathcal{L}^{\mathrm{per}}$ that gives an
analogous representation of the periodic KPZ fixed point:
\begin{equation}\label{eq:varipdls}
	\cH^{\mathrm{PKPZ}}(\XX,\TT;\h)
	\stackrel{d}{=}
	\max_{\XX'\in\bbR}
	\left\{ \h(\XX') + \mathcal{L}^{\mathrm{per}}(\XX',0;\XX,\TT) \right\}.
\end{equation}
After the present work was completed, \cite{Aggarwal-Corwin-Schmid26} constructed the periodic directed landscape by gluing copies of the full-space directed landscape, defined the periodic KPZ fixed point through the variational formula~\eqref{eq:varipdls}, and proved that it is equal in distribution to our periodic KPZ fixed point for any given initial condition. 
The variational representation provides a coupling of the periodic KPZ fixed point across multiple initial data.
The same paper also proved the large-period
limit~\eqref{eq:conjecture} of Conjecture~\ref{conj:largesmallPKPZ}\textnormal{(a)}.

\subsection{Proof strategy and organization of the paper}

Our starting point is the multipoint distribution formula for the PTASEP derived in~\cite{Baik-Liu21}.
The initial condition $Y$ enters this formula through the \emph{PTASEP energy function} and the \emph{modified PTASEP characteristic function}, which are defined in Definition~\ref{def:characteristic}.
The main novelty of this paper is a probabilistic representation of each function in terms of a geometric random walk and two stopping times $\tau$ and $\tau^*$.
The stopping time $\tau$ is the first hitting time of the initial profile, while $\tau^*$ is the first hitting time at or after one full period.
These representations are given in Theorems~\ref{thm:characteristic} and~\ref{thm:energy}.
On the relaxation time scale, the geometric random walks converge to Brownian motions, and the stopping times converge to their Brownian analogues.

The representation of the modified PTASEP characteristic function adapts ideas from~\cite{Matetski-Quastel-Remenik21,Liao-Liu25} for the KPZ fixed point on the line, with the new stopping time $\tau^*$ encoding the periodic geometry. 
The PTASEP energy function, by contrast, has no direct analogue on the line. 
We first express it as a Fredholm determinant, and then represent its kernel using transition probabilities of the geometric random walk, with one term constrained by the stopping time $\tau$. 
This representation is the main technical step of the proof.

The analysis developed in this paper can be adapted, with only minor modifications, to some other periodic KPZ models with a determinantal structure, including discrete-time TASEP, Brownian last-passage percolation, and the polynuclear growth model. In each case, the characteristic and energy functions admit analogous probabilistic representations involving model-specific random walks. These models will be studied in a separate paper.

\medskip

The paper is organized as follows. Section~\ref{sec:PKPZ_formula} defines the limiting functions $\mathbb{F}_\h^{(m)}$, and Section~\ref{sec:propertymulti} establishes their well-definedness and several structural properties. Section~\ref{sec:PTASEP_formula} reviews the PTASEP multipoint distribution formula from~\cite{Baik-Liu21}. Section~\ref{sec:proof_PTASEP}, the main technical part of the paper, derives probabilistic representations of the PTASEP energy function and the modified PTASEP characteristic function. 

The asymptotic analysis is carried out in Sections~\ref{sec:asymptotics} and~\ref{sec:asymptotics_general}. Section~\ref{sec:asymptotics} treats generic time-ordered parameters under the additional assumption that $y_1^{(L)}=-1$ for all $L$, and Section~\ref{sec:asymptotics_general} removes this assumption using shift invariance and continuity of the one-point distribution. This completes the proof of Theorem~\ref{thm:main}. Section~\ref{sec:particletoheight} then derives Theorem~\ref{thm:mainheight} from Theorem~\ref{thm:main}.

The remaining properties of Theorem~\ref{thm:main2} are proved in Section~\ref{sec:proof_thm12}. Section~\ref{sec:proofofsmalltconj} proves Proposition~\ref{prop:onepoint_b}, the one-point small-period limit. Section~\ref{sec:examples} evaluates the hitting-time formulas for the periodic narrow-wedge and flat initial conditions and verifies agreement with~\cite{Baik-Liu19,Baik-Liu21}. Finally, Section~\ref{sec:additional} collects additional formulas and structural properties of the quantities appearing in the limiting multipoint distributions defined in Section~\ref{sec:PKPZ_formula}.


\section{Multipoint distribution formulas for the periodic KPZ fixed point}
\label{sec:PKPZ_formula}

In this section, we define the limiting functions  appearing in the main theorems. 
In Section~\ref{sec:efdpcf}, for each periodic upper semicontinuous function $\h$, we introduce two Brownian hitting times and define the periodic characteristic function $\pchlim_\h$ and the operator $\limKe_{\h}(\rz)$; these are the two initial-condition-dependent ingredients. 
In Section~\ref{sec:def_rCh}, we use them to define two auxiliary functions $\rC_\h(\bz)$ and $\rD_\h(\bz)$. 
In Section~\ref{sec:limitdistr}, we define $\mathbb{F}_\h^{(m)}$ as a contour integral of $\rC_\h(\bz)\rD_\h(\bz)$. 
The well-definedness and structural properties of these quantities are established in Section~\ref{sec:propertymulti}. 
Some additional formulas and properties are collected in Section~\ref{sec:additional}.

\subsection{Brownian hitting times} 
\label{sec:efdpcf}

\begin{defn} 
\label{def:hittingtimes}
Let $\B$ be a Brownian motion with diffusion coefficient $1$.\footnote{Throughout this paper, all Brownian motions have diffusion coefficient $1$. This differs from~\cite{Matetski-Quastel-Remenik21,Liao-Liu25}, where the diffusion coefficient is $2$. The discrepancy comes from different choices of spatial units in the periodic and infinite-line settings, following the convention of~\cite{Baik-Liu19,Baik-Liu21}; see also the footnote to \eqref{eq:conjecture}.}
For every upper semicontinuous function $\h:\bbR_{\le0}\to\bbR\cup\{-\infty\}$, define
\begin{equation*}
	\btau:=\inf\{\sfx\ge0:\B(\sfx)\le \h(-\sfx)\},
	\qquad
	\bntau:=\inf\{\sfx\ge1:\B(\sfx)\le \h(-\sfx)\}.
\end{equation*}
\end{defn}

These are the hitting times of the hypograph of the function $\sfx\mapsto\h(-\sfx)$. 
The negative argument reflects the right-to-left convention for TASEP particle labels. 
By a space-reversal property, the quantities below can also be written in terms of the hitting times using $\h(\sfx)$ instead; see Section~\ref{sec:reversal}.

In Definition~\ref{def:hittingtimes}, the dependence on $\h$ is suppressed from the notation; it will always be clear from context.

\begin{defn}\label{def:pch_lim} 
For $\h\in\uc_1$, $\arg(\eta)\in(-\frac{\pi}{4},\frac{\pi}{4})$, and $\Re(\xi)<0$, define the \emph{periodic characteristic function} associated with $\h$ by
\begin{equation}\label{eq:pch_lim}
	\pchlim_\h(\eta,\xi)
	:=
	\int_\bbR \rd s\, e^{-s\xi}
	\bbE_{\B(0)=s}\left[
	    e^{\eta\B(\btau)-\frac12\eta^2\btau}\mathbf{1}_{\btau<1}
	    -
	    e^{\eta\B(\bntau)-\frac12\eta^2\bntau}
	    \mathbf{1}_{\bntau<\infty}\mathbf{1}_{\btau<1}
	\right].
\end{equation}
\end{defn}

Lemma~\ref{prop:plim_bound} shows that $\pchlim_\h$ is well-defined, and Section~\ref{sec:examples} evaluates it explicitly for the periodic narrow-wedge and flat initial conditions. For nonperiodic $\h$, an analogous function was introduced in~\cite[Definition~2.2]{Liao-Liu25} and used to describe the KPZ fixed point on the line. The new feature in the periodic setting is the stopping time $\bntau$, which records the first hit at or after one full period.

The second initial-condition-dependent quantity uses the normalized subspace
\begin{equation*}
	\uc_1^0:=\{\h\in\uc_1:\h(0)=0\}.
\end{equation*}

\begin{defn} \label{def:Koperdef}
For $\h\in\uc_1^0$ and $|\rz|<1$, define the operator 
$\limKe_{\h}(\rz):L^2(\bbR)\to L^2(\bbR)$ by\footnote{Here and below, $\mathbf{1}_{>\beta}$ denotes the multiplication operator by the indicator function $\mathbf{1}_{x>\beta}$.}
\begin{equation}\label{eq:limkernel}
        \limKe_{\h}(\rz)
        :=
        \rz\,\mathbf{1}_{>0}(\rI-\rz\mathsf{G})^{-1}\sfT_{\h}\mathbf{1}_{>0},
\end{equation}
where $\mathsf{G}$ and $\sfT_{\h}$ have kernels
\begin{equation}\label{eq:limkernel2TT}
        \mathsf{G}(\sfx,\sfy)
        =
        \frac{1}{\sqrt{2\pi}}e^{-\frac{(\sfx-\sfy)^2}{2}},
        \qquad 
        \sfT_{\h}(\sfx,\sfy)
        :=
        \frac{
        \prob_{\B(0)=\sfx}\bigl(\B(1)\in\rd\sfy,\ \btau<1\bigr)
        }{\rd\sfy},
\end{equation}
for $\sfx,\sfy\in\bbR$.
\end{defn}

Proposition~\ref{result:KasQ} shows that $\limKe_\h(\rz)$ is trace class and gives a factorization of its Fredholm determinant relative to the periodic narrow-wedge case. The resolvent of the heat-kernel operator $\mathsf G$ admits the expansion
\begin{equation}\label{eq:limkernel2SS}
	(\rI-\rz\mathsf G)^{-1}
	=
	\rI+\sfS_\rz,
	\qquad
	\sfS_\rz(\sfx,\sfy)
	:=
	\sum_{k\geq1}
	\frac{\rz^k}{\sqrt{2\pi k}}
	e^{-\frac{(\sfx-\sfy)^2}{2k}}.
\end{equation}
See Subsection~\ref{sec:STalt} for alternative expressions for $\sfS_\rz$ and $\sfT_\h$.

\subsection{The functions $\rC_\h(\bz)$ and $\rD_\h(\bz)$} 
\label{sec:def_rCh}

Let $\polylog_s(\rz)$ denote the polylogarithm of order $s$. For $|\rz|<1$, 
\begin{equation*}
	\polylog_s(\rz)
	:=
	\sum_{k=1}^\infty \frac{\rz^k}{k^s}, 
\end{equation*}
and set 
\begin{equation}\label{eq:Afunction}
	A_1(\rz):=-\frac{1}{\sqrt{2\pi}}\polylog_{3/2}(\rz),
	\qquad
	A_2(\rz):=-\frac{1}{\sqrt{2\pi}}\polylog_{5/2}(\rz).
\end{equation}
For $0<|\rz|,|\rz'|<1$, define
\begin{equation}\label{eq:def_Bz}
	B(\rz,\rz')
	:=
	\frac{1}{4\pi}
	\sum_{k,k'\ge1}
	\frac{\rz^k(\rz')^{k'}}{(k+k')\sqrt{kk'}}, 
	\qquad 
	B(\rz):=B(\rz,\rz).
\end{equation}

\begin{defn} \label{def:Cfunction2}
Let $m\ge1$, $\HH_1,\ldots,\HH_m\in\bbR$, and 
$\TT_1,\ldots,\TT_m\in\bbR_{+}$. 
For $\bz=(\rz_1,\ldots,\rz_m)$ with $0<|\rz_i|<1$ and 
$\rz_i\ne\rz_{i+1}$ for $1\le i\le m-1$, set $\rz_{m+1}:=0$ and define
\begin{equation}\label{eq:Climnoh}
	\rC(\bz)
	:=
	\left[\prod_{\ell=1}^{m} \frac{\rz_\ell}{\rz_\ell-\rz_{\ell+1}}\right]
	\left[
	\prod_{\ell=1}^{m}
	\frac{
	e^{\HH_\ell A_1(\rz_\ell)+\TT_\ell A_2(\rz_\ell)}
	}{
	e^{\HH_\ell A_1(\rz_{\ell+1})+\TT_\ell A_2(\rz_{\ell+1})}
	}
	\right]
	\left[
	\prod_{\ell=2}^m e^{2B(\rz_\ell)-2B(\rz_\ell,\rz_{\ell-1})}
	\right].
\end{equation}
For $\h\in\uc_1$, choose a point $\ra$ in the effective domain of $\h$, set
\begin{equation*}
    \h_\ra(\sfx):=\h(\sfx+\ra)-\h(\ra),
\end{equation*}
and define
\begin{equation}\label{eq:Clim}
	\rC_\h(\bz):=\rC(\bz)\,\rE_\h(\rz_1),
	\qquad
	\rE_\h(\rz):=
	e^{-\h(\ra)A_1(\rz)}
	\det(\I+\limKe_{\h_\ra}(\rz))_{L^2(\bbR)}.
\end{equation}
\end{defn}

We call $\rE_\h(\rz)$ the \emph{energy function}. Corollary~\ref{lm:translation_rC} in Subsection~\ref{sec:FdetlimKtr} shows that it is independent of the choice of $\ra$. In Section~\ref{sec:examples}, we evaluate the energy function explicitly for the periodic narrow-wedge and flat initial conditions.

We now define $\rD_\h(\bz)$ by its Fredholm series. The corresponding determinant representation is given in Subsection~\ref{sec:Fredholm_D}.
For $|\rz|<1$, define\footnote{This function was denoted by $\mathsf{h}$ in~\cite{Baik-Liu18,Baik-Liu19,Baik-Liu21}. We use $\hftn$ to avoid confusion with the initial condition $\h$.}
\begin{equation}\label{eq:def_h_R}
	\hftn(\zeta,\rz)
	:=
    \begin{cases}
	 -\frac{1}{\sqrt{2\pi}}
	 \int_{-\infty}^{\zeta}
	 \polylog_{1/2}\bigl(\rz e^{(\zeta^2-y^2)/2}\bigr)\,\rd y,
     & \Re(\zeta)<0,\\
	 -\frac{1}{\sqrt{2\pi}}
	 \int_{-\infty}^{-\zeta}
	 \polylog_{1/2}\bigl(\rz e^{(\zeta^2-y^2)/2}\bigr)\,\rd y,
     & \Re(\zeta)>0,
    \end{cases}
\end{equation}
where the integration contour lies in the half-plane $\Re(y)<0$. 
Note that  $\hftn(\zeta,\rz)=O(\zeta^{-1})$ as $|\zeta|\to\infty$ away from the imaginary axis.
For $\rz_1,\ldots,\rz_m$ with $|\rz_i|<1$, define
\begin{equation*}
    \mathrm{g}_\ell(\zeta)
    :=
    e^{2\hftn(\zeta,\rz_\ell)
    -\hftn(\zeta,\rz_{\ell-1})
    -\hftn(\zeta,\rz_{\ell+1})},
    \qquad 1\le \ell\le m,
\end{equation*}
with the convention $\hftn(\zeta,\rz_0)=\hftn(\zeta,\rz_{m+1})=0$.

Let $\HH_1,\ldots,\HH_m,\XX_1,\ldots,\XX_m\in\bbR$ and 
$\TT_1,\ldots,\TT_m\in\bbR_{+}$. 
Set $\TT_0=\XX_0=\HH_0:=0$, and define
\begin{equation}\label{eq:f_lim}
	\mathrm{f}_\ell(\zeta)
	:=
	\begin{cases}
	e^{-\frac13(\TT_\ell-\TT_{\ell-1})\zeta^3
	+\frac12(\XX_\ell-\XX_{\ell-1})\zeta^2
	+(\HH_\ell-\HH_{\ell-1})\zeta},
	& \Re(\zeta)<0,\\[3pt]
	e^{+\frac13(\TT_\ell-\TT_{\ell-1})\zeta^3
	-\frac12(\XX_\ell-\XX_{\ell-1})\zeta^2
	-(\HH_\ell-\HH_{\ell-1})\zeta},
	& \Re(\zeta)>0,
	\end{cases}
\end{equation}
for $1\le \ell\le m$.

For $W=(w_1,\ldots,w_n)\in\bbC^n$ and 
$W'=(w'_1,\ldots,w'_m)\in\bbC^m$, write
\begin{equation*}
	W\sqcup W'
	:=
	(w_1,\ldots,w_n,w'_1,\ldots,w'_m)\in\bbC^{n+m}.
\end{equation*}
When $n=m$, define the Cauchy determinant
\begin{equation} \label{eq:Cauchydt}
	\Cau(W;W')
	:=
	\det\left[\frac{1}{w_i-w'_j}\right]_{1\le i,j\le n}
	=
	(-1)^{\frac{n(n-1)}{2}}
	\frac{ \prod_{1\le i<j\le n}(w_j-w_i)(w'_j-w'_i) }{ \prod_{1\le i,j\le n}(w_i-w'_j) }.
\end{equation}

For $0<|\rz|<1$, define the discrete sets
\begin{equation}
	\rL_\rz
	:=
	\{\zeta\in\bbC:e^{-\zeta^2/2}=\rz,\ \Re(\zeta)<0\},
	\qquad  
	\rR_\rz
	:=
	\{\zeta\in\bbC:e^{-\zeta^2/2}=\rz,\ \Re(\zeta)>0\}.
\end{equation}
See Figure~\ref{fig:root_limit}.

\begin{figure}
\centering
\begin{tikzpicture}[x=0.33cm, y=0.33cm,
	lroot/.style={circle, fill=black, inner sep=0pt, minimum size=3pt},
	rroot/.style={rectangle, fill=black, inner sep=0pt, minimum size=2.66pt}]

\draw[gray!60, thin] (-6.4,0) -- (6.4,0);
\draw[gray!60, thin] (0,-6.4) -- (0,6.4);
\foreach \t in {-6,-4,-2,2,4,6}{
	\draw[gray!60, thin] (\t,-0.12) -- (\t,0.12) node[below=2pt, black!65, font=\tiny] {$\t$};
	\draw[gray!60, thin] (-0.12,\t) -- (0.12,\t) node[left=2pt, black!65, font=\tiny] {$\t$};
}

\begin{scope}
	\clip (-6,-6) rectangle (6,6);
	\draw[thin] plot[domain=-6.2:6.2, samples=121, variable=\t]
		({sqrt(\t*\t + 1.832581)}, {\t});
	\draw[thin] plot[domain=-6.2:6.2, samples=121, variable=\t]
		({-sqrt(\t*\t + 1.832581)}, {\t});
\end{scope}

\foreach \x/\y in {%
	1.5191/-0.6893, 2.4964/2.0974, 2.8816/-2.5438, 3.5316/3.2618,
	3.8158/-3.5676, 4.3292/4.1121, 4.5644/-4.3591, 5.0019/4.8153,
	5.2069/-5.0279, 5.5945/5.4283, 5.7786/-5.6178}{
	\node[rroot] at (\x,\y) {};
	\node[lroot] at (-\x,-\y) {};
}

\node at (-4.5,2.2) {$\rL_\rz$};
\node at (4.5,2.2) {$\rR_\rz$};

\end{tikzpicture}
\caption{The roots of $e^{-\zeta^2/2}=\rz$ for $\rz=0.4e^{\ri\pi/3}$. The roots in $\rL_\rz$ are marked by circles and those in $\rR_\rz$ by squares. The line is the level curve $|e^{-\zeta^2/2}|=|\rz|$.}
\label{fig:root_limit}
\end{figure}

\begin{defn}[Series formula for $\rD_{\h}(\bz)$]
\label{def:D_series}
Let $m\ge1$, $\HH_1,\ldots,\HH_m,\XX_1,\ldots,\XX_m\in\bbR$, and 
$\TT_1,\ldots,\TT_m\in\bbR_{+}$.  
Let $\bz=(\rz_1,\ldots,\rz_m)$ with $0<|\rz_i|<1$. 
For $\mathbf{n}=(n_1,\ldots,n_m)\in\bbZ_{\ge0}^m$, set
\begin{equation*}
	\xib=(\xib^{(1)},\ldots,\xib^{(m)}), \qquad \etab=(\etab^{(1)},\ldots,\etab^{(m)}),
\end{equation*} 
where
$\xib^{(\ell)}=(\xi_1^{(\ell)},\ldots,\xi_{n_\ell}^{(\ell)})$, $\etab^{(\ell)}=(\eta_1^{(\ell)},\ldots,\eta_{n_\ell}^{(\ell)})$ for $1\le \ell\le m$. 
Define\footnote{We use the convention $\xib^{(m+1)}=\etab^{(m+1)}=\varnothing$.} 
\begin{equation}\label{eq:Dn_limnoh}
	\begin{aligned}
	\mathrm{D}^{(\mathbf{n})}(\bz;\xib,\etab)
	&=
	\left[
	\prod_{\ell=1}^{m-1}
	\left(1-\frac{\rz_{\ell+1}}{\rz_\ell}\right)^{n_\ell}
	\left(1-\frac{\rz_\ell}{\rz_{\ell+1}}\right)^{n_{\ell+1}}
	\right] 
	\left[
	\prod_{\ell=1}^m\prod_{i=1}^{n_\ell}
	\frac{
	\mathrm{f}_\ell(\xi_i^{(\ell)})
	\mathrm{f}_\ell(\eta_i^{(\ell)})
	\mathrm{g}_\ell(\xi_i^{(\ell)})
	\mathrm{g}_\ell(\eta_i^{(\ell)})
	}{
	\xi_i^{(\ell)}\eta_i^{(\ell)}
	}
	\right]  \\
	&\quad \times
	\prod_{i=1}^{n_1}
	e^{-\hftn(\xi_i^{(1)},\rz_1)-\hftn(\eta_i^{(1)},\rz_1)}
	\left[
	\prod_{\ell=1}^{m}
	\Cau\bigl(\xib^{(\ell)}\sqcup\etab^{(\ell+1)};
	\etab^{(\ell)}\sqcup\xib^{(\ell+1)}\bigr)
	\right].
    \end{aligned}
\end{equation}
For $\h\in\uc_1$, define\footnote{Throughout the paper, empty products and $0\times0$ determinants are understood to be $1$.}
\begin{equation}\label{eq:Dn_lim}
	\mathrm{D}^{(\mathbf{n})}_{\h}(\bz)
	:=
	\sum_{\substack{ \xib^{(\ell)}\in(\rL_{\rz_\ell})^{n_\ell},\ \etab^{(\ell)}\in(\rR_{\rz_\ell})^{n_\ell}\\ 1\le \ell\le m}}
	\det \left[ \pchlim_{\h}(\eta_i^{(1)},\xi_j^{(1)}) \right]_{i,j=1}^{n_1}
	\mathrm{D}^{(\mathbf{n})}(\bz;\xib,\etab),
\end{equation}
and
\begin{equation}\label{eq:series}
    \rD_{\h}(\bz)
    :=
    \sum_{\mathbf{n}\in\bbZ_{\ge0}^m}
    \frac{(-1)^{n_1+\cdots+n_m}}{(n_1!\cdots n_m!)^2}
    \mathrm{D}^{(\mathbf{n})}_{\h}(\bz).
\end{equation}
\end{defn}

\subsection{Definition of the multipoint distributions}
\label{sec:limitdistr}

We now define the multipoint distributions of the periodic KPZ fixed point. 
For $0<\TT_1\le\cdots\le\TT_m$, set
\begin{equation*}
    \dom^m(\TT_1,\ldots,\TT_m)
    :=
    \{\bm{\HH}\in\bbR^m:\HH_i\le \HH_{i+1}
    \text{ whenever } \TT_i=\TT_{i+1}\}
\end{equation*}
and
\begin{equation*}
    \dom_+^m(\TT_1,\ldots,\TT_m)
    :=
    \{\bm{\HH}\in\bbR^m:\HH_i< \HH_{i+1}
    \text{ whenever } \TT_i=\TT_{i+1}\}.
\end{equation*}

\begin{defn}
\label{def:F}
Let $\h\in\uc_1$ and $m\ge1$, and let $\rC_\h(\bz)$ and $\rD_\h(\bz)$ be the functions from Definitions~\ref{def:Cfunction2} and~\ref{def:D_series}. 
We define 
$\mathbb{F}_\h^{(m)}\bigl(\bm{\HH};(\XX_i,\TT_i)_{i=1}^m\bigr)$ as follows.
\begin{enumerate}[(i)]
\item If $0<\TT_1\le\cdots\le\TT_m$ and 
$\bm{\HH}\in\dom_+^m(\TT_1,\ldots,\TT_m)$, set
\begin{equation}\label{eq:multi_time1}
	\mathbb{F}_\h^{(m)}
	\bigl(\bm{\HH};(\XX_i,\TT_i)_{i=1}^m\bigr)
	:=
	\oint \frac{\rd\rz_1}{2\pi\ri\,\rz_1}\cdots
	\oint \frac{\rd\rz_m}{2\pi\ri\,\rz_m}
	\,\rC_\h(\bz)\rD_\h(\bz),
\end{equation}  
where $\bz=(\rz_1,\ldots,\rz_m)$ and the contours are nested circles satisfying
$0<|\rz_m|<\cdots<|\rz_1|<1$. 

\item If $0<\TT_1\le\cdots\le\TT_m$ and 
$\bm{\HH}\in\dom^m(\TT_1,\ldots,\TT_m)\setminus
\dom_+^m(\TT_1,\ldots,\TT_m)$, set
\begin{equation}\label{eq:multi_time2}
	\mathbb{F}_\h^{(m)}
	\bigl(\bm{\HH};(\XX_i,\TT_i)_{i=1}^m\bigr)
	:=
	\lim_{\substack{\bm{\HH}'\to\bm{\HH}\\
	\bm{\HH}'\in\dom_+^m(\TT_1,\ldots,\TT_m)}}
	\mathbb{F}_\h^{(m)}
	\bigl(\bm{\HH}';(\XX_i,\TT_i)_{i=1}^m\bigr).
\end{equation}  

\item For general $(\XX_1,\TT_1),\ldots,(\XX_m,\TT_m)\in\bbR\times\bbR_{+}$ and 
$\bm{\HH}\in\bbR^m$, choose a permutation $\sigma$ of $\{1,\ldots,m\}$ such that
\begin{equation*}
    \TT_{\sigma(1)}\le\cdots\le\TT_{\sigma(m)}
    \quad\text{and}\quad
    (\HH_{\sigma(i)})_{i=1}^m
    \in\dom^m(\TT_{\sigma(1)},\ldots,\TT_{\sigma(m)}).
\end{equation*}
Set
\begin{equation} \label{eq:multi_time3}
	\mathbb{F}_\h^{(m)}
	\bigl(\bm{\HH};(\XX_i,\TT_i)_{i=1}^m\bigr)
	:=
	\mathbb{F}_\h^{(m)}
	\bigl((\HH_{\sigma(i)})_{i=1}^m;\,
	(\XX_{\sigma(i)},\TT_{\sigma(i)})_{i=1}^m\bigr).
\end{equation}
\end{enumerate}     
\end{defn}

The limit in \eqref{eq:multi_time2} exists. 
An analogous boundary limit was considered for the periodic narrow wedge initial condition in~\cite[Lemma~A.1]{Baik-Liu2024}. The proof does not depend on the initial condition and applies here as well. 
The same argument also shows that the definition in \eqref{eq:multi_time3} is independent of the choice of the permutation $\sigma$.
Several  properties of $\mathbb{F}_\h^{(m)}$ are stated in Theorem~\ref{thm:main2} and proved in Sections~\ref{sec:propertymulti}, \ref{sec:consistency}, and~\ref{sec:proof_thm12}.

\section{Analytic properties of the limiting formulas}
\label{sec:propertymulti}

We establish several structural properties of the limiting quantities entering $\mathbb{F}_\h^{(m)}$: the trace-class property and determinant factorization of $\limKe_\h(\rz)$, independence of the normalization point in the energy function, bounds ensuring that $\rD_\h(\bz)$ is well defined, and shift invariance.

\subsection{Fredholm determinant of $\limKe_\h(\rz)$}
\label{sec:FdetlimKtr}

We record the trace-class property of $\limKe_\h(\rz)$ and a useful factorization of its Fredholm determinant.
In Proposition~\ref{prop:pnw22}, we show that 
$\det(\I+\limKe_{\h_{\mathrm{pnw}}}(\rz))_{L^2(\bbR)}=e^{2B(\rz)}$, where $B(\rz)$ is defined in \eqref{eq:def_Bz}.

\begin{prop}\label{result:KasQ}
Let $\h\in\uc_1^0$. Let $\mathsf G_+$ and $\sfTp_\h$ be the operators on $L^2(\bbR_+)$ with kernels
\begin{equation*}
	\mathsf G_+(\sfx,\sfy):=\mathsf G(\sfx,\sfy),
	\qquad
	\sfTp_\h(\sfx,\sfy):=\sfT_\h(\sfx,\sfy),
	\qquad \sfx,\sfy>0,
\end{equation*}
and let $\I_+$ be the identity operator on $L^2(\bbR_+)$. Then the following hold.
\begin{enumerate}[(a)]
\item The operator $\sfTp_\h$ is trace class on $L^2(\bbR_+)$.

\item For $|\rz|<1$, the operator $\limKe_\h(\rz)$ is trace class on $L^2(\bbR)$.

\item For $|\rz|<1$, the following factorization holds:
\begin{equation}\label{eq:KasQ_factorization}
	\det(\I+\limKe_\h(\rz))_{L^2(\bbR)}
	=
	\det(\I+\limKe_{\h_{\mathrm{pnw}}}(\rz))_{L^2(\bbR)}
	\det\left(\I_++\rz(\I_+-\rz\mathsf G_+)^{-1}\sfTp_\h\right)_{L^2(\bbR_+)},
\end{equation}
where $\h_{\mathrm{pnw}}(\alpha)=-\infty\mathbf{1}_{\alpha\notin\bbZ}$ denotes the periodic narrow-wedge initial condition.
\end{enumerate}
\end{prop}

\begin{proof}
(a) Set
\begin{equation*}
	\sfG_t(\sfx,\sfy):=\frac{1}{\sqrt{2\pi t}}e^{-(\sfx-\sfy)^2/(2t)}.
\end{equation*}
Define
\begin{equation*}
	a(\sfx,\sfy):=\frac{\prob_{\B(0)=\sfx}\bigl(\B(1/2)\in\rd \sfy,\ \btau\le1/2\bigr)}{\rd \sfy},
	\qquad
	q(\sfx,\sfy):=\frac{\prob_{\B(0)=\sfx}\bigl(\B(1/2)\in\rd \sfy,\ \btau>1/2\bigr)}{\rd \sfy}.
\end{equation*}
For a Brownian motion started from $x$ at time $1/2$, set
\begin{equation*}
	\sigma:=\inf\{t\in[1/2,1]:\B(t)\le\h(-t)\},
	\qquad
	r(\sfx,\sfy):=\frac{\prob_{\B(1/2)=\sfx}\bigl(\B(1)\in\rd \sfy,\ \sigma<1\bigr)}{\rd \sfy},
\end{equation*}
with $\inf\varnothing=\infty$. The Markov property at time $1/2$ gives
\begin{equation}\label{eq:G_minus_Q_midpoint}
	\sfTp_\h(\sfx,\sfy)
	=
	\int_\bbR a(\sfx,\sfu)\sfG_{1/2}(\sfu,\sfy)\,\rd \sfu
	+
	\int_\bbR q(\sfx,\sfu)r(\sfu,\sfy)\,\rd \sfu.
\end{equation}
Thus
\begin{equation*}
	\sfTp_\h=A_1B_1+A_2B_2,
\end{equation*}
where $A_1,A_2:L^2(\bbR)\to L^2(\bbR_+)$ and $B_1,B_2:L^2(\bbR_+)\to L^2(\bbR)$ have kernels
\begin{equation*}
	A_1(\sfx,\sfu):=a(\sfx,\sfu)e^{|\sfu|},
	\quad
	B_1(\sfu,\sfy):=e^{-|\sfu|}\sfG_{1/2}(\sfu,\sfy),
	\quad 
	A_2(\sfx,\sfu):=q(\sfx,\sfu)e^{-|\sfu|},
	\quad
	B_2(\sfu,\sfy):=e^{|\sfu|}r(\sfu,\sfy).
\end{equation*}
We show that these four operators are Hilbert--Schmidt.

Since $q(\sfx,\sfu)\le \sfG_{1/2}(\sfx,\sfu)=\sfG_{1/2}(\sfu,\sfx)$,
\begin{equation*}
	\|B_1\|_{\mathrm{HS}}^2,\ \|A_2\|_{\mathrm{HS}}^2
	\le
	\frac{1}{\sqrt{2\pi}}\int_\bbR e^{-2|\sfu|}\,\rd \sfu
	<\infty.
\end{equation*}

Since $a(\sfx,\sfu)\le \sfG_{1/2}(\sfx,\sfu)$ and
\begin{equation*}
	e^{2|\sfu|}\sfG_{1/2}(\sfx,\sfu)\le\frac{e}{\sqrt{\pi}}e^{2\sfx},
	\qquad \sfx>0,
\end{equation*}
we have
\begin{equation*}
	\int_\bbR |A_1(\sfx,\sfu)|^2\,\rd \sfu
	\le
	\frac{e}{\sqrt{\pi}}e^{2\sfx}\prob_{\B(0)=\sfx}(\btau\le 1/2).
\end{equation*}
Since $\h$ is one-periodic and upper semicontinuous, there is $M\ge0$ such that $\h\le M$. The event $\{\btau\le1/2\}$ implies $\min_{0\le t\le1/2}\B(t)\le M$. Hence, by the reflection principle, there are constants $C,c>0$ such that
\begin{equation*}
	\prob_{\B(0)=\sfx}(\btau\le1/2)
	\le
	\mathbf{1}_{\sfx<M+1}
	+
	Ce^{-c(\sfx-M)^2}\mathbf{1}_{\sfx\ge M+1}.
\end{equation*}
Thus $A_1$ is Hilbert--Schmidt.

Finally, since $r(\sfu,\sfy)\le \sfG_{1/2}(\sfu,\sfy)\le\pi^{-1/2}$,
\begin{equation*}
	\int_0^\infty r(\sfu,\sfy)^2\,\rd \sfy
	\le
	\frac{1}{\sqrt{\pi}}\prob_{\B(1/2)=\sfu}\bigl(\B(1)>0,\ \sigma<1\bigr).
\end{equation*}
By the reflection principle and the Gaussian tail bound, there are constants $C,c>0$ such that
\begin{equation*}
	\prob_{\B(1/2)=\sfu}\bigl(\B(1)>0,\ \sigma<1\bigr)
	\le
	Ce^{-c\sfu^2}\mathbf{1}_{\sfu\le-1}
	+
	\mathbf{1}_{-1<\sfu<M+1}
	+
	Ce^{-c(\sfu-M)^2}\mathbf{1}_{\sfu\ge M+1}.
\end{equation*}
It follows that
\begin{equation*}
	\|B_2\|_{\mathrm{HS}}^2
	=
	\int_\bbR e^{2|\sfu|}\left(\int_0^\infty r(\sfu,\sfy)^2\,\rd \sfy\right)\rd \sfu
	<\infty.
\end{equation*}
Therefore, $A_1B_1$ and $A_2B_2$ are trace class, and \eqref{eq:G_minus_Q_midpoint} proves (a).

(b) Let $R:L^2(\bbR)\to L^2(\bbR_+)$ be restriction and let $E:L^2(\bbR_+)\to L^2(\bbR)$ be extension by zero. Since
$\limKe_\h(\rz)=E\bigl(R\limKe_\h(\rz)E\bigr)R$, 
it is enough to prove that $R\limKe_\h(\rz)E$ is trace class on $L^2(\bbR_+)$.

Set 
$\mathsf Q_\h:=\mathsf G_+-\sfTp_\h$. 
Since $\h(0)=\h(-1)=0$, its kernel is
\begin{equation}\label{eq:Qs_definition}
	\mathsf Q_\h(\sfx,\sfy)
	:=
	\frac{\prob_{\B(0)=\sfx}\bigl(\B(1)\in\rd\sfy,\ \B(t)>\h(-t)\text{ for all }t\in[0,1]\bigr)}{\rd\sfy},
	\qquad \sfx,\sfy>0.
\end{equation}
Moreover, from \eqref{eq:limkernel2TT}, 
\begin{equation*}
	\sfT_\h E=\mathsf GE-E\mathsf Q_\h
\end{equation*}
as operators from $L^2(\bbR_+)$ to $L^2(\bbR)$. Hence, using \eqref{eq:limkernel},
\begin{equation} \label{eq:IRKEf}
	\I_++R\limKe_\h(\rz)E
	=
	R(\I-\rz\mathsf G)^{-1} E (\I_+-\rz \mathsf Q_\h).
\end{equation}
We decompose it as 
\begin{equation}\label{eq:RlimKeAB}
	\I_++R\limKe_\h(\rz)E=AB,
\end{equation}
where
\begin{equation*}
	A:=R(\I-\rz\mathsf G)^{-1}E(\I_+-\rz\mathsf G_+),
	\qquad
	B:=(\I_+-\rz\mathsf G_+)^{-1}(\I_+-\rz\mathsf Q_\h).
\end{equation*}

By part (a),
\begin{equation}\label{eq:KeB_rewriting}
	B-\I_+
	=
	\rz(\I_+-\rz\mathsf G_+)^{-1}(\mathsf G_+-\mathsf Q_\h)
	=
	\rz(\I_+-\rz\mathsf G_+)^{-1}\sfTp_\h
\end{equation}
is trace class.

Let $P_-:=\I-ER:L^2(\bbR)\to L^2(\bbR)$ be the projection onto $L^2(( -\infty,0])$. Since $\I_+-\rz\mathsf G_+=R(\I-\rz\mathsf G)E$,
\begin{equation*}
	A-\I_+
	=
	\rz R(\I-\rz\mathsf G)^{-1}P_-\mathsf GE.
\end{equation*}
Using \eqref{eq:limkernel2SS} and $R P_-=0$, we find $R(\I-\rz\mathsf G)^{-1}P_-
	=R \sfS_\rz P_-$, and thus, 
\begin{equation*}
	\|R(\I-\rz\mathsf G)^{-1}P_-\|_{\mathrm{HS}}^2
	=
	\frac{1}{2\pi}\sum_{k,\ell\ge1}\frac{\sqrt{k\ell}\,\rz^k\overline{\rz}^{\ell}}{k+\ell}
	<\infty.
\end{equation*}
Also,
\begin{equation*}
	\|P_-\mathsf GE\|_{\mathrm{HS}}^2
	=
	\int_{-\infty}^0\int_0^\infty\mathsf G(\sfx,\sfy)^2\,\rd\sfy\,\rd\sfx
	=
	\frac{1}{4\pi}.
\end{equation*}
Thus $A-\I_+$ is trace class. It follows from \eqref{eq:RlimKeAB} that $R\limKe_\h(\rz)E$ is trace class, and hence so is $\limKe_\h(\rz)$. 

(c) From \eqref{eq:RlimKeAB}, 
\begin{equation}\label{eq:Kh_AB_det}
	\det(\I+\limKe_\h(\rz))_{L^2(\bbR)}
	=
	\det(\I_++R\limKe_\h(\rz)E)_{L^2(\bbR_+)}
	=
	\det(A)_{L^2(\bbR_+)}\det(B)_{L^2(\bbR_+)}.
\end{equation}
For $\h=\h_{\mathrm{pnw}}$, the no-hit condition in \eqref{eq:Qs_definition} is automatic for paths with endpoints in $\bbR_+$, and therefore $\mathsf Q_{\h_{\mathrm{pnw}}}=\mathsf G_+$. Thus the corresponding operator $B$ is $\I_+$, and
\begin{equation*}
	\det(A)_{L^2(\bbR_+)}
	=
	\det(\I+\limKe_{\h_{\mathrm{pnw}}}(\rz))_{L^2(\bbR)}.
\end{equation*}
Substituting this identity into \eqref{eq:Kh_AB_det} and using \eqref{eq:KeB_rewriting} proves \eqref{eq:KasQ_factorization}.
\end{proof}

\begin{cor}\label{lm:translation_rC}
Let $\h\in\uc_1$. For every $|\rz|<1$, the energy function $\rE_\h(\rz)$ is independent of the choice of $\ra$ in the effective domain of $\h$. 
\end{cor}

\begin{proof}
The claim is immediate for $\rz=0$. Fix $0<|\rz|<1$. Let $\ra$ and $\ra'$ be points in the effective domain of $\h$. By periodicity and symmetry in $\ra,\ra'$, it is enough to assume $0\le\ra<\ra'<1$. 

For $\ra$ in the effective domain of $\h$, define
\begin{equation*}
	\Delta_\ra(\rz)
	:=
	\det\left(\I_++\rz(\I_+-\rz\mathsf G_+)^{-1}\sfTp_{\h_\ra}\right)_{L^2(\bbR_+)}
	=
	\det\left((\I_+-\rz\mathsf G_+)^{-1}(\I_+-\rz\mathsf Q_{\h_\ra})\right)_{L^2(\bbR_+)},
\end{equation*}
where $\mathsf G_+$ and $\sfTp_\h$ are as in Proposition~\ref{result:KasQ}, and
$\mathsf Q_\h:=\mathsf G_+-\sfTp_\h$, whose kernel is given by \eqref{eq:Qs_definition}. Proposition~\ref{result:KasQ}(c) then gives
\begin{equation}\label{eq:K_Delta_a}
	\det(\I+\limKe_{\h_\ra}(\rz))_{L^2(\bbR)}\Delta_{\ra'}(\rz)
	=
	\det(\I+\limKe_{\h_{\ra'}}(\rz))_{L^2(\bbR)}\Delta_\ra(\rz).
\end{equation}

Both $\mathsf G_+$ and $\mathsf Q_{\h_\ra}$ are contractions, and their difference is trace class by Proposition~\ref{result:KasQ}(a). Hence the logarithmic expansion of the relative Fredholm determinant gives
\begin{equation*}
	\log\Delta_\ra(\rz)
	=
	\sum_{n\ge1}\frac{\rz^n}{n}\operatorname{Tr}\left(\mathsf G_+^n-(\mathsf Q_{\h_\ra})^n\right),
\end{equation*}
where the logarithm is chosen to vanish at $\rz=0$. The series is absolutely convergent since
$\|\mathsf G_+^n-(\mathsf Q_{\h_\ra})^n\|_1 \le n\|\mathsf G_+-\mathsf Q_{\h_\ra}\|_1$. 
Therefore,
\begin{equation}\label{eq:Delta_ratio_log}
	\log\frac{\Delta_\ra(\rz)}{\Delta_{\ra'}(\rz)}
	=
	\sum_{n\ge1}\frac{\rz^n}{n}\operatorname{Tr}\left((\mathsf Q_{\h_{\ra'}})^n-(\mathsf Q_{\h_\ra})^n\right).
\end{equation}

By periodicity and the Markov property,
\begin{equation*}
	(\mathsf Q_{\h_\ra})^n(\sfx,\sfy)
	=
	\frac{\prob_{\B(0)=\sfx}\bigl(\B(n)\in\rd\sfy,\ \B(t)>\h_\ra(-t)\text{ for all }t\in[0,n]\bigr)}{\rd\sfy}.
\end{equation*}
Let $\hat{\B}$ be a standard Brownian bridge from $0$ to $0$ over $[0,n]$, and set
\begin{equation*}
	M_\ra(\hat{\B})
	:=
	\sup_{0\le t\le n}\bigl(\h_\ra(-t)-\hat{\B}(t)\bigr).
\end{equation*}
Conditioning on the endpoint gives
\begin{equation}\label{eq:Q_diag_M}
	(\mathsf Q_{\h_\ra})^n(\sfx,\sfx)
	=
	\frac{1}{\sqrt{2\pi n}}\bbE\left[\mathbf{1}_{\sfx>M_\ra(\hat{\B})}\right].
\end{equation}

Note that 
\begin{equation*}
	\h_{\ra'}(-t)=\h_\ra(-t+c)-d, 
    \qquad
    c:=\ra'-\ra,
	\qquad
	d:=\h(\ra')-\h(\ra).
\end{equation*}
Define 
\begin{equation*}
	\gamma(u)
	:=
	\begin{cases}
	\hat{\B}(u+c)-\hat{\B}(c),&0\le u\le n-c,\\
	\hat{\B}(u-n+c)-\hat{\B}(c),&n-c\le u\le n.
	\end{cases}
\end{equation*}
A direct covariance computation shows that $\gamma$ is again a standard Brownian bridge over $[0,n]$. Using the one-periodicity of $\h_\ra$ and the fact that $n$ is an integer, we obtain
\begin{equation*}
	M_{\ra'}(\hat{\B})=M_\ra(\gamma)-d-\hat{\B}(c).
\end{equation*}
Thus, 
\begin{equation}\label{eq:M_expectation_shift}
	\bbE[M_{\ra'}(\hat{\B})]
	=
	\bbE[M_\ra(\hat{\B})]-d.
\end{equation}

By Proposition~\ref{result:KasQ}(a), the operators $\mathsf G_+^n-(\mathsf Q_{\h_\ra})^n$ and $(\mathsf Q_{\h_{\ra'}})^n-(\mathsf Q_{\h_\ra})^n$ are trace class. The Brownian-bridge representation shows that the latter has a continuous kernel. Moreover, $M_\ra(\hat{\B})$ and $M_{\ra'}(\hat{\B})$ have finite first moments because $\h_\ra$ and $\h_{\ra'}$ are bounded above and the extrema of a Brownian bridge have finite first moments. Using \eqref{eq:Q_diag_M}, we obtain
\begin{equation*}
\begin{aligned}
	\operatorname{Tr}\left((\mathsf Q_{\h_{\ra'}})^n-(\mathsf Q_{\h_\ra})^n\right)
	&=
	\lim_{R\to\infty}\operatorname{Tr}\left( \mathbf{1}_{(0,R)}\bigl((\mathsf Q_{\h_{\ra'}})^n-(\mathsf Q_{\h_\ra})^n\bigr) \mathbf{1}_{(0,R)} \right)\\
	&=
	\frac{1}{\sqrt{2\pi n}}\lim_{R\to\infty}\bbE\left[(R-M_{\ra'}(\hat{\B}))_+-(R-M_\ra(\hat{\B}))_+\right]
	=
	\frac{d}{\sqrt{2\pi n}},
\end{aligned}
\end{equation*}
where the last equality follows from dominated convergence and \eqref{eq:M_expectation_shift}. Substituting this identity into \eqref{eq:Delta_ratio_log}, we find
\begin{equation*}
	\log\frac{\Delta_\ra(\rz)}{\Delta_{\ra'}(\rz)}
	=
	\frac{d}{\sqrt{2\pi}}\sum_{n\ge1}\frac{\rz^n}{n^{3/2}}
	=
	-dA_1(\rz).
\end{equation*}
Combining this identity with \eqref{eq:K_Delta_a} gives
\begin{equation*}
	\det(\I+\limKe_{\h_\ra}(\rz))_{L^2(\bbR)}
	=
	e^{-dA_1(\rz)}\det(\I+\limKe_{\h_{\ra'}}(\rz))_{L^2(\bbR)}.
\end{equation*}
Since $d=\h(\ra')-\h(\ra)$, the definition of $\rE_\h(\rz)$ now proves the result.
\end{proof}

\subsection{Bounds for the periodic characteristic function and well-definedness of $\rD_\h(\bz)$}
\label{sec:pcfbound}

We first establish a bound for $\pchlim_\h$.

\begin{lm}\label{prop:plim_bound}
Let $\h\in\uc_1$, and let $\Mx\ge0$ satisfy $\max_{\XX\in[0,1]}\h(\XX)\le\Mx$. Then, for every $\Re(\xi)<0$ and $\arg(\eta)\in(-\frac{\pi}{4},\frac{\pi}{4})$,
\begin{equation}\label{eq:plim_bound}
    \left|\pchlim_\h(\eta,\xi)\right|
    \le
    \frac{\re^{(\Mx+1)(|\eta|+|\xi|)}}{1-\re^{-\frac12\Re(\eta^2)}}
    \left(\frac{1}{-\Re(\xi)}+\re^{\frac12(\Re(\xi))^2}\right).
\end{equation}
\end{lm}

\begin{proof}
By \eqref{eq:pch_lim},
\begin{equation*}
    \pchlim_\h(\eta,\xi)=\int_\bbR \re^{-s\xi}E(s,\eta)\,\rd s,
    \qquad E(s,\eta):=\bbE_{\B(0)=s}\left[\re^{\eta\B(\btau)-\frac12\eta^2\btau}\bfone_{\btau<1}-\re^{\eta\B(\bntau)-\frac12\eta^2\bntau}\bfone_{\btau<1,\,\bntau<\infty}\right].
\end{equation*}
On $\{\btau<1\}$ and $\{\bntau<\infty\}$, respectively, $\B(\btau)\le\h(-\btau)\le\Mx$ and $\B(\bntau)\le\h(-\bntau)\le\Mx$. 
Since $\Re(\eta^2)>0$, 
\begin{equation*}
    \big|\re^{-\frac12\eta^2\btau} \big|\le 1
    \quad\text{on }\{\btau<1\},
    \qquad
    \big|\re^{-\frac12\eta^2\bntau}\big|\le \re^{-\frac12\Re(\eta^2)k}
    \quad\text{on }\{\bntau\in[k,k+1)\}.
\end{equation*}
Hence
\begin{equation*}
\begin{aligned}
	|E(s,\eta)|
	&\le
	\re^{\Mx\Re(\eta)}\prob_{\B(0)=s}(\btau<1)
	\left(1+\sum_{k=1}^\infty\re^{-\frac12\Re(\eta^2)k}\right) 
	=
	\frac{\re^{\Mx\Re(\eta)}}{1-\re^{-\frac12\Re(\eta^2)}}\prob_{\B(0)=s}(\btau<1).
\end{aligned}
\end{equation*}
Consequently,
\begin{equation}\label{eq:after_bracket}
    \left|\pchlim_\h(\eta,\xi)\right|
    \le
    \frac{\re^{\Mx\Re(\eta)}}{1-\re^{-\frac12\Re(\eta^2)}}
    \int_\bbR \re^{-s\Re(\xi)}\prob_{\B(0)=s}(\btau<1)\,\rd s.
\end{equation}
By the reflection principle,
\begin{equation*}
    \prob_{\B(0)=s}(\btau<1)
    \le
    \mathbf{1}_{s\le\Mx+1}
    +
    \sqrt{\frac{2}{\pi}}\re^{-(s-\Mx)^2/2}\mathbf{1}_{s>\Mx+1}.
\end{equation*}
Thus, with $a:=-\Re(\xi)>0$, 
\begin{equation}\label{eq:transform_bound}
    \int_\bbR \re^{sa}\prob_{\B(0)=s}(\btau<1)\,\rd s
    \le
    \frac{\re^{(\Mx+1)a}}{a}
    +
    \re^{\frac12a^2+(\Mx+1)a}.
\end{equation}
Inserting \eqref{eq:transform_bound} into \eqref{eq:after_bracket}, and using $\Re(\eta)\le|\eta|$ and $a=-\Re(\xi)\le|\xi|$, proves \eqref{eq:plim_bound}.
\end{proof}

The preceding bound implies that $\rD_\h(\bz)$ is well-defined. 

\begin{cor}
Let $m\ge1$, $\XX_1,\ldots,\XX_m\in\bbR$, $0<\TT_1\le\cdots\le\TT_m$, and $(\HH_1,\ldots,\HH_m)\in\dom_+^m(\TT_1,\ldots,\TT_m)$. For $\h\in\uc_1$ and $\bz=(\rz_1,\ldots,\rz_m)$ with $0<|\rz_i|<1$, the series \eqref{eq:series} defining $\rD_\h(\bz)$ converges absolutely.
\end{cor}

\begin{proof}
We estimate the terms $\mathrm D_\h^{(\mathbf n)}(\bz)$ in \eqref{eq:series}. The functions $\hftn(\zeta,\rz_\ell)$ satisfy $\hftn(\zeta,\rz_\ell)=O(\zeta^{-1})$ as $|\zeta|\to\infty$ away from the imaginary axis. Hence the factors $\re^{\pm\hftn}$ and $\mathrm g_\ell$ are bounded on $\rL_{\rz_\ell}\cup\rR_{\rz_\ell}$, while the Cauchy determinants in \eqref{eq:Dn_limnoh} grow at most polynomially in the root variables.

The decay comes from the factors $\mathrm f_\ell$ in \eqref{eq:f_lim}. For $\zeta\in\rL_{\rz_\ell}\cup\rR_{\rz_\ell}$, the term involving $\zeta^2$ is bounded because $\re^{-\zeta^2/2}=\rz_\ell$. Thus, it remains to control the terms involving $\zeta^3$ and $\zeta$. If $\TT_\ell>\TT_{\ell-1}$, then there exist constants $C,c>0$ such that
\begin{equation*}
	|\mathrm f_\ell(\zeta)|\le C\re^{-c|\zeta|^3},
	\qquad
	\zeta\in\rL_{\rz_\ell}\cup\rR_{\rz_\ell}.
\end{equation*}
This applies in particular to $\ell=1$, since $\TT_0=0$ and $\TT_1>0$. If $\TT_\ell=\TT_{\ell-1}$, then $\beta_\ell-\beta_{\ell-1}>0$, and hence
\begin{equation*}
	|\mathrm f_\ell(\zeta)|\le C\re^{-c|\zeta|}, \qquad
	\zeta\in\rL_{\rz_\ell}\cup\rR_{\rz_\ell}. 
\end{equation*}

Lemma~\ref{prop:plim_bound} implies that $\pchlim_\h(\eta,\xi)$ has at most quadratic exponential growth in the root variables. Since it is evaluated only for $\eta\in\rR_{\rz_1}$ and $\xi\in\rL_{\rz_1}$, the product $\pchlim_\h(\eta,\xi)\mathrm f_1(\eta)\mathrm f_1(\xi)$ 
has cubic exponential decay. Consequently, the summand in \eqref{eq:Dn_lim} decays.

Combining these estimates with Hadamard's inequality and summing over the root sets, we obtain a constant $C>0$ such that $|\mathrm D_\h^{(\mathbf n)}(\bz)| \le \prod_{\ell=1}^m n_\ell^{n_\ell}C^{n_\ell}$. 
(The bound can be improved to $\prod_{\ell=1}^m C^{n_\ell}$ if we combine the Cauchy determinant with the functions $\mathrm{f}_\ell$.)
Thus, the series \eqref{eq:series} converges absolutely. 
\end{proof}

\subsection{Shift invariance}
\label{sec:proof_transition_rc}

We prove that $\mathbb{F}_{\h}$ is invariant under horizontal and vertical shifts of $\h$. Once the convergence of the PTASEP distribution functions to $\mathbb{F}_{\h}$ is established, this invariance follows immediately from the invariance of the PTASEP under relabeling and common translations of the particle positions. Here, we give a direct proof from the formula for $\mathbb{F}_{\h}$. We begin with the following lemma. 

\begin{lm}\label{lm:change_start}
Let $\h\in\uc_1$, and let $\eta,\xi\in\bbC$ satisfy 
\begin{equation}\label{eq:eta_xi_same_root}
	\arg(\eta)\in(-\pi/4,\pi/4),\qquad \Re(\xi)<0,\qquad \re^{-\eta^2/2}=\re^{-\xi^2/2}.
\end{equation} 
Define
\begin{equation*}
	M(t):=\re^{\eta\B(t)-\frac12\eta^2t},
\end{equation*}
with the convention $M(\infty):=0$, where $\B$ is a Brownian motion. 
For $c\in\bbR$, define 
\begin{equation}\label{eq:Jc_definition}
	J_c(\eta,\xi)
	:=
	\int_\bbR \re^{-s\xi}\bbE_{\B(c)=s}\left[M(\btau_c)-M(\btau_{c+1})\right] \,\rd s, 
	\qquad 
	\btau_c:=\inf\{\sfx\ge c:\B(\sfx)\le\h(-\sfx)\}.
\end{equation}
Then, for every $c,c'\in\bbR$, 
\begin{equation}\label{eq:Jc_change_start}
	\re^{c\xi^2/2}J_c(\eta,\xi)
	=
	\re^{c'\xi^2/2}J_{c'}(\eta,\xi).
\end{equation}

\end{lm}

\begin{proof}
Since $\Re(\eta^2)>0$ and $\B(t)/t\to0$ almost surely, $M(t)\to0$ almost surely as $t\to\infty$. 
Note that $\btau=\btau_0$ and $\bntau=\btau_1$. After translating time by $c$, the estimates in the proof of Lemma~\ref{prop:plim_bound} show that the integral in \eqref{eq:Jc_definition} is absolutely convergent.

It is enough to prove \eqref{eq:Jc_change_start} when $c<c'<c+1$. The general case follows by applying this case successively along a finite partition from $c$ to $c'$.

For $c<c'<c+1$, write
\begin{equation}\label{eq:decomposition}
	M(\btau_c)-M(\btau_{c+1})
	=
	\bigl(M(\btau_c)-M(\btau_{c'})\bigr)
	+
	\bigl(M(\btau_{c'})-M(\btau_{c'+1})\bigr)
	-
	\bigl(M(\btau_{c+1})-M(\btau_{c'+1})\bigr).
\end{equation}
Set
\begin{equation*}
	A:=\int_\bbR \re^{-s\xi}\bbE_{\B(c)=s}\left[M(\btau_c)-M(\btau_{c'})\right] \,\rd s,
	\qquad
	C:=\int_\bbR \re^{-s\xi}\bbE_{\B(c)=s}\left[M(\btau_{c+1})-M(\btau_{c'+1})\right] \,\rd s.
\end{equation*}
Conditioning on $\B(c+1)$ and evaluating the Gaussian integral, we obtain
\begin{equation}\label{eq:C_condition}
	C
	=
	\re^{\xi^2/2}
	\int_\bbR \re^{-r\xi}\bbE_{\B(c+1)=r}\left[M(\btau_{c+1})-M(\btau_{c'+1})\right] \,\rd r.
\end{equation}
Shifting the Brownian time parameter by $1$ gives $C=\re^{(\xi^2-\eta^2)/2}A$. 
By \eqref{eq:eta_xi_same_root}, we find that $C=A$.

Taking expectations in \eqref{eq:decomposition} and integrating over $s$, the first and third terms therefore cancel, and
\begin{equation*}
	J_c(\eta,\xi)
	=
	\int_\bbR \re^{-s\xi}\bbE_{\B(c)=s}\left[M(\btau_{c'})-M(\btau_{c'+1})\right] \,\rd s.
\end{equation*}
Conditioning on $\B(c')$ and evaluating the Gaussian integral, we find
\begin{equation*}
	J_c(\eta,\xi)
	=
	\int_\bbR\int_\bbR
	\re^{-s\xi}
	\frac{\re^{-\frac{(r-s)^2}{2(c'-c)}}}{\sqrt{2\pi(c'-c)}}
	\bbE_{\B(c')=r}\left[M(\btau_{c'})-M(\btau_{c'+1})\right] \,\rd r\,\rd s
	=
	\re^{\frac12(c'-c)\xi^2}J_{c'}(\eta,\xi).
\end{equation*}
\end{proof}

We next deduce the shift property of the periodic characteristic function.

\begin{cor}[Shift property of $\pchlim_\h$]\label{prop:translation_pchlim}
Let $\h\in\uc_1$, let $c_1,c_2\in\bbR$, and set
\begin{equation*}
	\h'(\XX):=\h(\XX-c_1)+c_2.
\end{equation*}
Then, for every $0<|\rz|<1$, $\eta\in\rR_\rz$, and $\xi\in\rL_\rz$,
\begin{equation}\label{eq:pchilmshftid}
	\pchlim_{\h'}(\eta,\xi)
	=
	\pchlim_\h(\eta,\xi)
	\re^{-\frac12c_1(\xi^2-\eta^2)+c_2(\eta-\xi)}.
\end{equation}
\end{cor}

\begin{proof}
For a vertical shift, translating the Brownian motion and the integration variable in \eqref{eq:pch_lim} gives
\begin{equation}\label{eq:pch_vertical_shift}
	\pchlim_{\h+c_2}(\eta,\xi)
	=
	\re^{c_2(\eta-\xi)}\pchlim_\h(\eta,\xi).
\end{equation}

For a horizontal shift, let $\B$ be a Brownian motion indexed by $t\ge0$, and define
\begin{equation*}
	\widehat{\B}(t):=\B(t-c_1),
	\qquad t\ge c_1.
\end{equation*}
If
\begin{equation*}
	T_0:=\inf\{t\ge0:\B(t)\le\h(-t-c_1)\},
	\qquad
	T_1:=\inf\{t\ge1:\B(t)\le\h(-t-c_1)\},
\end{equation*}
then
\begin{equation*}
	T_0+c_1=\widehat{\btau}_{c_1},
	\qquad
	T_1+c_1=\widehat{\btau}_{c_1+1},
\end{equation*}
where the stopping times on the right are defined using $\widehat{\B}$ and the boundary $\h(-\cdot)$. Hence
\begin{equation*}
	\pchlim_{\h(\cdot-c_1)}(\eta,\xi)
	=
	\re^{c_1\eta^2/2}J_{c_1}(\eta,\xi),
\end{equation*}
where $J_{c_1}$ is defined using $\widehat{\B}$. Here the indicator $\mathbf{1}_{T_1<\infty}$ is removed by the convention $M(\infty)=0$, while the indicator $\mathbf{1}_{T_0<1}$ may be omitted because $T_0=T_1$ on $\{T_0\ge1\}$.

Since $\eta\in\rR_\rz$ and $\xi\in\rL_\rz$, we have $\re^{-\eta^2/2}=\rz=\re^{-\xi^2/2}$. Lemma~\ref{lm:change_start}, applied with $c=0$ and $c'=c_1$, gives
\begin{equation*}
	J_{c_1}(\eta,\xi)=\re^{-c_1\xi^2/2}J_0(\eta,\xi).
\end{equation*}
The same indicator-removal argument gives $J_0(\eta,\xi)=\pchlim_\h(\eta,\xi)$. Therefore,
\begin{equation}\label{eq:pch_horizontal_shift}
	\pchlim_{\h(\cdot-c_1)}(\eta,\xi)
	=
	\re^{-\frac12c_1(\xi^2-\eta^2)}\pchlim_\h(\eta,\xi).
\end{equation}
Combining \eqref{eq:pch_vertical_shift} and \eqref{eq:pch_horizontal_shift} proves \eqref{eq:pchilmshftid}.
\end{proof}

We can now prove the shift invariance property in Theorem~\ref{thm:main2}\textnormal{(g)}.

\begin{cor}\label{cor:shiftinvlmf}
Let $\h\in\uc_1$ and $m\ge1$. If $\h'(\XX)=\h(\XX-a)+b$ for some $a,b\in\bbR$, then
\begin{equation}\label{eq:Feqst}
	\mathbb{F}_{\h'}^{(m)}
	\bigl(\bm{\HH}+b(1,\ldots,1);\,(\XX_i+a,\TT_i)_{i=1}^m\bigr)
	=
	\mathbb{F}_\h^{(m)}
	\bigl(\bm{\HH};(\XX_i,\TT_i)_{i=1}^m\bigr).
\end{equation}
\end{cor}

\begin{proof}
It is enough to consider the parameters in part~\textnormal{(i)} of Definition~\ref{def:F}. 
Write $\rC(\bz;\HH_1,\ldots,\HH_m)$ and $\rC_\h(\bz;\HH_1,\ldots,\HH_m)$ to indicate the dependence on the height parameters. By \eqref{eq:Climnoh}, 
$\rC(\bz;\HH_1+b,\ldots,\HH_m+b) =\re^{bA_1(\rz_1)}\rC(\bz;\HH_1,\ldots,\HH_m)$ and 
by Corollary~\ref{lm:translation_rC} and the definition of the energy function, $\rE_{\h'}(\rz_1)=\re^{-bA_1(\rz_1)}\rE_\h(\rz_1)$. 
Thus, 
\begin{equation*}
	\rC_{\h'}(\bz;\HH_1+b,\ldots,\HH_m+b)
	=
	\rC_\h(\bz;\HH_1,\ldots,\HH_m).
\end{equation*}

As for the $\rD$ factors, replacing every $\XX_i$ by $\XX_i+a$ and every $\HH_i$ by $\HH_i+b$ changes only the $\ell=1$ factors in \eqref{eq:f_lim}, multiplying $\mathrm f_1(\xi)$ and $\mathrm f_1(\eta)$ by $\re^{\frac12a\xi^2+b\xi}$ and $\re^{-\frac12a\eta^2-b\eta}$, 
respectively. On the other hand, Corollary~\ref{prop:translation_pchlim} multiplies the determinant in \eqref{eq:Dn_lim} by the inverse product of these factors. 
Thus every summand in \eqref{eq:series} is unchanged, and hence so is $\rD_\h(\bz)$. 
The result now follows from the contour integral formula \eqref{eq:multi_time1}. 
\end{proof}

\section{The PTASEP multipoint distribution formula}
\label{sec:PTASEP_formula}

We prove Theorem~\ref{thm:main} by taking the large-time, large-period limit of the PTASEP multipoint distribution formula. In this section, we recall the formula from~\cite{Baik-Liu21}. The initial condition enters through two functions, the \emph{PTASEP energy function} and the \emph{modified PTASEP characteristic function}, for which we derive new probabilistic representations in Section~\ref{sec:proof_PTASEP}.

\subsection{Bethe roots and initial-condition-dependent functions}

The following polynomial and its roots, introduced in~\cite{Baik-Liu18}, play a central role in the analysis of the PTASEP.

\begin{defn}\label{defn:Bethe}
Let $L>N$ be positive integers. The \emph{Bethe polynomial} associated with $z\in\bbC$ is
\begin{equation}
	q_z(w):=w^N(w+1)^{L-N}-z^L.
\end{equation}
Its roots are called the \emph{Bethe roots} associated with $z$. Define
\begin{equation}
	\bfr_{N,L}:=\left(\frac{N}{L}\right)^{N/L}\left(1-\frac{N}{L}\right)^{1-N/L}.
\end{equation}
\end{defn}

The Bethe roots lie on the level curve 
\begin{equation*}
    \big\{w\in\bbC:|w|^{N/L}|w+1|^{1-N/L}=|z| \big\}, 
\end{equation*}
a generalized Cassini oval. 
Its geometry depends on $|z|$ relative to $\bfr_{N,L}$:
\begin{itemize}
\item If $|z|<\bfr_{N,L}$, the level curve has two disjoint components, one in $\{\Re(w)<-N/L\}$ and the other in $\{\Re(w)>-N/L\}$.
\item If $|z|=\bfr_{N,L}$, the level curve has a self-intersection at $w=-N/L$ and is a generalized lemniscate of Bernoulli.
\item If $|z|>\bfr_{N,L}$, the level curve is a smooth Jordan curve.
\end{itemize}
See Figure~\ref{fig:roots_comparison}.

\begin{figure}[htbp]
     \centering
     \begin{subfigure}[b]{0.32\textwidth}
         \centering
         \includegraphics[width=\textwidth]{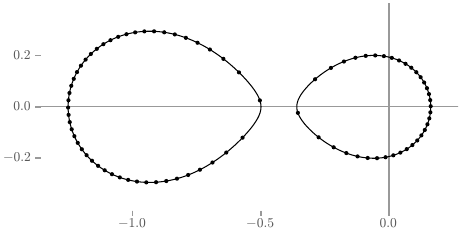}
         \caption{$z=0.99\bfr_{N,L}\re^{\ri\pi/20}$}
     \end{subfigure}
     \hfill
     \begin{subfigure}[b]{0.32\textwidth}
         \centering
         \includegraphics[width=\textwidth]{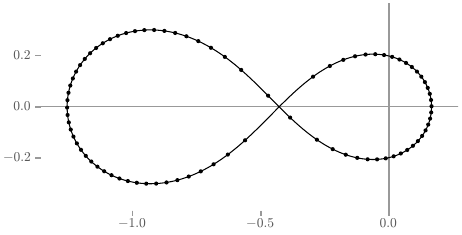}
         \caption{$z=\bfr_{N,L}\re^{\ri\pi/20}$}
     \end{subfigure}
     \hfill
     \begin{subfigure}[b]{0.32\textwidth}
         \centering
         \includegraphics[width=\textwidth]{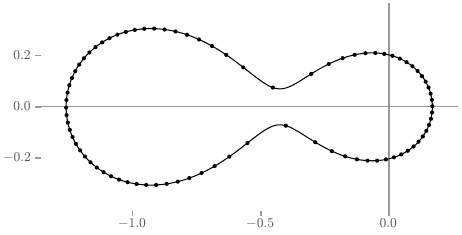}
         \caption{$z=1.01\bfr_{N,L}\re^{\ri\pi/20}$}
     \end{subfigure}
     \caption{Bethe roots and level curves for $w^N(w+1)^{L-N}=z^L$ with $N=36$, $L=84$. Here $\bfr_{N,L} =3^{3/7}4^{4/7}/7$.}
     \label{fig:roots_comparison}
\end{figure}

\begin{defn}\label{def:leftright}
Let $L>N$ be positive integers and let $0<|z|<\bfr_{N,L}$. Define the sets of \emph{left} and \emph{right Bethe roots} by
\begin{equation}\label{eq:lrBethe}
	\cL_z:=\{w\in\bbC:q_z(w)=0,\ \Re(w)<-N/L\},
	\qquad
	\cR_z:=\{w\in\bbC:q_z(w)=0,\ \Re(w)>-N/L\},
\end{equation}
and the \emph{left} and \emph{right Bethe polynomials} by
\begin{equation}
	q_{z,\rL}(w):=\prod_{u\in\cL_z}(w-u),
	\qquad
	q_{z,\rR}(w):=\prod_{v\in\cR_z}(w-v).
\end{equation}
\end{defn}

The sets $\cL_z$ and $\cR_z$ contain $L-N$ and $N$ roots, respectively; see~\cite[Section~7]{Baik-Liu18}.

\begin{defn}\label{def:cGlam}
For $Y=(y_1,\ldots,y_N)\in\bbZ^N$, define\footnote{This differs from the notation in~\cite{Baik-Liu21}, where the same function was written as
\begin{equation*}
	\cG_\lambda(W)
	=
	\frac{\det[w_i^{-j}(w_i+1)^{\lambda_j}]_{i,j=1}^N}
	{\det[w_i^{-j}]_{i,j=1}^N}
\end{equation*}
for $\lambda\in\bbZ^N$, with $\lambda_j=y_j+j$.} 
\begin{equation}
	\mcG_Y(W)
	:=
	\frac{\det[w_i^{-j}(w_i+1)^{y_j+j}]_{i,j=1}^N}
	{\det[w_i^{-j}]_{i,j=1}^N},
	\qquad W=(w_1,\ldots,w_N)\in\bbC^N.
\end{equation}
\end{defn}

Since $\mcG_Y(W)$ is symmetric in $w_1,\ldots,w_N$, we may regard $W$ as a set or a vector. 
This symmetric Laurent polynomial is related to the eigenfunctions of the PTASEP; see~\cite[Section~6]{Motegi-Sakai13}.
It is also related to the symmetric functions in~\cite{Motegi-Sakai13} and to the inhomogeneous Schur functions in~\cite{Borodin17}.

We now define the two initial-condition-dependent functions appearing in the distribution formula.

\begin{defn}\label{def:characteristic}
Let $L>N$ be positive integers, let $Y=(y_1,\ldots,y_N)\in\bbZ^N$, and let $0<|z|<\bfr_{N,L}$. Define the \emph{PTASEP energy function} associated with $Y$ by
\begin{equation}\label{eq:energyfunc}
	\cE_Y(z):=\mcG_Y(\cR_z)
	=
	\frac{\det[v_i^{-j}(v_i+1)^{y_j+j}]_{i,j=1}^N}
	{\det[v_i^{-j}]_{i,j=1}^N},
\end{equation}
where $\cR_z=\{v_1,\ldots,v_N\}$. If $\cE_Y(z)\neq0$, define\footnote{These functions appeared in~\cite[(3.8), (3.9)]{Baik-Liu21} as the ``global energy function'' and the ``characteristic function''; the latter did not include the factor $1/(v-u)$. We use ``PTASEP'' to distinguish these prelimit objects from their scaling limits and ``modified'' to indicate this additional factor. Later we also introduce normalized versions and the PTASEP characteristic function.}
the \emph{modified PTASEP characteristic function} associated with $Y$ by
\begin{equation}\label{eq:charY}
	\mdpch_Y(v,u;z)
	:=
	\frac{\mcG_Y(\cR_z\cup\{u\}\setminus\{v\})}
	{(v-u)\mcG_Y(\cR_z)},
	\qquad
	v\in\cR_z,
	\quad
	u\in\bbC\setminus\{-1,0,v\}.
\end{equation}
\end{defn}

As defined, $\mdpch_Y(v,u;z)$ may have poles at zeros of $\cE_Y(z)$. 
However, Theorem~\ref{thm:characteristic} shows that, for $Y\in\Pconfno_{N,L}$, it extends analytically through these possible poles.

For the periodic step initial condition $y_i=-i$, $1\le i\le N$, we find~\cite[Lemma~3.8]{Baik-Liu21}
\begin{equation*}
	\cE_{\mathrm{step}}(z)=1,
	\qquad
	\mdpch_{\mathrm{step}}(v,u;z)=\frac{1}{v-u}.
\end{equation*}

\subsection{The PTASEP multipoint distribution formula}

The following theorem gives the multipoint distribution formula for the PTASEP with an arbitrary periodic initial condition.

\begin{thm}[Theorem~3.1 of~\cite{Baik-Liu21}]
\label{thm:PTASEP_multi}
Let $L>N$ be positive integers, and consider $\mathrm{PTASEP}_Y(L,N)$ with $Y\in\Pconf$. Let $m\ge1$, and let $(k_i,t_i)_{i=1}^m$ be distinct points in $\bbZ\times[0,\infty)$ satisfying $0<t_1\le t_2\le\cdots\le t_m$. Then, for all $a_1,\ldots,a_m\in\bbZ$,
\begin{equation}\label{eq:PTASEP_multi}
	\prob_Y\left(\bigcap_{i=1}^m\{\sx_{k_i}(t_i)\ge a_i\}\right)
	=
	\oint\cdots\oint
	\mathscr{C}_Y(\bm{z})\mathscr{D}_Y(\bm{z})
	\frac{\rd z_1}{2\pi\ri z_1}\cdots\frac{\rd z_m}{2\pi\ri z_m},
\end{equation}
where $\bm{z}=(z_1,\ldots,z_m)$ and the contours are nested circles centered at the origin satisfying $0<|z_m|<\cdots<|z_1|<\bfr_{N,L}$. The functions $\mathscr{C}_Y(\bm{z})$ and $\mathscr{D}_Y(\bm{z})$ are defined in Definitions~\ref{def:Czoriginal} and~\ref{def:Dz}, respectively.
\end{thm}


In the next two definitions, we fix the data $Y$, $\bm{z}=(z_1,\ldots,z_m)$, $(k_i,t_i)_{i=1}^m$, and $a_1,\ldots,a_m$ as in Theorem~\ref{thm:PTASEP_multi}. 
For an ordered tuple $W=(w_1,\ldots,w_n)$ and finite sets $A,B\subset\bbC$, define
\begin{equation*}
	\Delta(W):=\prod_{1\le i<j\le n}(w_j-w_i),
	\qquad
	\Delta(A;B):=\prod_{a\in A}\prod_{b\in B}(a-b).
\end{equation*}

\begin{defn}[The function $\mathscr{C}_Y(\bm{z})$]
\label{def:Czoriginal}
Define
\begin{equation}\label{eq:CPTASEP}
\begin{aligned}
	\mathscr{C}_Y(\bm{z})
	:=\cE_Y(z_1)
	&\left[\prod_{\ell=1}^m\frac{\rG_\ell(z_\ell)}{\rG_{\ell-1}(z_\ell)}\right]
	\left[\prod_{\ell=1}^m
	\frac{\prod_{u\in\cL_{z_\ell}}(-u)^N\prod_{v\in\cR_{z_\ell}}(v+1)^{L-N}}
	{\Delta(\cR_{z_\ell};\cL_{z_\ell})}\right]\\
	&\quad\times
	\left[\prod_{\ell=2}^m\frac{z_{\ell-1}^L}{z_{\ell-1}^L-z_\ell^L}\right]
	\left[\prod_{\ell=2}^m
	\frac{\Delta(\cR_{z_\ell};\cL_{z_{\ell-1}})}
	{\prod_{u\in\cL_{z_{\ell-1}}}(-u)^N\prod_{v\in\cR_{z_\ell}}(v+1)^{L-N}}\right],
\end{aligned}
\end{equation}
where  $\cE_Y(z)$ is the PTASEP energy function from Definition~\ref{def:characteristic}, $\rG_0(z):=1$, and 
\begin{equation*}
	\rG_\ell(z)
	:=
	\prod_{u\in\cL_z}(-u)^{k_\ell}
	\prod_{v\in\cR_z}(v+1)^{-a_\ell-k_\ell}\re^{t_\ell v},
	\qquad 1\le\ell\le m,
\end{equation*}
\end{defn}

The function $\mathscr{D}_Y(\bm{z})$ admits both Fredholm determinant and series representations. We state only the series representation, which is the prelimit analogue of Definition~\ref{def:D_series}. 
The operator form is analogous to that in Section~\ref{sec:Fredholm_D}. For $1\le\ell\le m$, define
\begin{equation}\label{eq:f}
	f_\ell(w):=
	\begin{dcases}
	\frac{w^{k_\ell}(w+1)^{a_{\ell-1}+k_{\ell-1}}\re^{t_\ell w}}
	{w^{k_{\ell-1}}(w+1)^{a_\ell+k_\ell}\re^{t_{\ell-1}w}}
	\quad &\text{if }\Re(w)<-N/L,\\
	\frac{w^{k_{\ell-1}}(w+1)^{a_\ell+k_\ell}\re^{t_{\ell-1}w}}
	{w^{k_\ell}(w+1)^{a_{\ell-1}+k_{\ell-1}}\re^{t_\ell w}}
	\quad &\text{if }\Re(w)>-N/L,
	\end{dcases}
\end{equation}
with the convention $a_0=k_0=t_0:=0$. Define also
\begin{equation*}
	g_\ell(w):=\frac{H_{z_\ell}(w)^2}{H_{z_{\ell-1}}(w)H_{z_{\ell+1}}(w)},
\end{equation*}
where $z_0=z_{m+1}:=0$, $H_0(w):=1$, and
\begin{equation*}
	H_z(w):=
	\begin{cases}
	\displaystyle\frac{q_{z,\rL}(w)}{(w+1)^{L-N}},&\text{if }\Re(w)>-N/L,\\[6pt]
	\displaystyle\frac{q_{z,\rR}(w)}{w^N},&\text{if }\Re(w)<-N/L.
	\end{cases}
\end{equation*}
Finally, set
\begin{equation*}
	A_\ell(w):=f_\ell(w)g_\ell(w)\frac{w(w+1)}{Lw+N}.
\end{equation*}

\begin{defn}[Series formula for $\mathscr{D}_Y(\bm{z})$]
\label{def:Dz}
Let $\mdpch_Y$ be the modified PTASEP characteristic function from Definition~\ref{def:characteristic}, and recall the Cauchy determinant notation in \eqref{eq:Cauchydt}. For $\mathbf{n}=(n_1,\ldots,n_m)\in(\bbZ_{\ge0})^m$, set
\begin{equation*}
\begin{aligned}
	d^{(\mathbf{n})}(\bm{z};U,V)
	:=&
	\prod_{\ell=1}^{m-1}
	\left(1-\frac{z_{\ell+1}^L}{z_\ell^L}\right)^{n_\ell}
	\left(1-\frac{z_\ell^L}{z_{\ell+1}^L}\right)^{n_{\ell+1}}\\
	&\quad\times
	\left[\prod_{\ell=1}^m\prod_{i=1}^{n_\ell}
	A_\ell(u_i^{(\ell)})A_\ell(v_i^{(\ell)})\right]
	\left[\prod_{\ell=1}^m
	\Cau\left(U^{(\ell)}\sqcup V^{(\ell+1)};
	V^{(\ell)}\sqcup U^{(\ell+1)}\right)\right].
\end{aligned}
\end{equation*}
Here $U=U^{(1)}\sqcup\cdots\sqcup U^{(m)}$ and $V=V^{(1)}\sqcup\cdots\sqcup V^{(m)}$, where\footnote{We set $U^{(m+1)}=V^{(m+1)}:=\varnothing$.}
\begin{equation*}
	U^{(\ell)}=(u_1^{(\ell)},\ldots,u_{n_\ell}^{(\ell)}),
	\qquad
	V^{(\ell)}=(v_1^{(\ell)},\ldots,v_{n_\ell}^{(\ell)}),
	\qquad 1\le\ell\le m. 
\end{equation*}
For $Y\in\Pconfno_{N,L}$, set
\begin{equation}\label{eq:Dn}
	\mathscr{D}^{(\mathbf{n})}_Y(\bm{z})
	:=
	\sum_{\substack{
	U^{(\ell)}\in(\cL_{z_\ell})^{n_\ell},\ V^{(\ell)}\in(\cR_{z_\ell})^{n_\ell}\\
	1\le\ell\le m}}
	\det\left[\mdpch_Y(v_i^{(1)},u_j^{(1)};z_1)\right]_{1\le i,j\le n_1}
	d^{(\mathbf{n})}(\bm{z};U,V).
\end{equation}
Finally, define
\begin{equation*}
	\mathscr{D}_Y(\bm{z})
	:=
	\sum_{\mathbf{n}\in(\bbZ_{\ge0})^m}
	\frac{(-1)^{n_1+\cdots+n_m}}{(n_1!\cdots n_m!)^2}
	\mathscr{D}^{(\mathbf{n})}_Y(\bm{z}).
\end{equation*}
\end{defn}

\section{Probabilistic representations for the PTASEP energy and characteristic functions}
\label{sec:proof_PTASEP}

In this section, we derive probabilistic representations of the PTASEP energy function and the modified PTASEP characteristic function after simple algebraic normalizations. These representations are well suited to asymptotic analysis. The main results are Theorems~\ref{thm:characteristic} and~\ref{thm:energy}.

Throughout this section, let $L>N$ be positive integers and let $0<|z|<\bfr_{N,L}$. 
We begin by removing explicit algebraic factors from the two quantities that depend on the initial condition.

\begin{defn}\label{def:nepcf}
Let $\cE_Y(z)$ and $\mdpch_Y(v,u;z)$ be the PTASEP energy function and the modified PTASEP characteristic function from Definition~\ref{def:characteristic}. Define the \emph{normalized PTASEP energy function}
\begin{equation}\label{eq:nef}
    \ncE_Y(z)
    :=
    \frac{\prod_{u\in\cL_z}(-u)^N\prod_{v\in\cR_z}(v+1)^{L-N}}
    {\prod_{v\in\cR_z,\,u\in\cL_z}(v-u)}
    \,\cE_Y(z).
\end{equation}
If $\cE_Y(z)\neq0$, define the \emph{normalized PTASEP characteristic function} by
\begin{equation}\label{eq:npcf}
    \npch_Y(v,u;z)
    :=
    \frac{q_{z,\rL}(v)q_{z,\rR}(u)}{u^N(v+1)^{L-N}}
    \,\mdpch_Y(v,u;z) 
\end{equation}
for $v\in\cR_z$ and $u\in\bbC\setminus\{-1,0,v\}$. 
\end{defn}

We will use the following radii associated with the right and left Bethe roots.

\begin{defn}\label{defn:r12}
For $0<|z|<\bfr_{N,L}$, define
\begin{equation*}
    r_{1,N,L}(z):=\max_{v\in\cR_z}|v|,
    \qquad
    r_{2,N,L}(z):=\max_{u\in\cL_z}|u+1|.
\end{equation*}
\end{defn}

\begin{lm}
For $0<|z|<\bfr_{N,L}$,
\begin{equation*}
    r_{2,N,L}(z)<1-\frac{N}{L}<1-r_{1,N,L}(z).
\end{equation*}
\end{lm}

\begin{proof}
Set $\alpha:=N/L$. If $q_z(w)=0$, then $|w|^{\alpha}|w+1|^{1-\alpha}=|z| <\bfr_{N,L}= \alpha^{\alpha}(1-\alpha)^{1-\alpha}$. 
Suppose $v\in\cR_z$ and $|v|\ge\alpha$. Since $\Re(v)>-\alpha$,
\begin{equation*}
    |v+1|^2=|v|^2+1+2\Re(v)>\alpha^2+1-2\alpha=(1-\alpha)^2.
\end{equation*}
Thus $|v+1|>1-\alpha$, contradicting the inequality $|v|^{\alpha}|v+1|^{1-\alpha}<\alpha^{\alpha}(1-\alpha)^{1-\alpha}$. Hence $|v|<\alpha$ for every $v\in\cR_z$, and therefore $r_{1,N,L}(z)<N/L$. The bound $r_{2,N,L}(z)<1-N/L$ follows by the same argument after replacing $w$ by $-1-w$ and $\alpha$ by $1-\alpha$.
\end{proof}

\subsection{Geometric random walks and hitting times}
\label{sec:probtheorems}

We first introduce the geometric random walk and the hitting times appearing in the probabilistic representations.

\begin{defn}\label{def:char}
Fix $0<\rho<1$. Let $(G_k)_{k\ge0}$ be the geometric random walk with transition probabilities
\begin{equation}\label{eq:transition_general}
    \prob(G_{k+1}=b\mid G_k=a)
    :=
    \frac{\rho}{1-\rho}(1-\rho)^{a-b}
    \mathbf{1}_{b-a\in\bbZ_{<0}},
\end{equation}
where $\bbZ_{<0}:=\{-1,-2,-3,\ldots\}$. Given a sequence $Y=(y_i)_{i\ge1}$, define the stopping times\footnote{Here $\tau$ and $\ntau$ denote stopping times of the geometric random walk and should not be confused with the time parameters $\tau_i$ from \eqref{eq:scaling_1}. When the dependence on the initial condition needs to be emphasized, we write $\tau_Y$ and $\tau_Y^*$, respectively.}
\begin{equation}\label{eq:hitting_time}
    \tau=\tau_Y:=\inf\{m\ge0:G_m>y_{m+1}\},
    \qquad
    \ntau=\tau_Y^*:=\inf\{m\ge N:G_m>y_{m+1}\}.
\end{equation}
\end{defn}

Thus, $\tau$ is the first time the random walk enters the strict epigraph of $Y$, while $\tau^*$ is the first such time at or after time $N$.

\begin{defn}\label{def:chardef2}
For a sequence $Y=(y_i)_{i\ge1}$, define the \emph{characteristic function} associated with $Y$ by
\begin{equation}\label{eq:ess_hitting}
    \rch_Y(v,u)
    :=
    \sum_{x\in\bbZ}
    \left(\frac{u+1}{1-\rho}\right)^x
    \bbE_{G_0=x}\left[
    \frac{(1-\rho)^{G_\tau}}{(v+1)^{G_\tau+1}}
    \left(\frac{-v(1-\rho)}{(v+1)\rho}\right)^\tau
    \mathbf{1}_{\tau<N}
    \right].
\end{equation}
If, in addition, $Y\in\Pconf$, define
\begin{equation}\label{eq:ess_hittingstar}
    \rch_Y^*(v,u)
    :=
    \sum_{x\in\bbZ}
    \left(\frac{u+1}{1-\rho}\right)^x
    \bbE_{G_0=x}\left[
    \frac{(1-\rho)^{G_\ntau}}{(v+1)^{G_\ntau+1}}
    \left(\frac{-v(1-\rho)}{(v+1)\rho}\right)^\ntau
    \mathbf{1}_{\tau<N,\,\ntau<\infty}
    \right],
\end{equation}
and the \emph{PTASEP characteristic function}
\begin{equation}\label{eq:mpchdefn}
    \mpch_Y(v,u):=\rch_Y(v,u)-\rch_Y^*(v,u).
\end{equation}
\end{defn}

We prove the well-definedness and basic properties of $\rch_Y$ and $\mpch_Y$ in Subsection~\ref{sec:grw}. When $\rho=1/2$, the function $\rch_Y$ agrees with the characteristic function introduced in~\cite[Theorem~3.4]{Liao-Liu25} for TASEP on the infinite line. The additional term $\rch_Y^*$ records the first hit at or after one full period and is the new feature of the periodic setting. 
The first main result identifies this probabilistic function with the normalized algebraic quantity in Definition~\ref{def:nepcf}.

\begin{thm}\label{thm:characteristic}
Let $0<|z|<\bfr_{N,L}$, and let $r_{1,N,L}(z)$ and $r_{2,N,L}(z)$ be as in Definition~\ref{defn:r12}. Let $Y\in\Pconf$, and let $\rho\in(0,1)$ satisfy
\begin{equation*}
    r_{1,N,L}(z)<\rho<1-r_{2,N,L}(z),
    \qquad
    \rho^N(1-\rho)^{L-N}>|z|^L.
\end{equation*}
Assume that $\cE_Y(z)\neq0$. Then, for every $v\in\cR_z$ and every $u$ satisfying $0<|u+1|<1-\rho$,
\begin{equation}\label{eq:char_rewriting}
    \npch_Y(v,u;z)=\mpch_Y(v,u),
\end{equation}
where $\mpch_Y$ is defined using the geometric random walk parameter $\rho$.
\end{thm}

Once $\rho$ is fixed, the definition of $\mpch_Y(v,u)$ does not involve $z$.
In \eqref{eq:char_rewriting}, it is evaluated at $v\in\cR_z$. Since $\mpch_Y$ is defined without dividing by $\cE_Y(z)$, the right-hand side provides a natural extension of $\npch_Y$ through possible zeros of $\cE_Y$.

The second main result gives a Fredholm determinant representation of the normalized PTASEP energy function. We first introduce the operator appearing in the formula.

\begin{defn}\label{def:chardef2K}
Fix $\rho\in(0,1)$, and let $(G_k)_{k\ge0}$ be the geometric random walk with parameter $\rho$ from Definition~\ref{def:char}. For $|\rz|<1$, $Y\in\Pconf$ with $y_1<0$, and $x,y\in\bbZ$, define\footnote{It is straightforward to check that $I+S_\rz=(I-\rz H)^{-1}$ 
on $\ell^2(\bbZ)$, where $H(x,y):=\prob_{G_0=x}(G_N=-L+y)$.}
\begin{equation}\label{eq:formaS}
    S_{\rz}(x,y)
    :=
    \sum_{k=1}^\infty \rz^k\prob_{G_0=x}(G_{kN}=-kL+y),
    \qquad
    T_Y(x,y)
    :=
    \prob_{G_0=x}(G_N=-L+y-1,\ \tau<N).
\end{equation}
Define the operator $\Ke_{Y,\rz}$ on $\ell^2(\bbZ_{\le0})$ by
\begin{equation}\label{eq:onemrhoK}
    \Ke_{Y,\rz}
    :=
    -\rz\,\mathbf{1}_{\le0}\Sigma(I+S_{\rz})T_Y\mathbf{1}_{\le0},
\end{equation}
where $S_\rz$ and $T_Y$ denote the operators on $\ell^2(\bbZ)$ with the kernels  \eqref{eq:formaS}, and $\Sigma$ is the shift operator
\begin{equation*}
    (\Sigma f)(x):=f(x-1).
\end{equation*}
\end{defn}

\begin{thm}\label{thm:energy}
Let $0<|z|<\bfr_{N,L}$, and let $r_{1,N,L}(z)$ and $r_{2,N,L}(z)$ be as in Definition~\ref{defn:r12}. Let $\rho\in(0,1)$ satisfy
\begin{equation*}
    r_{1,N,L}(z)<\rho<1-r_{2,N,L}(z),
    \qquad
    \rho^N(1-\rho)^{L-N}>|z|^L,
\end{equation*}
and set
\begin{equation*}
    \rz:=\frac{z^L}{(-\rho)^N(1-\rho)^{L-N}}.
\end{equation*}
Then $|\rz|<1$. For every $Y\in\Pconf$ with $y_1<0$, the operator $\Ke_{Y,\rz}$ defined in Definition~\ref{def:chardef2K} using the geometric random walk parameter $\rho$ is trace class on $\ell^2(\bbZ_{\le0})$, and
\begin{equation}\label{eq:energy_fredholm}
    \ncE_Y(z)=\det(I-\Ke_{Y,\rz})_{\ell^2(\bbZ_{\le0})}. 
\end{equation}
\end{thm}

Although the representations in Theorems~\ref{thm:characteristic} and~\ref{thm:energy} involve the geometric random walk parameter $\rho$, we show below that their values are independent of its choice whenever the corresponding admissibility conditions are satisfied. In the asymptotic analysis, however, we make the specific choice $\rho=N/L$.

\subsection{Proof outline}
\label{sec:pfofalg}

The proofs of Theorems~\ref{thm:characteristic} and~\ref{thm:energy} both consist of an algebraic reduction followed by a probabilistic solution of the resulting equations.

We first perform the algebraic reductions. For the normalized PTASEP characteristic function, Lemma~\ref{prop:key_identity} shows that $\npch_Y(v,u;z)$, viewed as a function of $v\in\cR_z$, is characterized as the unique solution of the finite system \eqref{eq:identity_generaldensity}. Thus, to prove Theorem~\ref{thm:characteristic}, it is enough to show that $\mpch_Y(v,u)$ satisfies the same system.

The reduction for the normalized PTASEP energy function is different. Lemma~\ref{lm:energy_factorize} first rewrites $\ncE_Y(z)$ as the finite determinant $\det(M_Y)$. Using the general result in Proposition~\ref{prop:orthogonalization}, Corollary~\ref{lm:MYassumeH} shows that this determinant can be converted into a Fredholm determinant on $\ell^2(\bbZ_{\le0})$, once one finds a function $H(v,u)$ satisfying the reproducing relations \eqref{eq:Hdefiningreal}.
We then obtain the required probabilistic solutions from the geometric random walk. 

We establish basic properties of $\rch_Y$ and $\mpch_Y$ in Section~\ref{sec:grw}, prove two reproducing identities in Section~\ref{sec:chmpchadv}, and complete the proofs of Theorems~\ref{thm:characteristic} and~\ref{thm:energy} in Section~\ref{sec:proof_char}.

\subsection{Algebraic characterization of the normalized PTASEP characteristic function}
\label{sec:npcfchar}

We show that $\npch_Y(v,u;z)$, viewed as a function of $v\in\cR_z$, is the unique solution of an explicit system of linear equations.

\begin{lm}\label{prop:key_identity}
Fix $u\in\bbC\setminus\{-1,0\}$ and $z$ with $0<|z|<\bfr_{N,L}$ such that $\cE_Y(z)\neq0$. Then
\begin{equation*}
    \alpha(v):=\frac{v(v+1)}{Lv+N}\npch_Y(v,u;z),
    \qquad v\in\cR_z,
\end{equation*}
is the unique solution of the system of $N$ linear equations
\begin{equation}\label{eq:identity_generaldensity}
    \sum_{v\in\cR_z}v^{-i}(v+1)^{y_i+i}\alpha(v)
    =
    -u^{-i}(u+1)^{y_i+i},
    \qquad 1\le i\le N.
\end{equation}
\end{lm}

\begin{proof}
Write $\cR_z=\{v_1,\ldots,v_N\}$ and set $\alpha_j:=\alpha(v_j)$. Then \eqref{eq:identity_generaldensity} becomes
\begin{equation*}
    \sum_{j=1}^N v_j^{-i}(v_j+1)^{y_i+i}\alpha_j
    =
    -u^{-i}(u+1)^{y_i+i},
    \qquad 1\le i\le N.
\end{equation*}
Since
\begin{equation*}
    \det\left[v_j^{-i}(v_j+1)^{y_i+i}\right]_{i,j=1}^N
    =
    \cE_Y(z)\det\left[v_j^{-i}\right]_{i,j=1}^N
\end{equation*}
is not zero, the system has a unique solution.

By Cramer's rule,
\begin{equation*}
    \alpha_k
    =
    \frac{
    \det\left[
    v_j^{-i}(v_j+1)^{y_i+i}\mathbf{1}_{j\ne k}
    -
    u^{-i}(u+1)^{y_i+i}\mathbf{1}_{j=k}
    \right]_{i,j=1}^N
    }{
    \det\left[v_j^{-i}(v_j+1)^{y_i+i}\right]_{i,j=1}^N
    }.
\end{equation*}
Using Definition~\ref{def:cGlam} and \eqref{eq:charY}, we obtain
\begin{equation*}
    \alpha_k
    =
    (v_k-u)\mdpch_Y(v_k,u;z)
    \frac{
    \det\left[
    v_j^{-i}\mathbf{1}_{j\ne k}
    -
    u^{-i}\mathbf{1}_{j=k}
    \right]_{i,j=1}^N
    }{
    \det\left[v_j^{-i}\right]_{i,j=1}^N
    }.
\end{equation*}
Since $q_{z,\rR}(w)=\prod_{j=1}^N(w-v_j)$, the Vandermonde determinant formula gives
\begin{equation*}
    \det\left[
    v_j^{-i}\mathbf{1}_{j\ne k}
    -
    u^{-i}\mathbf{1}_{j=k}
    \right]_{i,j=1}^N
    =
    \frac{v_k^Nq_{z,\rR}(u)}
    {u^Nq'_{z,\rR}(v_k)(v_k-u)}
    \det\left[v_j^{-i}\right]_{i,j=1}^N.
\end{equation*}
Therefore,
\begin{equation}\label{eq:alphak_cramer}
    \alpha_k
    =
    \frac{v_k^Nq_{z,\rR}(u)}
    {u^Nq'_{z,\rR}(v_k)}
    \mdpch_Y(v_k,u;z)
    =
    \frac{v_k^N(v_k+1)^{L-N}}
    {q'_{z,\rR}(v_k)q_{z,\rL}(v_k)}
    \npch_Y(v_k,u;z).
\end{equation}
Since $v_k\in\cR_z$, $q'_{z,\rR}(v_k)q_{z,\rL}(v_k)=q'_z(v_k)$. 
Moreover, $q'_z(w) =\frac{Lw+N}{w(w+1)}w^N(w+1)^{L-N}$. 
Hence \eqref{eq:alphak_cramer} becomes
\begin{equation*}
    \alpha_k
    =
    \frac{v_k(v_k+1)}{Lv_k+N}\npch_Y(v_k,u;z),
\end{equation*}
as desired.
\end{proof}

\subsection{Rewriting the normalized PTASEP energy function}
\label{sec:nefrewrite}

We first express $\ncE_Y(z)$ as an $N\times N$ determinant. We then convert this finite determinant into a Fredholm determinant on $\ell^2(\bbZ_{\le0})$ using a general  identity.

\begin{lm}\label{lm:energy_factorize}
The normalized PTASEP energy function satisfies
\begin{equation}\label{eq:MYentry}
	\ncE_Y(z)=\det(M_Y),
	\qquad
	M_Y(i,j):=
	\sum_{v\in\cR_z}
	v^{i-j-1}(v+1)^{L-N+y_j+j}
	\frac{v(v+1)}{Lv+N},
\end{equation}
for $1\le i,j\le N$.
\end{lm}

\begin{proof}
Write $\cR_z=\{v_1,\ldots,v_N\}$. By the Vandermonde determinant formula,
\begin{equation*}
	\frac{1}{\det[v_i^{-j}]_{i,j=1}^N}
	=
	(-1)^{\frac{N(N-1)}{2}}
	\frac{\prod_{v\in\cR_z}v^N}{\Delta(\cR_z)^2}
	\det[v_i^{j-1}]_{i,j=1}^N.
\end{equation*}
Hence, by \eqref{eq:energyfunc} and the Cauchy--Binet formula,
\begin{equation*}
	\cE_Y(z)
	=
	\frac{(-1)^{\frac{N(N-1)}{2}}\prod_{v\in\cR_z}v^N}
	{\Delta(\cR_z)^2}
	\det\left[
	\sum_{v\in\cR_z}v^{i-j-1}(v+1)^{y_j+j}
	\right]_{i,j=1}^N.
\end{equation*}
It remains to compare the normalization factors. For $v\in\cR_z$, since
$q_z'(v)=q_{z,\rL}(v)q'_{z,\rR}(v)$, we have
\begin{equation*}
	\frac{Lv+N}{v(v+1)}
	=
	\frac{q_{z,\rL}(v)q'_{z,\rR}(v)}
	{v^N(v+1)^{L-N}}.
\end{equation*}
Combining this with
\begin{equation*}
	\prod_{v\in\cR_z}q'_{z,\rR}(v)
	=
	(-1)^{N(N-1)/2}\Delta(\cR_z)^2,
	\quad
	\prod_{u\in\cL_z}(-u)^N
	=
	\prod_{v\in\cR_z}(v+1)^{L-N},
	\quad
	\prod_{v\in\cR_z}q_{z,\rL}(v)=\Delta(\cR_z;\cL_z),
\end{equation*}
from~\cite[(4.42) and (4.52)]{Baik-Liu19}, 
we obtain
\begin{equation*}
	\frac{(-1)^{\frac{N(N-1)}{2}}}{\Delta(\cR_z)^2}
	\left[\prod_{v\in\cR_z}\frac{v^N}{(v+1)^{L-N}}\right]
	\left[\prod_{v\in\cR_z}\frac{Lv+N}{v(v+1)}\right]
	=
	\frac{\Delta(\cR_z;\cL_z)}
	{\prod_{u\in\cL_z}(-u)^N\prod_{v\in\cR_z}(v+1)^{L-N}}.
\end{equation*}
Thus, multiplying by the normalization factor in \eqref{eq:nef} gives \eqref{eq:MYentry}.
\end{proof}

\begin{lm}\label{lm:MY_contour}
Let $Y\in\bbZ_{<0}^N$. Then the matrix entries in \eqref{eq:MYentry} can be written as
\begin{equation*}
	M_Y(i,j)
	=
	\oint_0 v^{i-j-1}(v+1)^{L-N+y_j+j}\frac{\rd v}{2\pi\ri}
	-
	\oint_{\Gamma_\rL}
	u^{i-j-1}
	\frac{z^L(u+1)^{L-N+y_j+j}}
	{u^N(u+1)^{L-N}-z^L}
	\frac{\rd u}{2\pi\ri},
\end{equation*}
where\footnote{Throughout the paper, $\oint_0$ denotes integration over a small positively oriented circle around $0$.} $\Gamma_\rL$ is a positively oriented simple closed contour enclosing $-1$ and all left Bethe roots, while excluding $0$ and all right Bethe roots.
\end{lm}

\begin{proof}
Let $\Gamma_\rR$ be a positively oriented simple closed contour enclosing $0$ and all right Bethe roots, while excluding $-1$ and all left Bethe roots. By the residue theorem,
\begin{equation*}
	M_Y(i,j)
	=
	\oint_{\Gamma_\rR}
	v^{i-j-1}(v+1)^{L-N+y_j+j}
	\frac{v^N(v+1)^{L-N}}
	{v^N(v+1)^{L-N}-z^L}
	\frac{\rd v}{2\pi\ri}.
\end{equation*}
Using
\begin{equation*}
	\frac{v^N(v+1)^{L-N}}
	{v^N(v+1)^{L-N}-z^L}
	=
	1+\frac{z^L}{v^N(v+1)^{L-N}-z^L},
\end{equation*}
and noting 
\begin{equation*}
	v^{i-j-1}(v+1)^{L-N+y_j+j}
	\frac{1}{v^N(v+1)^{L-N}-z^L}
	=
	O(v^{-2})
\end{equation*}
as $v\to\infty$, since $i\le N$ and $y_j\le-1$, the result follows. 
\end{proof}

The next proposition converts an $N\times N$ determinant of the form in Lemma~\ref{lm:MY_contour} into a Fredholm determinant on $\ell^2(\bbZ_{\le0})$. 
After the first version of this paper was written, this identity was extended in~\cite{Liu-Tejaswi26}. In addition, \cite{Petrov26} identified the matrix on the left-hand side of \eqref{eq:orthogonalization} as a special case of a tilted Toeplitz matrix and obtained a different Fredholm determinant representation. 
The functions $\alpha_i(v)$ in \eqref{eq:alphafcnzl} appeared previously, in a slightly different form, in~\cite[Section~2.1.3]{Liu22}.

\begin{prop}\label{prop:orthogonalization}
Let $r_1,r_2\in(0,1)$ satisfy $r_2<1-r_1$. Let $F_1,\ldots,F_N$ be analytic in $\{v\in\bbC:|v|<r_1\}$ and satisfy $F_j(0)=1$ for all $j$. Let $\gamma$ be an oriented contour contained in $\{u\in\bbC:|u+1|<r_2\}$, and let $G_1,\ldots,G_N$ be continuous functions on $\gamma$. Suppose there exists a continuous function $H(v,u)$, analytic in $v$ at $v=0$, such that
\begin{equation}\label{eq:Hdefining}
	\oint_0 v^{-j}F_j(v)H(v,u)\frac{\rd v}{2\pi\ri}
	=
	-G_j(u),
	\qquad 1\le j\le N.
\end{equation}
Then, for every $r\in(r_2,1-r_1)$, the operator $K$ on $\ell^2(\bbZ_{\le0})$ with kernel
\begin{equation}\label{eq:kernel_doublecontour}
	K(x,y)
	:=
	r^{x-y}
	\oint_0\frac{\rd v}{2\pi\ri}
	\int_\gamma\frac{\rd u}{2\pi\ri}
	\frac{u^N(v+1)^{y-1}}{(u+1)^xv^N}
	H(v,u),
	\qquad x,y\in\bbZ_{\le0},
\end{equation}
is trace class, and
\begin{equation}\label{eq:orthogonalization}
	\det\left[
	\oint_0 v^{i-j-1}F_j(v)\frac{\rd v}{2\pi\ri}
	+
	\int_\gamma u^{i-1}G_j(u)\frac{\rd u}{2\pi\ri}
	\right]_{i,j=1}^N
	=
	\det(I+K)_{\ell^2(\bbZ_{\le0})}.
\end{equation}
\end{prop}

\begin{proof}
Denote the matrix on the left-hand side of \eqref{eq:orthogonalization} by $R$, and write $R=P+Q$, where
\begin{equation*}
	P(i,j):=\oint_0 v^{i-j-1}F_j(v)\frac{\rd v}{2\pi\ri},
	\qquad
	Q(i,j):=\int_\gamma u^{i-1}G_j(u)\frac{\rd u}{2\pi\ri}.
\end{equation*}
Since $P$ is upper triangular with diagonal entries equal to $1$, we have
$\det(R)=\det(I+P^{-1}Q)$. 

We first compute $P^{-1}$. Let $\alpha_1,\ldots,\alpha_N$ be functions of the form
\begin{equation}\label{eq:alphafcnzl}
	\alpha_i(v)=\sum_{k=i}^N c_{ik}\frac{v^{k-1}}{F_k(v)}
\end{equation}
satisfying
\begin{equation}\label{eq:lineareq}
	\oint_0\alpha_i(v)\frac{F_j(v)}{v^j}\frac{\rd v}{2\pi\ri}
	=
	\delta_{ij},
	\qquad 1\le i,j\le N.
\end{equation}
The coefficients $c_{ik}$ are determined successively by an upper-triangular system, with $c_{ii}=1$. Hence the functions $\alpha_i$ exist and are unique, and $\alpha_i$ has a zero of order exactly $i-1$ at $v=0$.
We claim that
\begin{equation*}
	P^{-1}(i,j)
	=
	\oint_0\alpha_i(v)v^{-j}\frac{\rd v}{2\pi\ri}.
\end{equation*}
Indeed, for sufficiently small $0<\epsilon<\epsilon'$,
\begin{equation*}
\begin{aligned}
	\sum_{k=1}^N
	\left(\oint_{|v|=\epsilon}\alpha_i(v)v^{-k}\frac{\rd v}{2\pi\ri}\right)
	P(k,j)
	&=
	\oint_{|v|=\epsilon}\frac{\rd v}{2\pi\ri}
	\oint_{|w|=\epsilon'}\frac{\rd w}{2\pi\ri}
	\frac{\alpha_i(v)w^{-j}F_j(w)}{v-w}
	\left(1-\frac{w^N}{v^N}\right).
\end{aligned}
\end{equation*}
The contribution from the first term in the parenthesis vanishes because its integrand is analytic in $v$ at $0$. For the second term, the residue at $w=v$ gives
\begin{equation*}
	-
	\oint_{|v|=\epsilon}\frac{\rd v}{2\pi\ri}
	\oint_{|w|=\epsilon'}\frac{\rd w}{2\pi\ri}
	\frac{\alpha_i(v)w^{N-j}F_j(w)}{(v-w)v^N}
	=
	\oint_{|v|=\epsilon}\alpha_i(v)v^{-j}F_j(v)\frac{\rd v}{2\pi\ri}
	=
	\delta_{ij}.
\end{equation*}
This proves the claim.

Summing over the intermediate index gives
\begin{equation*}
	(P^{-1}Q)(i,j)
	=
	\oint_0\frac{\rd v}{2\pi\ri}
	\int_\gamma\frac{\rd u}{2\pi\ri}
	\frac{\alpha_i(v)G_j(u)}{v-u}
	\left(1-\frac{u^N}{v^N}\right)
	=
	-\oint_0\frac{\rd v}{2\pi\ri}
	\int_\gamma\frac{\rd u}{2\pi\ri}
	\frac{\alpha_i(v)G_j(u)u^N}{(v-u)v^N},
\end{equation*}
where the term without $u^N/v^N$ vanishes because it is analytic in $v$ at $0$.

Choose the $v$-contour small enough that $|v|<r_1$. Since $r\in(r_2,1-r_1)$, we find that $|u+1|<r<|v+1|$ 
for $u\in\gamma$ and $v$ on the $v$-contour. Expanding
\begin{equation*}
	\frac{1}{v-u}
	=
	\frac{1}{v+1}
	\sum_{k=0}^\infty
	\left(\frac{u+1}{r}\right)^k
	\left(\frac{r}{v+1}\right)^k,
\end{equation*}
we obtain
\begin{equation*}
	(P^{-1}Q)(i,j)
	=
	\sum_{x\in\bbZ_{\le0}}B(i,x)A(x,j),
\end{equation*}
where
\begin{equation*}
	B(i,x)
	:=
	r^{-x}
	\oint_0
	\frac{\alpha_i(v)(v+1)^{x-1}}{v^N}
	\frac{\rd v}{2\pi\ri},
	\qquad
	A(x,j)
	:=
	-r^x
	\int_\gamma
	\frac{u^NG_j(u)}{(u+1)^x}
	\frac{\rd u}{2\pi\ri}.
\end{equation*}
Viewing $A:\ell^2\{1,\ldots,N\}\to\ell^2(\bbZ_{\le0})$ and
$B:\ell^2(\bbZ_{\le0})\to\ell^2\{1,\ldots,N\}$ as operators with these kernels, we have $\det(R)=\det(I+BA)$. 
Since $\alpha_1,\ldots,\alpha_N$ and $G_1,\ldots,G_N$ are bounded on the relevant contours, and $r\in(r_2,1-r_1)$, we find that $A$ and $B$ are Hilbert--Schmidt. 
Thus, $AB$ and $BA$ are trace class and 
\begin{equation*}
	\det(R)=\det(I+AB).
\end{equation*}

It remains to evaluate the kernel of $AB$. For $x,y\in\bbZ_{\le0}$,
\begin{equation*}
	(AB)(x,y)
	=
	-r^{x-y}
	\int_\gamma\frac{\rd u}{2\pi\ri}
	\frac{u^N}{(u+1)^x}
	\oint_0\frac{\rd w}{2\pi\ri}
	\frac{(w+1)^{y-1}}{w^N}
	\sum_{i=1}^N\alpha_i(w)G_i(u).
\end{equation*}
Using the formula \eqref{eq:Hdefining} for $G_i(u)$, we obtain
\begin{equation*}
\begin{aligned}
	(AB)(x,y)
	=
	r^{x-y}
	\int_\gamma\frac{\rd u}{2\pi\ri}
	\frac{u^N}{(u+1)^x}
	\oint_0\frac{\rd v}{2\pi\ri}
	H(v,u)
	\sum_{i=1}^N
	\left(
	\oint_0\frac{(w+1)^{y-1}\alpha_i(w)}{w^N}
	\frac{\rd w}{2\pi\ri}
	\right)
	\frac{F_i(v)}{v^i}.
\end{aligned}
\end{equation*}
Define
\begin{equation*}
	E(v)
	:=
	\frac{(v+1)^{y-1}}{v^N}
	-
	\sum_{i=1}^N
	\left(
	\oint_0
	\frac{(w+1)^{y-1}\alpha_i(w)}{w^N}
	\frac{\rd w}{2\pi\ri}
	\right)
	\frac{F_i(v)}{v^i}.
\end{equation*}
The function $E(v)$ has a pole at $v=0$ of order at most $N$. By \eqref{eq:lineareq},
\begin{equation*}
	\oint_0E(v)\alpha_j(v)\frac{\rd v}{2\pi\ri}=0,
	\qquad 1\le j\le N.
\end{equation*}
Since $\alpha_j(v)=v^{j-1}+O(v^j)$ as $v\to0$, it follows that $E$ is analytic at $v=0$. Since $H(v,u)$ is also analytic there,
\begin{equation*}
	\oint_0H(v,u)E(v)\frac{\rd v}{2\pi\ri}=0.
\end{equation*}
Consequently,
\begin{equation*}
	(AB)(x,y)
	=
	r^{x-y}
	\int_\gamma\frac{\rd u}{2\pi\ri}
	\frac{u^N}{(u+1)^x}
	\oint_0\frac{\rd v}{2\pi\ri}
	H(v,u)\frac{(v+1)^{y-1}}{v^N}
	=
	K(x,y).
\end{equation*}
This completes the proof.
\end{proof}

The principal consequence of Proposition~\ref{prop:orthogonalization} is the following.

\begin{cor}\label{lm:MYassumeH}
Let $Y\in\bbZ_{<0}^N$. Let $0<|z|<\bfr_{N,L}$, and let $r_{1,N,L}(z)$ and $r_{2,N,L}(z)$ be as in Definition~\ref{defn:r12}. Fix
\begin{equation*}
	r\in\bigl(r_{2,N,L}(z),1-r_{1,N,L}(z)\bigr),
\end{equation*}
and let $\Gamma_\rL$ be a positively oriented simple closed contour contained in
\begin{equation*}
	\{u\in\bbC:|u+1|<r\}
\end{equation*}
that encloses all left Bethe roots and $-1$. 
Suppose there exists a continuous function $H(v,u)$, analytic in $v$ at $v=0$, such that, for $u\in \Gamma_\rL$, 
\begin{equation}\label{eq:Hdefiningreal}
	\oint_0v^{-j}(v+1)^{L-N+y_j+j}H(v,u)\frac{\rd v}{2\pi\ri}
	=
	\frac{z^L(u+1)^{L-N+y_j+j}}
	{u^j\bigl(u^N(u+1)^{L-N}-z^L\bigr)},
	\qquad 1\le j\le N.
\end{equation}
Then the operator $K$ in \eqref{eq:kernel_doublecontour}, with $\gamma=\Gamma_\rL$, is trace class on $\ell^2(\bbZ_{\le0})$, and
\begin{equation*}
	\ncE_Y(z)=\det(I+K)_{\ell^2(\bbZ_{\le0})}.
\end{equation*}
\end{cor}

\begin{proof}
It follows from applying Proposition~\ref{prop:orthogonalization} with
\begin{equation*}
	F_j(v):=(v+1)^{L-N+y_j+j},
	\qquad
	G_j(u):=
	-\frac{z^L(u+1)^{L-N+y_j+j}}
	{u^j\bigl(u^N(u+1)^{L-N}-z^L\bigr)} 
\end{equation*}
to the matrix in Lemma~\ref{lm:MY_contour}.
\end{proof}

The next subsections solve the systems \eqref{eq:identity_generaldensity} and \eqref{eq:Hdefiningreal} probabilistically.

\subsection{Basic properties of \texorpdfstring{$\rch_Y$ and $\mpch_Y$}{chY and pchY}}
\label{sec:grw}

We prove some basic  properties of $\rch_Y$ and $\mpch_Y$, introduced in Definition~\ref{def:chardef2}.

\begin{lm}\label{prop:char_welldefine}
Let $Y=(y_1,\ldots,y_N)\in\bbZ^N$ with $y_1>y_2>\cdots>y_N$. Then the series defining $\rch_Y(v,u)$ in \eqref{eq:ess_hitting} converges absolutely and uniformly on every compact subset of
\begin{equation*}
	\{(v,u)\in\bbC^2:0<|u+1|<|v+1|\},
\end{equation*}
and $\rch_Y(v,u)$ is analytic there. Moreover, $\rch_Y(v,u)$ is independent of $\rho$.
\end{lm}

\begin{proof}
Consider the random walk started from $G_0=x_0$. If $x_0<y_N+N$, then $G_i\le y_{i+1}$ for $0\le i\le N-1$, since $G$ decreases by at least one at each step and $y_{i+1}+i+1\ge y_N+N$. 
Hence $\mathbf{1}_{\tau<N}=0$. If $x_0>y_1$, then $\tau=0$ and $G_\tau=x_0$. Therefore,
\begin{equation*}
	\rch_Y(v,u)
	=
	\sum_{x_0=y_1+1}^\infty
	\frac{(u+1)^{x_0}}{(v+1)^{x_0+1}}
	+
	\sum_{x_0=y_N+N}^{y_1}
	\left(\frac{u+1}{1-\rho}\right)^{x_0}a_{x_0},
\end{equation*}
where
\begin{equation*}
	a_{x_0}
	:=
	\bbE_{G_0=x_0}\left[
	\frac{(1-\rho)^{G_\tau}}{(v+1)^{G_\tau+1}}
	\left(\frac{-v(1-\rho)}{(v+1)\rho}\right)^\tau
	\mathbf{1}_{\tau<N}
	\right].
\end{equation*}
The first series converges absolutely and uniformly on every compact subset of the stated domain. For each fixed $x_0$, only finitely many paths contribute to $a_{x_0}$, so the second sum is finite and its summands are analytic on the same domain. This proves the convergence and analyticity.

To prove independence of $\rho$, condition on the path up to time $\tau$. Then
\begin{equation*}
\begin{aligned}
	a_{x_0}
	=
	\sum_{k=0}^{N-1}
	\left(\frac{-v(1-\rho)}{(v+1)\rho}\right)^k
	\sum_{x_1,\ldots,x_k\in\bbZ}
	\frac{(1-\rho)^{x_k}}{(v+1)^{x_k+1}}
	\prob_{G_0=x_0}(G_1=x_1,\ldots,G_k=x_k)
	\mathbf{1}_{(x_0,\ldots,x_k)\in A_k},
\end{aligned}
\end{equation*}
where
\begin{equation*}
	\mathbf{1}_{(x_0,\ldots,x_k)\in A_k}
	:=
	\mathbf{1}_{x_0>x_1>\cdots>x_k}
	\left[\prod_{i=0}^{k-1}\mathbf{1}_{x_i\le y_{i+1}}\right]
	\mathbf{1}_{x_k>y_{k+1}}.
\end{equation*}
Since 
\begin{equation*}
	\prob_{G_0=x_0}(G_1=x_1,\ldots,G_k=x_k)
	=
	\rho^k(1-\rho)^{x_0-x_k-k}, 
\end{equation*}
we obtain
\begin{equation*}
	\rch_Y(v,u)
	=
	\sum_{x_0=y_1+1}^\infty
	\frac{(u+1)^{x_0}}{(v+1)^{x_0+1}}
	+
	\sum_{x_0=y_N+N}^{y_1}
	(u+1)^{x_0}
	\sum_{k=0}^{N-1}
	\left(\frac{-v}{v+1}\right)^k
	\sum_{x_1,\ldots,x_k\in\bbZ}
	\frac{\mathbf{1}_{(x_0,\ldots,x_k)\in A_k}}{(v+1)^{x_k+1}},
\end{equation*}
which is independent of $\rho$.
\end{proof}

\begin{lm}\label{prop:mpchwelldefined}
Let $Y\in\Pconf$. The series defining $\mpch_Y(v,u)$ in \eqref{eq:mpchdefn} converges absolutely and uniformly on every compact subset of
\begin{equation}\label{eq:mpch_region}
	\left\{(v,u)\in\bbC^2:
	0<|u+1|<1-\rho<|v+1|,\ 
	|v^N(v+1)^{L-N}|<\rho^N(1-\rho)^{L-N}
	\right\},
\end{equation}
and $\mpch_Y(v,u)$ is analytic there. Moreover, for fixed $(v,u)$ with $0<|u+1|<|v+1|$, the value $\mpch_Y(v,u)$ is independent of any choice of $\rho\in(0,1)$ satisfying
\begin{equation*}
	|u+1|<1-\rho<|v+1|,
	\qquad
	|v^N(v+1)^{L-N}|<\rho^N(1-\rho)^{L-N}.
\end{equation*}
\end{lm}

\begin{proof}
By Lemma~\ref{prop:char_welldefine}, it is enough to consider $\rch_Y^*(v,u)$. Since $\mathbf{1}_{\tau<N}=0$ for $x<y_N+N$,
\begin{equation}\label{eq:Ivua}
	\rch_Y^*(v,u)
	=
	\sum_{x=y_N+N}^\infty
	\left(\frac{u+1}{1-\rho}\right)^x b_x,
	\qquad
	b_x
	:=
	\bbE_{G_0=x}\left[
	\frac{(1-\rho)^{G_\ntau}}{(v+1)^{G_\ntau+1}}
	\left(\frac{-v(1-\rho)}{(v+1)\rho}\right)^\ntau
	\mathbf{1}_{\tau<N,\,\ntau<\infty}
	\right].
\end{equation}

For every $x\ge y_N+N$, partitioning according to the events $\{\ntau=m,G_\ntau=n\}$, and using the bounds $G_\ntau\ge y_{\ntau+1}+1$ and 
$\prob_{G_0=x}(\ntau=m,\ G_\ntau=n)\le1$, we obtain
\begin{equation*}
	|b_x|
	\le
	\sum_{m=N}^\infty
	\left|\frac{v(1-\rho)}{(v+1)\rho}\right|^m
	\sum_{n=y_{m+1}+1}^\infty
	\frac{(1-\rho)^n}{|v+1|^{n+1}}.
\end{equation*}
Set
\begin{equation*}
	\alpha_v:=\frac{1-\rho}{|v+1|},
	\qquad
	\beta_v:=\frac{|v|(1-\rho)}{\rho|v+1|},
	\qquad
	\gamma_v:=\frac{|v|^N|v+1|^{L-N}}{\rho^N(1-\rho)^{L-N}}.
\end{equation*}
On the region \eqref{eq:mpch_region}, we have $\alpha_v<1$ and $\gamma_v<1$. Since
\begin{equation*}
	\sum_{n=y_{m+1}+1}^\infty
	\frac{(1-\rho)^n}{|v+1|^{n+1}}
	=
	\frac{\alpha_v^{y_{m+1}+1}}{|v+1|-(1-\rho)},
\end{equation*}
and $y_{kN+i}=y_i-kL$, regrouping the sum according to $m+1=kN+i$, with $1\le i\le N$ and $k\ge1$, gives
\begin{equation*}
	|b_x|
	\le
	\frac{\alpha_v}{|v+1|-(1-\rho)}
	\left(\sum_{i=1}^N\beta_v^{i-1}\alpha_v^{y_i}\right)
	\left(\sum_{k=1}^\infty\gamma_v^k\right).
\end{equation*}
This bound is independent of $x$ and is locally bounded throughout \eqref{eq:mpch_region}. 
On every compact subset of this region, the geometric series appearing in the bound converge uniformly. 
It follows that $\rch_Y^*(v,u)$, and hence $\mpch_Y(v,u)$, is analytic on \eqref{eq:mpch_region}.

It remains to prove independence of $\rho$. 
Conditioning on the full path up to time $\ntau$ and rearranging the resulting sums, we obtain
\begin{equation*}
\begin{aligned}
	b_{x_0}
	=
	\sum_{k=0}^{N-1}
	\sum_{\ell=N}^\infty
	\sum_{x_1,\ldots,x_\ell\in\bbZ}
	\frac{(1-\rho)^{x_\ell}}{(v+1)^{x_\ell+1}}
	\left(\frac{-v(1-\rho)}{(v+1)\rho}\right)^\ell
	\prob_{G_0=x_0}(G_1=x_1,\ldots,G_\ell=x_\ell)
	\mathbf{1}_{(x_0,\ldots,x_\ell)\in B_{k,\ell}},
\end{aligned}
\end{equation*}
where
\begin{equation*}
	\mathbf{1}_{(x_0,\ldots,x_\ell)\in B_{k,\ell}}
	:=
	\mathbf{1}_{x_0>x_1>\cdots>x_\ell}
	\left[\prod_{i=0}^{k-1}\mathbf{1}_{x_i\le y_{i+1}}\right]
	\mathbf{1}_{x_k>y_{k+1}}
	\left[\prod_{i=N}^{\ell-1}\mathbf{1}_{x_i\le y_{i+1}}\right]
	\mathbf{1}_{x_\ell>y_{\ell+1}}.
\end{equation*}
When we substitute 
\begin{equation*}
	\prob_{G_0=x_0}(G_1=x_1,\ldots,G_\ell=x_\ell)
	=
	\rho^\ell(1-\rho)^{x_0-x_\ell-\ell}
\end{equation*}
into the preceding formula and then into \eqref{eq:Ivua}, all factors involving $\rho$ cancel, yielding
\begin{equation*}
\begin{aligned}
	\rch_Y^*(v,u)
	=
	\sum_{x_0=y_N+N}^\infty
	(u+1)^{x_0}
	\sum_{k=0}^{N-1}
	\sum_{\ell=N}^\infty
	\sum_{x_1,\ldots,x_\ell\in\bbZ}
	\frac{1}{(v+1)^{x_\ell+1}}
	\left(\frac{-v}{v+1}\right)^\ell
	\mathbf{1}_{(x_0,\ldots,x_\ell)\in B_{k,\ell}}.
\end{aligned}
\end{equation*}
Thus $\rch_Y^*(v,u)$ is independent of $\rho$, and so is $\mpch_Y(v,u)$. 
\end{proof}

\subsection{Probabilistic solutions of the linear systems}
\label{sec:chmpchadv}

We show that $\rch_Y$ and $\mpch_Y$ satisfy the two systems of linear equations in Subsections~\ref{sec:npcfchar} and~\ref{sec:nefrewrite}. 

\begin{prop}\label{prop:charlineareqs}
Let $Y=(y_1,\ldots,y_N)\in\bbZ^N$ satisfy $y_1>\cdots>y_N$. Then the characteristic function $\rch_Y(v,u)$ defined in \eqref{eq:ess_hitting} satisfies
\begin{equation}\label{eq:reproducing}
	\oint_0 v^{-i}(v+1)^{y_i+i}\rch_Y(v,u)\frac{\rd v}{2\pi\ri}
	=
	-u^{-i}(u+1)^{y_i+i}
\end{equation}
for $1\le i\le N$ and $0<|u+1|<1$.
\end{prop}

\begin{prop}\label{prop:mpchlineareq}
Let $0<|z|<\bfr_{N,L}$, and let $r_{1,N,L}(z)$ and $r_{2,N,L}(z)$ be as in Definition~\ref{defn:r12}. Let $\rho\in(0,1)$ satisfy
\begin{equation*}
	r_{1,N,L}(z)<\rho<1-r_{2,N,L}(z),
	\qquad
	\rho^N(1-\rho)^{L-N}>|z|^L.
\end{equation*}
Then, for every $Y\in\Pconf$, the PTASEP characteristic function $\mpch_Y(v,u)$ defined in \eqref{eq:mpchdefn}, using the geometric random walk parameter $\rho$, satisfies
\begin{equation}\label{eq:mpchdefineq}
	\sum_{v\in\cR_z}
	v^{-i}(v+1)^{y_i+i}\mpch_Y(v,u)\frac{v(v+1)}{Lv+N}
	=
	-u^{-i}(u+1)^{y_i+i}
\end{equation}
for $1\le i\le N$ and $0<|u+1|<1-\rho$.
\end{prop}

Proposition~\ref{prop:charlineareqs} for $\rch_Y$ was established in~\cite[Theorem~3.4]{Liao-Liu25}. Here we give a different proof. Proposition~\ref{prop:mpchlineareq} for $\mpch_Y$ is new.
The remainder of this subsection is devoted to the proofs of the two propositions.
We begin the proofs by introducing a martingale.

\begin{lm}\label{lem:martingalecomp}
Let $Y=(y_i)_{i\ge1}$, and let $(G_k)_{k\ge 0}$, $\tau$, and $\ntau$ be as in Definition~\ref{def:char}. For $i\ge1$ and $0\le n\le i-1$, define
\begin{equation*}
	X_n^{(i)}
	:=
	(1-\rho)^{G_n}
	\left(\frac{1-\rho}{\rho}\right)^n
	\binom{G_n-y_i-1}{i-n-1}
	\mathbf{1}_{G_n+n\ge y_i+i}.
\end{equation*}
Then the following hold.
\begin{enumerate}[(a)]
\item For each $i$, $(X_n^{(i)})_{n=0}^{i-1}$ is a martingale with respect to the filtration $\mathcal F_n:=\sigma(G_0,\ldots,G_n)$.

\item For every $x\in\bbZ$ and $i\ge1$,
\begin{equation*}
	\bbE_{G_0=x}\left[X_\tau^{(i)}\mathbf{1}_{\tau<i}\right]
	=
	(1-\rho)^x
	\binom{x-y_i-1}{i-1}
	\mathbf{1}_{x\ge y_i+i}.
\end{equation*}

\item For every $x\in\bbZ$ and $i>N$,
\begin{equation*}
	\bbE_{G_0=x}\left[
	X_\ntau^{(i)}\mathbf{1}_{\tau<N,\,\ntau<i}
	\right]
	=
	\bbE_{G_0=x}\left[
	X_\tau^{(i)}\mathbf{1}_{\tau<N}
	\right].
\end{equation*}
\end{enumerate}
\end{lm}

\begin{proof}
For part~\textnormal{(a)}, $X_n^{(i)}$ is clearly $\mathcal F_n$-measurable. Write $G_n=G_{n-1}+Z$, where $Z$ is independent of $\mathcal F_{n-1}$ and
\begin{equation*}
	\prob(Z=a)
	=
	\frac{\rho}{1-\rho}(1-\rho)^{-a}\mathbf{1}_{a<0}.
\end{equation*}
Then
\begin{equation*}
\begin{aligned}
	\bbE\left[X_n^{(i)}\mid\mathcal F_{n-1}\right]
	&=
	\mathbf{1}_{G_{n-1}+n-1\ge y_i+i}
	(1-\rho)^{G_{n-1}}
	\left(\frac{1-\rho}{\rho}\right)^{n-1}
	\sum_{a=y_i+i-n-G_{n-1}}^{-1}
	\binom{G_{n-1}+a-y_i-1}{i-n-1}\\
	&=
	(1-\rho)^{G_{n-1}}
	\left(\frac{1-\rho}{\rho}\right)^{n-1}
	\binom{G_{n-1}-y_i-1}{i-n}
	\mathbf{1}_{G_{n-1}+n-1\ge y_i+i},
\end{aligned}
\end{equation*}
where the last equality follows from the hockey-stick identity. This is $X_{n-1}^{(i)}$.

For part~\textnormal{(b)}, we claim that
\begin{equation*}
	X_\tau^{(i)}\mathbf{1}_{\tau<i}
	=
	X_{\tau\wedge(i-1)}^{(i)}.
\end{equation*}
Indeed, if $\tau\ge i$ and $X_{i-1}^{(i)}\neq0$, then $G_{i-1}\ge y_i+1$, which implies $\tau\le i-1$, a contradiction. The optional stopping theorem applied to the bounded stopping time $\tau\wedge(i-1)$ therefore gives
\begin{equation*}
\begin{aligned}
	\bbE_{G_0=x}\left[X_\tau^{(i)}\mathbf{1}_{\tau<i}\right]
	&=
	\bbE_{G_0=x}\left[X_{\tau\wedge(i-1)}^{(i)}\right]
	=
	X_0^{(i)}
	=
	(1-\rho)^x
	\binom{x-y_i-1}{i-1}
	\mathbf{1}_{x\ge y_i+i}.
\end{aligned}
\end{equation*}

For part~\textnormal{(c)}, suppose $i>N$. Since $\mathbf{1}_{\tau<N}$ is $\mathcal F_N$-measurable,
\begin{equation*}
\begin{aligned}
	\bbE_{G_0=x}\left[
	X_\ntau^{(i)}\mathbf{1}_{\tau<N,\,\ntau<i}
	\right]
	&=
	\bbE_{G_0=x}\left[
	\mathbf{1}_{\tau<N}
	\bbE\left[
	X_\ntau^{(i)}\mathbf{1}_{\ntau<i}\mid\mathcal F_N
	\right]
	\right].
\end{aligned}
\end{equation*}
As in part~\textnormal{(b)},
\begin{equation*}
	X_\ntau^{(i)}\mathbf{1}_{\ntau<i}
	=
	X_{\ntau\wedge(i-1)}^{(i)}.
\end{equation*}
Since $N\le\ntau\wedge(i-1)\le i-1$, conditional optional stopping gives
\begin{equation*}
	\bbE\left[
	X_\ntau^{(i)}\mathbf{1}_{\ntau<i}\mid\mathcal F_N
	\right]
	=
	X_N^{(i)}.
\end{equation*}
Hence
\begin{equation*}
	\bbE_{G_0=x}\left[
	X_\ntau^{(i)}\mathbf{1}_{\tau<N,\,\ntau<i}
	\right]
	=
	\bbE_{G_0=x}\left[
	\mathbf{1}_{\tau<N}X_N^{(i)}
	\right].
\end{equation*}
Finally, optional sampling at the bounded stopping time $\tau\wedge N$ yields
\begin{equation*}
	\bbE_{G_0=x}\left[
	\mathbf{1}_{\tau<N}X_N^{(i)}
	\right]
	=
	\bbE_{G_0=x}\left[
	\mathbf{1}_{\tau<N}X_\tau^{(i)}
	\right],
\end{equation*}
which proves part~\textnormal{(c)}.
\end{proof}

We record two elementary residue evaluations and a Taylor expansion. 
Their proofs are straightforward. 

\begin{lm}\label{lm:intzmo}
For integers $n$ and $\ell$,
\begin{equation*}
	\oint_0\frac{1}{v^n(v+1)^\ell}\frac{\rd v}{2\pi\ri}
	=
	(-1)^{n-1}\binom{\ell+n-2}{n-1}\mathbf{1}_{n>0},
	\qquad
	\oint_{-1}\frac{1}{u^n(u+1)^\ell}\frac{\rd u}{2\pi\ri}
	=
	(-1)^n\binom{\ell+n-2}{\ell-1}\mathbf{1}_{\ell>0}.
\end{equation*}
Moreover, for integers $i\ge1$ and $k$, and for $|u+1|<1$,
\begin{equation}\label{eq:taylor}
	\sum_{x\in\bbZ}
	(-1)^{i-1}(u+1)^x
	\binom{x-k-1}{i-1}
	\mathbf{1}_{x\ge k+i}
	=
	-u^{-i}(u+1)^{k+i}.
\end{equation}
\end{lm}

\begin{cor}\label{cor:szerocomp}
For every $x\in\bbZ$ and $1\le i\le N$,
\begin{equation*}
	\bbE_{G_0=x}\left[
	\frac{(1-\rho)^{G_\tau+\tau}}{(-\rho)^\tau}
	\mathbf{1}_{\tau<N}
	\oint_0
	\frac{1}{v^{i-\tau}(v+1)^{G_\tau+1+\tau-y_i-i}}
	\frac{\rd v}{2\pi\ri}
	\right]
	=
	(-1)^{i-1}(1-\rho)^x
	\binom{x-y_i-1}{i-1}
	\mathbf{1}_{x\ge y_i+i}.
\end{equation*}
\end{cor}

\begin{proof}
By Lemma~\ref{lm:intzmo}, the left-hand side equals
\begin{equation*}
	(-1)^{i-1}
	\bbE_{G_0=x}\left[
	(1-\rho)^{G_\tau}
	\left(\frac{1-\rho}{\rho}\right)^\tau
	\binom{G_\tau-y_i-1}{i-\tau-1}
	\mathbf{1}_{\tau<N,\,\tau<i}
	\right].
\end{equation*}
Since $i\le N$, the two indicators reduce to $\mathbf{1}_{\tau<i}$. Moreover, on $\{\tau<i\}$, 
$G_\tau+\tau \ge y_{\tau+1}+\tau+1 \ge y_i+i$ since the sequence $y_j+j$ is nonincreasing. Thus the preceding expression is $(-1)^{i-1} \bbE_{G_0=x} [ X_\tau^{(i)}\mathbf{1}_{\tau<i}]$ 
and the result follows from Lemma~\ref{lem:martingalecomp}\textnormal{(b)}.
\end{proof}

We now prove the propositions. 

\begin{proof}[Proof of Proposition~\ref{prop:charlineareqs}]
Fix $1\le i\le N$ and $u$ with $0<|u+1|<1$. Choose the circle in \eqref{eq:reproducing} small enough that 
$|u+1|<|v+1|$ for all $v$ on the $v$-contour. 
The normal convergence in Lemma~\ref{prop:char_welldefine} then justifies interchanging the contour integral, the sum over $x$, and the expectation. The left-hand side of \eqref{eq:reproducing} becomes
\begin{equation*}
\begin{aligned}
	\sum_{x\in\bbZ}
	\left(\frac{u+1}{1-\rho}\right)^x
	\bbE_{G_0=x}\left[
	(1-\rho)^{G_\tau}
	\left(\frac{1-\rho}{-\rho}\right)^\tau
	\mathbf{1}_{\tau<N}
	\oint_0
	\frac{1}{v^{i-\tau}(v+1)^{G_\tau+1+\tau-y_i-i}}
	\frac{\rd v}{2\pi\ri}
	\right].
\end{aligned}
\end{equation*}
By Corollary~\ref{cor:szerocomp}, this equals
\begin{equation*}
	\sum_{x\in\bbZ}
	(-1)^{i-1}(u+1)^x
	\binom{x-y_i-1}{i-1}
	\mathbf{1}_{x\ge y_i+i},
\end{equation*}
which is $-u^{-i}(u+1)^{y_i+i}$ by \eqref{eq:taylor}.
\end{proof}

\begin{proof}[Proof of Proposition~\ref{prop:mpchlineareq}]
For every function $f$ analytic on a domain containing
\begin{equation*}
	\left\{v\in\bbC:
	|v^N(v+1)^{L-N}|\le |z|^L,\ \Re(v)>-N/L
	\right\},
\end{equation*}
the residue theorem gives
\begin{equation*}
	\sum_{v\in\cR_z}
	\frac{f(v)}{v^N(v+1)^{L-N}}
	\frac{v(v+1)}{Lv+N}
	=
	\sum_{v\in\cR_z}\frac{f(v)}{q_z'(v)}
	=
	\oint_{\Gamma_\rR}
	\frac{f(v)}{v^N(v+1)^{L-N}-z^L}
	\frac{\rd v}{2\pi\ri},
\end{equation*}
where $\Gamma_\rR$ is a positively oriented simple closed contour enclosing all right Bethe roots. We choose $\Gamma_\rR$ within the analyticity region of $\mpch_Y$ so that
\begin{equation*}
	|z|^L
	<
	|v^N(v+1)^{L-N}|
	<
	\rho^N(1-\rho)^{L-N}
\end{equation*}
for all $v\in \Gamma_\rR$. Expanding the denominator as a geometric series and deforming each term to a small circle around $0$, we obtain
\begin{equation}\label{eq:sumfqint}
	\sum_{v\in\cR_z}
	\frac{f(v)}{v^N(v+1)^{L-N}}
	\frac{v(v+1)}{Lv+N}
	=
	\sum_{k=0}^\infty z^{kL}
	\oint_0
	\frac{f(v)}
	{\bigl(v^N(v+1)^{L-N}\bigr)^{k+1}}
	\frac{\rd v}{2\pi\ri}.
\end{equation}

Fix $1\le i\le N$ and apply \eqref{eq:sumfqint} with
\begin{equation*}
	f(v):=
	v^{N-i}(v+1)^{L-N+y_i+i}\mpch_Y(v,u).
\end{equation*}
By Lemma~\ref{prop:mpchwelldefined} and the assumptions on $\rho$, this function is analytic in the required region. 
Denoting the left-hand side of \eqref{eq:mpchdefineq} by $\mathrm{LHS}$ and using \eqref{eq:mpchdefn}, we obtain
\begin{equation}\label{eq:sum_integral_2}
	\mathrm{LHS}
	=
	\sum_{x\in\bbZ}
	\left(\frac{u+1}{1-\rho}\right)^x
	\sum_{k=0}^\infty
	\bigl(s_k(x)-s_k^*(x)\bigr)z^{kL},
\end{equation}
where
\begin{equation*}
	s_k(x)
	:=
	\oint_0
	\bbE_{G_0=x}\left[
	\frac{(1-\rho)^{G_\tau}}{(v+1)^{G_\tau+1}}
	\left(\frac{-v(1-\rho)}{(v+1)\rho}\right)^\tau
	\mathbf{1}_{\tau<N}
	\right]
	\frac{(v+1)^{y_i+i}}
	{v^i\bigl(v^N(v+1)^{L-N}\bigr)^k}
	\frac{\rd v}{2\pi\ri},
\end{equation*}
and
\begin{equation*}
	s_k^*(x)
	:=
	\oint_0
	\bbE_{G_0=x}\left[
	\frac{(1-\rho)^{G_\ntau}}{(v+1)^{G_\ntau+1}}
	\left(\frac{-v(1-\rho)}{(v+1)\rho}\right)^\ntau
	\mathbf{1}_{\tau<N,\,\ntau<\infty}
	\right]
	\frac{(v+1)^{y_i+i}}
	{v^i\bigl(v^N(v+1)^{L-N}\bigr)^k}
	\frac{\rd v}{2\pi\ri}.
\end{equation*}

For $k=0$, Corollary~\ref{cor:szerocomp} gives
\begin{equation*}
	s_0(x)
	=
	(-1)^{i-1}(1-\rho)^x
	\binom{x-y_i-1}{i-1}
	\mathbf{1}_{x\ge y_i+i}.
\end{equation*}
On the other hand,
\begin{equation*}
\begin{aligned}
	s_0^*(x)
	=
	\bbE_{G_0=x}\left[
	\frac{(1-\rho)^{G_\ntau+\ntau}}{(-\rho)^\ntau}
	\mathbf{1}_{\tau<N,\,\ntau<\infty}
	\oint_0
	\frac{v^{\ntau-i}}
	{(v+1)^{G_\ntau+1+\ntau-y_i-i}}
	\frac{\rd v}{2\pi\ri}
	\right]
	=
	0,
\end{aligned}
\end{equation*}
because $\ntau\ge N\ge i$ makes the integrand analytic at $v=0$.

Now let $k\ge1$ and set $I:=kN+i$. Recall that $Y\in\Pconf$. By periodicity, $y_I=y_i-kL$. Thus,
\begin{equation*}
\begin{aligned}
	s_k(x)
	=
	\bbE_{G_0=x}\left[
	(1-\rho)^{G_\tau}
	\left(\frac{1-\rho}{-\rho}\right)^\tau
	\mathbf{1}_{\tau<N}
	\oint_0
	\frac{1}{v^{I-\tau}(v+1)^{G_\tau+1+\tau-y_I-I}}
	\frac{\rd v}{2\pi\ri}
	\right]
\end{aligned}
\end{equation*}
and
\begin{equation*}
\begin{aligned}
	s_k^*(x)
	=
	\bbE_{G_0=x}\left[
	(1-\rho)^{G_\ntau}
	\left(\frac{1-\rho}{-\rho}\right)^\ntau
	\mathbf{1}_{\tau<N,\,\ntau<I}
	\oint_0
	\frac{1}{v^{I-\ntau}(v+1)^{G_\ntau+1+\ntau-y_I-I}}
	\frac{\rd v}{2\pi\ri}
	\right],
\end{aligned}
\end{equation*}
where $\mathbf{1}_{\ntau<\infty}$ has been replaced by $\mathbf{1}_{\ntau<I}$ because the integral vanishes when $\ntau\ge I$. 
Evaluating the integrals using Lemma~\ref{lm:intzmo}, we obtain
\begin{equation}\label{eq:expectation}
	s_k(x)
	=
	(-1)^{I-1}
	\bbE_{G_0=x}\left[
	(1-\rho)^{G_\tau}
	\left(\frac{1-\rho}{\rho}\right)^\tau
	\binom{G_\tau-y_I-1}{I-\tau-1}
	\mathbf{1}_{\tau<N}
	\right]
\end{equation}
and
\begin{equation}\label{eq:expectation_new}
	s_k^*(x)
	=
	(-1)^{I-1}
	\bbE_{G_0=x}\left[
	(1-\rho)^{G_\ntau}
	\left(\frac{1-\rho}{\rho}\right)^\ntau
	\binom{G_\ntau-y_I-1}{I-\ntau-1}
	\mathbf{1}_{\tau<N,\,\ntau<I}
	\right], 
\end{equation}
where no additional indicator arises in \eqref{eq:expectation} since $\tau<N<I$. 
Moreover, on $\{\tau<N\}$,
\begin{equation*}
	G_\tau+\tau
	\ge
	y_{\tau+1}+\tau+1
	\ge
	y_I+I,
\end{equation*}
and the analogous inequality holds on $\{\ntau<I\}$. 
Thus the indicators appearing in the definition of $X^{(I)}$ in Lemma~\ref{lem:martingalecomp} are equal to one in \eqref{eq:expectation} and \eqref{eq:expectation_new}. 
Consequently, Lemma~\ref{lem:martingalecomp}\textnormal{(c)} implies that 
\begin{equation*}
	s_k(x)=s_k^*(x),
	\qquad k\ge1.
\end{equation*}

Substituting these identities into \eqref{eq:sum_integral_2}, only the term $k=0$ remains. Hence
\begin{equation*}
	\mathrm{LHS}
	=
	(-1)^{i-1}
	\sum_{x\in\bbZ}
	(u+1)^x
	\binom{x-y_i-1}{i-1}
	\mathbf{1}_{x\ge y_i+i}
	=
	-\frac{(u+1)^{y_i+i}}{u^i},
\end{equation*}
where the last equality follows from \eqref{eq:taylor}.
\end{proof}

\subsection{Proofs of Theorems~\ref{thm:characteristic} and~\ref{thm:energy}}
\label{sec:proof_char}

We now prove Theorems~\ref{thm:characteristic} and~\ref{thm:energy}.

\begin{proof}[Proof of Theorem~\ref{thm:characteristic}]
By Lemma~\ref{prop:key_identity}, $v\mapsto\npch_Y(v,u;z)$ is the unique solution of the linear system \eqref{eq:identity_generaldensity}. 
By Proposition~\ref{prop:mpchlineareq}, $v\mapsto\mpch_Y(v,u)$ also satisfies the same system. 
The result follows by uniqueness.
\end{proof}

\begin{proof}[Proof of Theorem~\ref{thm:energy}]
Since $y_1<0$, we have $(y_1,\ldots,y_N)\in\bbZ_{<0}^N$. By Proposition~\ref{prop:charlineareqs}, the function $H(v,u)$ in Corollary~\ref{lm:MYassumeH} can be chosen as
\begin{equation*}
	H(v,u)
	:=
	-\frac{z^L(u+1)^{L-N}}
	{(v+1)^{L-N}\bigl(u^N(u+1)^{L-N}-z^L\bigr)}
	\rch_Y(v,u).
\end{equation*}
Let $r_0$ be a number satisfying $r_{2,N,L}(z)<r_0<1-\rho$ and $(1-r_0)^N r_0^{L-N}>|z|^L$. 
Then, $\Gamma_\rL:=\{u\in\bbC:|u+1|=r_0\}$ encloses all left Bethe roots and $-1$, but excludes $0$ and all right Bethe roots, and 
\begin{equation*}
	|u+1|<1-\rho,
	\qquad
	|u^N(u+1)^{L-N}|
	\ge
	(1-r_0)^N r_0^{L-N}
	>
	|z|^L
\end{equation*}
for $u\in\Gamma_\rL$. 
Corollary~\ref{lm:MYassumeH}, with $r=1-\rho$, implies $\ncE_Y(z)=\det(I+K)_{\ell^2(\bbZ_{\le0})}$, 
where $K$ is trace class and has kernel
\begin{equation}\label{eq:Knotfancy}
	K(x,y)
	=
	-(1-\rho)^{x-y}z^L
	\oint_0\frac{\rd v}{2\pi\ri}
	\oint_{\Gamma_\rL}\frac{\rd u}{2\pi\ri}
	\frac{
	u^N(u+1)^{L-N-x}\rch_Y(v,u)
	}{
	v^N(v+1)^{L-N-y+1}
	\bigl(u^N(u+1)^{L-N}-z^L\bigr)
	},
\end{equation}
for $x,y\in\bbZ_{\le0}$.

We rewrite this kernel in operator form. Inserting \eqref{eq:ess_hitting} into \eqref{eq:Knotfancy}, we obtain
\begin{equation*}
	K(x,y)
	=
	-\sum_{x'\in\bbZ}P(x,x')R(x',y),
\end{equation*}
where
\begin{equation*}
	P(x,x')
	:=
	(1-\rho)^{x-x'-1}
	\oint_{\Gamma_\rL}
	\frac{u^N(u+1)^{L-N-x+x'}}
	{u^N(u+1)^{L-N}-z^L}
	\frac{\rd u}{2\pi\ri}
\end{equation*}
and
\begin{equation*}
\begin{aligned}
	R(x',y)
	:=
	z^L(1-\rho)^{1-y}
	\bbE_{G_0=x'}\left[
	\mathbf{1}_{\tau<N}
	\frac{(-1)^\tau(1-\rho)^{G_\tau+\tau}}{\rho^\tau}
	\oint_0
	\frac{1}
	{v^{N-\tau}(v+1)^{G_\tau+\tau+2+L-N-y}}
	\frac{\rd v}{2\pi\ri}
	\right].
\end{aligned}
\end{equation*}

We first evaluate $R$. By Lemma~\ref{lm:intzmo},
\begin{equation*}
	R(x',y)
	=
	-z^L(-1)^N(1-\rho)^{1-y}
	\bbE_{G_0=x'}\left[
	\mathbf{1}_{\tau<N}
	\frac{(1-\rho)^{G_\tau+\tau}}{\rho^\tau}
	\binom{G_\tau+L-y}{N-\tau-1}
	\right].
\end{equation*}
Let $\widetilde G$ be an independent geometric random walk with the same parameter $\rho$. For $n\ge1$,
\begin{equation}\label{eq:grws}
	\prob(\widetilde G_n=-\ell\mid\widetilde G_0=0)
	=
	\rho^n(1-\rho)^{\ell-n}
	\binom{\ell-1}{n-1}
	\mathbf{1}_{\ell\ge n}.
\end{equation}
Hence, on $\{\tau<N\}$,
\begin{equation*}
\begin{aligned}
	\prob\bigl(
	\widetilde G_{N-\tau}=-L+y-1
	\mid\widetilde G_0=G_\tau
	\bigr)
	=
	\rho^{N-\tau}
	(1-\rho)^{L-y+1+G_\tau-N+\tau}
	\binom{L-y+G_\tau}{N-\tau-1}.
\end{aligned}
\end{equation*}
Using $z^L= (-\rho)^N(1-\rho)^{L-N} \rz$, 
we therefore obtain
\begin{equation*}
\begin{aligned}
	R(x',y)
	&=
	-\rz\,
	\bbE_{G_0=x'}\left[
	\mathbf{1}_{\tau<N}
	\prob\bigl(
	\widetilde G_{N-\tau}=-L+y-1
	\mid\widetilde G_0=G_\tau
	\bigr)
	\right]\\
	&=
	-\rz\,\prob_{G_0=x'}(G_N=-L+y-1,\ \tau<N)
	=
	-\rz\,T_Y(x',y),
\end{aligned}
\end{equation*}
where the second equality follows from the strong Markov property.

We next evaluate $P$. Since $|u^N(u+1)^{L-N}|>|z|^L$ for $u\in\Gamma_\rL$, 
we may expand the denominator as a geometric series and deform each resulting contour to a small circle around $-1$ to obtain 
\begin{equation*}
	\frac{P(x,x')}{(1-\rho)^{x-x'-1}}
	=
	\sum_{k\ge0}z^{kL}
	\oint_{-1}
	\frac{1}
	{u^{kN}(u+1)^{k(L-N)+x-x'}}
	\frac{\rd u}{2\pi\ri}.
\end{equation*}
By Lemma~\ref{lm:intzmo},
\begin{equation*}
	\frac{P(x,x')}{(1-\rho)^{x-x'-1}}
	=
	\delta_{x=x'+1}
	+ 
	\sum_{k\ge1}
	(-1)^{kN}z^{kL}
	\binom{kL+x-x'-2}{k(L-N)+x-x'-1}
	\mathbf{1}_{k(L-N)+x-x'\ge1}.
\end{equation*}
On the other hand, \eqref{eq:grws} gives
\begin{equation*}
	\prob_{G_0=x-1}(G_{kN}=-kL+x')
	=
	\rho^{kN}
	(1-\rho)^{kL+x-x'-1-kN}
	\binom{kL+x-x'-2}{kN-1}
	\mathbf{1}_{k(L-N)+x-x'\ge1}.
\end{equation*}
Thus, by the definition of $S_\rz$,
\begin{equation*}
	P(x,x')
	=
	\delta_{x=x'+1}+S_\rz(x-1,x').
\end{equation*}

Combining the evaluations of $P$ and $R$, we find
\begin{equation*}
	K(x,y)
	=
	-\sum_{x'\in\bbZ}P(x,x')R(x',y)
	=
	\rz\sum_{x'\in\bbZ}
	\bigl(\delta_{x=x'+1}+S_\rz(x-1,x')\bigr)T_Y(x',y).
\end{equation*}
By Definition~\ref{def:chardef2K}, this is the kernel of $-\Ke_{Y,\rz}$ on $\ell^2(\bbZ_{\le0})$. Hence $K=-\Ke_{Y,\rz}$. Since $K$ is trace class, so is $\Ke_{Y,\rz}$, and
\begin{equation*}
	\ncE_Y(z)
	=
	\det(I+K)_{\ell^2(\bbZ_{\le0})}
	=
	\det(I-\Ke_{Y,\rz})_{\ell^2(\bbZ_{\le0})}.
\end{equation*}
This proves the theorem.
\end{proof}

\section{Proof of Theorem~\ref{thm:main} for generic time-ordered parameters under the normalization \texorpdfstring{$y_1^{(L)}=-1$}{y\_1\textasciicircum(L) = -1}}
\label{sec:asymptotics}

In this section, we prove Theorem~\ref{thm:main} under the additional assumptions
\begin{equation*}
    0<\TT_1\le\cdots\le\TT_m,
    \qquad
    (\HH_1,\ldots,\HH_m)\in\dom_+^m(\TT_1,\ldots,\TT_m),
\end{equation*}
and
\begin{equation*}
    y_1^{(L)}=-1 \qquad \text{for all $L\ge 1$.}
\end{equation*}
The time-ordering and threshold conditions correspond to part~\textnormal{(i)} of Definition~\ref{def:F}, where $\Fscaled_\h$ is given directly by the contour integral \eqref{eq:multi_time1}. 
The normalization $y_1^{(L)}=-1$ is removed in Section~\ref{sec:asymptotics_general}. 

\subsection{Setup}
\label{sec:setup}

Fix $0<\rho_-<\rho_+<1$. 
Let $N_L$ be a sequence of integers such that
\begin{equation*}
	\rho_-\le \rho_L \le \rho_+,
	\qquad
	\rho_L:=N_L/L,
\end{equation*}
for all sufficiently large $L$. 
Fix $\h\in\uc_1$, and let $Y_L=(y_i^{(L)})_{i\in\bbZ}\in\Pconfno_{N_L,L}$ satisfy
\begin{equation*}
	y_1^{(L)}=-1
\end{equation*}
for all $L$, and 
\begin{equation*}
	\h_L(\XX)
	:=
	\frac{-y^{(L)}_{-\XX N_L+1}-1+\XX L}
	{\sqrt{\frac{1-\rho_L}{\rho_L}}\,L^{1/2}}
	\to \h(\XX)
	\qquad \text{in $\uc_1$,}
\end{equation*}
where $y_a^{(L)}$ for non-integer $a$ is defined by linear interpolation. 
The last two assumptions imply that $\h$ necessarily satisfies $\h\in \uc_1^0$, i.e.,
\begin{equation*}
	\h(0)=0.
\end{equation*}

Fix an integer $m\ge1$, and let $0<\TT_1\le\cdots\le\TT_m$, $(\HH_1,\ldots,\HH_m)\in\dom_+^m(\TT_1,\ldots,\TT_m)$, and $\XX_1,\ldots,\XX_m\in\bbR$. 
Set, as in \eqref{eq:scaling_1},
\begin{equation}\label{eq:tkii} 
        t_i:=\TT_i\frac{L^{3/2}}{\sqrt{\rho_L(1-\rho_L)}},
        \qquad
        k_i:=\lfloor \rho_L^2t_i-\XX_iN_L \rfloor,
        \qquad
	a_i:=\left\lfloor \XX_iL+(1-2\rho_L)t_i-\HH_i\sqrt{\frac{1-\rho_L}{\rho_L}}\,L^{1/2} \right\rfloor
\end{equation}
for $1\le i\le m$. 
Let $(\sx_k^{(L)}(\cdot))_{k\in\bbZ}$ be distributed as $\mathrm{PTASEP}_{Y_L}(L,N_L)$. 
By Theorem~\ref{thm:PTASEP_multi},
\begin{equation*}
	\prob\left(\bigcap_{i=1}^m \left\{\sx_{k_i}^{(L)}(t_i)\ge a_i\right\}\right)
	=
	\oint\cdots\oint
	\mathscr{C}_{Y_L}(\bmz)\mathscr{D}_{Y_L}(\bmz)\,
	\frac{\rd z_1}{2\pi\ri z_1}\cdots\frac{\rd z_m}{2\pi\ri z_m},
\end{equation*}
where $\bmz=(z_1,\ldots,z_m)$ and the contours are nested circles satisfying
\begin{equation*}
	0<|z_m|<\cdots<|z_1|<\bfr_{N_L,L}
	:=
	\rho_L^{\rho_L}(1-\rho_L)^{1-\rho_L}.
\end{equation*}
It is immediate from the definitions that $\mathscr{C}_{Y_L}(\bmz)$ and $\mathscr{D}_{Y_L}(\bmz)$ depend on $z_1,\ldots,z_m$ only through $z_1^L,\ldots,z_m^L$. 
We change variables to $\rz_1,\ldots,\rz_m$ by
\begin{equation} \label{eq:zL_to_rz}
	z_i^L=(-\rho_L)^{N_L}(1-\rho_L)^{L-N_L}\rz_i.
\end{equation}
Then, as in~\cite{Baik-Liu18, Baik-Liu21},
\begin{equation} \label{eq:F_L_rz}
	\prob\left(\bigcap_{i=1}^m \left\{\sx_{k_i}^{(L)}(t_i)\ge a_i\right\}\right)
	=
	\oint\cdots\oint
	\mathscr{C}_{Y_L}(\bz)\mathscr{D}_{Y_L}(\bz)\,
	\frac{\rd\rz_1}{2\pi\ri\rz_1}\cdots\frac{\rd\rz_m}{2\pi\ri\rz_m},
\end{equation}
where $\bz=(\rz_1,\ldots,\rz_m)$, and we abuse notation by writing 
$\mathscr{C}_{Y_L}(\bz)$ and $\mathscr{D}_{Y_L}(\bz)$ for $\mathscr{C}_{Y_L}(\bmz(\bz))$ and $\mathscr{D}_{Y_L}(\bmz(\bz))$. 
The new contours satisfy $0<|\rz_m|<\cdots<|\rz_1|<1$. 

\begin{prop}\label{prop:CYconvergence}
Under the above assumptions, 
\begin{equation*}
    \lim_{L\to\infty} \mathscr{C}_{Y_L}(\bz)=\rC_\h(\bz)
\end{equation*}
uniformly in $\bz$ on compact subsets of $\{(\rz_1,\ldots,\rz_m)\in\bbC^m:0<|\rz_m|<\cdots<|\rz_1|<1\}$.
\end{prop}

\begin{prop}\label{prop:DYconvergence00}
Under the above assumptions, 
\begin{equation*}
    \lim_{L\to\infty}\mathscr{D}_{Y_L}(\bz)=\rD_\h(\bz)
\end{equation*}
uniformly in $\bz$ on compact subsets of $\{(\rz_1,\ldots,\rz_m)\in\bbC^m:0<|\rz_m|<\cdots<|\rz_1|<1\}$.
\end{prop}

Since
\begin{equation*}
	\oint\cdots\oint
	\rC_\h(\bz)\rD_\h(\bz)\,
	\frac{\rd\rz_1}{2\pi\ri\rz_1}\cdots\frac{\rd\rz_m}{2\pi\ri\rz_m}
	=
	\mathbb{F}_\h(\HH_1,\ldots,\HH_m;(\XX_1,\TT_1),\ldots,(\XX_m,\TT_m)),
\end{equation*}
we conclude the following.

\begin{cor}\label{cor:thm13geny1}
Theorem~\ref{thm:main} holds when $\h\in\uc_1^0$, $0<\TT_1\le\cdots\le\TT_m$, and
$(\HH_1,\ldots,\HH_m)\in\dom_+^m(\TT_1,\ldots,\TT_m)$, 
under the assumption that $y_1^{(L)}=-1$ for all $L$. Moreover, the convergence is locally uniform in $(\XX_i,\TT_i,\HH_i)_{i=1}^m$ on this parameter domain.
\end{cor}

It remains to prove Propositions~\ref{prop:CYconvergence} and~\ref{prop:DYconvergence00}. Their dependence on the initial condition enters through the normalized PTASEP energy function $\ncE_{Y_L}$ and the normalized PTASEP characteristic function $\npch_{Y_L}$, treated in Subsections~\ref{sec:energyconv} and~\ref{sec:convergence_pcf}. The remaining asymptotic ingredients are largely available from~\cite{Baik-Liu16, Baik-Liu19, Baik-Liu21} and are assembled in Subsection~\ref{sec:completionthm12}.
Unless stated otherwise, all estimates and limits below hold locally uniformly in $\bz$ and in the parameters on compact subsets of their respective domains; we will not mention this explicitly each time.

We first record a geometric-random-walk convergence result and transition estimate that will be used repeatedly.

\begin{lm}\label{lm:rw_hitting_limit}
Let $G_k^{(L)}$ be a geometric random walk with parameter $\rho_L$, and define the rescaled random walk\footnote{In this paper, $\mathbb{Z}_+=\{0,1,2,\ldots\}$.}
\begin{equation}\label{eq:rescaled_rw}
    \B_L(\alpha)
    :=
    -\frac{G_{\alpha N_L}^{(L)}+\alpha L}
    {\sqrt{\frac{1-\rho_L}{\rho_L}}\,L^{1/2}},
    \qquad
    \alpha\in\frac{1}{N_L}\mathbb{Z}_{+} .
\end{equation}
Extend $\B_L(\alpha)$ to all $\alpha\ge0$ by linear interpolation, and define
\begin{equation}\label{eq:rescaled_hittingtime}
    \btau_L
    :=
    \inf \big\{\alpha\in\frac1{N_L}\mathbb{Z}_+:
    \B_L(\alpha)\le \h_L(-\alpha) \big\}.
\end{equation}
Let $\sfx_L$ be a sequence satisfying $\sfx_L\to\sfx\ne\h(0)$, and consider $\B_L$ with $\B_L(0)=\sfx_L$. Then, 
\begin{equation}\label{eq:btauweak}
    (\btau_L,\B_L(\btau_L))
    \Longrightarrow
    (\btau,\B(\btau))
\end{equation}
where $\B$ is a standard Brownian motion with $\B(0)=\sfx$, and 
$\btau:=\inf\{t\ge0:\B(t)\le\h(-t)\}$, as in Definition~\ref{def:hittingtimes}. 
Moreover, there exist constants $C,c>0$, depending only on $\rho_-$ and $\rho_+$, such that
\begin{equation}\label{eq:PGscaled}
    \sqrt{\frac{1-\rho_L}{\rho_L}}\,L^{1/2}
    \prob_{\B_L(\theta)=\sfx}
    \left(\B_L(\theta')=\sfx'\right)
    \le
    \frac{C}{\sqrt{\theta'-\theta}}
    \exp\left\{
    -c\left(
    \frac{(\sfx-\sfx')^2}{\theta'-\theta}
    \wedge
    L^{1/2}|\sfx-\sfx'|
    \right)
    \right\}
\end{equation}
for all $L\ge 1$, $0\le\theta<\theta'$ with
$\theta,\theta'\in N_L^{-1}\mathbb{Z}_+$, and $\sfx, \sfx'\in \bbR$. 
\end{lm}

\begin{proof}
The convergence \eqref{eq:btauweak} follows from the hitting-time invariance principle in~\cite[(B.20)]{Matetski-Quastel-Remenik21}; the condition $\sfx\ne\h(0)$ excludes the discontinuity of the hitting-time map at the initial point. The estimate \eqref{eq:PGscaled} follows from a local large/moderate-deviation bound for sums of geometric random variables, uniformly for $\rho_L\in[\rho_-,\rho_+]$, after rewriting it in the rescaled variables.
\end{proof}

\subsection{Convergence of the normalized PTASEP energy function}
\label{sec:energyconv}

We prove the convergence of the normalized PTASEP energy function appearing in $\mathscr{C}_{Y_L}(\bz)$. Since $0<|z|<\bfr_{N_L,L}$ and, by Lemma~\ref{defn:r12},
$r_{1,N_L,L}(z)<\rho_L<1-r_{2,N_L,L}(z)$, 
Theorem~\ref{thm:energy} applies. 
Setting $\rz:=\frac{z^L}{(-\rho_L)^{N_L}(1-\rho_L)^{L-N_L}}$, 
Theorem~\ref{thm:energy} gives
\begin{equation*}
	\ncE_{Y_L}(z)=\det(I-\Ke_{Y_L,\rz})_{\ell^2(\bbZ_{\le0})},
\end{equation*}
where $\Ke_{Y_L,\rz}$ is the operator defined in \eqref{eq:onemrhoK} using the geometric random walk with parameter $\rho_L$. 
Its kernel is
\begin{equation*}
	\Ke_{Y_L,\rz}(x,y)
	=
	-\rz \big(
	T_{Y_L}(x-1,y)
	+
	\sum_{x'\in\bbZ}S_\rz(x-1,x')T_{Y_L}(x',y)
	\big),
	\qquad x,y\in\bbZ_{\le0},
\end{equation*}
where
\begin{equation*}
	S_\rz(x,x')
	:=
	\sum_{k=1}^\infty
	\rz^k\,\prob_{G_0^{(L)}=x}
	\big(G_{kN_L}^{(L)}=-kL+x' \big)
\end{equation*}
and
\begin{equation*}
	T_{Y_L}(x,y)
	:=
	\prob_{G_0^{(L)}=x}
	\big(G_{N_L}^{(L)}=-L+y-1,\ \tau<N_L \big).
\end{equation*}

We scale the kernel and take the limit. 
Set
\begin{equation}
	\rhosr:=\sqrt{\frac{1-\rho_L}{\rho_L}},
\end{equation}
and define the scaled kernels
\begin{equation*}
	\widehat{S}_{L,\rz}(\sfx,\sfx')
	:=
	\rhosr L^{1/2}
	S_\rz \big(
	\lfloor-\rhosr L^{1/2}\sfx\rfloor,
	\lfloor-\rhosr L^{1/2}\sfx'\rfloor
	\big),
\end{equation*}
\begin{equation*}
	\widehat{T}_L(\sfx,\sfy)
	:=
	\rhosr L^{1/2}
	T_{Y_L} \big(
	\lfloor-\rhosr L^{1/2}\sfx\rfloor,
	\lfloor-\rhosr L^{1/2}\sfy\rfloor
	\big),
\end{equation*}
and
\begin{equation} \label{eq:Kerec}
\begin{split}
	\widehat{\Ke}_{L,\rz}(\sfx,\sfy)
	:=
	\mathbf{1}_{\sfx>0}\,
	\rhosr L^{1/2}
	\Ke_{Y_L,\rz} \big(
	\lfloor-\rhosr L^{1/2}\sfx\rfloor,
	\lfloor-\rhosr L^{1/2}\sfy\rfloor
	\big)
	\mathbf{1}_{\sfy>0}.
\end{split}
\end{equation}
By construction,
\begin{equation*}
	\det(I-\Ke_{Y_L,\rz})_{\ell^2(\bbZ_{\le0})}
	=
	\det(I-\widehat{\Ke}_{L,\rz})_{L^2(\bbR)}.
\end{equation*}

We will prove
\begin{equation} \label{eq:energy_conv}
	\lim_{L\to\infty}
	\det(I-\Ke_{Y_L,\rz})_{\ell^2(\bbZ_{\le0})}
	=
	\det(\rI+\limKe_\h(\rz))_{L^2(\bbR)},
\end{equation}
uniformly on compact subsets of $0<|\rz|<1$, where $\limKe_\h(\rz)$ is the operator given in \eqref{eq:limkernel}. 
Sections~\ref{sec:ptwisecon} and~\ref{sec:tracebounds} establish convergence and uniform bounds for $\widehat S_{L,\rz}$ and $\widehat T_L$, respectively. These estimates are combined in Section~\ref{sec:Frdconv} to prove \eqref{eq:energy_conv}. 

\subsubsection{Convergence and bounds for the $\widehat{S}_{L,\rz}$ kernel}
\label{sec:ptwisecon}

From the definitions,
\begin{equation*}
	\widehat{S}_{L,\rz}(\sfx,\sfx')
	=
	\rhosr L^{1/2}
	\sum_{k=1}^\infty \rz^k\,
	\prob_{G_0^{(L)}=0}
	\bigl(
	    G_{kN_L}^{(L)}
	    =
	    -kL+\lfloor-\rhosr L^{1/2}\sfx'\rfloor
	    -\lfloor-\rhosr L^{1/2}\sfx\rfloor
	\bigr).
\end{equation*}

\begin{lm} \label{lm:S_bound}
For every $\sfx,\sfx'\in\bbR$,
\begin{equation}\label{eq:Shatlimit}
     \lim_{L\to\infty}\widehat{S}_{L,\rz}(\sfx,\sfx')
     =
     \sfS_\rz(\sfx,\sfx'),
     \qquad
     \sfS_\rz(\sfx,\sfx')
     =
     \sum_{k\ge1}\rz^k
     \frac{1}{\sqrt{2\pi k}}
     \re^{-\frac{(\sfx-\sfx')^2}{2k}}.
\end{equation}
Furthermore, for every $\delta\in(0,1)$, there exist constants $C,c>0$ such that
\begin{equation}\label{eq:S_upperbound}
        |\widehat{S}_{L,\rz}(\sfx,\sfx')|
        \le
        C\re^{-c|\sfx-\sfx'|}
\end{equation}
for all $\sfx,\sfx'\in\bbR$, $|\rz|\le1-\delta$, and $L\ge 1$.
\end{lm}

\begin{proof}
Since $G_{kN_L}^{(L)}$ has mean $-kL$ and variance $kL\rhosr^2$, the local central limit theorem~\cite[Chapter~VII, Theorem~16]{Petrov75} implies 
\begin{equation*}
\begin{aligned}
    &\lim_{L\to\infty}
    \rhosr L^{1/2}
    \prob_{G_0^{(L)}=0}
    \bigl(
        G_{kN_L}^{(L)}
        =
        -kL+\lfloor-\rhosr L^{1/2}\sfx'\rfloor
        -\lfloor-\rhosr L^{1/2}\sfx\rfloor
    \bigr)
    =
    \frac{1}{\sqrt{2\pi k}}
    \re^{-\frac{(\sfx-\sfx')^2}{2k}}.
\end{aligned}
\end{equation*}
Moreover, \eqref{eq:PGscaled} implies that the left-hand side above is bounded by $C/\sqrt{k}$, uniformly in $L$. 
Since $\sum_{k\ge1}|\rz|^k/\sqrt{k}<\infty$, we obtain \eqref{eq:Shatlimit} by dominated convergence. 

By the estimate \eqref{eq:PGscaled},
\begin{equation*}
        |\widehat{S}_{L,\rz}(\sfx,\sfx')|
        \le
        C\sum_{k\ge1}
        |\rz|^k
        \re^{-c \big(
        \frac{(\sfx-\sfx')^2}{k}
        \wedge
        L^{1/2}|\sfx-\sfx'|
        \big) }.
\end{equation*}
Set $d:=|\sfx-\sfx'|$,  $a:=-\log|\rz|$, and $k_0:=d/\sqrt{a}$. 
For $1\le k\le k_0$, $\frac{d^2}{k}\wedge L^{1/2}d \ge (\sqrt{a}\wedge 1)d$. 
Thus
\begin{equation*}
\begin{aligned}
        |\widehat{S}_{L,\rz}(\sfx,\sfx')|
	&\le
	C\re^{-c(\sqrt{a}\wedge1)d}
        \sum_{1\le k\le k_0}|\rz|^k
        +
	C\sum_{k>k_0}|\rz|^k
	\le
	\frac{C}{1-|\rz|}
	\re^{-c(\sqrt{a}\wedge1)d}
        +
	\frac{C}{1-|\rz|}
	\re^{-\sqrt{a}d}.
\end{aligned}
\end{equation*}
Therefore \eqref{eq:S_upperbound} follows.
\end{proof}

\subsubsection{Convergence and bounds for the $\widehat{T}_{L}$ kernel}
\label{sec:tracebounds}

From the definitions,
\begin{equation}
\label{eq:TY}
    \widehat{T}_L(\sfx,\sfy)
    =
    \rhosr L^{1/2}
    \prob_{\B_L(0)=\sfx_L}
    \big(\B_L(1)=\sfy_L,\ \btau_L\le 1-\frac1{N_L} \big),
\end{equation}
where
\begin{equation*}
    \sfx_L
    =
    \frac{\lfloor-\rhosr L^{1/2}\sfx\rfloor}{-\rhosr L^{1/2}}
    =
    \sfx+O(L^{-1/2}),
    \qquad
    \sfy_L
    =
    \frac{\lfloor-\rhosr L^{1/2}\sfy-1\rfloor}{-\rhosr L^{1/2}}
    =
    \sfy+O(L^{-1/2}).
\end{equation*}
We first prove pointwise convergence.

\begin{lm}\label{lm:T_convergence}
For every $\sfx\ne0$ and $\sfy>0$,
\begin{equation}
    \lim_{L\to\infty}\widehat{T}_L(\sfx,\sfy)=\sfT_\h(\sfx,\sfy), 
\end{equation}
where $\sfT_\h$ is defined in \eqref{eq:limkernel2TT}.
\end{lm}

\begin{proof}
We split the analysis according to whether $\btau$ and $\btau_L$ are close to $1$ or away from it.

Fix $\sfy>0$. 
Since $\h(-1)=\h(0)=0$, there is $\delta_0\in(0,1)$ such that
$\h(-\alpha)\le \sfy/2$ for all $\alpha\in(1-\delta_0,1)$. 
For every $\delta\in(0,\delta_0)$,
\begin{equation*}
    \frac{\prob_{\B(0)=\sfx}
    (\B(1)\in\rd\sfy,\ \btau\in(1-\delta,1))}{\rd\sfy} 
    =
    \iint_{\{(s,\sfx'):\,s\in(1-\delta,1),\,\sfx'\le\h(-s)\}}
    \prob_{\B(0)=\sfx}(\btau\in\rd s,\ \B(s)\in\rd\sfx')\,
    \sfG_{1-s}(\sfx',\sfy),
\end{equation*}
where $\sfG_{1-s}(\sfx',\sfy)$ is the transition density of a standard Brownian motion. 
Choosing $\delta_0$ smaller if necessary, for example so that $\delta_0<\sfy^2/4$, we have
\begin{equation}\label{eq:small_prob1}
    \frac{\prob_{\B(0)=\sfx}
    (\B(1)\in\rd\sfy,\ \btau\in(1-\delta,1))}{\rd\sfy}
    \le
    \sup_{\substack{s\in(1-\delta,1)\\ \sfx'\le\sfy/2}}
    \sfG_{1-s}(\sfx',\sfy)
    \le
    \frac{1}{\sqrt{2\pi\delta}}\re^{-\sfy^2/(8\delta)}
\end{equation}
for all $\delta\in(0,\delta_0)$.

Similarly, since $\h_L\to\h$ in $\uc_1$, we may choose $\delta_0$ so that
$\h_L(-\alpha)\le \sfy/4\le \sfy_L/2$ for all $\alpha\in(1-\delta_0,1)$ and all sufficiently large $L$. 
Considering the possible values of $\btau_L$, for every $\delta\in(0,\delta_0)$,
\begin{equation*}
\begin{split}
	\prob_{\B_L(0)=\sfx_L}
	\big(\B_L(1)=\sfy_L,\ \btau_L\in(1-\delta,1-\frac1{N_L}] \big)
	\le
	\sup_{\substack{k/N_L\in(1-\delta,1)\\ \sfx'\le\sfy_L/2}}
	\prob_{\B_L(k/N_L)=\sfx'} \big(\B_L(1)=\sfy_L \big).
\end{split}
\end{equation*}
Hence, by \eqref{eq:PGscaled}, there are constants $C,c>0$ such that
\begin{equation}\label{eq:small_prob2}
    \limsup_{L\to\infty}
    \rhosr L^{1/2}
    \prob_{\B_L(0)=\sfx_L}
    \big(\B_L(1)=\sfy_L,\ \btau_L\in(1-\delta,1-\frac1{N_L}] \big)
    \le
    \frac{C}{\delta^{1/2}}\re^{-c\sfy^2/\delta}.
\end{equation}
Thus, for any given $\sfy>0$ and $\varepsilon>0$, we may choose $\delta>0$ such that the left-hand sides of both \eqref{eq:small_prob1} and \eqref{eq:small_prob2} are less than $\varepsilon$. 
We also choose $\delta$ so that $\prob(\btau=1-\delta)=0$, which excludes at most countably many values of $\delta$.

It remains to analyze the part with hitting time bounded away from $1$. 
Consider
\begin{equation} \label{eq:PBLtemp}
\begin{aligned}
	&\rhosr L^{1/2}
	\prob_{\B_L(0)=\sfx_L}
	\bigl(\B_L(1)=\sfy_L,\ \btau_L\in[0,1-\delta]\bigr) \\
	&\qquad =
	\int_{[0,1-\delta]\times\bbR}
	\prob_{\B_L(0)=\sfx_L}
	\bigl(\B_L(s)\in\rd b,\ \btau_L\in\rd s \bigr)\,
	f_L(s,b),
\end{aligned}
\end{equation}
where
\begin{equation*}
	f_L(s,b)
	:=
	\rhosr L^{1/2}
	\prob_{\B_L(s)=b}(\B_L(1)=\sfy_L).
\end{equation*}
The uniform local limit theorem implies that 
\begin{equation*}
    f_L(s,b)\to \sfG_{1-s}(b,\sfy)
\end{equation*}
uniformly for $s\in[0,1-\delta]\cap N_L^{-1}\bbZ$ and for $b$ in the corresponding lattice range.
By Lemma~\ref{lm:rw_hitting_limit}, applied with $\sfx_L\to\sfx\ne\h(0)=0$,
\begin{equation*}
    \prob_{\B_L(0)=\sfx_L}
    (\B_L(\btau_L)\in\rd b,\ \btau_L\in\rd s)
    \to
    \prob_{\B(0)=\sfx}
    (\B(\btau)\in\rd b,\ \btau\in\rd s)
\end{equation*}
weakly as measures. 
Since $\sfG_{1-s}(b,\sfy)$ is bounded and continuous on $[0,1-\delta]\times\bbR$, and since $\prob(\btau=1-\delta)=0$, it follows that \eqref{eq:PBLtemp} converges to
\begin{equation*}
	\int_{[0,1-\delta]\times\bbR}
	\prob_{\B(0)=\sfx}
	(\B(s)\in\rd b,\ \btau\in\rd s)\,
	\sfG_{1-s}(b,\sfy)
	=
	\frac{\prob_{\B(0)=\sfx}
	(\B(1)\in\rd\sfy,\ \btau\in[0,1-\delta])}{\rd\sfy}.
\end{equation*}
Combining this convergence with the estimates \eqref{eq:small_prob1} and \eqref{eq:small_prob2}, we obtain
\begin{equation*}
    \limsup_{L\to\infty}
    |\widehat{T}_L(\sfx,\sfy)-\sfT_\h(\sfx,\sfy)|
    <2\varepsilon.
\end{equation*}
Since $\varepsilon>0$ is arbitrary, this proves the pointwise convergence.
\end{proof}

The next uniform estimate is obtained by comparison with the flat initial condition.

\begin{lm}\label{lm:trace_class_prelimit}
Let $\Mx\ge0$ be such that $\max_{\XX\in[0,1]}\h_L(-\XX)\le \Mx$ for all $L$. 
Then there exist constants $c,C>0$ such that
\begin{equation}\label{eq:T_superexponential}
\begin{aligned}
        \widehat{T}_L(\sfx,\sfy)
        &\le
        C e^{-c\left((\sfx-\sfy)^2\wedge L^{1/2}|\sfx-\sfy|\right)}
        \left(1-\bfone_{\sfx\ge\Mx+1,\,\sfy\ge\Mx+1}\right) \\
        &\quad
        + C e^{-c\left((\sfy-\Mx)^2\wedge L^{1/2}(\sfy-\Mx)\right)
        -\frac{(\sfx-\Mx)^2}{16}}
        \bfone_{\sfx\ge\Mx+1,\,\sfy\ge\Mx+1}
\end{aligned}
\end{equation}
for all $\sfx,\sfy\in\bbR$ and all sufficiently large $L$.
\end{lm}

\begin{proof}
If $\sfx<\Mx+1$ or $\sfy<\Mx+1$, then by \eqref{eq:TY} and \eqref{eq:PGscaled}, 
\begin{equation*}
	\widehat{T}_L(\sfx,\sfy)
	\le
	\rhosr L^{1/2}
	\prob_{\B_L(0)=\sfx_L}(\B_L(1)=\sfy_L)
	\le
	C e^{-c\left((\sfx-\sfy)^2\wedge L^{1/2}|\sfx-\sfy|\right)}.
\end{equation*}

Suppose now that $\sfx,\sfy\ge\Mx+1$. 
Introduce the stopping time
\begin{equation*}
        \bsigma_L
        :=
        \inf \big\{\XX\in\frac1{N_L}\bbZ_+:\B_L(\XX)\le\Mx \big\}.
\end{equation*}
This is the hitting time corresponding to the flat boundary $\h_L^{\mathrm{flat}}(\alpha)=\Mx$. 
Since $\h_L(-\XX)\le\Mx$ for $\XX\in[0,1]$, we have $\bsigma_L\le\btau_L$, and hence from \eqref{eq:TY}, 
\begin{equation*}
	\widehat{T}_L(\sfx,\sfy)
	\le
	\rhosr L^{1/2}
	\prob_{\B_L(0)=\sfx_L}
	\left(\B_L(1)=\sfy_L,\ \bsigma_L<1\right).
\end{equation*}
Thus,
\begin{equation*}
	\widehat{T}_L(\sfx,\sfy)
	\le
	\rhosr L^{1/2}
	\sum_{\substack{0\le k<N_L\\ \sfz\in D_{L,k}\cap(-\infty,\Mx]}}
	\prob_{\B_L(0)=\sfx_L}
	\big( \bsigma_L=k/N_L,\ \B_L (k/N_L)=\sfz \big) 
	\prob_{\B_L(k/N_L)=\sfz}(\B_L(1)=\sfy_L),
\end{equation*}
where $D_{L,k}$ denotes the range of $\B_L(k/N_L)$. 
The estimate \eqref{eq:PGscaled} implies that 
\begin{equation*}
	\rhosr L^{1/2}
	\prob_{\B_L(k/N_L)=\sfz}(\B_L(1)=\sfy_L)
	\le
	C\sqrt{\frac{N_L}{N_L-k}}
	e^{-c \big(
	\frac{N_L(\sfy-\sfz)^2}{N_L-k}
	\wedge L^{1/2}|\sfy-\sfz|
	\big)} 
	\le
	C e^{-c \big((\sfy-\Mx)^2\wedge L^{1/2}(\sfy-\Mx) \big)}
\end{equation*}
uniformly for $0\le k<N_L$ and $\sfz\le\Mx$, where we used $\sfy-\sfz\ge1$ in the last step. 
Therefore,
\begin{equation*}
\begin{aligned}
	\widehat{T}_L(\sfx,\sfy)
	&\le
	C e^{-c\left((\sfy-\Mx)^2\wedge L^{1/2}(\sfy-\Mx)\right)}
	\sum_{\substack{0\le k<N_L\\ \sfz\in D_{L,k}\cap(-\infty,\Mx]}}
	\prob_{\B_L(0)=\sfx_L}
	\big( \bsigma_L=k/N_L,\ \B_L (k/N_L)=\sfz \big)  \\
	&=
	C e^{-c\left((\sfy-\Mx)^2\wedge L^{1/2}(\sfy-\Mx)\right)}
	\prob_{\B_L(0)=\sfx_L}(\bsigma_L<1).
\end{aligned}
\end{equation*}
By Lemma~\ref{lm:flatstopping_tail} below, and using $\sfx-\Mx\ge1$ and $\sfx_L=\sfx+O(L^{-1/2})$, we have
\begin{equation*}
	\prob_{\B_L(0)=\sfx_L}(\bsigma_L<1)
	\le
	e^{-\frac{(\sfx-\Mx)^2}{16}}
\end{equation*}
for all sufficiently large $L$. 
This proves the result.
\end{proof}

The following result was used in the proof of the previous lemma. It will be used again in the next subsection.

\begin{lm}\label{lm:flatstopping_tail}
Let $\{\B_L(\XX)\}_{\XX\ge0}$ be the rescaled geometric random walk defined in \eqref{eq:rescaled_rw}, and let
\begin{equation*}
	\bsigma_L
	:=
	\inf \big\{\XX\in\frac1{N_L}\bbZ_+:\B_L(\XX)\le\Mx \big\}.
\end{equation*}
Then, for every $\sfx>\Mx$ and all sufficiently large $L$,
\begin{equation} \label{eq:PBsfl}
        \prob_{\B_L(0)=\sfx}(\bsigma_L<1)
        \le
        e^{-\frac{(\sfx-\Mx)^2}{8}}.
\end{equation}
\end{lm}

\begin{proof}
Let $\{Z_i\}_{i\ge1}$ be i.i.d. random variables with $\prob(Z_i=a)
= \frac{\rho_L}{1-\rho_L}(1-\rho_L)^{-a}\mathbf{1}_{a\in\bbZ_{<0}}$. 
Let $W_i:=Z_i-\bbE Z_i=Z_i+\frac1{\rho_L}$, and set $S_k:=\sum_{j=1}^k W_j$. 
From the definitions of $\B_L$ and $\bsigma_L$, for all sufficiently large $L$,
\begin{equation*}
	\prob_{\B_L(0)=\sfx}(\bsigma_L<1)
	\le
	\prob \big(
	\max_{1\le k<N_L}S_k
	\ge
	\frac12\rhosr L^{1/2}(\sfx-\Mx)
	\big),
\end{equation*}
where the factor $\frac12$ absorbs the rounding error between the discrete and rescaled variables.

By Doob's maximal inequality applied to the exponential submartingale,
\begin{equation*}
	\prob_{\B_L(0)=\sfx}(\bsigma_L<1)
	\le
	e^{-\frac12\lambda\rhosr L^{1/2}(\sfx-\Mx)}
	\bbE\left[e^{\lambda S_{N_L-1}}\right]
	=
	e^{-\frac12\lambda\rhosr L^{1/2}(\sfx-\Mx)+(N_L-1)\psi(\lambda)}
\end{equation*}
for all $\lambda>0$, where
\begin{equation*}
        \psi(\lambda)
        :=
        \log\bbE[e^{\lambda W_1}]
        =
        \frac{1-\rho_L}{\rho_L}\lambda
        -\log\bigl(1-(1-\rho_L)e^{-\lambda}\bigr)
        +\log\rho_L.
\end{equation*}
It is straightforward to check that 
\begin{equation*}
	0\le\psi(\lambda)
	\le
	\frac{1-\rho_L}{2\rho_L^2}\lambda^2,
	\qquad \lambda\ge0.
\end{equation*}
Thus, using $N_L=\rho_LL$ and $\rhosr^2=(1-\rho_L)/\rho_L$,
\begin{equation*}
	\prob_{\B_L(0)=\sfx}(\bsigma_L<1)
	\le
	\exp \big\{
	-\frac12\rhosr L^{1/2}(\sfx-\Mx)\lambda
	+\frac12\rhosr^2L\lambda^2
	\big\}
\end{equation*}
for all $\lambda>0$. 
Optimizing over $\lambda\ge0$ implies \eqref{eq:PBsfl}. 
\end{proof}

We use the following simplification of Lemma~\ref{lm:trace_class_prelimit} below. 

\begin{cor}\label{cor:T_arbitrary_exponential}
Let $\Mx\ge0$ be such that $\max_{\XX\in[0,1]}\h_L(-\XX)\le \Mx$ for all $L$.  
Set $b:=\Mx+1$. 
Then, for every $A>0$, there exist constants $C_A>0$ and $L_A\ge1$ such that
\begin{equation}\label{eq:T_arbitrary_exponential}
	\widehat T_L(\sfx,\sfy)
	\le
	C_A\re^{-A|\sfx-\sfy|}
	\left(1-\mathbf{1}_{\sfx\ge b,\,\sfy\ge b}\right)
	+
	C_A\re^{-A(\sfx+\sfy)}
	\mathbf{1}_{\sfx\ge b,\,\sfy\ge b}
\end{equation}
for all $\sfx,\sfy\in\bbR$ and all $L\ge L_A$.
\end{cor}

\begin{proof}
Let $c>0$ be the constant in \eqref{eq:T_superexponential}. For $s,q,r\ge 0$ with $q\ge r$,
we have $s(s\wedge q)\ge rs-\frac{r^2}{4}$, 
which is clear when $s>q$, and follows by completing the square when $s\le q$. 

Fix $A>0$ and set $r=A/c$. If $L\ge L_A:=r^2$, then
\begin{equation*}
	c (d^2\wedge L^{1/2}d )
	=c\,d (d\wedge L^{1/2} )
	\ge Ad-\frac{A^2}{4c} 
\end{equation*}
for all $d\ge 0$. 
Applying this with $d=|\sfx-\sfy|$ yields the desired bound for the first term in \eqref{eq:T_superexponential}.

Now suppose that $\sfx,\sfy\ge\Mx+1$. Applying the same estimate with $d=\sfy-\Mx$, and using
\begin{equation*}
	\frac{(\sfx-\Mx)^2}{16}
	\ge A(\sfx-\Mx)-4A^2,
\end{equation*}
we obtain
\begin{equation*}
	c \big((\sfy-\Mx)^2\wedge L^{1/2}(\sfy-\Mx) \big)
	+\frac{(\sfx-\Mx)^2}{16}
	\ge
	A(\sfx+\sfy)-\frac{A^2}{4c}  - 4A^2.
\end{equation*}
The result now follows from \eqref{eq:T_superexponential}.
\end{proof}

\subsubsection{Convergence of the determinant}
\label{sec:Frdconv}

From the definition of the rescaled kernel, viewed as a piecewise-constant kernel on $L^2(\bbR)$,
\begin{equation}\label{eq:Ktatexpl}
\begin{aligned}
	-\widehat{\Ke_{L,\rz}}(\sfx,\sfy)
	&=
	\rz\,\mathbf{1}_{\sfx>0}
	\left(
	\widehat{T}_L\big(\sfx+\frac{1}{\rhosr L^{1/2}},\sfy \big) 
	+
	\int_\bbR
	\widehat{S}_{L,\rz} \big(\sfx+\frac{1}{\rhosr L^{1/2}},\sfx' \big)
	\widehat{T}_L(\sfx',\sfy)\,\rd\sfx'
	\right)
	\mathbf{1}_{\sfy>0}.
\end{aligned}
\end{equation}

\begin{lm}\label{lm:Khatesunif}
For every $\delta\in(0,1)$, there exist constants $C,c>0$ such that
\begin{equation}
	|\widehat{\Ke_{L,\rz}}(\sfx,\sfy)|
	\le
	Ce^{-c(\sfx+\sfy)} \mathbf{1}_{\sfx,\sfy>0}
\end{equation}
for all $|\rz|<1-\delta$, $\sfx,\sfy\in \bbR$, and all sufficiently large $L$.
\end{lm}

\begin{proof}
It is enough to consider $\sfx,\sfy>0$. 
Since $\h$ is periodic and upper semicontinuous, it is bounded above. Thus, there exists $\Mx\ge0$ such that
$\max_{\XX\in[0,1]}\h_L(-\XX)\le \Mx$ 
for all sufficiently large $L$. 
Set $b:=\Mx+1$. 

By Corollary~\ref{cor:T_arbitrary_exponential} with $A=1$, there exists $C>0$ such that  
\begin{equation*}
        \mathbf{1}_{\sfx>0}\widehat{T}_L(\sfx,\sfy)\mathbf{1}_{\sfy>0}
        \le
        Ce^{-(\sfx+\sfy)}
\end{equation*}
for all $\sfx,\sfy\in\bbR$ and all sufficiently large $L$.
On the other hand, 
using \eqref{eq:Ktatexpl} and Lemma~\ref{lm:S_bound}, there are $C, c>0$ such that for $|\rz|\le1-\delta$,
\begin{equation*}
\begin{aligned}
	&\int_\bbR
	\left|
	\widehat{S}_{L,\rz} \big(\sfx+\frac{1}{\rhosr L^{1/2}},\sfx' \big)
	\widehat{T}_L(\sfx',\sfy)
	\right|\,\rd\sfx' 
	\le
	C\int_\bbR
	e^{-c|\sfx-\sfx'|}
	\left(
	e^{-c|\sfx'-\sfy|}
	\left(1-\bfone_{\sfx'\ge b,\,\sfy\ge b}\right)
	+
	e^{-c(\sfx'+\sfy)}
	\bfone_{\sfx'\ge b,\,\sfy\ge b}
	\right)\rd\sfx'.
\end{aligned}
\end{equation*}
By Lemma~\ref{lm:intesttmp} below, after rescaling the constants, the last integral is bounded by $C(1+\sfx)e^{-c(\sfx+\sfy)}$. 
Absorbing the factor $1+\sfx$ into the exponential by decreasing $c$, we obtain the result. 
\end{proof}

\begin{lm}\label{lm:intesttmp}
Fix $b>0$, and let
\begin{equation*}
	I(\sfx,\sfy)
	:=
	\int_\bbR e^{-|\sfx-\sfx'|}
	\left(
	e^{-|\sfx'-\sfy|}
	\left(1-\mathbf{1}_{\sfx'\ge b,\,\sfy\ge b}\right)
	+
	e^{-(\sfx'+\sfy)}
	\mathbf{1}_{\sfx'\ge b,\,\sfy\ge b}
	\right)\rd\sfx'.
\end{equation*}
Then
\begin{equation}\label{eq:Iintbd}
	I(\sfx,\sfy)
	\le
	e^{2b}(1+\sfx)e^{-\sfx-\sfy}, 
	\qquad \sfx,\sfy>0. 
\end{equation}
\end{lm}

\begin{proof}
For every $c\ge0$,
\begin{equation}\label{eq:intII}
	\int_\bbR e^{-|\sfx-\sfx'|-|\sfx'-c|}\,\rd\sfx'
	=
	(1+|\sfx-c|)e^{-|\sfx-c|}
	\le
	e^c(1+\sfx)e^{-\sfx}
\end{equation}
for all $\sfx\ge0$.
For $\sfx\ge c$, the inequality is immediate. For $0\le\sfx<c$, it follows from
\begin{equation*}
	1+c-\sfx\le e^{c-\sfx}(1+\sfx).
\end{equation*}

If $\sfy\in[0,b)$, then by \eqref{eq:intII} with $c=\sfy$,
\begin{equation*}
	I(\sfx,\sfy)
	=
	\int_\bbR e^{-|\sfx-\sfx'|-|\sfx'-\sfy|}\,\rd\sfx'
	\le
	e^\sfy(1+\sfx)e^{-\sfx}.
\end{equation*}
Since $\sfy<b$, we have $e^\sfy\le e^{2b-\sfy}$, and hence \eqref{eq:Iintbd} follows in this case.

If $\sfy\ge b$, then
\begin{equation*}
	\begin{aligned}
	I(\sfx,\sfy)
	&=
	\int_{-\infty}^b
	e^{-|\sfx-\sfx'|-|\sfx'-\sfy|}
	\,\rd\sfx'
	+
	\int_b^\infty
	e^{-|\sfx-\sfx'|-(\sfx'+\sfy)}
	\,\rd\sfx'.
	\end{aligned}
\end{equation*}
For $\sfx'<b$,
\begin{equation*}
	e^{-|\sfx'-\sfy|}
	=
	e^{-(\sfy-\sfx')}
	=
	e^b e^{-\sfy}e^{-|\sfx'-b|},
\end{equation*}
while for $\sfx'\ge b$,
\begin{equation*}
	e^{-(\sfx'+\sfy)}
	=
	e^{-b}e^{-\sfy}e^{-|\sfx'-b|}
	\le
	e^b e^{-\sfy}e^{-|\sfx'-b|}.
\end{equation*}
Using \eqref{eq:intII} with $c=b$, we obtain
\begin{equation*}
	\begin{aligned}
	I(\sfx,\sfy)
	&\le
	e^b e^{-\sfy}
	\int_\bbR
	e^{-|\sfx-\sfx'|-|\sfx'-b|}
	\,\rd\sfx'
	\le
	e^{2b}(1+\sfx)e^{-\sfx-\sfy}.
	\end{aligned}
\end{equation*}
This proves \eqref{eq:Iintbd}.
\end{proof}

\begin{cor}\label{lm:energy_kernelconvergence}
For every $\sfx,\sfy>0$,
\begin{equation}
	\lim_{L\to\infty}-\widehat{\Ke_{L,\rz}}(\sfx,\sfy)
	=
	\limKe_\h(\rz)(\sfx,\sfy).
\end{equation}
\end{cor}

\begin{proof}
The first term in \eqref{eq:Ktatexpl} converges to $\rz\,\sfT_\h(\sfx,\sfy)$ by Lemma~\ref{lm:T_convergence}. 
For the second term, Lemma~\ref{lm:S_bound} and Corollary~\ref{cor:T_arbitrary_exponential} provide an integrable dominating function, while Lemmas~\ref{lm:S_bound} and~\ref{lm:T_convergence} give pointwise convergence of the integrand for almost every $\sfx'$. 
Therefore, by the dominated convergence theorem,
\begin{equation*}
	\lim_{L\to\infty}
	\int_\bbR
	\widehat{S}_{L,\rz}
	\big(\sfx+\frac{1}{\rhosr L^{1/2}},\sfx' \big)
	\widehat{T}_L(\sfx',\sfy)\,\rd\sfx' 
	=
	\int_\bbR
	\sfS_{\rz}(\sfx,\sfx')\sfT_\h(\sfx',\sfy)\,\rd\sfx'.
\end{equation*}
Combining this with \eqref{eq:Ktatexpl}, and using $\limKe_\h(\rz)
= \rz\,\mathbf{1}_{>0}(\rI+\sfS_\rz)\sfT_\h\mathbf{1}_{>0}$, 
gives the result.
\end{proof}

We obtain the main result of this subsection.

\begin{prop}
\label{prop:energy_convergence} 
Under the same assumptions as Corollary~\ref{cor:thm13geny1}, 
\begin{equation} \label{eq:enclimtt}
	\lim_{L\to\infty}
	\det\left(I-\Ke_{Y_L,\rz}\right)_{\ell^2(\bbZ_{\le0})}
	=
	\det\left(\rI+\limKe_\h(\rz)\right)_{L^2(\bbR)}
\end{equation}
uniformly for $\rz$ in compact subsets of $\{\rz\in\bbC:0<|\rz|<1\}$.
\end{prop}

\begin{proof}
Fix a compact subset of $\{\rz\in\bbC: |\rz|<1\}$. 
By Hadamard's inequality and Lemma~\ref{lm:Khatesunif}, there is a constant $C>0$, uniform for $\rz$ in this compact set, such that
\begin{equation*}
	\int_{\bbR^n}
	\left|
	\det\left[
	\widehat{\Ke}_{L,\rz}(\sfx_i,\sfx_j)
	\right]_{i,j=1}^n
	\right|
	\rd\sfx_1\cdots\rd\sfx_n
	\le
	n^{n/2}C^n
\end{equation*}
for all $n\ge1$ and all sufficiently large $L$. 
Thus,  Corollary~\ref{lm:energy_kernelconvergence}, together with the compact-uniform convergence in $\rz$, implies the result. 
\end{proof}

\subsection{Uniform bounds and convergence of the PTASEP characteristic function}
\label{sec:convergence_pcf}

By Theorem~\ref{thm:characteristic}, the normalized PTASEP characteristic function can be analyzed using the probabilistic function $\mpch_{Y_L}(v,u)$. We establish a uniform bound in Section~\ref{sec:bound_char} and prove convergence to $\pchlim_\h(\eta,\xi)$ in Section~\ref{sec:converge_char}.

\subsubsection{Bounds for the PTASEP characteristic function}
\label{sec:bound_char}

Recall that we assumed $y_1^{(L)}=-1$ for all $L$.

\begin{lm}\label{prop:uniform_bound_char}
Let $\Mx\in[1,\infty)$ be such that $\max_{\XX\in[0,1]}\h_L(\XX)\le \Mx-1$ for all $L$. 
For every $\delta\in(0,1/2)$, there exist constants $C,c>0$, independent of $L$, such that
\begin{equation} \label{eq:uniform_bd_char}
    L^{-1/2}\left|\mpch_{Y_L}(v_L,u_L)\right|
    \le
    C e^{c\Mx(|\eta_L|+|\xi_L|)}
    \left(\frac{1}{\Re(-\xi_L)}+e^{c|\xi_L|^2}\right)
\end{equation}
for all complex sequences satisfying $\Re(u_L)<-\rho_L<\Re(v_L)$ and
\begin{equation}\label{eq:uvassumption_2}
	\delta\le
	\frac{|u_L|^{N_L}|1+u_L|^{L-N_L}}{\rho_L^{N_L}(1-\rho_L)^{L-N_L}}
	\le1-\delta,
	\qquad 
	\delta\le
	\frac{|v_L|^{N_L}|1+v_L|^{L-N_L}}{\rho_L^{N_L}(1-\rho_L)^{L-N_L}}
	\le1-\delta,
\end{equation}
where
\begin{equation}\label{eq:vunearo}
    \xi_L:=
    \frac{L^{1/2}(u_L+\rho_L)}{\sqrt{\rho_L(1-\rho_L)}},
    \qquad
    \eta_L:=
    \frac{L^{1/2}(v_L+\rho_L)}{\sqrt{\rho_L(1-\rho_L)}}.
\end{equation}
\end{lm}

\begin{proof}
To lighten notation, throughout this proof we write
\begin{equation*}
	u=u_L,\qquad v=v_L,\qquad \xi=\xi_L,\qquad \eta=\eta_L,
	\qquad \rho=\rho_L,\qquad \rho_0=\rhosr, \qquad N=N_L.
\end{equation*}
Recall that $\rho_L\in[\rho_-,\rho_+]$. 

We first claim that
\begin{equation}\label{eq:roots_bound}
    0<c_{\delta,\rho}\le |1+u|<1-\rho<|1+v|\le1+\rho,
    \qquad
    c_{\delta,\rho}:=
    \delta(1-\rho)\left(\frac{\rho}{2-\rho}\right)^{\frac{\rho}{1-\rho}}.
\end{equation}
Indeed, since $\Re(v)>-\rho$, $|1+v|\ge\Re(1+v)>1-\rho$. 
Thus, from  \eqref{eq:uvassumption_2}, 
\begin{equation*}
	|v|^N
	\le
	(1-\delta)\rho^N
	\frac{(1-\rho)^{L-N}}{|1+v|^{L-N}}
	<
	(1-\delta)\rho^N
	<
	\rho^N.
\end{equation*}
Hence, $|v|<\rho$, and so $|1+v|\le1+\rho$. 
For $u$, we have $|u|\ge|\Re(u)|>\rho$. Hence
\begin{equation*}
	|1+u|^{L-N}
	\le
	(1-\delta)(1-\rho)^{L-N}\frac{\rho^N}{|u|^N}
	<
	(1-\rho)^{L-N},
\end{equation*}
which implies $|1+u|<1-\rho$. 
Then $|u|\le1+|1+u|\le2-\rho$. 
Using the lower bound in \eqref{eq:uvassumption_2}, we get
\begin{equation*}
    |1+u|^{L-N}
    \ge
    \delta(1-\rho)^{L-N}\frac{\rho^N}{|u|^N}
    \ge
    \delta(1-\rho)^{L-N}
    \left(\frac{\rho}{2-\rho}\right)^N.
\end{equation*}
Since $\delta^{1/(L-N)}\ge\delta$, this gives the lower bound in \eqref{eq:roots_bound}.

Let $\tau_L$ and $\tau_L^*$ be the stopping times associated with $Y_L$ and the geometric random walk $G^{(L)}$ with parameter $\rho_L$.
By definition,
\begin{equation*}
	\mpch_{Y_L}(v,u)=\rch_{Y_L}(v,u)-\rch_{Y_L}^*(v,u),
\end{equation*}
where
\begin{equation*}
        \rch_{Y_L}(v,u)
        =
        \sum_{x\in\bbZ}
        \left(\frac{u+1}{1-\rho}\right)^x
        \bbE_{G_0^{(L)}=x}\left[
        \frac{(1-\rho)^{G_{\tau_L}^{(L)}}}{(v+1)^{G_{\tau_L}^{(L)}+1}}
        \left(\frac{-v(1-\rho)}{(v+1)\rho}\right)^{\tau_L}
        \mathbf{1}_{\tau_L<N}
        \right]
\end{equation*}
and
\begin{equation*}
\begin{split}
        \rch_{Y_L}^*(v,u)
        &=
        \sum_{x\in\bbZ}
        \left(\frac{u+1}{1-\rho}\right)^x
        \sum_{k=1}^\infty
        \bbE_{G_0^{(L)}=x}\bigg[
        \frac{(1-\rho)^{G_{\tau_L^*}^{(L)}}}{(v+1)^{G_{\tau_L^*}^{(L)}+1}}
        \left(\frac{-v(1-\rho)}{(v+1)\rho}\right)^{\tau_L^*}
        \mathbf{1}_{\tau_L<N,\ \tau_L^*\in[kN,(k+1)N)}
        \bigg].
\end{split}
\end{equation*}
By the assumption \eqref{eq:uvassumption_2} on $v$,
\begin{equation*}
\begin{aligned}
    \left|
    \frac{(1-\rho)^{G_{\tau_L}^{(L)}}}{(1+v)^{G_{\tau_L}^{(L)}+1}}
    \left(\frac{-v(1-\rho)}{(v+1)\rho}\right)^{\tau_L}
    \right|
    & \le
    \frac{(1-\rho)^{G_{\tau_L}^{(L)}+\tau_L/\rho}}
    {|1+v|^{G_{\tau_L}^{(L)}+\tau_L/\rho+1}}
    (1-\delta)^{\tau_L/N}.
\end{aligned}
\end{equation*}
Let $\B_L$ be the rescaled random walk from \eqref{eq:rescaled_rw} and 
$\btau_L$ be defined in \eqref{eq:rescaled_hittingtime}. 
Then, $\btau_L=\tau_L/N$ and
$G_{\tau_L}^{(L)}+\frac{\tau_L}{\rho}
= -\rho_0 L^{1/2}\B_L(\btau_L)$. 
Since $|1+v|\ge1-\rho$ and
$\B_L(\btau_L)\le\h_L(-\btau_L)\le\Mx-1\le\Mx$ by the periodicity of $\h_L$, 
we find that there exists $C>0$ such that, for all sufficiently large $L$,
\begin{equation*}
     \left|
     \frac{(1-\rho)^{G_{\tau_L}^{(L)}+\tau_L/\rho}}
     {(1+v)^{G_{\tau_L}^{(L)}+\tau_L/\rho+1}}
     \right|
     \le \frac1{|1+v|} \left|\frac{1+v}{1-\rho}\right|^{\rho_0 L^{1/2}\Mx}
     \le \frac1{1-\rho_+} \left|\frac{1+v}{1-\rho}\right|^{\rho_0 L^{1/2}\Mx}.
\end{equation*}
Since $|w|\le1+|w-1|\le e^{|w-1|}$ for all $w\in \bbC$, from the definition \eqref{eq:vunearo} of $\eta$,
\begin{equation*}
     \left|\frac{1+v}{1-\rho}\right|^{\rho_0  L^{1/2}\Mx}
     \le e^{\Mx |\eta |}.
\end{equation*}
Hence, using $(1-\delta)^{\tau_L/N}\le 1$, 
\begin{equation*}
	|\rch_{Y_L}(v,u)|
	\le
	\frac{e^{\Mx|\eta|}}{1-\rho_+}
	\sum_{x\in\bbZ}
	\left|\frac{1+u}{1-\rho}\right|^x
	\prob_{G_0^{(L)}=x}(\tau_L<N).
\end{equation*}
The same calculation, using $(1-\delta)^{\tau_L^*/N}\le(1-\delta)^k$ on
$\{\tau_L^*\in[kN,(k+1)N)\}$, gives
\begin{equation*}
\begin{aligned}
	|\rch_{Y_L}^*(v,u)|
	&\le
	\frac{e^{\Mx|\eta|}}{1-\rho_+}
	\sum_{k=1}^\infty(1-\delta)^k
	\sum_{x\in\bbZ}
	\left|\frac{1+u}{1-\rho}\right|^x
	\prob_{G_0^{(L)}=x}(\tau_L<N).
\end{aligned}
\end{equation*}
Thus
\begin{equation} \label{eq:char_bound3}
        |\mpch_{Y_L}(v,u)|
        \le
        C_\delta e^{\Mx|\eta|}
        \sum_{x\in\bbZ}
        \left|\frac{1+u}{1-\rho}\right|^x
        \prob_{G_0^{(L)}=x}(\tau_L<N).
\end{equation}

Recall that $y_1^{(L)}=-1$. 
If $x\ge y_1^{(L)}+1=0$, then $\tau_L=0$, and hence $\prob_{G_0^{(L)}=x}(\tau_L<N)=1$. 
Since $|1+u|< 1-\rho$, the definition \eqref{eq:vunearo} and the upper bound in \eqref{eq:uvassumption_2} imply that there is $C_1>0$ such that
\begin{equation}\label{eq:char_bound1}
	\sum_{x\ge0}
	\left|\frac{1+u}{1-\rho}\right|^x
	\prob_{G_0^{(L)}=x}(\tau_L<N)
	=
	\frac{1}{1-\frac{|1+u|}{1-\rho}}
	\le
	\frac{C_1 L^{1/2}}{\Re(-\xi)}.
\end{equation}

Consider now $x\le-1$. 
Set
\begin{equation*}
	a_L:=\rho_0  L^{1/2},\qquad
	x_L(\sfx):=\lfloor-a_L\sfx\rfloor,\qquad
	\sfx_L:= -\frac{x_L(\sfx)}{a_L}.
\end{equation*}
Then $\sfx\le\sfx_L<\sfx+a_L^{-1}$, and the initial condition $G_0^{(L)}=x_L(\sfx)$ corresponds exactly to $\B_L(0)=\sfx_L$. 
Let 
\begin{equation*}
	\bsigma_L:=\inf \big\{\XX\in\frac1{N}\bbZ_+:\B_L(\XX)\le\Mx \big\}.
\end{equation*}
Since $\B_L(\btau_L)\le \h_L(-\btau_L)\le \Mx$ on $\{\btau_L<1\}$,
we have
\begin{equation*}
	\prob_{G_0=x_L(\sfx)}(\tau_L<N)
	=
	\prob_{\B_L(0)=\sfx_L}(\btau_L<1)
	\le
	\prob_{\B_L(0)=\sfx_L}(\bsigma_L<1).
\end{equation*}
Thus, there is $C_2>0$ such that 
\begin{equation*}
\begin{aligned}
	\sum_{x\le-1}
	\left|\frac{1+u}{1-\rho}\right|^x
	\prob_{G_0^{(L)}=x}(\tau_L<N)
	&\le
	C_2 a_L
	\int_0^\infty
	\left|\frac{1+u}{1-\rho}\right|^{x_L(\sfx)}
	\prob_{\B_L(0)=\sfx_L}(\bsigma_L<1)
	\,\rd\sfx .
\end{aligned}
\end{equation*}
Since $\log\frac{1-\rho}{|1+u|}  \le \frac{1-\rho}{|1+u|} -1 \le \frac{|u+\rho|}{c_{\delta,\rho}}$ by \eqref{eq:roots_bound}, there is $C_3>0$ such that 
\begin{equation*}
	\left|\frac{1+u}{1-\rho}\right|^{x_L(\sfx)}
	\le
	C_3 e^{ \frac{|u+\rho|}{c_{\delta,\rho}} \rho_0 L^{1/2} \sfx }
	= C_3 e^{ \frac{1-\rho}{c_{\delta,\rho}}  |\xi| \sfx }.
\end{equation*}
Thus, there exist $C_4, c_1>0$ such that 
\begin{equation}\label{eq:estimate_tau}
\begin{aligned}
	\sum_{x\le-1}
	\left|\frac{1+u}{1-\rho}\right|^x
	\prob_{G_0^{(L)}=x}(\tau_L<N)  
	& \le
	C_4 L^{1/2} 
	\int_0^\infty
	e^{c_1 |\xi| \sfx}
	\prob_{\B_L(0)=\sfx_L}(\bsigma_L<1)
	\,\rd\sfx .
\end{aligned}
\end{equation}
By Lemma~\ref{lm:flatstopping_tail}, since $\sfx_L\ge\sfx$, there is $c_2>0$ such that 
\begin{equation*}
	\prob_{\B_L(0)=\sfx_L}(\bsigma_L<1)
	\le
	\mathbf{1}_{\sfx\le\Mx+1}
	+
	e^{-c_2(\sfx-\Mx)^2}\mathbf{1}_{\sfx>\Mx+1}.
\end{equation*}
Thus, we obtain from \eqref{eq:estimate_tau}
\begin{equation}\label{eq:estimate_upper}
    \sum_{x\le-1}
    \left|\frac{1+u}{1-\rho}\right|^x
    \prob_{G_0^{(L)}=x}(\tau_L<N) 
    \le
    C L^{1/2}e^{c\Mx|\xi|}
    \left(\frac{1}{|\xi|}+e^{c|\xi|^2}\right)
    \le
    C L^{1/2}e^{c\Mx|\xi|}
    \left(\frac{1}{\Re(-\xi)}+e^{c|\xi|^2}\right).
\end{equation}
Combining \eqref{eq:char_bound3}, \eqref{eq:char_bound1}, and \eqref{eq:estimate_upper}, we obtain \eqref{eq:uniform_bd_char}.
\end{proof}

\subsubsection{Pointwise convergence of the PTASEP characteristic function}
\label{sec:converge_char}

We prove the following pointwise convergence result.

\begin{lm}\label{prop:pch_convergence}
Let $\xi,\eta\in\bbC$ satisfy $\Re(\xi)<0$ and $\arg(\eta)\in(-\frac{\pi}{4},\frac{\pi}{4})$. 
Then
\begin{equation} \label{eq:pchar_convergence}
   \lim_{L\to\infty}
   L^{-1/2}\sqrt{\rho_L(1-\rho_L)}\,\mpch_{Y_L}(v_L,u_L)
   =
   \pchlim_\h(\eta,\xi)
\end{equation}
for all sequences of complex numbers $\{u_L\}$ and $\{v_L\}$ satisfying
\begin{equation}\label{eq:vunearo2}
    \lim_{L\to\infty}
    \frac{L^{1/2}(u_L+\rho_L)}{\sqrt{\rho_L(1-\rho_L)}}=\xi,
    \qquad
    \lim_{L\to\infty}
    \frac{L^{1/2}(v_L+\rho_L)}{\sqrt{\rho_L(1-\rho_L)}}=\eta.
\end{equation}
\end{lm}

\begin{proof}
We first rewrite $\mpch_Y$. 
Let
\begin{equation}\label{eq:phi_v}
    \phi_v(k,y):=
    \left(\frac{1-\rho}{v+1}\right)^{y+1}
    \left(\frac{-v(1-\rho)}{(v+1)\rho}\right)^k.
\end{equation}
By the martingale property and the optional stopping theorem,
\begin{equation*}
\begin{aligned}
	\bbE_{G_0=x}
	\left[\phi_v(\tau,G_\tau)\mathbf{1}_{\tau<N}\right]
	&=
	\bbE_{G_0=x}
	\left[\phi_v(N,G_N)\mathbf{1}_{\tau<N}\right] \\
	&=
	\bbE_{G_0=x}
	\left[\phi_v(N,G_N)\mathbf{1}_{\tau<N,\ G_N\le y_{N+1}}\right] 
	+
	\bbE_{G_0=x}
	\left[\phi_v(N,G_N)\mathbf{1}_{\tau<N,\ G_N>y_{N+1}}\right].
\end{aligned}
\end{equation*}
On the other hand,
\begin{equation*}
\begin{aligned}
	&\bbE_{G_0=x}
	\left[\phi_v(\ntau,G_\ntau)\mathbf{1}_{\tau<N,\ \ntau<\infty}\right] \\
	&\quad =
	\bbE_{G_0=x}
	\left[\phi_v(\ntau,G_\ntau)\mathbf{1}_{\tau<N,\ \ntau<\infty,\ G_N\le y_{N+1}}\right] 
	+
	\bbE_{G_0=x}
	\left[\phi_v(\ntau,G_\ntau)\mathbf{1}_{\tau<N,\ \ntau<\infty,\ G_N>y_{N+1}}\right].
\end{aligned}
\end{equation*}
On the event $\{G_N>y_{N+1}\}$, we have $\ntau=N$. 
Thus the last terms of the above two displayed equations are the same. 
Taking the difference and summing over $x$ with the prefactor in Definition~\ref{def:chardef2}, we obtain
\begin{equation*}
\begin{aligned}
	\mpch_Y(v,u)
	&=
	\frac1{1-\rho}
	\sum_{x\in\bbZ}
	\left(\frac{u+1}{1-\rho}\right)^x
	\bbE_{G_0=x}
	\left[\phi_v(N,G_N)\mathbf{1}_{\tau<N,\ G_N\le y_{N+1}}\right] \\
	&\quad -
	\frac1{1-\rho}
	\sum_{x\in\bbZ}
	\left(\frac{u+1}{1-\rho}\right)^x
	\bbE_{G_0=x}
	\left[\phi_v(\ntau,G_\ntau)
	\mathbf{1}_{\tau<N,\ \ntau<\infty,\ G_N\le y_{N+1}}\right].
\end{aligned}
\end{equation*}
Conditioning on the possible values of $G_N$, and using $y_{N+1}=y_1-L$, we find 
\begin{equation*}
	\bbE_{G_0=x}
	\left[\phi_v(N,G_N)\mathbf{1}_{\tau<N,\ G_N\le y_{N+1}}\right] 
	=
	\sum_{y\le y_1+1}
	\phi_v(N,-L+y-1)
	\prob_{G_0=x}
	\left(G_N=-L+y-1,\ \tau<N\right).
\end{equation*}
Similarly, by the Markov property and periodicity,
\begin{equation*}
\begin{aligned}
	&\bbE_{G_0=x}
	\left[\phi_v(\ntau,G_\ntau)
	\mathbf{1}_{\tau<N,\ \ntau<\infty,\ G_N\le y_{N+1}}\right] \\
	&\quad =
	\sum_{y\le y_1+1}
	\prob_{G_0=x}
	\left(G_N=-L+y-1,\ \tau<N\right)
	\bbE_{\widehat G_0=y-1}
	\left[
	\phi_v(N+\widehat\tau,-L+\widehat G_{\widehat\tau})
	\mathbf{1}_{\widehat\tau<\infty}
	\right],
\end{aligned}
\end{equation*}
where $\widehat G$ is an independent geometric random walk started from $y-1$, and
\begin{equation*}
	\widehat\tau:=\min\{k\ge0:\widehat G_k>y_{k+1}\}.
\end{equation*}
Thus, inserting the formula \eqref{eq:phi_v}, we obtain
\begin{equation} \label{eq:pchysuma22}
    (1-\rho)\mpch_Y(v,u)
    =
    \frac{v^N(1+v)^{L-N}}{(-\rho)^N(1-\rho)^{L-N}}
    \left[\mpch_Y^{<}(v,u)-\mpch_Y^{>}(v,u)\right],
\end{equation}
where
\begin{equation*}
\begin{aligned}
    \mpch_Y^{<}(v,u)
    &:=
    \sum_{x\in\bbZ}\sum_{y\le y_1+1}
    \left(\frac{u+1}{1-\rho}\right)^x
    \left(\frac{1-\rho}{v+1}\right)^y
    \prob_{G_0=x}
    \left(G_N=-L+y-1,\ \tau<N\right)
\end{aligned}
\end{equation*}
and
\begin{equation*}
\begin{aligned}
    \mpch_Y^{>}(v,u)
    &:=
    \sum_{x\in\bbZ}\sum_{y\le y_1+1}
    \left(\frac{u+1}{1-\rho}\right)^x
    \prob_{G_0=x}
    \left(G_N=-L+y-1,\ \tau<N\right) 
    \bbE_{\widehat G_0=y-1}
    \left[
    \phi_v(\widehat\tau,\widehat G_{\widehat\tau})
    \mathbf{1}_{\widehat\tau<\infty}
    \right].
\end{aligned}
\end{equation*}

We now evaluate the above when $Y=Y_L$, $u=u_L$, $v=v_L$, $\rho=\rho_L$, and $N=N_L$, and take the limit as $L\to\infty$. 
From the assumption \eqref{eq:vunearo2},
\begin{equation*}
    \lim_{L\to\infty}
    \frac{v_L^{N_L}(1+v_L)^{L-N_L}}
    {(-\rho_L)^{N_L}(1-\rho_L)^{L-N_L}}
    =
    \re^{-\frac12\eta^2}.
\end{equation*}

Set
\begin{equation*}
	a_L:=\rhosr L^{1/2},\qquad
	x_L(\sfx):=\lfloor-a_L\sfx\rfloor,\qquad
	y_L(\sfy):=\lfloor-a_L\sfy\rfloor+1,
\end{equation*}
and
\begin{equation*}
	\sfx_L:=-\frac{x_L(\sfx)}{a_L} = -\frac{\lfloor-a_L\sfx\rfloor}{a_L},
	\qquad
	\sfy_L:=-\frac{y_L(\sfy)-1}{a_L} = -\frac{\lfloor-a_L\sfy\rfloor}{a_L}.
\end{equation*}
Let $\B_L$ be the rescaled random walk from \eqref{eq:rescaled_rw}, let $\btau_L$ be as in \eqref{eq:rescaled_hittingtime}, and let 
$\widehat{\B}_L$ be an independent copy of $\B_L$ started from $\sfy_L$, with $\widehat{\btau}_L$ denoting its hitting time.

Writing the series as Riemann sums, we obtain
\begin{equation} \label{eq:pchgrthn0}
\begin{aligned}
    \mpch_{Y_L}^{<}(v_L,u_L)
    &=
    a_L^2
    \int_\bbR \rd\sfx
    \int_{\sfy>\h_L(-1)}\rd\sfy\,
    \left(\frac{u_L+1}{1-\rho_L}\right)^{x_L(\sfx)}
    \left(\frac{1-\rho_L}{v_L+1}\right)^{y_L(\sfy)} \\
    &\qquad\qquad \times
    \prob_{\B_L(0)=\sfx_L}
    \left(\B_L(1)=\sfy_L,\ \btau_L<1\right),
\end{aligned}
\end{equation}
and
\begin{equation} \label{eq:pchgrthn}
\begin{aligned}
    \mpch_{Y_L}^{>}(v_L,u_L)
    &=
    a_L^2
    \int_\bbR \rd\sfx
    \int_{\sfy>\h_L(-1)}\rd\sfy\,
    \left(\frac{u_L+1}{1-\rho_L}\right)^{x_L(\sfx)}
    \prob_{\B_L(0)=\sfx_L}
    \left(\B_L(1)=\sfy_L,\ \btau_L<1\right) \\
    &\qquad\qquad \times
    \bbE_{\widehat{\B}_L(0)=\sfy_L}
    \left[
    \psi_L(\widehat{\btau}_L,\widehat{\B}_L(\widehat{\btau}_L))
    \mathbf{1}_{\widehat{\btau}_L<\infty}
    \right].
\end{aligned}
\end{equation}
Here, since $y_1^{(L)}=-1$, the condition $y\le y_1^{(L)}+1=0$ corresponds to $\sfy>\h_L(-1)=0$ in the Riemann-sum notation above, up to endpoints of mesh size $a_L^{-1}$. 
The function inside the expectation in \eqref{eq:pchgrthn} is 
\begin{equation*}
    \psi_L(s,\sfz)
    :=
    \left(\frac{1-\rho_L}{1+v_L}\right)^{
    \left\lfloor -\rhosr L^{1/2}\sfz-\frac{L}{N_L}\lfloor N_Ls\rfloor\right\rfloor+1}
    \left(\frac{-v_L(1-\rho_L)}{(v_L+1)\rho_L}\right)^{\lfloor N_Ls\rfloor}.
\end{equation*}
Equivalently,
\begin{equation*}
    \psi_L(s,\sfz)
    =
    \left(\frac{1-\rho_L}{1+v_L}\right)^{-\rhosr L^{1/2}\sfz+O(1)}
    \left(
    \frac{v_L^{N_L}(1+v_L)^{L-N_L}}
    {(-\rho_L)^{N_L}(1-\rho_L)^{L-N_L}}
    \right)^{s+O(L^{-1})},
\end{equation*}
where the $O(1)$ error in the first exponent and the $O(L^{-1})$ error in the second exponent are uniform in $s,\sfz$. 
Consequently, $\psi_L(s,\sfz)\to \re^{\eta\sfz-\frac12\eta^2s}$ uniformly on compact subsets of $\bbR_+\times\bbR$.

For every $\sfx,\sfy\in\bbR$, \eqref{eq:vunearo2} implies
\begin{equation*}
    \lim_{L\to\infty}
    \left(\frac{u_L+1}{1-\rho_L}\right)^{x_L(\sfx)}
    =
    \re^{-\xi\sfx},
    \qquad
    \lim_{L\to\infty}
    \left(\frac{1-\rho_L}{v_L+1}\right)^{y_L(\sfy)}
    =
    \re^{\eta\sfy}.
\end{equation*}
By Lemma~\ref{lm:T_convergence},
\begin{equation*}
    \lim_{L\to\infty}
    a_L\,
    \prob_{\B_L(0)=\sfx_L}
    \left(\B_L(1)=\sfy_L,\ \btau_L<1\right)
    =
    \frac{\prob_{\B(0)=\sfx}(\B(1)\in \rd\sfy,\ \btau<1)}{\rd\sfy}
\end{equation*}
for all $\sfx\neq0$ and $\sfy>\h(-1)$. 
The exceptional set $\{\sfx=0\}$ has Lebesgue measure zero and is irrelevant for the dominated convergence argument below.

Moreover,
\begin{equation*}
  \lim_{L\to\infty}
  \bbE_{\widehat{\B}_L(0)=\sfy_L}
  \left[
  \psi_L(\widehat{\btau}_L,\widehat{\B}_L(\widehat{\btau}_L))
  \mathbf{1}_{\widehat{\btau}_L<\infty}
  \right]
  =
  \bbE_{\widehat{\B}(0)=\sfy}
  \left[
  \re^{\eta\widehat{\B}(\widehat{\btau})-\frac12\eta^2\widehat{\btau}}
  \mathbf{1}_{\widehat{\btau}<\infty}
  \right]
\end{equation*}
by Lemma~\ref{lm:rw_hitting_limit}, the compact-uniform convergence of $\psi_L$, and uniform integrability. The latter follows from the estimates in the proof of Lemma~\ref{prop:uniform_bound_char}: on the hitting event, the spatial factor is uniformly bounded by $Ce^{c\Mx|\eta|}$, while the time factor decays geometrically because $\Re(\eta^2)>0$.

Equation \eqref{eq:TY} and Corollary~\ref{cor:T_arbitrary_exponential} imply 
that the integrands of \eqref{eq:pchgrthn0} and \eqref{eq:pchgrthn} are bounded, for all sufficiently large $L$, by an integrable function on $\bbR\times\bbR_+$. 
Thus by the dominated convergence theorem, 
\begin{equation*}
    \lim_{L\to\infty}
    \frac{1-\rho_L}{a_L}\mpch_{Y_L}(v_L,u_L)
    =
    \pchlim_\h^{<}(\eta,\xi)-\pchlim_\h^{>}(\eta,\xi),
\end{equation*}
where
\begin{equation*}
	\pchlim_\h^{<}(\eta,\xi)
	:=
	\re^{-\frac12\eta^2}
	\int_\bbR \rd\sfx
	\int_{\sfy>\h(-1)}\rd\sfy\,
	\re^{\eta\sfy-\xi\sfx}
	\frac{\prob_{\B(0)=\sfx}(\B(1)\in\rd\sfy,\ \btau<1)}{\rd\sfy}
\end{equation*}
and
\begin{equation} \label{eq:pclimlimi}
\begin{aligned}
    \pchlim_\h^{>}(\eta,\xi)
    &:=
    \re^{-\frac12\eta^2}
    \int_\bbR \re^{-\sfx\xi}\,\rd\sfx
    \int_{\sfy>\h(-1)}\rd\sfy\,
    \frac{\prob_{\B(0)=\sfx}(\B(1)\in\rd\sfy,\ \btau<1)}{\rd\sfy} \\
    &\qquad\qquad \times
    \bbE_{\widehat{\B}(0)=\sfy}
    \left[
    \re^{\eta\widehat{\B}(\widehat{\btau})-\frac12\eta^2\widehat{\btau}}
    \mathbf{1}_{\widehat{\btau}<\infty}
    \right].
\end{aligned}
\end{equation}

It remains to identify this difference with $\pchlim_\h(\eta,\xi)$. 
In the definition \eqref{eq:pch_lim} of $\pchlim_\h(\eta,\xi)$, the contribution from the event 
$\{\btau<1,\ \B(1)\le \h(-1)\}$ cancels between the two terms, since on this event $\bntau=1$. 
Using the optional stopping theorem for the exponential martingale on the remaining part, we get
\begin{equation*}
\begin{aligned}
    \pchlim_\h(\eta,\xi)
    &=
    \int_\bbR \rd s\, e^{-s\xi}
    \bbE_{\B(0)=s}
    \left[
    e^{\eta\B(1)-\frac12\eta^2}
    \mathbf{1}_{\btau<1,\ \B(1)>\h(-1)}
    \right] \\
    &\quad -
    \int_\bbR \rd s\, e^{-s\xi}
    \bbE_{\B(0)=s}
    \left[
    e^{\eta\B(\bntau)-\frac12\eta^2\bntau}
    \mathbf{1}_{\btau<1,\ \bntau<\infty,\ \B(1)>\h(-1)}
    \right].
\end{aligned}
\end{equation*}
The first integral is $\pchlim_\h^{<}(\eta,\xi)$. 
For the second integral, use the strong Markov property at time $1$ on the event
$\{\btau<1,\ \B(1)>\h(-1)\}$. 
Since $\h$ is $1$-periodic, the shifted post-time-$1$ hitting problem has the same boundary $\h(-\cdot)$, and the second integral equals
\begin{equation*}
\begin{aligned}
	\re^{-\frac12\eta^2}
	\int_\bbR e^{-\sfx\xi}\,\rd\sfx
	\int_{\sfy>\h(-1)}
	\frac{\prob_{\B(0)=\sfx}(\B(1)\in\rd\sfy,\ \btau<1)}{\rd\sfy}
	\bbE_{\widehat{\B}(0)=\sfy}
	\left[
	e^{\eta\widehat{\B}(\widehat{\btau})-\frac12\eta^2\widehat{\btau}}
	\mathbf{1}_{\widehat{\btau}<\infty}
	\right]\rd\sfy
	=
	\pchlim_\h^{>}(\eta,\xi).
\end{aligned}
\end{equation*}
Thus $\pchlim_\h(\eta,\xi)=\pchlim_\h^{<}(\eta,\xi)-\pchlim_\h^{>}(\eta,\xi)$. 
Finally, since $\frac{1-\rho_L}{a_L} = L^{-1/2}\sqrt{\rho_L(1-\rho_L)}$, 
the convergence \eqref{eq:pchar_convergence} follows.
\end{proof}

\subsection{Proof of Corollary~\ref{cor:thm13geny1}}
\label{sec:completionthm12}

We now combine the preceding convergence results with the initial-condition-independent asymptotics from~\cite{Baik-Liu16,Baik-Liu19,Baik-Liu21} to prove Propositions~\ref{prop:CYconvergence} and~\ref{prop:DYconvergence00}.

\subsubsection{Proof of Proposition~\ref{prop:CYconvergence}} 
\label{sec:CYconvergence}

\begin{proof}[Proof of Proposition~\ref{prop:CYconvergence}] 
In the formula \eqref{eq:CPTASEP} for $\mathscr{C}_Y(\bmz)$, the only term that depends on the initial condition is the energy function $\cE_Y(z_1)$. 
Thus, using \eqref{eq:nef}, for any other initial condition $Y'$ with the same $N$ and $L$,
\begin{equation*}
\begin{aligned}
	\mathscr{C}_Y(\bz)
	=
	\frac{\cE_Y(z_1)}{\cE_{Y'}(z_1)}\mathscr{C}_{Y'}(\bz)
	=
	\frac{\ncE_Y(z_1)}{\ncE_{Y'}(z_1)}\mathscr{C}_{Y'}(\bz).
\end{aligned}
\end{equation*}
Let $Y'$ be the periodic step initial condition, which we denote by $\mathrm{step}_L$. 
For this initial condition, it was proved in~\cite[(6.20)]{Baik-Liu19} that
\begin{equation*}
	\lim_{L\to\infty}\mathscr{C}_{\mathrm{step}_L}(\bz)
	=
	e^{2B(\rz_1)}\rC(\bz),
\end{equation*}
where $\rC(\bz)$ is defined in \eqref{eq:Climnoh}. 
Thus, by Proposition~\ref{prop:energy_convergence},
\begin{equation*}
\begin{aligned}
	\lim_{L\to\infty}\mathscr{C}_{Y_L}(\bz)
	&=
	\frac{
	\det(\rI+\limKe_\h(\rz_1))_{L^2(\bbR)}
	}{
	\det(\rI+\limKe_{\mathrm{pnw}}(\rz_1))_{L^2(\bbR)}
	}
	e^{2B(\rz_1)}\rC(\bz),
\end{aligned}
\end{equation*}
where $\mathrm{pnw}$ denotes the periodic narrow-wedge initial condition. 
The Fredholm determinant in the denominator can be evaluated explicitly. 
In Proposition~\ref{prop:pnw22} below, we show that
\begin{equation*}
	\det(\rI+\limKe_{\mathrm{pnw}}(\rz))_{L^2(\bbR)}
	=
	e^{2B(\rz)}.
\end{equation*}
Therefore,
\begin{equation*}
\begin{aligned}
	\lim_{L\to\infty}\mathscr{C}_{Y_L}(\bz)
	&=
	\det(\rI+\limKe_\h(\rz_1))_{L^2(\bbR)}\rC(\bz)
	=
	\rC_\h(\bz).
\end{aligned}
\end{equation*}
Alternatively, the analysis of~\cite[(6.20)]{Baik-Liu19} shows directly that $\mathscr{C}_{\mathrm{step}_L}(\bz)/\ncE_{\mathrm{step}_L}(z_1)\to\rC(\bz)$, which yields the same conclusion.  
This proves Proposition~\ref{prop:CYconvergence}. 
\end{proof}

\subsubsection{Proof of Proposition~\ref{prop:DYconvergence00}}
\label{sec:convergence_scrD_Y}

From Definition~\ref{def:Dz},
\begin{equation} \label{eq:Dzhere}
	\mathscr{D}_{Y}(\bm{z})
	=
	\sum_{\mathbf{n}\in(\bbZ_{\geq0})^m}
	\frac{(-1)^{n_1+\cdots+n_m}}{(n_1!\cdots n_m!)^2}
	\mathscr{D}^{(\mathbf{n})}_{Y}(\bm{z}), 
	\qquad
	\mathscr{D}^{(\mathbf{n})}_{Y}(\bm{z})
	:=
	\sum_{\substack{
    U^{(\ell)}\in(\cL_{z_\ell})^{n_\ell},\ V^{(\ell)}\in(\cR_{z_\ell})^{n_\ell}\\
    1\le\ell\le m}}
	d_Y^{(\mathbf{n})}(\bm{z};U,V),
\end{equation}
where
\begin{equation}\label{eq:dYform}
	d_Y^{(\mathbf{n})}(\bm{z};U,V)
	=
	\det\left[
	\mdpch_{Y}(v^{(1)}_i,u^{(1)}_j;z_1)
	\right]_{1\le i,j\le n_1}
	d^{(\mathbf{n})}(\bm{z};U,V).
\end{equation}
The formula involves sums over the Bethe roots $\cL_{z_\ell}$ and $\cR_{z_\ell}$, the roots of
\begin{equation} \label{eq:brt}
	w^N(w+1)^{L-N}=z_\ell^L.
\end{equation}
When
\begin{equation*}
	z_\ell^L=(-\rho_L)^{N_L}(1-\rho_L)^{L-N_L}\rz_\ell
\end{equation*}
as in \eqref{eq:zL_to_rz}, and
\begin{equation} \label{eq:chov}
	w=-\rho_L+\sqrt{\rho_L(1-\rho_L)}\,\frac{\zeta}{L^{1/2}},
\end{equation}
with $\zeta=O(1)$, the Bethe-root equation \eqref{eq:brt} converges to
\begin{equation*}
	e^{-\frac12\zeta^2}=\rz_\ell.
\end{equation*}
The limiting roots split into $\rL_{\rz_\ell}$ and $\rR_{\rz_\ell}$. 
More precisely, for every $0<\epsilon<1/8$, the Bethe roots in the $O(L^{-1/2+\epsilon})$-neighborhood of $-\rho_L$ converge, under the scaling \eqref{eq:chov}, to the roots in $\rL_{\rz_\ell}$ and $\rR_{\rz_\ell}$; see~\cite[Lemma~8.1]{Baik-Liu18}. 
Thus, to prove Proposition~\ref{prop:DYconvergence00}, it remains to establish the pointwise limit and a tail bound under the scaling \eqref{eq:chov} for the summation variables.

For the periodic step initial condition, the pointwise limit and the tail estimates were obtained in~\cite{Baik-Liu19}. 
For a general initial condition, the only change is that the determinant
\begin{equation*}
	\det\left[
	\mdpch_{Y}(v^{(1)}_i,u^{(1)}_j;z_1)
	\right]_{1\le i,j\le n_1}
\end{equation*}
replaces the corresponding determinant for the periodic step initial condition. 
By Lemma~\ref{prop:pch_convergence}, the proof of the pointwise convergence is unchanged from the periodic step case.

The tail estimate requires one additional modification. 
For the periodic step initial condition, the modified PTASEP characteristic function is uniformly bounded. 
For general initial conditions, we instead have the bound from Lemma~\ref{prop:uniform_bound_char}, which grows super-exponentially as the argument tends to infinity. 
This growth is compensated by the decay of the factor $f_\ell$ with $\ell=1$ in $d^{(\mathbf{n})}(\bm{z};U,V)$. 
Lemma~\ref{lm:main_decay} below shows that this factor decays fast enough so that its product with the modified PTASEP characteristic function still satisfies the required tail estimate.

In~\cite{Baik-Liu19}, only a uniform bound on $f_\ell$ was stated, since this was sufficient for the periodic step initial condition. 
Lemma~\ref{lm:main_decay} is the strengthened estimate needed here. 
All remaining factors in $d^{(\mathbf{n})}(\bm{z};U,V)$ are independent of the initial condition and were already estimated in~\cite{Baik-Liu19}, as well as in~\cite{Baik-Liu18,Baik-Liu21}. 
After proving Lemma~\ref{lm:main_decay}, we complete the proof of Proposition~\ref{prop:DYconvergence00}.

We first state the improved estimate for $f_\ell$. 
We only need the result for $\ell=1$, but we state it for general $\ell$, since the analysis is the same. 
For $\ell=1$, we have $\TT_1>\TT_0=0$, and hence the estimate \eqref{eq:bound_generic} applies. 
We also recall that, in the notation of the next result, it was shown in~\cite[Lemma~8.8]{Baik-Liu18} that, with $\mathrm{f}_\ell$ defined in \eqref{eq:f_lim},
\begin{equation} \label{eq:f_asymptotics}
	\frac{f_\ell(u_L)}{Q_{\ell,L}}
	=
	\mathrm{f}_\ell(\xi_L)
	\left(1+O(L^{-1/2+\epsilon})\right),
	\qquad
	Q_{\ell,L} f_\ell(v_L)
	=
	\mathrm{f}_\ell(\eta_L)
	\left(1+O(L^{-1/2+\epsilon})\right),
\end{equation}
provided that $\xi_L,\eta_L=O(L^\epsilon)$ for some $\epsilon>0$. 
The estimates \eqref{eq:bound_generic} and \eqref{eq:bound_degenerate} are consistent with the tail behavior of the limiting formula for $\mathrm{f}_\ell$.

\begin{lm}\label{lm:main_decay} 
Let $\delta\in(0,1/2)$. 
Let $\{u_L\}$ and $\{v_L\}$ be sequences of complex numbers satisfying the assumptions of Lemma~\ref{prop:uniform_bound_char}; that is,
\begin{equation}\label{eq:uvassumption_4}
     \Re(u_L)<-\rho_L<\Re(v_L),
     \qquad
     \delta\le
     \frac{|u_L|^{N_L}|1+u_L|^{L-N_L}}
     {\rho_L^{N_L}(1-\rho_L)^{L-N_L}}
     \le 1-\delta,
     \qquad
     \delta\le
     \frac{|v_L|^{N_L}|1+v_L|^{L-N_L}}
     {\rho_L^{N_L}(1-\rho_L)^{L-N_L}}
     \le 1-\delta
\end{equation}
for all sufficiently large $L$. 
Define
\begin{equation*}
    \xi_L:=
    \frac{L^{1/2}(u_L+\rho_L)}{\sqrt{\rho_L(1-\rho_L)}},
    \qquad
    \eta_L:=
    \frac{L^{1/2}(v_L+\rho_L)}{\sqrt{\rho_L(1-\rho_L)}}.
\end{equation*}
For each $1\le \ell\le m$, set
\begin{equation*}
    Q_{\ell,L}
    :=
    \frac{
    (-\rho_L)^{k_\ell}(1-\rho_L)^{a_{\ell-1}+k_{\ell-1}}e^{-t_\ell\rho_L}
    }{
    (-\rho_L)^{k_{\ell-1}}(1-\rho_L)^{a_\ell+k_\ell}e^{-t_{\ell-1}\rho_L}
    }.
\end{equation*}
Let $f_\ell$ be the function in \eqref{eq:f}. 
For each $\ell$, the following hold.
\begin{enumerate}[(i)]
    \item If $\TT_{\ell-1}<\TT_\ell$, then there exist constants $C,c>0$ such that
    \begin{equation}\label{eq:bound_generic}
        \left|\frac{f_\ell(u_L)}{Q_{\ell,L}}\right|
        \le
        C e^{-c |\xi_L|^3},
        \qquad
        \left|Q_{\ell,L}f_\ell(v_L)\right|
        \le
        C e^{-c |\eta_L|^3}
    \end{equation}
    for all sufficiently large $L$, uniformly over all such sequences $\{u_L\}$ and $\{v_L\}$.

    \item If $\TT_{\ell-1}=\TT_\ell$ and $\HH_{\ell-1}<\HH_\ell$, then there exist constants $C,c>0$ such that
    \begin{equation} \label{eq:bound_degenerate}
        \left|\frac{f_\ell(u_L)}{Q_{\ell,L}}\right|
        \le
        C e^{-c |\xi_L|},
        \qquad
        \left|Q_{\ell,L}f_\ell(v_L)\right|
        \le
        C e^{-c |\eta_L|}
    \end{equation}
    for all sufficiently large $L$, uniformly over all such sequences $\{u_L\}$ and $\{v_L\}$.
\end{enumerate}
\end{lm}

\begin{proof}
To lighten the notation, we suppress the subscript $L$ in $u_L,v_L,\xi_L,\eta_L,\rho_L$, and $\rhosr$ throughout the proof. 
We also write $\rho_0:=\rhosr=\sqrt{\frac{1-\rho}{\rho}}$. 
Fix $1\le \ell\le m$ and denote
\begin{equation*}
	\Delta \TT:=\TT_\ell-\TT_{\ell-1},
	\qquad
	\Delta\XX:=\XX_\ell-\XX_{\ell-1},
	\qquad
	\Delta \HH:=\HH_\ell-\HH_{\ell-1}.
\end{equation*}

Consider (i), so that $\Delta \TT>0$. 
We prove only the first inequality in \eqref{eq:bound_generic}; the second is similar. 
Recalling \eqref{eq:tkii}, we have
\begin{equation}\label{eq:exponent}
	\hat f_\ell(u)
	:=
	\frac{f_\ell(u)}{Q_{\ell,L}}
	=
	e^{  L^{3/2} g_3(\xi) \Delta \TT 
	+ L g_2(\xi) \Delta \XX 
	+ L^{1/2} g_1(\xi) \Delta \HH 
	+ O(1) }, 
	\qquad
	\xi:=\frac{L^{1/2}(u+\rho)}{\sqrt{\rho(1-\rho)}} ,
\end{equation}
where the $O(1)$ term comes from rounding errors and is uniform in $L$ and $u$. 
Here
\begin{equation*}
\begin{aligned}
	g_1(\xi)
	&:=
	\rho_0\log\big(1+\frac{\xi}{\rho_0 L^{1/2}}\big), \\
	g_2(\xi)
	&:=
	-\rho\log\big(1-\frac{\rho_0\xi}{L^{1/2}}\big)
	-(1-\rho)\log\big(1+\frac{\xi}{\rho_0 L^{1/2}}\big), \\
	g_3(\xi)
	&:=
	L^{-1/2}\xi
	+\frac{\rho}{\rho_0}
	\log\big(1-\frac{\rho_0\xi}{L^{1/2}}\big)
	-(1-\rho)\rho_0
	\log\big(1+\frac{\xi}{\rho_0 L^{1/2}}\big).
\end{aligned}
\end{equation*}
We shall use the identity
\begin{equation}\label{eq:g2_real_bound}
	L\Re\, g_2(\xi)
	=
	-\log\left(
	\frac{|u|^N|1+u|^{L-N}}
	{\rho^N(1-\rho)^{L-N}}
	\right).
\end{equation}
Thus, by \eqref{eq:uvassumption_4}, $|e^{L g_2(\xi) \Delta \XX }|$ is uniformly bounded.

We consider two cases, according to whether $|\xi|\le \varepsilon L^{1/2}$ or $|\xi|>\varepsilon L^{1/2}$, where $\varepsilon>0$ will be chosen small enough.

Assume first that $|\xi|\le\varepsilon L^{1/2}$. 
Taylor expansion gives constants $c_1,c_2>0$ such that
\begin{equation*}
	|\Re\,g_1(\xi)|
	\le
	c_1L^{-1/2}|\xi|,
	\qquad
	|\Re\,g_2(\xi)|
	\le
	c_2L^{-1}|\xi|^2.
\end{equation*}
It also gives
\begin{equation}\label{eq:g3_taylor}
	g_3(\xi)
	=
	-\frac13 L^{-3/2}\xi^3
	+
	O(L^{-2}|\xi|^4).
\end{equation}
We claim that
\begin{equation}\label{eq:g3_decay_claim}
	\Re\,g_3(\xi)
	\le
	-c_3L^{-3/2}|\xi|^3+c_4L^{-3/2}
\end{equation}
for some constants $c_3,c_4>0$. 
To prove this claim, we first show that $\Re(\xi^3)$ is comparable with $|\xi|^3$ for $\xi$ arising from a point $u$ satisfying \eqref{eq:uvassumption_4}. 
By Taylor expansion,
\begin{equation*}
	\log\left(
	\frac{|u|^N|1+u|^{L-N}}
	{\rho^N(1-\rho)^{L-N}}
	\right)
	=
	-\frac12\Re(\xi^2)
	+
	O(L^{-1/2}|\xi|^3).
\end{equation*}
Hence, using \eqref{eq:uvassumption_4}, there is a constant
$C_\delta:=\log\frac1\delta$ 
such that, after decreasing $\varepsilon$ if necessary,
\begin{equation*}
	|\Re(\xi^2)|
	\le
	2C_\delta+O(L^{-1/2}|\xi|^3)
	\le
	2C_\delta+\frac14|\xi|^2
\end{equation*}
whenever $|\xi|\le\varepsilon L^{1/2}$. 
It follows that
\begin{equation*}
	(\Re\,\xi)^2
	=
	\frac{|\xi|^2+\Re(\xi^2)}{2}
	\ge
	\frac38|\xi|^2-C_\delta.
\end{equation*}
Since $\Re(\xi)<0$ by \eqref{eq:uvassumption_4}, this implies
$\Re(\xi)\le -\frac12|\xi|$ if $|\xi|\ge\sqrt{8C_\delta}$. 
Using $\Re(\xi^3) = \Re(\xi)\left(2\Re(\xi^2)-|\xi|^2\right)$, 
we obtain, for $|\xi|\ge\sqrt{16C_\delta}$,
\begin{equation*}
	\Re(\xi^3)
	\ge
	\frac12|\xi|\left(\frac12|\xi|^2-4C_\delta\right)
	\ge
	\frac18|\xi|^3.
\end{equation*}
For bounded $|\xi|$, the trivial estimate
$\Re(\xi^3)\ge-|\xi|^3$ shows that, after increasing the constant if necessary,
\begin{equation*}
	\Re(\xi^3)
	\ge
	\frac18|\xi|^3-C_\delta'
\end{equation*}
for a constant $C_\delta'>0$. 
Combining this with \eqref{eq:g3_taylor}, we find that there exists $\varepsilon>0$ such that \eqref{eq:g3_decay_claim} holds for $|\xi|\le \varepsilon L^{1/2}$. 
Substituting the bounds above into \eqref{eq:exponent}, and using the uniform boundedness of $L\Re\,g_2(\xi)$ from \eqref{eq:g2_real_bound}, we find 
\begin{equation*}
	|\hat f_\ell(u)|
	\le
	Ce^{-c \Delta \TT|\xi|^3}
\end{equation*}
for all $|\xi|\le\varepsilon L^{1/2}$.

Now assume that $|\xi|\ge\varepsilon L^{1/2}$. 
Since $|u|$ and $|1+u|$ are uniformly bounded above and below under \eqref{eq:uvassumption_4}, there is a constant $A>0$ such that
$|\xi|\le A L^{1/2}$. 
Moreover, $g_1(\xi)$ is uniformly bounded, and $L\Re\,g_2(\xi)$ is uniformly bounded by \eqref{eq:g2_real_bound}. 
We claim that there exists $c>0$ such that
\begin{equation}\label{eq:g3_large_claim}
	\Re\,g_3(\xi)\le -c
\end{equation}
for all $\varepsilon L^{1/2}\le|\xi|\le A L^{1/2}$ and all large enough $L$.
To prove the claim, consider
\begin{equation*}
	\Phi(u)
	:=
	\sqrt{\rho(1-\rho)}\,\Re\,g_3(\xi)
	=
	\Re(u+\rho)
	+
	\rho^2\log\left|\frac{u}{\rho}\right|
	-
	(1-\rho)^2\log\left|\frac{1+u}{1-\rho}\right|.
\end{equation*}
Using $\Re(u) = \frac12\left(|1+u|^2-|u|^2-1\right)$, 
we evaluate $\Phi$ on the level curve
\begin{equation*}
	\Sigma_\rho
	:=
	\big\{
	u\in\bbC:
	\frac{|u|^N|1+u|^{L-N}}
	{\rho^N(1-\rho)^{L-N}}
	=
	1,\
	\Re(u)\le-\rho
	\big\}.
\end{equation*}
On this curve, setting $s:=\log(|u|/\rho)$, we have
$|1+u| = (1-\rho) e^{-\frac{\rho}{1-\rho}s}$. 
Therefore
\begin{equation*}
	\Phi(u)
	=
	\frac12
	\left(
	(1-\rho)^2e^{-\frac{2\rho}{1-\rho}s}
	-\rho^2e^{2s}
	-1
	\right)
	+\rho+\rho s
	=:F(s).
\end{equation*}
A direct calculation gives $F(0)=0$ and $F'(s)<0$ for all $s>0$. 
Hence $F(s)<0$ for all $s>0$. 
Moreover, uniformly for $\rho$ in a compact subset of $(0,1)$, $F(s)$ is bounded above by a negative constant whenever $s$ is bounded away from $0$. 
The points satisfying \eqref{eq:uvassumption_4} are not exactly on $\Sigma_\rho$, but if they are bounded away from $-\rho$, then they are at distance $O(L^{-1})$ from $\Sigma_\rho$. 
Since $|\xi|\ge\varepsilon L^{1/2}$ implies $|u+\rho| \ge \varepsilon\sqrt{\rho(1-\rho)}$, 
compactness and continuity imply that there exists $c>0$ such that 
$\Phi(u)\le -c$ 
for all $u$ and large enough $L$. 
This proves \eqref{eq:g3_large_claim}. 
Consequently,
\begin{equation*}
	|\hat f_\ell(u)|
	\le
	Ce^{-c \Delta \TT L^{3/2}}
	\le
	Ce^{-\frac{c}{A^3} \Delta \TT|\xi|^3},
\end{equation*}
because $|\xi|\le A L^{1/2}$. 
This completes the proof of the first inequality in \eqref{eq:bound_generic}. 
The proof of the second inequality is the same, using the corresponding expression for $Q_{\ell,L}f_\ell(v)$ and the fact that $\Re(\eta)>0$.

We now prove (ii). 
Assume that $\TT_{\ell-1}=\TT_\ell$ and $\HH_{\ell-1}<\HH_\ell$, so that $\Delta \TT=0$ and $\Delta \HH>0$. 
Again we prove only the estimate involving $u$; the estimate involving $v$ is analogous. 
By \eqref{eq:exponent},
\begin{equation*}
	\hat f_\ell(u)
	=
	e^{ L g_2(\xi) \Delta \XX
	+ L^{1/2} g_1(\xi) \Delta \HH
	+ O(1)}.
\end{equation*}
As above, $L\Re g_2(\xi)$ is uniformly bounded by \eqref{eq:g2_real_bound}. 
Considering $|\xi|\le\varepsilon L^{1/2}$ and $|\xi|\ge\varepsilon L^{1/2}$ separately as in part~(i), we obtain
\begin{equation*}
	L^{1/2}\Re\,g_1(\xi)
	\le
	C-c|\xi|
\end{equation*}
for some constants $C,c>0$. 
Indeed, in the first region the Taylor expansion of $g_1$ and the estimate $\Re(\xi)\le C-c|\xi|$ give the claim. In the second region, compactness away from $u=-\rho$ gives $\Re\,g_1(\xi)\le-c$, and $|\xi|=O(L^{1/2})$. 
Therefore
\begin{equation*}
	| \hat f_\ell(u) |
	\le
	Ce^{-c\Delta \HH|\xi|}.
\end{equation*}
This proves the first inequality in \eqref{eq:bound_degenerate}. 
The second inequality is proved in the same way for $Q_{\ell,L}f_\ell(v)$. 
The lemma follows.
\end{proof}

We now prove Proposition~\ref{prop:DYconvergence00}.

\begin{proof}[Proof of Proposition~\ref{prop:DYconvergence00}]
The convergence follows from pointwise convergence of the summands in $\mathscr{D}^{(\mathbf{n})}_{Y_L}(\bm{z})$ and a uniform tail bound.
Since the convergence of \eqref{eq:Dzhere} was proved for the periodic step initial condition in~\cite{Baik-Liu19}, we only explain the new part of the proof, namely the effect of the general initial condition.

Factoring out 
$\prod_{i=1}^{n_1} f_1(u_i^{(1)})f_1(v_i^{(1)})$ 
from $d^{(\mathbf{n})}(\bm{z};U,V)$, we consider the combination
\begin{equation*}
\begin{aligned}
	&\det\left[
	\mdpch_{Y_L}(v_i^{(1)},u_j^{(1)};z_1)
	\right]_{1\le i,j\le n_1}
	\prod_{i=1}^{n_1} f_1(u_i^{(1)})f_1(v_i^{(1)}) \\
	&\quad =
	\det\left[
	\mpch_{Y_L}(v_i^{(1)},u_j^{(1)})
	f_1(v_i^{(1)})f_1(u_j^{(1)})
	\right]_{1\le i,j\le n_1}
	\prod_{i=1}^{n_1}
	\frac{1}{H_{z_1}(u_i^{(1)})H_{z_1}(v_i^{(1)})}.
\end{aligned}
\end{equation*}
The functions $1/H_z(w)$ were shown to be bounded in~\cite[Lemma~8.2]{Baik-Liu21} for all $w$ that appear in the above formula. 
Lemma~\ref{prop:uniform_bound_char} shows that $\mpch_Y$ can grow like $e^{c|\xi|^2}$ in the scaled variable. 
However, using Lemma~\ref{lm:main_decay}(i) with $\ell=1$ and recalling that $\TT_1>\TT_0=0$, we find that there are $C, c>0$ such that 
\begin{equation} \label{eq:here22}
\begin{aligned}
	&L^{-1/2}
	\left|
	\mpch_{Y_L}(v_i^{(1)},u_j^{(1)})
	f_1(v_i^{(1)})f_1(u_j^{(1)})
	\right| \\
	&\qquad \le
	C
	e^{c\Mx(|\eta_i^{(1)}|+|\xi_j^{(1)}|)
		-c(|\eta_i^{(1)}|^3+|\xi_j^{(1)}|^3)}
	\left(
	\frac{1}{\Re(-\xi_j^{(1)})}
	+
	e^{c|\xi_j^{(1)}|^2}
	\right),
\end{aligned}
\end{equation}
where $\Mx\in[1,\infty)$ is chosen so that 
$\max_{\XX\in[0,1]}\h_L(\XX)\le \Mx-1$ for all $L$. 
The right-hand side gives a decaying bound as the scaled variables tend to infinity. 
The remaining factors in $d^{(\mathbf{n})}(\bm{z};U,V)$ were estimated in~\cite[Lemmas~8.2 and~8.3]{Baik-Liu18}. 
The factor $L^{-1/2}$ in \eqref{eq:here22} is cancelled by the remaining factors.\footnote{For the periodic step initial condition, the same factor $L^{-1/2}$ appears because $\mpch_{\mathrm{step}_L}(v,u)=\frac1{v-u}=O(L^{1/2})$ under the scaling near $-\rho_L$, and it is cancelled by the same remaining factors.}
Consequently, the estimates of~\cite[Lemmas~8.2 and~8.3]{Baik-Liu18}, together with Hadamard's inequality, give a constant $C>0$ such that
\begin{equation}\label{eq:Dn_uniform_bound}
    \left|\mathscr{D}^{(\mathbf{n})}_{Y_L}(\bm{z})\right|
    \le
    \prod_{\ell=1}^m n_\ell^{n_\ell}C^{n_\ell}
\end{equation}
for all $\mathbf{n}$ and all sufficiently large $L$, with the convention $0^0=1$. This bound is summable after multiplication by the factorial factors in \eqref{eq:Dzhere}.

We now prove pointwise convergence. 
Observe from \eqref{eq:dYform},  \eqref{eq:npcf},  and \eqref{eq:char_rewriting} that 
\begin{equation*}
	d_{Y_L}^{(\mathbf{n})}(\bm{z};U,V)
	=
	\frac{
	\det\left[
	\mpch_{Y_L}(v_i^{(1)},u_j^{(1)})
	\right]_{1\le i,j\le n_1}
	}{
	\det\left[
	\mpch_{\mathrm{step}_L}(v_i^{(1)},u_j^{(1)})
	\right]_{1\le i,j\le n_1}
	}
	d_{\mathrm{step}_L}^{(\mathbf{n})}(\bm{z};U,V).
\end{equation*}
Lemma~\ref{prop:pch_convergence} implies that, for all sequences of complex numbers $\{u_L\}$ and $\{v_L\}$ satisfying \eqref{eq:vunearo2},
\begin{equation*}
	\lim_{L\to\infty}
	\frac{
	\det\left[
	\mpch_{Y_L}(v_i^{(1)},u_j^{(1)})
	\right]_{1\le i,j\le n_1}
	}{
	\det\left[
	\mpch_{\mathrm{step}_L}(v_i^{(1)},u_j^{(1)})
	\right]_{1\le i,j\le n_1}
	}
	=
	\frac{
	\det\left[
	\pchlim_\h(\eta_i^{(1)},\xi_j^{(1)})
	\right]_{1\le i,j\le n_1}
	}{
	\det\left[
	\pchlim_{\mathrm{pnw}}(\eta_i^{(1)},\xi_j^{(1)})
	\right]_{1\le i,j\le n_1}
	}.
\end{equation*}
Here the common scaling factor in Lemma~\ref{prop:pch_convergence} cancels between the numerator and denominator. 
By Proposition~\ref{prop:pnw},
\begin{equation*}
	\pchlim_{\mathrm{pnw}}(\eta,\xi)
	=
	\frac{\re^{\hftn(\xi,\rz)+\hftn(\eta,\rz)}}{\eta-\xi},
\end{equation*}
and hence the denominator of the limiting ratio does not vanish. 
On the other hand, the pointwise convergence of 
$d_{\mathrm{step}_L}^{(\mathbf{n})}(\bm{z};U,V)$ was computed in~\cite[Section~6.3]{Baik-Liu19}. 
In our notation,
\begin{equation*}
	d_{\mathrm{step}_L}^{(\mathbf{n})}(\bm{z};U,V)
	\to
	\det\left[
	\pchlim_{\mathrm{pnw}}(\eta_i^{(1)},\xi_j^{(1)})
	\right]_{i,j=1}^{n_1}
	\mathrm{D}^{(\mathbf{n})}(\bz;\xib,\etab).
\end{equation*}
Therefore,
\begin{equation*}
	d_{Y_L}^{(\mathbf{n})}(\bm{z};U,V)
	\to
	\det\left[
	\pchlim_\h(\eta_i^{(1)},\xi_j^{(1)})
	\right]_{i,j=1}^{n_1}
	\mathrm{D}^{(\mathbf{n})}(\bz;\xib,\etab).
\end{equation*}

The uniform bound \eqref{eq:Dn_uniform_bound}, together with the factorial factors in \eqref{eq:Dzhere}, justifies termwise convergence in the series defining $\mathscr{D}_{Y_L}(\bm{z})$. 
Thus the desired convergence of $\mathscr{D}_{Y_L}(\bm{z})$ follows, completing the proof of Proposition~\ref{prop:DYconvergence00}.
\end{proof}

This proves the two propositions announced at the beginning of this section, and hence completes the proof of Corollary~\ref{cor:thm13geny1}.

\section{Completion of the proof of Theorem~\ref{thm:main}: General parameters}
\label{sec:asymptotics_general}

In Corollary~\ref{cor:thm13geny1}, we proved Theorem~\ref{thm:main} for parameters satisfying part~\textnormal{(i)} of Definition~\ref{def:F}, under the assumption that $y_1^{(L)}=-1$ for all $L$. In Subsection~\ref{sec:proof_general_h}, we remove this assumption. We then prove continuity of the one-point distribution in Subsection~\ref{sec:consistency} and complete the proof for all parameters in Subsection~\ref{sec:prtheoremm1}.

\subsection{Removal of the normalization assumption}
\label{sec:proof_general_h}

We first prove the following lemma. 

\begin{lm}\label{lm:particlelabel}
Let $\h\in\uc_1$, and let $\ra$ be a point in the effective domain of $\h$. Let $Y_L=(y_i^{(L)})_{i\in\bbZ}\in\Pconfno_{N_L,L}$ be a sequence of configurations whose linearly interpolated rescaled profiles $\h_L$, defined in \eqref{eq:initial_scaling}, satisfy $\h_L\to\h$ in $\uc_1$. Then there exists a sequence $\ra_L\in N_L^{-1}\bbZ$ such that
\begin{equation}\label{eq:aLaandhaL}
    \ra_L\to\ra,
    \qquad
    \h_L(\ra_L)\to\h(\ra).
\end{equation}
Set $p_L:=-\ra_LN_L$ and define
\begin{equation*}
    \hat y_i^{(L)}:=y_{p_L+i}^{(L)}-y_{p_L+1}^{(L)}-1,
    \qquad i\in\bbZ.
\end{equation*}
Let $\hat\h_L$ be the corresponding linearly interpolated rescaled profile. Then $\hat y_1^{(L)}=-1$ and
\begin{equation}\label{eq:hathLtoh}
    \hat\h_L(\alpha)\to\hat\h(\alpha):=\h(\alpha+\ra)-\h(\ra)
    \qquad\text{in $\uc_1$.}
\end{equation}
\end{lm}

\begin{proof}
Since $\h(\ra)>-\infty$, the convergence $\h_L\to\h$ provides a sequence $\beta_L\to\ra$ such that
$\liminf_{L\to\infty}\h_L(\beta_L)\ge\h(\ra)$. 
Let $\ra_L$ be one of the two endpoints of the interval of the mesh $N_L^{-1}\bbZ$ containing $\beta_L$, chosen so that $\h_L(\ra_L)\ge\h_L(\beta_L)$. Such a choice is possible because $\h_L$ is linear on that interval. Then $\ra_L\to\ra$ and $\liminf_{L\to\infty} \h_L(\ra_L)\ge\h(\ra)$. The upper-bound condition in \eqref{eq:UC_criterion} implies 
$\limsup_{L\to\infty}\h_L(\ra_L)\le\h(\ra)$. 
Thus \eqref{eq:aLaandhaL} holds.

Since $p_L=-\ra_LN_L\in\bbZ$, the shifted configuration $\hat Y_L=(\hat y_i^{(L)})_{i\in\bbZ}$ belongs to $\Pconfno_{N_L,L}$ and satisfies $\hat y_1^{(L)}=-1$. Directly from the definition of the rescaled profiles,
\begin{equation}\label{eq:hhasum}
    \hat\h_L(\alpha)=\h_L(\alpha+\ra_L)-\h_L(\ra_L),
    \qquad \alpha\in\bbR.
\end{equation}
We verify \eqref{eq:hathLtoh} using \eqref{eq:UC_criterion}. For every sequence $\alpha_L\to\alpha$, \eqref{eq:hhasum} and \eqref{eq:aLaandhaL} imply
\begin{equation*}
    \limsup_{L\to\infty}\hat\h_L(\alpha_L)
    \le
    \h(\alpha+\ra)-\h(\ra).
\end{equation*}
Conversely, if $\h(\alpha+\ra)>-\infty$, choose $\beta_L\to\alpha+\ra$ such that
$\liminf_{L\to\infty}\h_L(\beta_L)\ge\h(\alpha+\ra),
$
and set $\gamma_L:=\beta_L-\ra_L$. Then $\gamma_L\to\alpha$, and
\begin{equation*}
    \liminf_{L\to\infty}\hat\h_L(\gamma_L)
    \ge
    \h(\alpha+\ra)-\h(\ra).
\end{equation*}
The lower-bound condition is automatic when $\h(\alpha+\ra)=-\infty$. Hence \eqref{eq:hathLtoh} follows.
\end{proof}

We now remove the normalization assumptions from Corollary~\ref{cor:thm13geny1}.

\begin{cor}\label{cor:thm13generic}
The conclusion of Corollary~\ref{cor:thm13geny1} remains valid without the assumption that $y^{(L)}_1=-1$ for all $L$.
\end{cor}

\begin{proof}
Let $Y_L=(y_i^{(L)})_{i\in\bbZ}\in\Pconfno_{N_L,L}$ be a sequence of initial configurations whose rescaled profiles satisfy $\h_L\to\h$ in $\uc_1$. Choose a point $\ra$ in the effective domain of $\h$, and let $\ra_L$, $\hat Y_L$, and $\hat\h_L$ be as in Lemma~\ref{lm:particlelabel}. Thus
\begin{equation*}
    \hat y_1^{(L)}=-1,
    \qquad
    \hat\h_L\to\hat\h:=\h(\,\cdot+\ra)-\h(\ra)
    \qquad\text{in $\uc_1$.}
\end{equation*}

For the original system, set
\begin{equation}\label{eq:parameters}
    t_i:=\TT_i\frac{L^{3/2}}{\sqrt{\rho_L(1-\rho_L)}},
    \qquad
    k_i:=\left\lfloor\rho_L^2t_i-\XX_iN_L\right\rfloor,
    \qquad
    a_i:=\left\lfloor
    \XX_iL+(1-2\rho_L)t_i-\HH_i\sqrt{\frac{1-\rho_L}{\rho_L}}\,L^{1/2}
    \right\rfloor.
\end{equation}
For the shifted system, define 
\begin{equation*}
    \hat\XX_i:=\XX_i-\ra_L,
    \qquad
    \hat\HH_i:=\HH_i-\h_L(\ra_L),
\end{equation*}
and the corresponding indices $\hat k_i$ and $\hat a_i$ as in \eqref{eq:parameters} with $\XX_i$ replaced by $\hat{\XX}_i$ and $\HH_i$ replaced by $\hat{\HH_i}$. Since $\ra_LN_L\in\bbZ$ and
\begin{equation*}
    \h_L(\ra_L)\sqrt{\frac{1-\rho_L}{\rho_L}}\,L^{1/2}
    =-y_{-\ra_LN_L+1}^{(L)}-1+\ra_LL,
\end{equation*}
we have
\begin{equation*}
    \hat k_i=k_i+\ra_LN_L,
    \qquad
    \hat a_i=a_i-y_{-\ra_LN_L+1}^{(L)}-1.
\end{equation*}
Then, 
\begin{equation*}
    \hat\sx_{\hat k}^{(L)}(t)
    =
    \sx_{-\ra_LN_L+\hat k}^{(L)}(t)-y_{-\ra_LN_L+1}^{(L)}-1, 
\end{equation*}
and hence 
\begin{equation*}
    \prob\left(\bigcap_{i=1}^m\{\sx_{k_i}^{(L)}(t_i)\ge a_i\}\right)
    =
    \prob\left(\bigcap_{i=1}^m\{\hat\sx_{\hat k_i}^{(L)}(t_i)\ge\hat a_i\}\right).
\end{equation*}

The shifted parameters satisfy the same generic time-ordering conditions as the original parameters, and
\begin{equation*}
    \hat\XX_i\to\XX_i-\ra,
    \qquad
    \hat\HH_i\to\HH_i-\h(\ra).
\end{equation*}
Applying Corollary~\ref{cor:thm13geny1} to $\hat Y_L$ and using the shift identity \eqref{eq:Feqst}, 
along with its local uniformity in the observation parameters, gives
\begin{equation*}
\begin{aligned}
    \prob\left(\bigcap_{i=1}^m\{\hat\sx_{\hat k_i}^{(L)}(t_i)\ge\hat a_i\}\right)
    \to \mathbb F_\h^{(m)}
    \bigl(
    \HH_1,\ldots,\HH_m;
    (\XX_1,\TT_1),\ldots,(\XX_m,\TT_m)
    \bigr).
\end{aligned}
\end{equation*}
This proves the result. 
\end{proof}

\subsection{Continuity of the one-point distribution}
\label{sec:consistency}

By Corollary~\ref{cor:thm13generic}, Theorem~\ref{thm:main} holds under the assumptions
\begin{equation}\label{eq:paratemp}
    0<\TT_1\le\cdots\le\TT_m,
    \qquad
    (\HH_1,\ldots,\HH_m)\in\dom_+^m(\TT_1,\ldots,\TT_m).
\end{equation}
To extend the result to the remaining parameter regimes in Definition~\ref{def:F}, we need continuity of the one-point distribution.

For $m=1$, the restrictions involving equal-time parameters are vacuous, and
\begin{equation}\label{eq:Foneexp}
    \mathbb F_\h^{(1)}(\HH;(\XX,\TT))
    =
    \oint\frac{\rd\rz}{2\pi\ri\rz}\,\rC_\h(\rz)\rD_\h(\rz).
\end{equation}

\begin{lm}\label{lem:consistency}
For every $\h\in\uc_1$, the function $\mathbb F_\h^{(1)}(\HH;(\XX,\TT))$ is a cumulative distribution function in $\HH$ and is jointly continuous in $(\HH,\XX,\TT)\in\bbR\times\bbR\times\bbR_{+}$.
\end{lm}

\begin{proof}
By Corollary~\ref{cor:thm13generic} with $m=1$, $\mathbb F_\h^{(1)}(\HH;(\XX,\TT))$ is a pointwise limit of cumulative distribution functions in $\HH$. Hence it takes values in $[0,1]$ and is nondecreasing in $\HH$.

We first prove joint continuity. In \eqref{eq:Foneexp}, the dependence on $(\HH,\XX,\TT)$ occurs through the explicit exponential factors in $\rC_\h$ and through the functions $\mathrm f_1$ in the series defining $\rD_\h$. The absolute-convergence estimates of Section~\ref{sec:propertymulti} are locally uniform for $(\HH,\XX,\TT)\in\bbR\times\bbR\times\bbR_{+}$. Hence the series and the contour integral in \eqref{eq:Foneexp} converge locally uniformly, proving joint continuity.

It remains to show
\begin{equation}\label{eq:onepoint_tail}
    \lim_{\HH\to-\infty}\mathbb F_\h^{(1)}(\HH;(\XX,\TT))=0,
    \qquad
    \lim_{\HH\to\infty}\mathbb F_\h^{(1)}(\HH;(\XX,\TT))=1.
\end{equation}
We use monotonicity in the initial condition. If $\h',\h''\in\uc_1$ satisfy $\h'\le\h''$, then the monotonicity of PTASEP implies 
\begin{equation}\label{eq:monotonicity}
    \mathbb F_{\h''}^{(1)}(\HH;(\XX,\TT))
    \le
    \mathbb F_{\h'}^{(1)}(\HH;(\XX,\TT)).
\end{equation}
Since $\h$ is one-periodic and upper semicontinuous, there exists $c_1\in\bbR$ such that $\h\le c_1$. Since $\h\not\equiv-\infty$, there exist $\XX_0,c_2\in\bbR$ such that $\h(\XX_0)\ge c_2$. Therefore
\begin{equation*}
    c_2-\infty\mathbf 1_{\XX\notin\XX_0+\intZ}
    \le
    \h(\XX)
    \le
    c_1,
    \qquad \XX\in\bbR.
\end{equation*}
By \eqref{eq:monotonicity},
\begin{equation*}
    \mathbb F_{c_1}^{(1)}(\HH;(\XX,\TT))
    \le
    \mathbb F_\h^{(1)}(\HH;(\XX,\TT))
    \le
    \mathbb F_{c_2-\infty\mathbf 1_{\XX\notin\XX_0+\intZ}}^{(1)}(\HH;(\XX,\TT)).
\end{equation*}
By \eqref{eq:Feqst}, the two bounding functions reduce, up to shifts of $\HH$ and $\XX$, to the flat and periodic narrow-wedge one-point distribution functions. Both are cumulative distribution functions by~\cite[Sections~4.1 and 4.2]{Baik-Liu18}. Letting $\HH\to\pm\infty$ proves \eqref{eq:onepoint_tail}.
\end{proof}

\subsection{Completion of the proof of Theorem~\ref{thm:main}}
\label{sec:prtheoremm1}

\begin{proof}[Proof of Theorem~\ref{thm:main} for all parameters]
For the periodic step initial condition, the extension from parameters satisfying \eqref{eq:paratemp} to arbitrary parameters, using parts~\textnormal{(ii)} and~\textnormal{(iii)} of Definition~\ref{def:F}, was proved in~\cite[Appendix~A]{Baik-Liu2024}. The argument uses only convergence in the generic time-ordered regime and continuity of the one-point limiting distribution in $\HH$; it does not use any special property of the periodic step initial condition. These two inputs are provided by Corollary~\ref{cor:thm13generic} and Lemma~\ref{lem:consistency}, respectively. Therefore the proof of~\cite[Lemma~A.1]{Baik-Liu2024} applies and extends the convergence to all parameter regimes in Definition~\ref{def:F}. This completes the proof of Theorem~\ref{thm:main}.
\end{proof}

\section{Proof of Theorem~\ref{thm:mainheight}: From particle positions to height functions}
\label{sec:particletoheight}

We deduce the periodic corner growth limit theorem, Theorem~\ref{thm:mainheight}, from the particle-location limit theorem, Theorem~\ref{thm:main}. 
We assume the normalization $\HPT_L(0,0)=0$. 
The general case follows from the same relabeling-and-shift argument used in Section~\ref{sec:proof_general_h}, together with the shift invariance in Theorem~\ref{thm:main2}; we omit the details.

We first recall the relation between height functions and particle configurations. A height function $\HPT:\bbZ\to\bbZ$ with $\HPT(0)=0$ determines a particle configuration $Y=(y_i)_{i\in\bbZ}$ by
\begin{equation}\label{eq:yiitfhpt}
	y_i
	=
	\min\left\{\ell\in\bbZ:\HPT(\ell+1)\le\ell+2i-1\right\}.
\end{equation}
Since $\HPT(0)=0$ and $|\HPT(\ell)-\HPT(0)|\le|\ell|$, we have
\begin{equation*}
	\cdots<y_1\le-1<y_0<\cdots.
\end{equation*}
Conversely, given such a configuration, we recover $\HPT$ from
\begin{equation}\label{eq:height_particle}
	\HPT(\ell)
	=
	2\bigl(\sxx^{-1}(\ell-1)-1\bigr)+\ell,
	\qquad
	\sxx^{-1}(\ell):=\min\{k\in\bbZ:y_k\le\ell\}.
\end{equation}
In particular,
\begin{equation}\label{eq:height_particle2}
	\HPT(y_k+1)=y_k+2k-1,
	\qquad k\in\bbZ.
\end{equation}
These relations preserve periodicity: $\HPT\in\confH(N,L)$ with $\HPT(0)=0$ if and only if $Y\in\Pconfno_{N,L}$ with $y_1\le-1<y_0$.

Applying \eqref{eq:yiitfhpt} at each time to a periodic corner growth model produces a PTASEP. To apply Theorem~\ref{thm:main}, we first verify that convergence of the rescaled height profiles is equivalent to convergence of the rescaled particle-location profiles. The analogous equivalence for TASEP on the line is stated in~\cite[(3.2) and (3.4)]{Matetski-Quastel-Remenik21}. Since it is local, the same argument applies in the periodic setting. We include the proof for completeness.

\begin{lm}\label{prop:particle_to_height}
Let $\h\in\uc_1$, and let $N_L$ be a sequence such that $\rho_L:=N_L/L\in[\rho_-,\rho_+]$ for all sufficiently large $L$. For each $L$, let $H_L\in\confH(N_L,L)$ satisfy $H_L(0)=0$, and let $Y_L=(y_i^{(L)})_{i\in\bbZ}\in\Pconfno_{N_L,L}$ satisfy $y_1^{(L)}\le-1<y_0^{(L)}$. Suppose that $H_L$ and $Y_L$ are related by \eqref{eq:yiitfhpt}. Define
\begin{equation}\label{eq:two_profiles}
	\mathfrak H_L(\XX)
	:=
	\frac{H_L(\XX L)-(L-2N_L)\XX}
	{-2\sqrt{\rho_L(1-\rho_L)}\,L^{1/2}},
	\qquad
	\h_L(\XX)
	:=
	\frac{-y^{(L)}_{-\XX N_L+1}-1+\XX L}
	{\sqrt{\frac{1-\rho_L}{\rho_L}}\,L^{1/2}},
\end{equation}
where $H_L(a)$ and $y_a^{(L)}$ for noninteger $a$ are defined by linear interpolation. Then
\begin{equation*}
	\mathfrak H_L\to\h\ \text{in $\uc_1$}
	\quad\Longleftrightarrow\quad
	\h_L\to\h\ \text{in $\uc_1$}.
\end{equation*}
\end{lm}

\begin{proof}
By \eqref{eq:height_particle2},
\begin{equation}\label{eq:HL_ya_identity}
	H_L(y_a^{(L)}+1)
	=
	y_a^{(L)}+2a-1+O(1),
	\qquad a\in\bbR,
\end{equation}
where the error is uniform in $a$ and $L$. Set
\begin{equation*}
	\rho_{L,0}:=\sqrt{\frac{1-\rho_L}{\rho_L}},
	\qquad
	\sfz_L(\XX):=\frac{y^{(L)}_{-\XX N_L+1}+1}{L}.
\end{equation*}
The map $\sfz_L$ is continuous and strictly increasing, with
\begin{equation*}
	\sfz_L(\XX+1)=\sfz_L(\XX)+1.
\end{equation*}
Let $\mathsf w_L:=\sfz_L^{-1}$. We have
\begin{align}
	\sfz_L(\XX)-\XX
	&=
	-\rho_{L,0}L^{-1/2}\h_L(\XX), \label{eq:zL_to_alpha_exact}\\
	\mathfrak H_L(\sfz_L(\XX))
	&=
	\h_L(\XX)+O(L^{-1/2}), \label{eq:reparam}\\
	\h_L(\mathsf w_L(\beta))
	&=
	\mathfrak H_L(\beta)+O(L^{-1/2}), \label{eq:reparam_inv}\\
	\mathsf w_L(\beta)-\beta
	&=
	\rho_{L,0}L^{-1/2}\mathfrak H_L(\beta)+O(L^{-1}). \label{eq:reparam_inv2}
\end{align}
Indeed, the first identity follows directly from \eqref{eq:two_profiles}. Using \eqref{eq:HL_ya_identity} with $a=-\XX N_L+1$, we obtain
\begin{equation*}
	H_L(\sfz_L(\XX)L)
	=
	y^{(L)}_{-\XX N_L+1}-2\XX N_L+1+O(1),
\end{equation*}
which gives \eqref{eq:reparam}. The remaining two identities follow by setting $\XX=\mathsf w_L(\beta)$ in \eqref{eq:reparam} and then using \eqref{eq:zL_to_alpha_exact}.

Suppose first that $\h_L\to\h$ in $\uc_1$. By periodicity and the limsup condition in \eqref{eq:UC_criterion}, the functions $\h_L$ are uniformly bounded above. Hence, by \eqref{eq:reparam_inv}, the same is true of $\mathfrak H_L$.
Fix $\beta\in\bbR$. 

Suppose $\beta_L\to\beta$. Set  
\begin{equation*}
	M=\limsup_{L\to \infty} \mathfrak H_L(\beta_L) .
\end{equation*}
There is nothing to prove if $M=-\infty$. 
Thus, passing to a subsequence, assume that
\begin{equation*}
	\mathfrak H_L(\beta_L)\to M \in \bbR. 
\end{equation*}
By \eqref{eq:reparam_inv2}, we find $\mathsf w_L(\beta_L)\to\beta$. 
Thus, using \eqref{eq:reparam_inv} and the limsup half of \eqref{eq:UC_criterion} for $\h_L\to\h$, we obtain
\begin{equation*}
	M
	= \lim_{L\to\infty}\mathfrak H_L(\beta_L)
	= \lim_{L\to\infty}\h_L(\mathsf w_L(\beta_L))
	\le \h(\beta).
\end{equation*}
Thus
\begin{equation*}
	\limsup_{L\to\infty}\mathfrak H_L(\beta_L)\le \h(\beta).
\end{equation*}

For the lower bound, suppose that $\h(\beta)>-\infty$. Choose $\XX_L\to\beta$ such that
\begin{equation*}
	\liminf_{L\to\infty}\h_L(\XX_L)\ge\h(\beta).
\end{equation*}
The limsup condition shows that $\h_L(\XX_L)$ is bounded, and hence \eqref{eq:zL_to_alpha_exact} implies $\beta_L:=\sfz_L(\XX_L)\to \beta$. 
Thus, by \eqref{eq:reparam},
\begin{equation*}
	\liminf_{L\to\infty}\mathfrak H_L(\beta_L)
	=
	\liminf_{L\to\infty}\h_L(\XX_L)
	\ge
	\h(\beta).
\end{equation*} 
This result also holds trivially if $\h(\beta) = -\infty$.
Thus $\mathfrak H_L\to\h$ in $\uc_1$.

The converse follows by the same argument with $(\mathfrak H_L,\sfz_L)$ and $(\h_L,\mathsf w_L)$ interchanged, using \eqref{eq:reparam}--\eqref{eq:reparam_inv2}.
\end{proof}

We now prove Theorem~\ref{thm:mainheight}.

\begin{proof}[Proof of Theorem~\ref{thm:mainheight}]
Let $\HPT_L(\ell,t)$ be the periodic corner growth model of type $(N_L,L)$ appearing in the theorem, and assume that $\HPT_L(0,0)=0$ for all $L$. 
Define the particles $(\sx_i^{(L)}(t))_{i\in\bbZ}$ by \eqref{eq:yiitfhpt}. Then they form a $\mathrm{PTASEP}_{Y_L}(L,N_L)$, where $Y_L=(y_i^{(L)})_{i\in\bbZ}$ is determined from $\HPT_L(\cdot,0)$ by the same relation.

The hypothesis \eqref{eq:initial_height_convergence} states that the rescaled initial height profiles $\mathfrak H_L$ in Lemma~\ref{prop:particle_to_height} converge to $\h$ in $\uc_1$. Hence the corresponding rescaled particle-location profiles satisfy \eqref{eq:initial_scaling}, and Theorem~\ref{thm:main} applies.

It remains to translate the height events into particle-location events. Fix the parameters $(\XX_i,\TT_i,\HH_i)_{i=1}^m$ from Theorem~\ref{thm:mainheight}, and set
\begin{equation*}
	\rho_{L,0}:=\sqrt{\frac{1-\rho_L}{\rho_L}},
	\qquad
	t_i:=\frac{\TT_iL^{3/2}}{\sqrt{\rho_L(1-\rho_L)}},
	\qquad
	x_i:=\XX_iL+(1-2\rho_L)t_i.
\end{equation*}
Define
\begin{equation*}
	\XX_{i,L}
	:=
	\XX_i+\rho_{L,0}L^{-1/2}\HH_i,
\end{equation*}
so that $\XX_{i,L}\to\XX_i$.

Fix $i$ and set $\ell_i:=\lfloor x_i\rfloor$. Since $\HPT_L(\cdot,t_i)$ is linearly interpolated with slopes $\pm1$, replacing $x_i$ by $\ell_i$ changes the scaled height variable by $O(L^{-1/2})$. By the time-$t_i$ version of \eqref{eq:height_particle},
\begin{equation*}
	\HPT_L(\ell_i,t_i)
	=
	2\bigl(\sxx_{t_i}^{-1}(\ell_i-1)-1\bigr)+\ell_i,
	\qquad
	\sxx_{t_i}^{-1}(\ell):=\min\{k\in\bbZ:\sx_k^{(L)}(t_i)\le\ell\}.
\end{equation*}
Consequently, the height inequality
\begin{equation}\label{eq:height_event_i}
	\frac{\HPT_L(x_i,t_i)-(1-2\rho_L)x_i-2\rho_L(1-\rho_L)t_i}
	{-2\sqrt{\rho_L(1-\rho_L)}\,L^{1/2}}
	\le\HH_i
\end{equation}
is equivalent, up to an $O(1)$ perturbation of the integer threshold, to
\begin{equation*}
	\sxx_{t_i}^{-1}(\ell_i-1)
	\ge
	\rho_L^2t_i-\XX_{i,L}N_L+O(1).
\end{equation*}

The inverse relation
\begin{equation*}
	\sxx_{t_i}^{-1}(\ell)\le k
	\quad\Longleftrightarrow\quad
	\sx_k^{(L)}(t_i)\le\ell
\end{equation*}
therefore sandwiches \eqref{eq:height_event_i} between particle-location events of the form
\begin{equation}\label{eq:particle_event_shifted_height}
	\left\{
	\frac{
	-\sx_{\lfloor\rho_L^2t_i-\widetilde\XX_{i,L}N_L\rfloor}^{(L)}(t_i)
	+\widetilde\XX_{i,L}L+(1-2\rho_L)t_i}
	{\rho_{L,0}L^{1/2}}
	\le
	\widetilde\HH_{i,L}
	\right\},
\end{equation}
where
\begin{equation*}
	\widetilde\XX_{i,L}
	=
	\XX_{i,L}+O(L^{-1}),
	\qquad
	\widetilde\HH_{i,L}
	=
	\HH_i+O(L^{-1/2}).
\end{equation*}
The errors are deterministic and uniform for the fixed parameters.

Intersecting these comparisons over $1\le i\le m$, the event in \eqref{eq:lim_bbF_height} is sandwiched between intersections of particle events of the form \eqref{eq:particle_event_shifted_height}. Theorem~\ref{thm:main}, together with the local uniformity of the convergence in the parameters, gives the corresponding values of $\mathbb F_\h^{(m)}$ at the perturbed parameters. Since
$\widetilde\XX_{i,L}\to \XX_i$ and $\widetilde\HH_{i,L}\to \HH_i$, 
their limits equal $\mathbb F_\h^{(m)} \bigl(\bm{\HH};(\XX_i,\TT_i)_{i=1}^m\bigr)$
by the joint continuity in Theorem~\ref{thm:main2}. This proves \eqref{eq:lim_bbF_height}.
\end{proof}

\section{Proof of Theorem~\ref{thm:main2}: Properties of the periodic KPZ fixed point distribution functions} 
\label{sec:proof_thm12}

We prove the properties of $\mathbb{F}_\h^{(m)}$ stated in Theorem~\ref{thm:main2}.

The symmetry property~\textnormal{(d)} is immediate from Definition~\ref{def:F}, and the shift invariance property~\textnormal{(g)} was proved for all parameters in Corollary~\ref{cor:shiftinvlmf}. The periodicity property~\textnormal{(a)} follows from Theorem~\ref{thm:main} and the spatial periodicity of PTASEP, while the monotonicity property~\textnormal{(h)} follows from the basic coupling of PTASEP and Theorem~\ref{thm:main}.

We next consider properties~\textnormal{(b)} and~\textnormal{(c)}. For the periodic narrow-wedge initial condition, they were proved in~\cite[Proposition~A.3]{Baik-Liu2024}. The argument uses only that the one-point function is a continuous cumulative distribution function, which was established for every $\h\in\uc_1$ in Lemma~\ref{lem:consistency}. The same argument therefore proves properties~\textnormal{(b)} and~\textnormal{(c)} for every $\h\in\uc_1$.

We now prove the joint continuity property~\textnormal{(e)}. First consider parameters satisfying part~\textnormal{(i)} of Definition~\ref{def:F}. In \eqref{eq:multi_time1}, the dependence on $(\HH_i,\XX_i,\TT_i)_{i=1}^m$ appears explicitly in the exponential factors of $\rC_\h(\bz)$ in \eqref{eq:Climnoh} and in the factors $\mathrm f_\ell$ of $\rD_\h(\bz)$ in \eqref{eq:f_lim}. These factors are jointly continuous in the parameters. Moreover, the estimates proving the absolute convergence of the series defining $\rD_\h(\bz)$ and of the contour integrals are locally uniform in the parameters. Hence the series and contour integrals converge locally uniformly, proving the joint continuity of $\mathbb{F}_\h^{(m)}$ in the generic time-ordered regime. 
For general parameters, \cite[Proposition~A.4]{Baik-Liu2024} shows that joint continuity of the one-point distribution implies joint continuity of the multipoint distributions defined by the extension rules in parts~\textnormal{(ii)} and~\textnormal{(iii)} of Definition~\ref{def:F}. The argument is independent of the initial condition. Together with Lemma~\ref{lem:consistency}, this proves property~\textnormal{(e)} for all parameters.

\medskip

We devote the remainder of this section to proving the initial condition property~\textnormal{(f)}. We begin by reducing it to the one-point case. 
For every $1\le\ell\le m$,
\begin{equation*}
	0
	\le
	\mathbb{F}^{(m)}_\h
	\bigl(\HH_1,\ldots,\HH_m;(\XX_1,\TT_1),\ldots,(\XX_m,\TT_m)\bigr)
	\le
	\mathbb{F}^{(1)}_\h(\HH_\ell;(\XX_\ell,\TT_\ell)),
\end{equation*}
and
\begin{equation*}
	0
	\le
	1-\mathbb{F}^{(m)}_\h
	\bigl(\HH_1,\ldots,\HH_m;(\XX_1,\TT_1),\ldots,(\XX_m,\TT_m)\bigr)
	\le
	\sum_{\ell=1}^m
	\bigl(
	1-\mathbb{F}^{(1)}_\h(\HH_\ell;(\XX_\ell,\TT_\ell))
	\bigr).
\end{equation*}
Thus, property~\textnormal{(f)} follows once we prove
\begin{equation}\label{eq:limit_tau0_special}
	\lim_{\TT\downarrow0}
	\mathbb{F}^{(1)}_\h(\HH;(\XX,\TT))
	=\begin{dcases}
        0,& \HH<\h(\XX), \\
        1,& \HH>\h(\XX). 
    \end{dcases}
\end{equation}
It is enough to consider $\h(\XX)>-\infty$. By the shift invariance \eqref{eq:translation_invariance}, we may further assume that $\XX=0$ and $\h(0)=0$. We split the proof into the cases $\HH<0$ and $\HH>0$.

First suppose  $\HH<0$. Let $\h_{\mathrm{pnw}}$ denote the periodic narrow-wedge initial condition $\h_{\mathrm{pnw}}(\XX')=-\infty\mathbf{1}_{\XX'\notin\intZ}$. Since $\h(0)=0$ and $\h$ is one-periodic, we have $\h_{\mathrm{pnw}}\le\h$. Monotonicity therefore gives
\begin{equation}\label{eq:onept_upper_bound_case1}
	\mathbb{F}^{(1)}_\h(\HH;(0,\TT))
	\le
	\mathbb{F}^{(1)}_{\h_{\mathrm{pnw}}}(\HH;(0,\TT)).
\end{equation}
For the periodic narrow-wedge initial condition, \cite[Theorem~1.6]{Baik-Liu-Silva22} shows that, for every fixed $x\in\bbR$,
\begin{equation*}
	\lim_{\TT\downarrow0}
	\mathbb{F}^{(1)}_{\h_{\mathrm{pnw}}}
	\bigl(\TT^{1/3}x;(0,\TT)\bigr)
	=
	\mathbf F_{\mathrm{GUE}}(x),
\end{equation*}
where $\mathbf F_{\mathrm{GUE}}$ is the GUE Tracy--Widom distribution function. Since $\HH<0$ is fixed, for every fixed $x\in\bbR$ we have $\HH<\TT^{1/3}x$ for all sufficiently small $\TT$. Hence, by monotonicity in $\HH$,
\begin{equation*}
	\limsup_{\TT\downarrow0}
	\mathbb{F}^{(1)}_{\h_{\mathrm{pnw}}}(\HH;(0,\TT))
	\le
	\mathbf F_{\mathrm{GUE}}(x)
\end{equation*}
for all $x\in \bbR$. 
Letting $x\to-\infty$ gives
\begin{equation}\label{eq:limit_tau0_special_a}
	\lim_{\TT\downarrow0}
	\mathbb{F}^{(1)}_\h(\HH;(0,\TT))
	=
	0, 
	\qquad
	\HH<0.
\end{equation}

It remains to show that
\begin{equation}\label{eq:limit_tau0_special_b}
	\lim_{\TT\downarrow0}
	\mathbb{F}^{(1)}_\h(\HH;(0,\TT))
	=
	1,
	\qquad
	\HH>0.
\end{equation}
Fix $\HH>0$. Since $\h$ is upper semicontinuous and one-periodic, it is bounded above. Choose $M>\HH$ such that
\begin{equation*}
	\sup_{\XX\in\bbR}\h(\XX)<M/2.
\end{equation*}
Moreover, since $\h(0)=0$, there exists $\delta\in(0,1/4)$ such that
\begin{equation*}
	\sup_{\XX\in[-\delta,\delta]}\h(\XX)<\HH_0:=\HH/2.
\end{equation*}
Define
\begin{equation*}
	\hat\h(\XX)
	:=
	\begin{cases}
		\HH_0, \quad &\operatorname{dist}(\XX,\intZ)<\delta,\\
		M,\quad &\operatorname{dist}(\XX,\intZ)\ge\delta.
	\end{cases}
\end{equation*}
The function $\hat\h$ is upper semicontinuous and one-periodic, and satisfies $\h\le\hat\h$. Hence, by monotonicity,
\begin{equation*}
	\mathbb{F}^{(1)}_\h(\HH;(0,\TT))
	\ge
	\mathbb{F}^{(1)}_{\hat\h}(\HH;(0,\TT)).
\end{equation*}
It is therefore enough to prove
\begin{equation}\label{eq:limit_tau0_special_c}
	\lim_{\TT\downarrow0}
	\mathbb{F}^{(1)}_{\hat\h}(\HH;(0,\TT))
	=
	1.
\end{equation}
We prove this by comparing the prelimit periodic corner-growth model with the standard corner-growth model.

Throughout the proof, $L$ tends to infinity through even positive integers. Set
\begin{equation*}
	r_L:=\lfloor\delta L\rfloor,
\end{equation*}
and choose an integer $k_L$ satisfying
\begin{equation*}
	k_L
	=
	r_L- \bigl(\lfloor ML^{1/2}\rfloor-\lfloor\HH_0L^{1/2}\rfloor \bigr)+O(1),
\end{equation*}
where the bounded adjustment is chosen so that the transition segments below are admissible. 
For each even $L$, define $h_L^{\mathrm{per}}\in\confH(L/2,L)$ by 
\begin{equation*}
	h_L^{\mathrm{per}}(k)
	=
	\begin{dcases}
		-\lfloor ML^{1/2}\rfloor+i_L(k),
		\quad &\operatorname{dist}(k,L\intZ) \ge r_L,\\
		-\lfloor\HH_0L^{1/2}\rfloor+i_L(k),
		\quad &\operatorname{dist}(k,L\intZ) \le k_L,
	\end{dcases}
\end{equation*}
where $i_L(k)\in\{0,1\}$ is the parity correction required for an admissible height function, and on each intervening interval satisfying
$k_L< \operatorname{dist}(k,L\intZ)<r_L$, 
define $h_L^{\mathrm{per}}(k)$ to be the admissible linear interpolation with slope $1$ or $-1$ between the two plateaus. Then
\begin{equation*}
	-\frac{h_L^{\mathrm{per}}(L\XX)}{L^{1/2}}
	\longrightarrow
	\hat\h(\XX)
	\qquad\text{in }\uc_1.
\end{equation*}
Let $H_L^{\mathrm{per}}(x,t)$ denote the height function\footnote{In the notation of Theorem~\ref{thm:mainheight}, $H_L^{\mathrm{per}}=\HPT_L$. We write $H_L^{\mathrm{per}}$ to emphasize the periodic setting.} of the periodic corner-growth model of type $(L/2,L)$ with initial condition
$H_L^{\mathrm{per}}(x,0)=h_L^{\mathrm{per}}(x)$.

We use the corner-growth/last-passage-percolation correspondence in coordinates obtained by rotating the usual LPP lattice counterclockwise by $45$ degrees and scaling it by $\sqrt{2}$. Couple the periodic and standard models using the same family
\begin{equation*}
	\{w(i,j):i+j\text{ is odd}\}
\end{equation*}
of independent mean-one exponential random variables. Define the periodic environment by
\begin{equation*}
	w^{\mathrm{per}}(i,j)
	:=
	w\left(i-L\left\lfloor \frac{i}{L} + \frac12\right\rfloor,j\right),
\end{equation*}
so that $w^{\mathrm{per}}(i+L,j)=w^{\mathrm{per}}(i,j)$. For odd lattice points $(i_1,j_1)$ and $(i_2,j_2)$ with $j_1\le j_2$, define
\begin{equation*}
	\mathbb L^{\mathrm{per}}(i_1,j_1;i_2,j_2)
	:=
	\max_\pi\sum_{(i,j)\in\pi}w^{\mathrm{per}}(i,j),
	\qquad
	\mathbb L(i_1,j_1;i_2,j_2)
	:=
	\max_\pi\sum_{(i,j)\in\pi}w(i,j),
\end{equation*}
where the maxima are taken over directed paths from $(i_1,j_1)$ to $(i_2,j_2)$ with steps $(\pm1,1)$. For an initial height function $g$, set
\begin{equation*}
	\mathbb L^{\mathrm{per}}(g;i,j)
	:=
	\max_{(i_0,j_0)}\mathbb L^{\mathrm{per}}(i_0,j_0;i,j),
	\qquad
	\mathbb L(g;i,j)
	:=
	\max_{(i_0,j_0)}\mathbb L(i_0,j_0;i,j),
\end{equation*}
where the maxima are taken over admissible odd lattice points $(i_0,j_0)$ above the initial curve $(k,g(k))$. We denote the corresponding periodic maximizing paths by $\gamma_L^{\mathrm{per}}(p;i,j)$ and $\gamma_L^{\mathrm{per}}(g;i,j)$ in the point-to-point and curve-to-point cases, respectively.

The standard corner-growth/LPP correspondence gives
\begin{equation}\label{eq:cornerLPPc}
	\prob\left(H_L^{\mathrm{per}}(i,t)\ge j\right)
	=
	\prob\left(\mathbb L^{\mathrm{per}}(h_L^{\mathrm{per}};i,j)\le t\right).
\end{equation}
Set
\begin{equation}\label{eq:tnn32}
	t_L:=2\tau L^{3/2},
\end{equation}
and choose an admissible integer $j_L$ of the correct parity such that
\begin{equation}\label{eq:jntn32}
	j_L
	=
	\frac12t_L-\HH L^{1/2}+O(1).
\end{equation}
By \eqref{eq:cornerLPPc} and Theorem~\ref{thm:mainheight},
$\prob\left(\mathbb L^{\mathrm{per}}(h_L^{\mathrm{per}};0,j_L)\le t_L\right)
\to \mathbb F_{\hat\h}^{(1)}(\HH;(0,\tau))$. 
Therefore, \eqref{eq:limit_tau0_special_c} follows once we prove
\begin{equation}\label{eq:to_be_compared}
	\limsup_{\tau\downarrow0}\limsup_{L\to\infty}
	\prob\left(
	\mathbb L^{\mathrm{per}}(h_L^{\mathrm{per}};0,j_L)>t_L
	\right)
	=0.
\end{equation}

We separate the contributions from the central part and the two tails of the periodic initial condition. Define the nonperiodic function
\begin{equation*}
	h_L(k)
	:=
	\begin{cases}
		-k+k_L+h_L^{\mathrm{per}}(-k_L),&k<-k_L,\\
		h_L^{\mathrm{per}}(k),&|k|\le k_L,\\
		k-k_L+h_L^{\mathrm{per}}(k_L),&k>k_L.
	\end{cases}
\end{equation*}
Define also
\begin{equation*}
	h_L^{\mathrm{ltail}}(k)
	:=
	\begin{cases}
		-\lfloor ML^{1/2}\rfloor+i_L(k),&k\le-r_L,\\
		k+r_L-\lfloor ML^{1/2}\rfloor+i_L(-r_L),&k>-r_L,
	\end{cases}
\end{equation*}
and
$h_L^{\mathrm{rtail}}(k):=h_L^{\mathrm{ltail}}(-k)$. 
The functions $h_L^{\mathrm{ltail}}$ and $h_L^{\mathrm{rtail}}$ are step-flat initial profiles. By construction,
\begin{equation}\label{eq:ic_decomposition}
	h_L^{\mathrm{per}}(k)
	\ge
	\min\left\{
	h_L(k),
	h_L^{\mathrm{ltail}}(k),
	h_L^{\mathrm{rtail}}(k)
	\right\}.
\end{equation}
Therefore, by monotonicity and reflection symmetry,
\begin{equation}\label{eq:Lper_split}
	\prob\left(
	\mathbb L^{\mathrm{per}}(h_L^{\mathrm{per}};0,j_L)>t_L
	\right)
	\le
	\prob\left(
	\mathbb L^{\mathrm{per}}(h_L;0,j_L)>t_L
	\right)
	+
	2\prob\left(
	\mathbb L^{\mathrm{per}}(h_L^{\mathrm{ltail}};0,j_L)>t_L
	\right).
\end{equation}
Thus, establishing the following two lemmas proves \eqref{eq:to_be_compared}, which completes the proof of Theorem~\ref{thm:main2} (f).  

\begin{lm}
\label{lm:lpp_center}
We have
\begin{equation*}
	\limsup_{\tau\downarrow0}\limsup_{L\to\infty}
	\prob\left(
	\mathbb L^{\mathrm{per}}(h_L;0,j_L)>t_L
	\right)
	=0.
\end{equation*}
\end{lm}

\begin{lm}
\label{lm:lpp_left}
We have
\begin{equation*}
	\limsup_{\tau\downarrow0}\limsup_{L\to\infty}
	\prob\left(
	\mathbb L^{\mathrm{per}}(h_L^{\mathrm{ltail}};0,j_L)>t_L
	\right)
	=0.
\end{equation*}
\end{lm}

\begin{proof}[Proof of Lemma~\ref{lm:lpp_center}]
Define
\begin{equation}\label{eq:def_pend1}
	\mathrm p_L^-
	:=
	\bigl(-k_L,h_L(-k_L)\bigr),
	\qquad
	\mathrm p_L^+
	:=
	\bigl(k_L,h_L(k_L)\bigr).
\end{equation}
We compare the periodic and standard passage times with the same initial condition $h_L$. Define 
\begin{equation*}
	A_L
	:=
	\left\{
	\gamma_L^{\mathrm{per}}
	\bigl(\mathrm p_L^-;0,j_L\bigr)
	\text{ intersects the line }i=-L/2
	\right\}.
\end{equation*}
The corresponding event for $\mathrm p_L^+$ and the line $i=L/2$ has the same probability. Because $h_L$ has maximal slopes outside $[-k_L,k_L]$, a maximizing source for $\mathbb L^{\mathrm{per}}(h_L;0,j_L)$ may be chosen on the portion of the initial curve between $\mathrm p_L^-$ and $\mathrm p_L^+$. Planar ordering then places the corresponding curve-to-point geodesic between the point-to-point geodesics from these two endpoints to $(0,j_L)$. Hence, if the curve-to-point geodesic exits the fundamental strip, one of the two endpoint geodesics must intersect the corresponding boundary line.

On the complement of the union of these two events, the periodic maximizing path remains in a single fundamental strip. Under the above coupling, it is then an admissible path in the standard environment with the same weight. Therefore,
\begin{equation}\label{eq:Lpermain_vs_L}
	\prob\left(
	\mathbb L^{\mathrm{per}}(h_L;0,j_L)>t_L
	\right)
	\le
	\prob\left(
	\mathbb L(h_L;0,j_L)>t_L
	\right)
	+
	2\prob(A_L).
\end{equation}

The straight segment joining $\mathrm p_L^-$ to $(0,j_L)$ remains at distance of order $L$ from the line $i=-L/2$. Moreover, by \eqref{eq:tnn32},
\begin{equation}\label{eq:nastau2332}
	L
	=
	(2\tau)^{-2/3}t_L^{2/3}.
\end{equation}
Thus, the constant-fraction transversal estimate proved in \cite[Proposition~3.12]{Schmid-Sly26}\footnote{This estimate is further refined in \cite[Lemma~3.9]{Aggarwal-Corwin-Schmid26}.} gives a constant $c>0$ such that
\begin{equation}\label{eq:transversal_bound}
	\prob(A_L)
	\le
	\re^{-\frac{c}{\tau^{2/3}}}
\end{equation}
for all sufficiently small $\tau$ and sufficiently large even $L$.

On the other hand, since
$h_L(k)\ge-\lfloor\HH_0L^{1/2}\rfloor-1$ 
for all $k$, monotonicity implies
\begin{equation*}
	\prob\left(
	\mathbb L(h_L;0,j_L)>t_L
	\right)
	\le
	\prob\left(
	\mathbb L(h_{\mathrm{flat}};0,j_L+\lfloor\HH_0L^{1/2}\rfloor+1)>t_L
	\right),
\end{equation*}
where $h_{\mathrm{flat}}(k)=\mathbf{1}_{k\text{ even}}$ denotes the flat initial condition. The limit theorem for line-to-point exponential LPP~\cite{Baik-Rains01,Ferrari04} gives
\begin{equation*}
	\prob\left(
	\mathbb L(h_{\mathrm{flat}};0,j_L+\lfloor\HH_0L^{1/2}\rfloor+1)>t_L
	\right)
	\longrightarrow
	1-F_{\mathrm{GOE}}\bigl((2\tau)^{-1/3}\HH\bigr)
\end{equation*}
as $L\to\infty$. Since $\HH>0$, the right-hand side tends to $0$ as $\tau\downarrow0$. Combining this with \eqref{eq:Lpermain_vs_L} and \eqref{eq:transversal_bound} proves the lemma.
\end{proof}

\begin{proof}[Proof of Lemma~\ref{lm:lpp_left}]
Define the point 
\begin{equation}\label{eq:def_pend}
	\mathrm p_L
	:=
	\bigl(-r_L,h_L^{\mathrm{ltail}}(-r_L)\bigr).
\end{equation}
By planarity, the maximizing paths for 
$\mathbb L^{\mathrm{per}}(h_L^{\mathrm{ltail}};0,j_L)$ and
$\mathbb L^{\mathrm{per}}(\mathrm p_L;-r_L,j_L)$ 
intersect. Switching their remaining portions after the first intersection gives
\begin{equation*}
	\mathbb L^{\mathrm{per}}(h_L^{\mathrm{ltail}};0,j_L)
	+
	\mathbb L^{\mathrm{per}}(\mathrm p_L;-r_L,j_L)
	\le
	\mathbb L^{\mathrm{per}}(h_L^{\mathrm{ltail}};-r_L,j_L)
	+
	\mathbb L^{\mathrm{per}}(\mathrm p_L;0,j_L).
\end{equation*}
Thus,
\begin{equation}\label{eq:bound_Lper_P123}
	\prob\left(
	\mathbb L^{\mathrm{per}}(h_L^{\mathrm{ltail}};0,j_L)>t_L
	\right)
	\le
	Q_{1,L}+Q_{2,L}+Q_{3,L},
\end{equation}
where
\begin{equation*}
\begin{aligned}
	Q_{1,L}
	&:=
	\prob\left(
	\mathbb L^{\mathrm{per}}(h_L^{\mathrm{ltail}};-r_L,j_L)
	>
	t_L+3ML^{1/2}
	\right),\\
	Q_{2,L}
	&:=
	\prob\left(
	\mathbb L^{\mathrm{per}}(\mathrm p_L;0,j_L)
	>
	t_L-(3M+2\HH)L^{1/2}
	\right),\\
	Q_{3,L}
	&:=
	\prob\left(
	\mathbb L^{\mathrm{per}}(\mathrm p_L;-r_L,j_L)
	<
	t_L-2\HH L^{1/2}
	\right), 
\end{aligned}
\end{equation*}
since, if none of these three events occurs, then
\begin{equation*}
	\mathbb L^{\mathrm{per}}(h_L^{\mathrm{ltail}};0,j_L)
	\le
	\bigl(t_L+3ML^{1/2}\bigr)
	+
	\bigl(t_L-(3M+2\HH)L^{1/2}\bigr)
	-
	\bigl(t_L-2\HH L^{1/2}\bigr)
	=
	t_L.
\end{equation*}

We first estimate $Q_{1,L}$. Translating the endpoint and comparing with the flat periodic initial condition gives
\begin{equation*}
	Q_{1,L}
	\le
	\prob \bigl(
	\mathbb L^{\mathrm{per}}
	\bigl(h_L^{\mathrm{flat}};0,j_L+ML^{1/2}+O(1)\bigr)
	>
	t_L+3ML^{1/2}
	\bigr).
\end{equation*}
The one-point relaxation-scale limit for periodic LPP with flat initial condition~\cite{Baik-Liu18} gives
\begin{equation}\label{eq:bound_P1}
	\limsup_{L\to\infty}Q_{1,L}
	\le
	1-
	\mathbb F_{\h_{\mathrm{flat}}}^{(1)}
	\bigl(\HH+M/2;(0,\tau)\bigr).
\end{equation}
It was shown in  \cite[Theorem~8.2]{Baik-Liu-Silva22} that, for every fixed $x\in\bbR$, 
$\mathbb F_{\h_{\mathrm{flat}}}^{(1)}
\bigl(\tau^{1/3}x;(0,\tau)\bigr)
\to F_{\mathrm{GOE}}(2^{2/3}x)$
as $\tau\downarrow0$. Since $\HH+M/2>0$, this implies
\begin{equation*}
	\mathbb F_{\h_{\mathrm{flat}}}^{(1)}
	\bigl(\HH+M/2;(0,\tau)\bigr)
	\longrightarrow
	1
\end{equation*}
as $\tau\downarrow0$. Hence
\begin{equation}\label{eq:P1goestoz}
	\lim_{\tau\downarrow0}
	\limsup_{L\to\infty}Q_{1,L}
	=
	0.
\end{equation}

We next estimate $Q_{2,L}$ and $Q_{3,L}$. 
The estimate in \cite[Corollary~3.14]{Schmid-Sly26} compares the periodic last-passage time with the expected value of the corresponding standard last-passage time.
After rotating to the usual up-right coordinates, the shape theorem for standard point-to-point exponential LPP~\cite{Rost80,Johansson00} gives
\begin{equation*}
	\bbE\left[\mathbb L(a,b;0,c)\right]
	=
	(c-b)+\sqrt{(c-b)^2-a^2}
	+
	O\bigl((c-b)^{1/3}\bigr)
\end{equation*}
as $c-b\to\infty$. Therefore,
\begin{equation*}
	\bbE\left[\mathbb L(\mathrm p_L;0,j_L)\right]
	=
	t_L+2(M-\HH)L^{1/2}
	-
	\frac{\delta^2}{2\tau}L^{1/2}
	+
	O\bigl(\tau^{1/3}L^{1/2}\bigr)
	+
	o\bigl(L^{1/2}\bigr),
\end{equation*}
and
\begin{equation*}
	\bbE\left[\mathbb L(\mathrm p_L;-r_L,j_L)\right]
	=
	t_L+2(M-\HH)L^{1/2}
	+
	O\bigl(\tau^{1/3}L^{1/2}\bigr)
	+
	o\bigl(L^{1/2}\bigr).
\end{equation*}
Consequently,
\begin{equation*}
\begin{aligned}
	Q_{2,L}
	&=
	\prob \left(
	\mathbb L^{\mathrm{per}}(\mathrm p_L;0,j_L)
	-
	\bbE\left[\mathbb L(\mathrm p_L;0,j_L)\right]
	>
	\bigl(
	\frac{\delta^2}{2\tau}-5M+O(\tau^{1/3})
	\bigr)L^{1/2}
	+
	o\bigl(L^{1/2}\bigr)
	\right),
\end{aligned}
\end{equation*}
and
\begin{equation*}
\begin{aligned}
	Q_{3,L}
	&=
	\prob \left(
	\mathbb L^{\mathrm{per}}(\mathrm p_L;-r_L,j_L)
	-
	\bbE\left[\mathbb L(\mathrm p_L;-r_L,j_L)\right]
	<
	-
	\bigl(2M+O(\tau^{1/3})\bigr)L^{1/2}
	+
	o\bigl(L^{1/2}\bigr)
	 \right).
\end{aligned}
\end{equation*}
Using \eqref{eq:nastau2332}, \cite[Corollary~3.14]{Schmid-Sly26} gives constants $C,c>0$ such that
\begin{equation}\label{eq:P2P3es}
	Q_{2,L}
	\le
	\re^{-C\left(\frac{\delta^2}{2\tau^{3/2}}-\frac{5M}{\tau^{1/2}}\right)},
	\qquad
	Q_{3,L}
	\le
	\re^{-\frac{cM}{\tau^{1/2}}}
\end{equation}
for all sufficiently small $\tau$ and sufficiently large even $L$.

Lemma~\ref{lm:lpp_left} follows from \eqref{eq:bound_Lper_P123}, \eqref{eq:P1goestoz}, and \eqref{eq:P2P3es}.
\end{proof}

Combining Lemmas~\ref{lm:lpp_center} and~\ref{lm:lpp_left} with \eqref{eq:Lper_split} proves \eqref{eq:to_be_compared}, and hence \eqref{eq:limit_tau0_special_c}.

\section{Proof of Proposition~\ref{prop:onepoint_b}: The one-point small-period limit}
\label{sec:proofofsmalltconj}

\begin{proof}
Set
\begin{equation*}
	c_0:=\frac{\sqrt{2}}{\pi^{1/4}},
	\qquad
	Z_\rp(\h):=
	c_0\rp^{1/4}
	\left(\cH^{\mathrm{PKPZ}}_\rp(\XX,\TT;\h)+\rp^{-1}\tau\right).
\end{equation*}
Fix $\epsilon>0$. 
By assumption, for all sufficiently small $\rp$,
\begin{equation*}
	\h_\rp(\XX)\le \epsilon\rp^{-1/4}
	\qquad \text{for all }\XX.
\end{equation*}
By the monotonicity and shift invariance of Theorem~\ref{thm:main2}, 
\begin{equation*}
	\prob(Z_\rp(\h_\rp)>z)
	\le \prob(Z_\rp(\epsilon\rp^{-1/4})>z)
	= \prob(Z_\rp(\h_{\mathrm{flat}})>z-c_0\epsilon), 
\end{equation*}
where $\h_{\mathrm{flat}}\equiv 0$. 
For the flat initial condition, it was shown in~\cite[Theorem~8.1]{Baik-Liu-Silva22} that $\prob(Z_\rp(\h_{\mathrm{flat}})>z-c_0\epsilon)$ converges to $\prob(\B(\TT)>z-c_0\epsilon)$ as $\rp\to 0$. 
Hence, 
\begin{equation} \label{eq:prpup}
	\limsup_{\rp\to0}\prob(Z_\rp(\h_\rp)>z)
	\le
	\prob(\B(\TT)>z-c_0\epsilon).
\end{equation}

For the lower bound, the assumption also implies that, for all sufficiently small $\rp$, there exists $\XX_\rp$ such that
\begin{equation*}
	\h_\rp(\XX_\rp)\ge -\epsilon\rp^{-1/4}.
\end{equation*}
Thus, by the periodicity, 
\begin{equation*}
	\h_\rp(\alpha) \ge  \hat{\h}_{\rp}(\alpha):= -\epsilon\rp^{-1/4}+\h_{\mathrm{pnw}}(\alpha-\XX_\rp), 
\end{equation*}
where $\h_{\mathrm{pnw}}(\alpha)= -\infty \mathbf{1}_{\alpha\notin \rp\bbZ}$. 
By the monotonicity and shift invariance, 
\begin{equation*}
\begin{aligned}
	\prob(Z_\rp(\h_\rp)\le z)
	&\le 
	\prob(Z_\rp(\hat{\h}_{\rp}) \le z) 
	=
	\prob\left(
	c_0\rp^{1/4}
	\left(
	\cH^{\mathrm{PKPZ}}_\rp(\XX-\XX_\rp,\TT;\h_{\mathrm{pnw}})
	+\rp^{-1}\tau
	\right)
	\le z+c_0\epsilon
	\right). 
\end{aligned}
\end{equation*}
It was shown in~\cite[Theorem~1.5]{Baik-Liu-Silva22} that for $\XX=0$ the right-hand side converges to 
$\prob\left(\B(\TT)\le z+c_0\epsilon\right)$. 
Although the theorem is stated there at $\XX=0$, the proof is spatially uniform and does not depend on the value of the spatial parameter; hence the same limit holds with $\XX-\XX_\rp$ in place of $0$. 
Thus
\begin{equation} \label{eq:prplo}
	\liminf_{\rp\to0}\prob(Z_\rp(\h_\rp)>z)
	\ge
	\prob(\B(\TT)>z+c_0\epsilon).
\end{equation}

Since $\epsilon>0$ is arbitrary, \eqref{eq:prpup} and \eqref{eq:prplo} imply the proposition.
\end{proof}

\section{Periodic narrow-wedge and flat initial conditions}
\label{sec:examples}

In this section, we evaluate the two initial-condition-dependent quantities in the limiting formula,
\begin{equation*}
	\pchlim_\h(\eta,\xi)
	\qquad\text{and}\qquad
	\det(\rI+\limKe_\h(\rz))_{L^2(\bbR)},
\end{equation*}
for the periodic narrow-wedge initial condition $\h_{\mathrm{pnw}}(\XX)=-\infty\mathbf{1}_{\XX\notin\bbZ}$ 
and the flat initial condition $\h_{\mathrm{flat}}\equiv0$. The resulting formulas agree with those obtained in~\cite{Baik-Liu19,Baik-Liu21}.

The determinant factorization in Proposition~\ref{result:KasQ} reduces the flat case to the periodic narrow-wedge determinant and a relative determinant on $L^2(\bbR_+)$. These determinants, as well as the periodic narrow-wedge characteristic function, are evaluated using Wiener--Hopf methods. We collect the required preliminaries in Subsection~\ref{sec:WH_setup} and treat the two initial conditions in Subsections~\ref{sec:examples_pnw} and~\ref{sec:examples_flat}.

\subsection{Wiener--Hopf preliminaries}
\label{sec:WH_setup}

Let
\begin{equation}\label{eq:FTrr}
	\mathcal F(f)(\lambda)
	:=
	\int_\bbR \re^{-\ri\sfx\lambda}f(\sfx)\,\rd\sfx
\end{equation}
denote the Fourier transform. Let $E:L^2(\bbR_+)\to L^2(\bbR)$ be extension by zero and let $R:L^2(\bbR)\to L^2(\bbR_+)$ be restriction, so that $P_+:=ER:L^2(\bbR)\to L^2(\bbR)$ is the orthogonal projection onto $L^2(\bbR_+)$. 
For $b\in L^\infty(\bbR)$, define the Wiener--Hopf operator
\begin{equation*}
	W(b):=R\mathcal F^{-1}M_b\mathcal F E
\end{equation*}
on $L^2(\bbR_+)$, where $M_b$ denotes multiplication by $b$.

Suppose that $b=b_+b_-$ is a Wiener--Hopf factorization, where $b_+$ and $1/b_+$ extend to bounded analytic functions on $\bbC_+$, and $b_-$ and $1/b_-$ extend to bounded analytic functions on $\bbC_-$. Krein's theorem~\cite{Bottcher-Silbermann06} gives\footnote{The book~\cite{Bottcher-Silbermann06} uses the convention $\mathcal F(f)(\lambda)=\int_\bbR \re^{\ri\sfx\lambda}f(\sfx)\,\rd\sfx$, under which \eqref{eq:WH_inverse} takes the form $W(b)^{-1}=W(1/b_+)W(1/b_-)$. That convention maps $L^2(\bbR_+)$ to $H^2(\bbC_+)$, whereas \eqref{eq:FTrr} maps it to $H^2(\bbC_-)$, accounting for the reversed order of the factors.}
\begin{equation}\label{eq:WH_inverse}
	W(b)^{-1}=W(1/b_-)\,W(1/b_+).
\end{equation}

Recall the heat-kernel operator $\mathsf G$ from \eqref{eq:limkernel2TT}. Its Fourier multiplier is $\re^{-\lambda^2/2}$, and hence $\mathcal F(\rI-\rz\mathsf G)\mathcal F^{-1}=M_a$ with the symbol 
\begin{equation} \label{eq:asymb}
	a(\lambda):=1-\rz\re^{-\lambda^2/2}.
\end{equation}
Consequently, for $0<|\rz|<1$,
\begin{equation}\label{eq:WH_IzG}
	R(\rI-\rz\mathsf G)E=W(a),
	\qquad
	R(\rI-\rz\mathsf G)^{-1}E=W(1/a).
\end{equation}
Define
\begin{equation*}
	a_\pm(\lambda)
	:=
	\exp\left[
	\pm\frac{1}{2\pi\ri}
	\int_\bbR
	\frac{\log(1-\rz\re^{-t^2/2})}{t-\lambda}\,\rd t
	\right],
	\qquad \lambda\in\bbC_\pm.
\end{equation*}
By the Sokhotski--Plemelj formula, the boundary values satisfy $a=a_+a_-$, and
$\lim_{|\lambda|\to\infty}a_\pm(\lambda)=1$ in $\bbC_\pm$. 
Since~\cite[(2.25)]{Baik-Liu19}
\begin{equation*}
	\int_{-\ri\infty}^{\ri\infty}
	\frac{\log(1-\rz\re^{w^2/2})}{w-\zeta}\,
	\frac{\rd w}{2\pi\ri}
	=
	-\frac{1}{\sqrt{2\pi}}
	\int_{-\infty}^{\zeta}
	\polylog_{1/2}\!\left(\rz\re^{(\zeta^2-y^2)/2}\right)
	\,\rd y
\end{equation*}
for $0<|\rz|<1$ and $\Re(\zeta)<0$, we find that 
\begin{equation}\label{eq:apm_h}
	a_+(\ri\eta)=\re^{\hftn(\eta,\rz)}
	\quad\text{for $\Re(\eta)>0$},
	\qquad
	a_-(\ri\xi)=\re^{\hftn(\xi,\rz)}
	\quad\text{for $\Re(\xi)<0$},
\end{equation}
where $\hftn$ is defined in \eqref{eq:def_h_R}.

We need the following formula for the Widom constant associated with the symbol \eqref{eq:asymb}.

\begin{lm}\label{lem:Widom}
For $0<|\rz|<1$, let $a(\lambda):=1-\rz\re^{-\lambda^2/2}$. 
Then, the operator $W(1/a)W(a)-\I_+$ is trace class on $L^2(\bbR_+)$, and
\begin{equation}\label{eq:Widom_constant}
	\det\!\bigl(W(1/a)W(a)\bigr)_{L^2(\bbR_+)}
	=
	\re^{2B(\rz)}
\end{equation}
with $B(\rz)$  defined in \eqref{eq:def_Bz}.
\end{lm}

\begin{proof}
By the Akhiezer--Kac formula~\cite[Section~10.84]{Bottcher-Silbermann06}, the operator $W(1/a)W(a)-\I_+$ is trace class and
\begin{equation}\label{eq:Akhiezer_Kac}
	\log\det\!\bigl(W(1/a)W(a)\bigr)_{L^2(\bbR_+)}
	=
	\int_0^\infty \sfx\,s(\sfx)s(-\sfx)\,\rd\sfx,
\end{equation}
where $s:=\mathcal F^{-1}(\log a)$. 
Since $|\rz|<1$,
\begin{equation}\label{eq:sxformula}
	\log a(\lambda)
	=
	-\sum_{k\ge1}\frac{\rz^k}{k}\re^{-k\lambda^2/2}, 
	\qquad 
	s(\sfx)
	=
	-\frac{1}{\sqrt{2\pi}}
	\sum_{k\ge1}\frac{\rz^k}{k^{3/2}}
	\re^{-\sfx^2/(2k)}.
\end{equation}
Substituting \eqref{eq:sxformula} into \eqref{eq:Akhiezer_Kac} implies 
\begin{equation} \label{eq:intofss}
	\int_0^\infty \sfx\,s(\sfx)s(-\sfx)\,\rd\sfx
	=
	\frac{1}{2\pi}
	\sum_{k,k'\ge1}
	\frac{\rz^{k+k'}}{(k+k')\sqrt{kk'}}. 
\end{equation}
From the definition \eqref{eq:def_Bz}, the result follows. 
\end{proof}

\subsection{The periodic narrow-wedge initial condition}
\label{sec:examples_pnw}

Throughout this subsection, let $\h_{\mathrm{pnw}}$ denote the periodic narrow-wedge initial condition, defined by $\h_{\mathrm{pnw}}(\XX)=-\infty\mathbf{1}_{\XX\notin\bbZ}$. 

\begin{prop}\label{prop:pnw}
For $\h=\h_{\mathrm{pnw}}$,
\begin{equation}\label{eq:pnw_pchlim}
	\pchlim_{\mathrm{pnw}}(\eta,\xi)
	=
	\frac{\re^{\hftn(\xi,\rz)+\hftn(\eta,\rz)}}{\eta-\xi},
	\qquad
	\rz=\re^{-\eta^2/2},
\end{equation}
for $\arg(\eta)\in(-\pi/4,\pi/4)$ and $\Re(\xi)<0$.
\end{prop}

\begin{proof}
Recall the definition \eqref{eq:pch_lim} of $\pchlim_\h$. For $\h=\h_{\mathrm{pnw}}$, the event $\{\btau<1\}$ is $\{\B(0)\le0\}$, and on this event $\btau=0$. Hence the first term in $\pchlim_{\mathrm{pnw}}(\eta,\xi)$ is
\begin{equation}\label{eq:pnw_term1_compute}
	\int_{-\infty}^0\re^{(\eta-\xi)s}\,\rd s
	=
	\frac{1}{\eta-\xi}.
\end{equation}

The stopping time $\bntau$ takes values in $\bbZ_{>0}$, with $\bntau=k$ if and only if
\begin{equation*}
	\B(1),\ldots,\B(k-1)>0,
	\qquad
	\B(k)\le0.
\end{equation*}
Therefore, the second term in $\pchlim_{\mathrm{pnw}}(\eta,\xi)$ equals
\begin{equation*}
\begin{aligned}
	\sum_{k=1}^\infty
	\int_{-\infty}^0\re^{-s_0\xi}\,\rd s_0
	\int_0^\infty\rd s_1\cdots
	\int_0^\infty\rd s_{k-1}
	\int_{-\infty}^0\re^{\eta s_k-\frac12\eta^2k}
	\prod_{j=0}^{k-1}\mathsf G(s_j,s_{j+1})\,\rd s_k.
\end{aligned}
\end{equation*}
Changing variables $s_j\mapsto-s_j$ gives
\begin{equation}\label{eq:pnw_term2_compute}
	\sum_{k=1}^\infty \rz^k 
	\int_0^\infty\re^{\xi s_0}\,\rd s_0
	\int_{-\infty}^0\rd s_1\cdots
	\int_{-\infty}^0\rd s_{k-1}
	\int_0^\infty
	\prod_{j=0}^{k-1}\mathsf G(s_j,s_{j+1})
	\re^{-\eta s_k}\,\rd s_k,
\end{equation}
where we set 
\begin{equation}\label{eq:zetesq}
	\rz:=\re^{-\eta^2/2}.
\end{equation}
Since $\arg(\eta)\in(-\pi/4,\pi/4)$, we have $|\rz|<1$. 

For $a,b\in\{+,-\}$, let $G_{ab}:L^2(\bbR_b)\to L^2(\bbR_a)$ denote the corresponding block of $\mathsf G$, and set
\begin{equation*}
	u_\eta(\sfx):=\re^{-\eta\sfx}.
\end{equation*}
Using the bilinear pairing $\langle f,g\rangle_{L^2(\bbR_+)}:= \int_0^\infty f(\sfx)g(\sfx)\,\rd\sfx
$, 
the expression \eqref{eq:pnw_term2_compute} becomes
\begin{equation*}
	\rz\langle G_{++}u_\eta,u_{-\xi}\rangle
	+
	\sum_{k\ge2}\rz^k
	\langle G_{+-}G_{--}^{k-2}G_{-+}u_\eta,u_{-\xi}\rangle.
\end{equation*}
Summing the series and subtracting it from \eqref{eq:pnw_term1_compute}, we obtain
\begin{equation*}
	\pchlim_{\mathrm{pnw}}(\eta,\xi)
	=
	\langle Su_\eta,u_{-\xi}\rangle_{L^2(\bbR_+)},
	\qquad
	S
	:=
	\rI_+-\rz G_{++}
	-\rz^2G_{+-}(\rI_--\rz G_{--})^{-1}G_{-+}.
\end{equation*}

With respect to $L^2(\bbR)=L^2(\bbR_+)\oplus L^2(\bbR_-)$,
\begin{equation*}
	\rI-\rz\mathsf G
	=
	\begin{bmatrix}
		\rI_+-\rz G_{++}&-\rz G_{+-}\\
		-\rz G_{-+}&\rI_--\rz G_{--}
	\end{bmatrix}.
\end{equation*}
The Schur complement formula and \eqref{eq:WH_IzG} give
\begin{equation*}
	S^{-1}
	=
	R(\rI-\rz\mathsf G)^{-1}E
	=
	W(1/a).
\end{equation*}
Applying \eqref{eq:WH_inverse} to the factorization $1/a=(1/a_+)(1/a_-)$ yields
\begin{equation*}
	\pchlim_{\mathrm{pnw}}(\eta,\xi)
	=
	\langle W(a_-)W(a_+)u_\eta,u_{-\xi}\rangle_{L^2(\bbR_+)}.
\end{equation*}

It remains to evaluate this pairing. Since $(\mathcal FEu_\eta)(\lambda) =\frac{1}{\eta+\ri\lambda}$, 
we have 
\begin{equation*}
	(M_{a_+} \mathcal{F}Eu_\eta)(\lambda)
	= \frac{a_+(\lambda)}{\eta+\ri\lambda}
	=
	\frac{a_+(\ri\eta)}{\eta+\ri\lambda}
	+
	\frac{a_+(\lambda)-a_+(\ri\eta)}{\eta+\ri\lambda}.
\end{equation*}
The second term is analytic in $\bbC_+$ and is annihilated by the projection onto $H^2(\bbC_-)$. Hence
\begin{equation*}
	(\mathcal FEW(a_+)u_\eta)(\lambda)
	=
	\frac{a_+(\ri\eta)}{\eta+\ri\lambda}.
\end{equation*}
Since multiplication by $a_-$ preserves $H^2(\bbC_-)$,
\begin{equation*}
	(\mathcal FEW(a_-)W(a_+)u_\eta)(\lambda)
	=
	\frac{a_-(\lambda)a_+(\ri\eta)}{\eta+\ri\lambda}.
\end{equation*}
For $\Re(\xi)<0$, evaluation at $\lambda=\ri\xi$ gives
\begin{equation*}
	\langle W(a_-)W(a_+)u_\eta,u_{-\xi}\rangle_{L^2(\bbR_+)}
	=
	(\mathcal FEW(a_-)W(a_+)u_\eta)(\ri\xi)
	=
	\frac{a_-(\ri\xi)a_+(\ri\eta)}{\eta-\xi}.
\end{equation*}
The identities \eqref{eq:apm_h} and \eqref{eq:zetesq} now give \eqref{eq:pnw_pchlim}.
\end{proof}

\begin{prop}\label{prop:pnw22}
For $\h=\h_{\mathrm{pnw}}$ and $|\rz|<1$,
\begin{equation}\label{eq:pnw_energy}
	\det(\rI+\limKe_{\mathrm{pnw}}(\rz))_{L^2(\bbR)}
	=
	\re^{2B(\rz)}.
\end{equation}
\end{prop}

\begin{proof}
For $\h=\h_{\mathrm{pnw}}$, the event $\{\btau<1\}$ is $\{\B(0)\le0\}$, and hence $	\sfT_{\mathrm{pnw}}(\sfx,\sfy) =	\mathbf{1}_{\sfx\leq 0}\mathsf G(\sfx,\sfy).
$
Thus $\sfT_{\mathrm{pnw}}=(\rI-P_+)\mathsf G$. Substituting this identity into \eqref{eq:limkernel} gives
\begin{equation*}
	\limKe_{\mathrm{pnw}}(\rz)
	=
	\rz P_+(\rI-\rz\mathsf G)^{-1}\mathsf GP_+
	-
	\rz P_+(\rI-\rz\mathsf G)^{-1}P_+\mathsf GP_+.
\end{equation*}
Using $\rz(\rI-\rz\mathsf G)^{-1}\mathsf G = -\rI+(\rI-\rz\mathsf G)^{-1}$, 
we obtain
\begin{equation}\label{eq:pnw_K_WH}
	\limKe_{\mathrm{pnw}}(\rz)
	=
	-P_+
	+
	P_+(\rI-\rz\mathsf G)^{-1}P_+
	(\rI-\rz\mathsf G)P_+\
    =
	-P_++EW(1/a)W(a)R,
\end{equation}
where we used \eqref{eq:WH_IzG} and $ER=P_+$.
The operator in \eqref{eq:pnw_K_WH} vanishes on $L^2(\bbR_-)$, while its restriction to $L^2(\bbR_+)$ is $-\rI_++W(1/a)W(a)$. 
Therefore, Lemma~\ref{lem:Widom} gives
\begin{equation*}
	\det(\rI+\limKe_{\mathrm{pnw}}(\rz))_{L^2(\bbR)}
	=
	\det\!\bigl(W(1/a)W(a)\bigr)_{L^2(\bbR_+)}
	=
	\re^{2B(\rz)}.
\end{equation*}
\end{proof}

\subsection{The flat initial condition}
\label{sec:examples_flat}

Throughout this subsection, let $\h_{\mathrm{flat}}\equiv0$ denote the flat initial condition.

\begin{prop}\label{prop:flat}
For $\h=\h_{\mathrm{flat}}$,
\begin{equation}\label{eq:flat_pchlim}
	\pchlim_{\mathrm{flat}}(\eta,\xi)
	=
	\frac{2\eta\bigl(\re^{(\xi^2-\eta^2)/2}-1\bigr)}
	{\xi^2-\eta^2}
\end{equation}
for $\arg(\eta)\in(-\pi/4,\pi/4)$ and $\Re(\xi)<0$, where the right-hand side is understood by continuity when $\xi=-\eta$.
\end{prop}

\begin{proof}
For $\h\equiv0$,
\begin{equation*}
	\btau=\inf\{\sfx\ge0:\B(\sfx)\le0\},
	\qquad
	\bntau=\inf\{\sfx\ge1:\B(\sfx)\le0\}.
\end{equation*}
If $\B(0)=s\le0$, then $\btau=0$. If $s>0$, then $\btau$ has density $\frac{s}{\sqrt{2\pi t^3}}\re^{-s^2/(2t)}\,\rd t$ for $t>0$, 
and $\B(\btau)=0$ on $\{\btau<\infty\}$. It follows that
\begin{equation}\label{eq:flat_E1}
	\bbE_{\B(0)=s}
	\left[
	\re^{\eta\B(\btau)-\frac12\eta^2\btau}
	\mathbf{1}_{\btau<1}
	\right]
	=
	\mathbf{1}_{s\le0}\re^{s\eta}
	+
	\mathbf{1}_{s>0}
	\left(
	\re^{-s\eta}\Phi(\eta-s)
	+
	\re^{s\eta}\Phi(-\eta-s)
	\right),
\end{equation}
where
\begin{equation*}
	\Phi(\zeta)
	:=
	\frac{1}{\sqrt{2\pi}}
	\int_{-\infty}^{\zeta}\re^{-u^2/2}\,\rd u.
\end{equation*}

For the second term, conditioning on $\B(1)$ and using the strong Markov property at time $1$ gives 
\begin{equation}\label{eq:Bntaaa}
\begin{aligned}
	&\bbE_{\B(0)=s}
	\left[
	\re^{\eta\B(\bntau)-\frac12\eta^2\bntau}
	\mathbf{1}_{\btau<1,\,\bntau<\infty}
	\right]\\
	&\qquad=
	\re^{-\frac12\eta^2}
	\int_\bbR
	\bbE_{\B(0)=y}
	\left[
	\re^{\eta\B(\btau)-\frac12\eta^2\btau}
	\right]
	\prob_{\B(0)=s}
	\bigl(\B(1)\in\rd y,\ \btau<1\bigr).
\end{aligned}
\end{equation}
Optional stopping gives
\begin{equation*}
	\bbE_{\B(0)=y}
	\left[
	\re^{\eta\B(\btau)-\frac12\eta^2\btau}
	\right]
	=
	\re^{\eta y}\mathbf{1}_{y\le0}
	+
	\re^{-\eta y}\mathbf{1}_{y>0},
\end{equation*}
while the reflection principle gives
\begin{equation*}
	\frac{
	\prob_{\B(0)=s}
	\bigl(\B(1)\in\rd y,\ \btau<1\bigr)}
	{\rd y}
	=
	\begin{cases}
		\mathsf G(s,y),&s\le0,\\[2pt]
		\mathsf G(s,y)\mathbf{1}_{y\le0}
		+
		\displaystyle\frac{1}{\sqrt{2\pi}}
		\re^{-(s+y)^2/2}\mathbf{1}_{y>0},&s>0.
	\end{cases}
\end{equation*}
Evaluating the Gaussian integrals in \eqref{eq:Bntaaa}, we obtain
\begin{equation}\label{eq:flat_E2}
	\mathbf{1}_{s\le0}
	\left(
	\re^{s\eta}\Phi(-(s+\eta))
	+
	\re^{-s\eta}\Phi(s-\eta)
	\right)
	+
	2\mathbf{1}_{s>0}\re^{s\eta}\Phi(-(s+\eta)).
\end{equation}
Subtracting \eqref{eq:flat_E2} from \eqref{eq:flat_E1} and using
$\Phi(\zeta)+\Phi(-\zeta)=1$,  
we find
\begin{equation*}
	\bbE_{\B(0)=s}
	\left[
	\left(
	\re^{\eta\B(\btau)-\frac12\eta^2\btau}
	-
	\re^{\eta\B(\bntau)-\frac12\eta^2\bntau}
	\mathbf{1}_{\bntau<\infty}
	\right)
	\mathbf{1}_{\btau<1}
	\right]
	=
	\mathbf{1}_{s\le0}\Psi_-(s,\eta)
	+
	\mathbf{1}_{s>0}\Psi_+(s,\eta),
\end{equation*}
with
\begin{equation*}
	\Psi_-(s,\eta)
	:=
	\re^{s\eta}\Phi(s+\eta)
	-
	\re^{-s\eta}\Phi(s-\eta),
	\qquad
	\Psi_+(s,\eta)
	:=
	\re^{-s\eta}\Phi(\eta-s)
	-
	\re^{s\eta}\Phi(-(s+\eta)).
\end{equation*}
Therefore,
\begin{equation}\label{eq:flat_pchlim_split}
	\pchlim_{\mathrm{flat}}(\eta,\xi)
	=
	\int_{-\infty}^0
	\re^{-s\xi}\Psi_-(s,\eta)\,\rd s
	+
	\int_0^\infty
	\re^{-s\xi}\Psi_+(s,\eta)\,\rd s.
\end{equation}

Direct Gaussian integration gives
\begin{align*}
	\int_{-\infty}^0
	\re^{-s(\xi\mp\eta)}\Phi(s\pm\eta)\,\rd s
	&=
	\frac{
	\re^{(\xi^2-\eta^2)/2}\Phi(\xi)-\Phi(\pm\eta)}
	{\xi\mp\eta},\\
	\int_0^\infty
	\re^{-s(\xi\pm\eta)}\Phi(-s\pm\eta)\,\rd s
	&=
	\frac{
	\Phi(\pm\eta)-\re^{(\xi^2-\eta^2)/2}\Phi(-\xi)}
	{\xi\pm\eta}.
\end{align*}
Substituting these identities into \eqref{eq:flat_pchlim_split} and using $\Phi(\zeta)+\Phi(-\zeta)=1$, we obtain 
\begin{equation*}
	\pchlim_{\mathrm{flat}}(\eta,\xi)
	=
	\frac{2\eta}{\xi^2-\eta^2}
	\left[
	\re^{(\xi^2-\eta^2)/2}
	\bigl(\Phi(\xi)+\Phi(-\xi)\bigr)
	-
	\bigl(\Phi(\eta)+\Phi(-\eta)\bigr)
	\right]
	=
	\frac{2\eta\bigl(\re^{(\xi^2-\eta^2)/2}-1\bigr)}
	{\xi^2-\eta^2}.
\end{equation*}
\end{proof}

\begin{cor}\label{cor:flat_pchlim_roots}
For $\h=\h_{\mathrm{flat}}$ and $0<|\rz|<1$,
\begin{equation}\label{eq:flat_pchlim_roots}
	\pchlim_{\mathrm{flat}}(\eta,\xi)
	=
	\eta\,\mathbf{1}_{\xi+\eta=0},
	\qquad
	\eta\in\rR_\rz,\quad \xi\in\rL_\rz.
\end{equation}
\end{cor}

\begin{proof}
Since $\rL_\rz=-\rR_\rz$, for each $\eta\in\rR_\rz$, the condition $\xi+\eta=0$ singles out the unique point $\xi=-\eta$ in $\rL_\rz$. 
Since $\re^{-\eta^2/2} = \rz = \re^{-\xi^2/2}$, 
the result follows from Proposition~\ref{prop:flat}. \end{proof}

\begin{prop}\label{prop:flat22}
For $\h=\h_{\mathrm{flat}}$ and $|\rz|<1$,
\begin{equation}\label{eq:flat_energy}
	\det(\I+\limKe_{\mathrm{flat}}(\rz))_{L^2(\bbR)}
	=
	\re^{B(\rz)-\frac14\log(1-\rz)}.
\end{equation}
\end{prop}

\begin{proof}
By \eqref{eq:KasQ_factorization} and Proposition~\ref{prop:pnw22},
\begin{equation*}
\begin{aligned}
	\det(\I+\limKe_{\mathrm{flat}}(\rz))_{L^2(\bbR)}
	&=
	\re^{2B(\rz)}
	\det\left(
	\I_+
	+
	\rz(\I_+-\rz\mathsf G_+)^{-1}
	\sfTp_{\mathrm{flat}}
	\right)_{L^2(\bbR_+)}.
\end{aligned}
\end{equation*}
Since $\h_{\mathrm{flat}}\equiv 0$, the reflection principle implies 
\begin{equation*}
	\sfTp_{\mathrm{flat}}(\sfx,\sfy)
	=
	\frac{1}{\sqrt{2\pi}}
	\re^{-(\sfx+\sfy)^2/2},
	\qquad
	\sfx,\sfy>0.
\end{equation*}

For $b\in L^\infty(\bbR)$, define the Hankel operator
\begin{equation*}
	H(b)
	:=
	R\mathcal F^{-1}M_b\mathcal FJE
\end{equation*}
on $L^2(\bbR_+)$, where $(Jf)(\sfx):=f(-\sfx)$. Set $g(\lambda):=\re^{-\lambda^2/2}$ so that 
$a(\lambda)=1-\rz g(\lambda)$. 
A direct computation gives $H(g)=\sfTp_{\mathrm{flat}}$. 
Hence, since $H(1)=RJE=0$,
\begin{equation*}
	H(a)=H(1-\rz g)= H(1)- \rz H(g)= -\rz\sfTp_{\mathrm{flat}}.
\end{equation*}
Moreover, by \eqref{eq:WH_IzG},
$W(a)=\I_+-\rz\mathsf G_+$. 
Thus,
\begin{equation}\label{eq:flat_widom_split}
	\det(\I+\limKe_{\mathrm{flat}}(\rz))_{L^2(\bbR)}
	=
	\re^{2B(\rz)}
	\det\left(
	\I_+-W(a)^{-1}H(a)
	\right)_{L^2(\bbR_+)}.
\end{equation}

The Wiener--Hopf-minus-Hankel determinant formula from~\cite[Theorem~5]{Basor97} and~\cite[Theorem~1.1]{BasorEhrhardt03} gives
\begin{equation}\label{eq:Basor}
	\log
	\det\left(
	\I_+-W(a)^{-1}H(a)
	\right)_{L^2(\bbR_+)}
	=
	-\frac12
	\int_0^\infty \sfx\,s(\sfx)^2\,\rd\sfx
	-\frac14\log a(0),
\end{equation}
where $s=\mathcal F^{-1}(\log a)$ is given by \eqref{eq:sxformula}. 
Since $s$ is even, the equation \eqref{eq:intofss} implies that the first term is $-B(\rz)$. 
Since $a(0)=1-\rz$, \eqref{eq:Basor} becomes
\begin{equation*}
	\log
	\det\left(
	\I_+-W(a)^{-1}H(a)
	\right)_{L^2(\bbR_+)}
	=
	-B(\rz)-\frac14\log(1-\rz).
\end{equation*}
Substituting this into \eqref{eq:flat_widom_split} proves \eqref{eq:flat_energy}.
\end{proof}

\section{Additional formulas and properties of the PKPZ multipoint distributions}
\label{sec:additional}

We collect several additional formulas and structural properties of the quantities appearing in the multipoint distributions of Section~\ref{sec:PKPZ_formula}. 
None of the results below is needed in the rest of the paper. 
Section~\ref{sec:Fredholm_D} gives an operator representation of $\rD_\h(\bz)$, Section~\ref{sec:STalt} gives Fourier--Laplace representations of $\sfS_\rz$ and $\sfT_\h$, and Section~\ref{sec:reversal} establishes space-reversal identities.

\subsection{Fredholm determinant formula for $\rD_\h(\bz)$}
\label{sec:Fredholm_D}

By the general Fredholm determinant identities in~\cite[Lemmas~4.8 and~4.9]{Baik-Liu19}, $\rD_\h(\bz)$ admits the following Fredholm determinant representation. 
We state it without proof, since it is a direct application of those lemmas.

\begin{prop}[Fredholm determinant formula for $\rD_\h(\bz)$]
\label{prop:Fredholm}
Under the assumptions and notation of Definition~\ref{def:D_series}, define the discrete sets
\begin{equation*}
	\Sp_1
	:=
	\bigcup_{\substack{1\le i\le m\\ i\text{ odd}}}\rL_{\rz_i}
	\cup
	\bigcup_{\substack{1\le i\le m\\ i\text{ even}}}\rR_{\rz_i},
	\qquad
	\Sp_2
	:=
	\bigcup_{\substack{1\le i\le m\\ i\text{ odd}}}\rR_{\rz_i}
	\cup
	\bigcup_{\substack{1\le i\le m\\ i\text{ even}}}\rL_{\rz_i}.
\end{equation*}
Set $\rz_0=\rz_{m+1}:=0$, and define
\begin{equation*}
	\rK_1:\ell^2(\Sp_2)\to\ell^2(\Sp_1),
	\qquad
	\rK_{2;\h}:\ell^2(\Sp_1)\to\ell^2(\Sp_2)
\end{equation*}
as follows. For $\zeta\in(\rL_{\rz_i}\cup\rR_{\rz_i})\cap\Sp_1$ and $\zeta'\in(\rL_{\rz_j}\cup\rR_{\rz_j})\cap\Sp_2$, 
set
\begin{equation*}
	\rK_1(\zeta,\zeta')
	:=
	\bigl(\delta_{i,j}+\delta_{i,\,j+(-1)^i}\bigr)
	\left(1-\frac{\rz_{j-(-1)^j}}{\rz_j}\right)
	\frac{
		\mathrm f_i(\zeta)
		\re^{2\hftn(\zeta,\rz_i)-\hftn(\zeta,\rz_{i-(-1)^i})-\hftn(\zeta',\rz_{j-(-1)^j})}
	}{
		\zeta(\zeta-\zeta')
	}.
\end{equation*}
For $\zeta\in(\rL_{\rz_i}\cup\rR_{\rz_i})\cap\Sp_2$ and $\zeta'\in(\rL_{\rz_j}\cup\rR_{\rz_j})\cap\Sp_1$, 
set
\begin{equation*}
\begin{aligned}
	\rK_{2;\h}(\zeta,\zeta')
	&:=
	\bigl(\delta_{i,j}+\delta_{i,\,j-(-1)^i}\bigr)
	\left(1-\frac{\rz_{j+(-1)^j}}{\rz_j}\right)
	\frac{
		\mathrm f_i(\zeta)
		\re^{2\hftn(\zeta,\rz_i)-\hftn(\zeta,\rz_{i+(-1)^i})-\hftn(\zeta',\rz_{j+(-1)^j})}
	}{\zeta}\\
	&\quad\times
	\begin{cases}
		\re^{-\hftn(\zeta,\rz_1)-\hftn(\zeta',\rz_1)}
		\pchlim_\h(\zeta,\zeta'),
		&j=1,\\[2pt]
		\dfrac{1}{\zeta-\zeta'},
		&j\ge2.
	\end{cases}
\end{aligned}
\end{equation*}
Then $\rK_1\rK_{2;\h}$ is trace class on $\ell^2(\Sp_1)$, and
\begin{equation}\label{eq:fredholm}
	\rD_\h(\bz)
	=
	\det(\rI-\rK_1\rK_{2;\h})_{\ell^2(\Sp_1)}.
\end{equation}
\end{prop}

For each $\ell\ge1$, omitting any nonexistent sets, $\rK_1$ maps 
$\rR_{\rz_{2\ell-1}}\cup\rL_{\rz_{2\ell}}$ to $\rL_{\rz_{2\ell-1}}\cup\rR_{\rz_{2\ell}}$, 
and $\rK_{2;\h}$ maps in the reverse direction. Hence $\rK_1\rK_{2;\h}$ is block diagonal in $\ell$, and the initial condition enters only through the first block of $\rK_{2;\h}$.

\subsection{Alternative expressions for $\sfS_\rz$ and $\sfT_\h$}
\label{sec:STalt}

The following identities give a contour representation of the resolvent kernel $\sfS_\rz$ and a Fourier--Laplace representation of the hitting kernel $\sfT_\h$.

\begin{lm}\label{lm:alt_S_T}
\begin{enumerate}[(a)]
\item For $0<|\rz|<1$ and $\sfx,\sfx'\in\bbR$,
\begin{equation}\label{eq:S_contour}
	\sfS_\rz(\sfx,\sfx')
	=
	\int_\gamma
	\frac{\rz\re^{\xi(\sfx-\sfx')}}{\re^{-\xi^2/2}-\rz}\,
	\frac{\rd\xi}{2\pi\ri},
\end{equation}
where $\gamma$ is a vertical line satisfying $-\sqrt{-\log|\rz|}<\Re(\xi)<0$. 

\item For $\h\in\uc_1^0$, $\arg(\eta)\in(-\frac{\pi}{4},\frac{\pi}{4})$, and $s\in\bbR$,
\begin{equation}\label{eq:T_OST}
	\re^{-\eta^2/2}
	\int_\bbR \re^{\eta\sfy}\sfT_\h(s,\sfy)\,\rd\sfy
	=
	\bbE_{\B(0)=s}
	\left[
	\re^{\eta\B(\btau)-\frac12\eta^2\btau}
	\mathbf{1}_{\btau<1}
	\right].
\end{equation}
\end{enumerate}
\end{lm}

\begin{proof}
For part~(a), expand
\begin{equation*}
    \frac{\rz}{e^{-\xi^2/2}-\rz}
    =
    \sum_{k\ge1}\rz^k e^{k\xi^2/2},
\end{equation*} 
which converges on the contour since $|e^{-\xi^2/2}|>|\rz|$. 
Integrating term by term gives the series definition \eqref{eq:limkernel2SS}.

For part~(b), the process 
$( e^{\eta\B(t)-\frac{1}{2}\eta^2t} )_{t\ge0}$ 
is a martingale. 
Since $\btau\wedge1$ is bounded, optional stopping gives
\begin{equation*}
	\bbE_{\B(0)=s} \left[ e^{\eta \B(\btau)-\frac{1}{2}\eta^2\btau} \mathbf{1}_{\btau<1} \right]
	=
	\bbE_{\B(0)=s} \left[ e^{\eta \B(1)-\frac{1}{2}\eta^2} \mathbf{1}_{\btau<1} \right] 
	=
	e^{-\eta^2/2} \int_{\bbR} e^{\eta \sfy}\sfT_\h(s,\sfy)\,\rd \sfy,
\end{equation*}
where the last equality follows from the definition \eqref{eq:limkernel2TT} of $\sfT_\h$.
\end{proof}

Define
\begin{equation}\label{eq:Xchh-def}
	\pchlimline_\h(\eta,\xi)
	:=
	\int_\bbR \re^{-s\xi}
	\bbE_{\B(0)=s}
	\left[
	\re^{\eta\B(\btau)-\frac12\eta^2\btau}
	\mathbf{1}_{\btau<1}
	\right]\rd s,
\end{equation}
which is the first term in \eqref{eq:pch_lim}. By \eqref{eq:T_OST},
\begin{equation}\label{eq:Xchh-Fourier}
	\pchlimline_\h(\eta,\xi)
	=
	\re^{-\eta^2/2}
	\int_\bbR\int_\bbR
	\re^{-s\xi+\eta\sfy}\sfT_\h(s,\sfy)
	\,\rd s\,\rd\sfy.
\end{equation}
Thus $\re^{\eta^2/2}\pchlimline_\h(\eta,\xi)$ is the two-sided Fourier--Laplace transform of $\sfT_\h$.

\subsection{Space-reversal formulas}
\label{sec:reversal}

By the space-reversal invariance of Brownian motion, both the periodic characteristic function and the energy function admit alternative expressions; see Propositions~\ref{prop:spacereversal} and~\ref{prop:spacereversal2} below. These identities are reminiscent of particle--hole duality for (P)TASEP.

For $\h\in\uc_1$, define
\begin{equation*}
	\btau_+:=\inf\{\sfx\ge0:\B(\sfx)\le\h(\sfx)\},
	\qquad
	\bntau_+:=\inf\{\sfx\ge1:\B(\sfx)\le\h(\sfx)\}.
\end{equation*}
Define the unit-time hitting kernels
\begin{equation*}
	\widehat{\sfT}_{\h}(\sfx,\sfy)
	:=
	\frac{\prob_{\B(0)=\sfx}
	\bigl(\B(1)\in\rd\sfy,\ \btau\le1\bigr)}{\rd\sfy},
	\qquad
	\widehat{\sfT}_{\h,+}(\sfx,\sfy)
	:=
	\frac{\prob_{\B(0)=\sfx}
	\bigl(\B(1)\in\rd\sfy,\ \btau_+\le1\bigr)}{\rd\sfy},
\end{equation*}
and the corresponding no-hit operators 
\begin{equation*}
	\bar{\sfT}_{\h}:=\mathsf G-\widehat{\sfT}_{\h},
	\qquad
	\bar{\sfT}_{\h,+}:=\mathsf G-\widehat{\sfT}_{\h,+}.
\end{equation*}

\begin{lm}\label{lm:spacereversal}
For every $\h\in\uc_1$,
\begin{equation}\label{eq:space_reversal} 
	\widehat{\sfT}_{\h,+}^{\top}
	=
	\widehat{\sfT}_{\h},
	\qquad
	\bar{\sfT}_{\h,+}^{\top}
	=
	\bar{\sfT}_{\h}.
\end{equation}
\end{lm}

\begin{proof}
Set $\widehat{\B}(t):=\B(1-t)$. By the one-periodicity of $\h$,
\begin{equation*}
	\{\B(t)\le\h(-t)\text{ for some }t\in[0,1]\}
	=
	\{\widehat{\B}(t)\le\h(t)\text{ for some }t\in[0,1]\}.
\end{equation*}
The time-reversal invariance of a Brownian bridge therefore gives the first identity in \eqref{eq:space_reversal}. Since $\mathsf G^\top=\mathsf G$, the second identity immediately follows.
\end{proof}

Recall the periodic characteristic function $\pchlim_\h(\eta,\xi)$ defined in \eqref{eq:pch_lim}.

\begin{prop}\label{prop:spacereversal}
For $\h\in\uc_1$, define
\begin{equation*}
	\pchlim_\h^+(\eta,\xi)
	:=
	\int_\bbR \re^{s\eta}
	\bbE_{\B(0)=s}
	\bigg[
	\re^{-\xi\B(\btau_+)-\frac12\xi^2\btau_+}
	\mathbf{1}_{\btau_+<1}
	-
	\re^{-\xi\B(\bntau_+)-\frac12\xi^2\bntau_+}
	\mathbf{1}_{\btau_+<1,\,\bntau_+<\infty}
	\bigg]\rd s
\end{equation*}
for $\arg(\xi)\in\left(\frac{3\pi}{4},\frac{5\pi}{4}\right)$ and $\Re(\eta)>0$. 
Then, for $0<|\rz|<1$,
\begin{equation*}
	\pchlim_\h(\eta,\xi)=\pchlim_\h^+(\eta,\xi),
	\qquad
	\eta\in\rR_\rz,\quad \xi\in\rL_\rz.
\end{equation*}
\end{prop}

\begin{proof}
Fix $\eta\in\rR_\rz$ and $\xi\in\rL_\rz$, so that $\rz=\re^{-\eta^2/2}=\re^{-\xi^2/2}$. 
On the event $\{\btau=1\}$, we have $\bntau=1$, so the two terms in the definition of $\pchlim_\h$ cancel. 
Thus, we may replace $\mathbf{1}_{\btau<1}$ by $\mathbf{1}_{\btau\le1}$, and similarly in the definition of $\pchlim_\h^+$.

Define
\begin{equation*}
	F_\h^\eta
	:=
	\sum_{k\ge0}
	\rz^{k+1}\bar{\sfT}_{\h}^{k}
	\widehat{\sfT}_{\h}u_\eta,
	\qquad
	F_{\h,+}^{-\xi}
	:=
	\sum_{k\ge0}
	\rz^{k+1}\bar{\sfT}_{\h,+}^{k}
	\widehat{\sfT}_{\h,+}u_{-\xi}, 
	\qquad 
	u_\lambda(s):=\re^{\lambda s}.
\end{equation*}
The series converge absolutely by the same estimates used in the proof of Lemma~\ref{prop:plim_bound}. Decomposing according to the unit interval containing the first hit and applying optional sampling on the final interval gives
\begin{equation*}
	F_\h^\eta(s)
	=
	\bbE_{\B(0)=s}
	\left[
	\re^{\eta\B(\btau)-\frac12\eta^2\btau}
	\mathbf{1}_{\btau<\infty}
	\right],
	\qquad 
	F_{\h,+}^{-\xi}(s)
	=
	\bbE_{\B(0)=s}
	\left[
	\re^{-\xi\B(\btau_+)-\frac12\xi^2\btau_+}
	\mathbf{1}_{\btau_+<\infty}
	\right].
\end{equation*}
The strong Markov property at time $1$ now yields
\begin{equation*} 
	\pchlim_\h(\eta,\xi)
	=
	\rz
	\langle
	u_{-\xi},
	\widehat{\sfT}_{\h}(u_\eta-F_\h^\eta)
	\rangle_{\bbR},
	\qquad 
	\pchlim_\h^+(\eta,\xi)
	=
	\rz
	\langle
	u_\eta,
	\widehat{\sfT}_{\h,+}
	(u_{-\xi}-F_{\h,+}^{-\xi})
	\rangle_{\bbR}, 
\end{equation*}
where we write the bilinear pairing $\langle f,g\rangle_{\bbR} =
\int_\bbR f(s)g(s)\,\rd s$. 
Lemma~\ref{lm:spacereversal} implies that for $\arg(\xi)\in\left(\frac{3\pi}{4},\frac{5\pi}{4}\right)$ and $\arg(\eta)\in\left(-\frac{\pi}{4},\frac{\pi}{4}\right)$, we have
\begin{equation*}
	\langle u_{-\xi},\widehat{\sfT}_{\h}u_\eta\rangle_{\bbR}
	=
	\langle u_\eta,\widehat{\sfT}_{\h,+}u_{-\xi}\rangle_{\bbR},
\end{equation*}
and, for every $k\ge0$,
\begin{equation*}
	\langle
	u_{-\xi},
	\widehat{\sfT}_{\h}\bar{\sfT}_{\h}^{k}
	\widehat{\sfT}_{\h}u_\eta
	\rangle_{\bbR}
	=
	\langle
	u_\eta,
	\widehat{\sfT}_{\h,+}\bar{\sfT}_{\h,+}^{k}
	\widehat{\sfT}_{\h,+}u_{-\xi}
	\rangle_{\bbR}.
\end{equation*}
Thus, the result follows. 
\end{proof}

A similar space-reversed representation holds for the energy function. It follows by first rewriting the operators in a form analogous to \eqref{eq:IRKEf}, as in the proof of Proposition~\ref{result:KasQ}, and then applying the space-reversal identity \eqref{eq:space_reversal}. We omit the details.

\begin{prop}\label{prop:spacereversal2}
Let $\h\in\uc_1^0$ and $|\rz|<1$. Define
\begin{equation}\label{eq:limkernel_plus} 
	\limKe_{\h,+}(\rz)
	:=
	\rz\,\mathbf{1}_{>0}
	(\rI-\rz\mathsf G)^{-1}
	\sfT_{\h,+}\mathbf{1}_{>0},
	\qquad 
	\sfT_{\h,+}(\sfx,\sfy)
	:=
	\frac{
	\prob_{\B(0)=\sfx}
	\bigl(\B(1)\in\rd\sfy,\ \btau_+<1\bigr)
	}{\rd\sfy}.
\end{equation}
Then $\limKe_{\h,+}(\rz)$ is trace class on $L^2(\bbR)$ and
\begin{equation*}
	\det(\rI+\limKe_{\h,+}(\rz))_{L^2(\bbR)}
	=
	\det(\rI+\limKe_{\h}(\rz))_{L^2(\bbR)}.
\end{equation*}
\end{prop}



\begin{thebibliography}{99}

\bibitem{Aggarwal-Corwin-Hegde24}
A.~Aggarwal, I.~Corwin, and M.~Hegde.
\newblock Scaling limit of the colored {ASEP} and stochastic six-vertex models.
\newblock {\em arXiv:2403.01341}, 2024.

\bibitem{Aggarwal-Corwin-Schmid26}
A.~Aggarwal, I.~Corwin, and D.~Schmid.
\newblock Periodic Directed Landscape.
\newblock {\em arXiv:2607.18104}, 2026.

\bibitem{Amir-Corwin-Quastel11}
G.~Amir, I.~Corwin, and J.~Quastel.
\newblock Probability distribution of the free energy of the continuum directed random polymer in {$1+1$} dimensions.
\newblock {\em Comm. Pure Appl. Math.}, 64(4):466--537, 2011.

\bibitem{Baik-Deift-Johansson99}
J.~Baik, P.~Deift, and K.~Johansson.
\newblock On the distribution of the length of the longest increasing subsequence of random permutations.
\newblock {\em J. Amer. Math. Soc.}, 12(4):1119--1178, 1999.

\bibitem{Baik-Liu16}
J.~Baik and Z.~Liu.
\newblock {TASEP on a ring in sub-relaxation time scale}.
\newblock {\em J. Stat. Phys.}, 165(6):1051--1085, 2016.

\bibitem{Baik-Liu18}
J.~Baik and Z.~Liu.
\newblock Fluctuations of {TASEP} on a ring in relaxation time scale.
\newblock {\em Comm. Pure Appl. Math.}, 71(4):747--813, 2018.

\bibitem{Baik-Liu19}
J.~Baik and Z.~Liu.
\newblock Multipoint distribution of periodic {TASEP}.
\newblock {\em J. Amer. Math. Soc.}, 32(3):609--674, 2019.

\bibitem{Baik-Liu21}
J.~Baik and Z.~Liu.
\newblock Periodic {TASEP} with general initial conditions.
\newblock {\em Probab. Theory Related Fields}, 179(3-4):1047--1144, 2021.


\bibitem{Baik-Liu2024}
J.~{Baik} and Z.~{Liu}.
\newblock {Pinched-up periodic KPZ fixed point}.
\newblock {\em Ann. Inst. Henri Poincar\'{e} Probab. Stat.}, to appear.

\bibitem{Baik-Liu-Silva22}
J.~Baik, Z.~Liu, and G.~L.~F.~Silva.
\newblock Limiting one-point distribution of periodic {TASEP}.
\newblock {\em Ann. Inst. Henri Poincar\'{e} Probab. Stat.}, 58(1):248--302, 2022.

\bibitem{Baik-Rains01}
J.~Baik and E.~M.~Rains.
\newblock Symmetrized random permutations.
\newblock In {\em Random matrix models and their applications}, volume 40 of {\em Math. Sci. Res. Inst. Publ.}, pages 1--19. Cambridge Univ. Press, Cambridge, 2001.


\bibitem{Basor97}
E.~L.~Basor.
\newblock Distribution functions for random variables for ensembles of positive
  Hermitian matrices.
\newblock {\em Comm. Math. Phys.}, 188(2):327--350, 1997.

\bibitem{BasorEhrhardt03}
E.~L.~Basor and T.~Ehrhardt.
\newblock Asymptotics of determinants of Bessel operators.
\newblock {\em Comm. Math. Phys.}, 234:491--516, 2003.

\bibitem{BLSZ23}
E.~Bisi, Y.~Liao, A.~Saenz, and N.~Zygouras.
\newblock Non-intersecting path constructions for TASEP with inhomogeneous rates and the KPZ fixed point.
\newblock {\em Commun. Math. Phys.}, 402:285--333, 2023.

\bibitem{Borodin17}
A.~Borodin.
\newblock On a family of symmetric rational functions.
\newblock {\em Adv. Math.}, 306:973 -- 1018, 2017.

\bibitem{Bottcher-Silbermann06}
A.~B\"ottcher and B.~Silbermann.
\newblock {\em Analysis of Toeplitz operators}, second edition. 
\newblock Springer Monographs in Mathematics, Springer, Berlin, 2006.

\bibitem{Brankov-Papoyan-Poghosyan-Priezzhev06}
J.~G.~Brankov, V.~V.~Papoyan, V.~S.~Poghosyan, and V.~B.~Priezzhev.
\newblock The totally asymmetric exclusion process on a ring: Exact relaxation dynamics and associated model of clustering transition.
\newblock {\em Phys. A}, 368(2):471--480, 2006.

\bibitem{Corwin-Gu-Sorensen26}
I.~Corwin, Y.~Gu, and E.~Sorensen.
\newblock Periodic {P}itman {T}ransforms and {J}ointly {I}nvariant {M}easures.
\newblock {\em Comm. Math. Phys.}, 407(3):Paper No. 49, 2026.

\bibitem{Corwin-OConnel-Seppalainen-Zygouras14}
I.~Corwin, N.~O'Connell, T.~Sepp\"al\"ainen, and N.~Zygouras.
\newblock Tropical combinatorics and {W}hittaker functions.
\newblock {\em Duke Math. J.}, 163(3):513--563, 2014.

\bibitem{Dauvergne-Ortmann-Virag22}
D.~Dauvergne, J.~Ortmann, and B.~Vir{\'a}g.
\newblock {The directed landscape}.
\newblock {\em Acta Mathematica}, 229(2):201 -- 285, 2022.

\bibitem{Dauvergne-Zhang24}
D.~Dauvergne and L.~Zhang.
\newblock Characterization of the directed landscape from the {KPZ} fixed point.
\newblock {\em arXiv:2412.13032}, 2024.

\bibitem{Derrida-Lebowitz98}
B.~Derrida and J.~L.~Lebowitz.
\newblock Exact large deviation function in the asymmetric exclusion process.
\newblock {\em Phys. Rev. Lett.}, 80(2):209--213, 1998.

\bibitem{Dunlap-Gu-Komorowski23}
A.~Dunlap, Y.~Gu, and T.~Komorowski.
\newblock Fluctuation exponents of the {KPZ} equation on a large torus.
\newblock {\em Comm. Pure Appl. Math.}, 76(11):3104--3149, 2023.

\bibitem{Ferrari04}
P.~L.~Ferrari.
\newblock Polynuclear growth on a flat substrate and edge scaling of {GOE} eigenvalues.
\newblock {\em Comm. Math. Phys.}, 252(1-3):77--109, 2004.

\bibitem{Golinelli-Mallick04}
O.~Golinelli and K.~Mallick.
\newblock Bethe ansatz calculation of the spectral gap of the asymmetric exclusion process.
\newblock {\em J. Phys. A}, 37(10):3321--3331, 2004.

\bibitem{Golinelli-Mallick05}
O.~Golinelli and K.~Mallick.
\newblock Spectral gap of the totally asymmetric exclusion process at arbitrary filling.
\newblock {\em J. Phys. A}, 38(7):1419--1425, 2005.

\bibitem{Gwa-Spohn92}
L.-H.~Gwa and H.~Spohn.
\newblock Six-vertex model, roughened surfaces, and an asymmetric spin {H}amiltonian.
\newblock {\em Phys. Rev. Lett.}, 68(6):725--728, 1992.

\bibitem{Johansson00}
K.~Johansson.
\newblock Shape fluctuations and random matrices.
\newblock {\em Comm. Math. Phys.}, 209(2):437--476, 2000.

\bibitem{Kardar-Parisi-Zhang86}
M.~Kardar, G.~Parisi, and Y.-C.~Zhang.
\newblock Dynamic scaling of growing interfaces.
\newblock {\em Phys. Rev. Lett.}, 56:889--892, Mar 1986.

\bibitem{Li-Saenz25}
J.-H.~Li and A.~Saenz.
\newblock Contour integral formulas for {P}ush{ASEP} on the ring.
\newblock {\em Ann. Probab.}, 53(4):1434--1490, 2025.

\bibitem{Liao22}
Y.~Liao.
\newblock Multi-point distribution of discrete time periodic {TASEP}.
\newblock {\em Probab. Theory Related Fields}, 182(3-4):1053--1131, 2022.

\bibitem{Liao-Liu25}
Y.~Liao and Z.~Liu.
\newblock Multipoint distributions of the {KPZ} fixed point with compactly supported initial conditions.
\newblock {\em arXiv:2509.03246}, 2025.

\bibitem{Liu16}
Z.~Liu.
\newblock Height fluctuations of stationary {TASEP} on a ring in relaxation time scale.
\newblock {\em Ann. Inst. Henri Poincar\'{e} Probab. Stat.}, 54(2):1031--1057, 2018.

\bibitem{Liu22}
Z.~Liu.
\newblock Multipoint distribution of {TASEP}.
\newblock {\em Ann. Probab.}, 50(4):1255--1321, 2022.

\bibitem{Liu-Saenz-Wang20}
Z.~Liu, A.~Saenz, and D.~Wang.
\newblock Integral formulas of {ASEP} and {$q$}-{TAZRP} on a ring.
\newblock {\em Comm. Math. Phys.}, 379(1):261--325, 2020.

\bibitem{Liu-Tejaswi26}
Z.~Liu and T.~Tripathi.
\newblock A determinant identity for the sum of contour integral matrices.
\newblock {\em arXiv:2604.24747}, 2026.

\bibitem{Matetski-Quastel-Remenik21}
K.~Matetski, J.~Quastel, and D.~Remenik.
\newblock The {KPZ} fixed point.
\newblock {\em Acta Math.}, 227(1):115--203, 2021.

\bibitem{Matetski-Quastel-Remenik25}
K.~Matetski, J.~Quastel, and D.~Remenik.
\newblock Polynuclear growth and the {T}oda lattice.
\newblock {\em J. Eur. Math. Soc. (JEMS)}, 2025.

\bibitem{Matetski-Remenik23a}
K.~Matetski and D.~Remenik.
\newblock TASEP and generalizations: method for exact solution.
\newblock {\em Probab. Theory Related Fields}, 185:615--698, 2023.

\bibitem{Matetski-Remenik23}
K.~Matetski and D.~Remenik.
\newblock Exact solution of {TASEP} and variants with inhomogeneous speeds and memory lengths.
\newblock {\em Electron. J. Probab.}, 30:Paper No. 167, 1--54, 2025.

\bibitem{Motegi-Sakai13}
K.~Motegi and K.~Sakai.
\newblock Vertex models, {TASEP} and {G}rothendieck polynomials.
\newblock {\em J. Phys. A}, 46(35):355201, 2013.

\bibitem{Nica-Quastel-Remenik20}
M.~Nica, J.~Quastel, and D.~Remenik.
\newblock One-sided reflected {B}rownian motions and the {KPZ} fixed point.
\newblock {\em Forum Math. Sigma}, 8:Paper No. e63, 16, 2020.

\bibitem{Petrov75}
V.~V.~Petrov.
\newblock {\em Sums of independent random variables}.
\newblock Springer-Verlag, New York, 1975.
\newblock Translated from the Russian by A. A. Brown, Ergebnisse der Mathematik und ihrer Grenzgebiete, Band 82.

\bibitem{Petrov26}
L.~Petrov.
\newblock A {B}orodin-{O}kounkov-{G}eronimo-{C}ase identity for tilted {T}oeplitz minors.
\newblock {\em arXiv:2605.24976}, 2026.

\bibitem{Priezzhev2003}
V.~B.~Priezzhev.
\newblock Exact nonstationary probabilities in the asymmetric exclusion process on a ring.
\newblock {\em Phys. Rev. Lett.}, 91(5):050601, 2003.

\bibitem{Prolhac16}
S.~Prolhac.
\newblock Finite-time fluctuations for the totally asymmetric exclusion process.
\newblock {\em Phys. Rev. Lett.}, 116:090601, 2016.

\bibitem{Prolhac20}
S.~Prolhac.
\newblock {Riemann surfaces for KPZ with periodic boundaries}.
\newblock {\em SciPost Phys.}, 8:008, 2020.

\bibitem{Rost80}
H.~Rost.
\newblock Nonequilibrium behaviour of a many particle process: density profile and local equilibria.
\newblock {\em Z. Wahrsch. Verw. Gebiete}, 58:41--53, 1980.

\bibitem{Schmid-Sly26}
D.~Schmid and A.~Sly.
\newblock Mixing times for the {TASEP} on the circle.
\newblock {\em Probab. Theory Related Fields}, 194:1161--1233, 2026.


\bibitem{Seppalainen12}
T.~Sepp\"al\"ainen.
\newblock Scaling for a one-dimensional directed polymer with boundary conditions.
\newblock {\em Ann. Probab.}, 40(1):19--73, 2012.

\bibitem{Virag20}
B.~Vir\'ag.
\newblock The heat and the landscape {I}.
\newblock {\em arXiv:2008.07241}, 2020.

\bibitem{Wu23}
X.~Wu.
\newblock The {KPZ} equation and the directed landscape.
\newblock {\em Ann. Probab.}, 54(3):1212--1257, 2026.


\end{thebibliography}
\end{document}